%% file: S_machine_arXiv.tex
\documentclass[11pt,letterpaper,reqno]{amsart}
\usepackage{scrextend}
\usepackage{amsmath}
\usepackage{amsbsy}
\usepackage{amssymb} 
\usepackage{paralist}

\DeclareFontFamily{U}{mathb}{\hyphenchar\font45}
\DeclareFontShape{U}{mathb}{m}{n}{
      <5> <6> <7> <8> <9> <10> gen * mathb
      <10.95> mathb10 <12> <14.4> <17.28> <20.74> <24.88> mathb12
      }{}
\DeclareSymbolFont{mathb}{U}{mathb}{m}{n}

\DeclareMathSymbol{\precneq}{3}{mathb}{"AC}

\usepackage{mathtools}

\usepackage{graphicx}
\usepackage{tikz}
\tikzset{
  font={\fontsize{10pt}{12}\selectfont}}
  \usetikzlibrary{arrows.meta,arrows}
  \tikzset{>=latex}
\usepackage{caption, cleveref} 
\usepackage{xfrac}  
\usetikzlibrary{arrows, fit, backgrounds, patterns, positioning, calc, intersections, decorations.markings, decorations.pathmorphing, decorations.pathreplacing, shapes.misc}
\usetikzlibrary{shapes.multipart}
\usepackage[utf8]{inputenc}
\usepackage{multicol}
\usepackage{subcaption}
\usepackage{enumitem}

    \advance \topmargin by -\headheight
    \advance \topmargin by -\headsep

    \evensidemargin \oddsidemargin
\usepackage{scalerel,stackengine}
\stackMath
\newcommand\reallywidehat[1]{%
\savestack{\tmpbox}{\stretchto{%
  \scaleto{%
    \scalerel*[\widthof{\ensuremath{#1}}]{\kern-.6pt\bigwedge\kern-.6pt}%
    {\rule[-\textheight/2]{1ex}{\textheight}}
  }{\textheight}%
}{0.5ex}}%
\stackon[1pt]{#1}{\tmpbox}%
}

\newtheorem{theorem}{Theorem}
\newtheorem{corollary}[theorem]{Corollary}
\newtheorem{proposition}[theorem]{Proposition}
\newtheorem{lemma}[theorem]{Lemma}

\theoremstyle{definition}
\newtheorem{conjecture}[theorem]{Conjecture}
\newtheorem{definition}[theorem]{Definition}
\newtheorem{example}[theorem]{Example}

\numberwithin{theorem}{section}

\numberwithin{equation}{section}

\newcommand{\E}{\mathcal{E}}
\newcommand{\I}{\mathcal{I}}

\newcommand{\A}{\mathcal{A}}
\renewcommand{\S}{\mathcal{S}}
\renewcommand{\H}{\mathcal{H}}
\newcommand{\T}{\mathcal{T}}
\newcommand{\TT}{\mathrm{TT}}
\newcommand{\G}{\mathcal{G}}
\newcommand{\Q}{\mathcal{Q}}
\newcommand{\N}{\mathbb{N}}

\newcommand{\Z}{\mathbb{Z}}
\newcommand{\F}{\mathcal{F}}
\renewcommand{\i}{\mathbf{i}}
\newcommand{\nan}{\text{NA}}
\newcommand{\eps}{\varepsilon}
\newcommand{\MGI}{\text{MGI}}
\newcommand{\gen}[1]{\langle #1 \rangle}

\DeclareMathOperator{\size}{size}

\newcommand{\TM}{\mathsf{TM}}
\newcommand{\TMSP}{\mathsf{TMSP}}

\newcommand{\TMSPT}{\widetilde{\mathsf{TMSP}}}
\newcommand{\WTMSP}{\mathsf{WTMSP}}
\newcommand{\SP}{\mathsf{SP}}
\newcommand{\AR}{\mathsf{AR}}
\newcommand{\tm}{\mathrm{tm}}
\renewcommand{\sp}{\mathrm{sp}}

\newcommand{\sph}{\widehat{\mathrm{sp}}}
\newcommand{\tmt}{\widetilde{\mathrm{tm}}}
\newcommand{\spt}{\widetilde{\mathrm{sp}}}
\newcommand{\ar}{\mathrm{ar}}
\newcommand{\REACH}{\mathsf{REACH}}

\newcommand{\IO}{\mathsf{IO}}
\newcommand{\CIO}{\mathsf{CIO}}
\newcommand{\VAL}{\mathsf{VAL}}

\newcommand{\FRAG}{\mathsf{FRAG}}
\newcommand{\ACT}{\mathsf{ACT}}

\newcommand{\SGRAPH}{\mathsf{SGRAPH}}
\newcommand{\OBJ}{\mathsf{OBJ}}
\newcommand{\OP}{\mathrm{OP}}
\newcommand{\act}{\mathrm{act}}
\newcommand{\inst}{\mathrm{inst}}
\renewcommand{\exp}{\mathrm{exp}}
\newcommand{\INST}{\mathsf{INST}}
\newcommand{\divtwo}{\mathrm{div2}}

\newcommand{\acc}{\mathrm{acc}}
\newcommand{\accept}{\mathrm{accept}}
\newcommand{\wacc}{\mathrm{wacc}}
\newcommand{\waccept}{\mathrm{waccept}}
\newcommand{\inc}{\mathrm{inc}}
\newcommand{\set}{\mathrm{set}}

\newcommand{\rea}{\mathrm{re}}
\renewcommand{\deg}{\mathrm{deg}}
\newcommand{\swap}{\mathrm{swap}}
\newcommand{\TTape}{\mathrm{TTape}}
\newcommand{\splt}{\mathrm{split}}
\newcommand{\op}{\mathrm{op}}
\newcommand{\Cntr}{\mathcal{C}}
\newcommand{\C}{\mathcal{C}}
\newcommand{\cpy}{\mathrm{copy}}

\newcommand{\pr}[1]{\bigl|_{#1}}
\newcommand{\SF}{\mathrm{SF}}
\newcommand{\GL}{\mathrm{GL}}
\newcommand{\ND}{\mathrm{ND}}
\renewcommand{\sf}{\mathrm{sf}}
\newcommand{\val}{\mathrm{val}}
\renewcommand{\int}{\mathrm{int}}
\newcommand{\ext}{\mathrm{ext}}
\newcommand{\inv}{\mathrm{inv}}
\newcommand{\lab}{\text{Lab}}
\newcommand{\Counter}{\text{Counter}}
\newcommand{\syn}{\text{syn}}
\newcommand{\defeq}{:=}
\newcommand{\SEQ}{\mathrm{SEQ}}
\newcommand{\IRT}{\mathrm{IRT}}

\newcommand{\Obj}{\mathrm{Obj}}
\newcommand{\TTL}{\mathrm{TTL}}
\newcommand{\TTR}{\mathrm{TTR}}

\newcommand{\EXP}{\mathsf{EXP}}

\newcommand{\RST}{\mathsf{RST}}
\newcommand{\dt}{\mathrm{div2}}
\def\b{\beta}

\renewcommand{\check}{\mathrm{check}}
\newcommand{\ch}{\mathrm{ch}}

\renewcommand{\le}{\mathrm{less}}

\newcommand{\cntr}{{\mathrm{cntr}}}
\newcommand{\scc}{{\mathrm{succ}}}
\def\a{\alpha}
\def\Area{\mathrm{Area}}

\newcommand{\Id}{\mathbb{I}}
\newcommand{\id}{\mathrm{id}}

\title{Quasilinear emulation of Turing machines by $S$-machines}
\author{Bogdan Chornomaz}
\author{Francis Wagner}
\begin{document}

\maketitle

\begin{center}

\emph{Dedicated to the memory of Mark V. Sapir}

\end{center}

\begin{abstract}
	We prove that for any $\varepsilon>0$, a non-deterministic Turing machine $\T$ with time complexity $T(n)$ can be emulated by an $S$-machine with time and space complexities at most $T(n)^{1+\varepsilon}$ and $T(n)$, respectively. This improves the bounds on the emulation in \cite{SBR} and leads to improved bounds in the main theorem of \cite{BORS}. In particular, for a non-hyperbolic finitely generated group $G$ whose word problem has linear time complexity, this yields an embedding of $G$ into a finitely presented group $H$ such that $G$ has bounded distortion in $H$ and the Dehn function of $G$ in $H$ is bounded above by $n^{2+\varepsilon}$, an optimal bound modulo the $\eps$ factor. As a means to this end, we introduce and develop the theory of $S$-graphs, giving a different perspective on the construction of $S$-machines akin to a crude object-oriented programming language.
\end{abstract}

\section{Introduction} \label{sec-intro}
First introduced in \cite{SBR}, an \emph{S-machine} is an intricate computational model that bears resemblance to that of a multi-tape Turing machine.  However, there are several fundamental differences between these two models.  For just one example, an $S$-machine works with group words while a Turing machine works over free monoids, with no inverse letters or reductions present.  Despite such differences, though, a correspondence is established in the fundamental result of \cite{SBR}: For every Turing machine $\T$, there exists an $S$-machine that `emulates' $\T$ in a particular sense.

Perhaps an even more powerful observation, though, is the association of $S$-machines and multiple HNN-extensions of free groups.  Such groups, first introduced in 1949 by their namesakes Higman, B.H Neumann, and H. Neumann \cite{HNN}, are constructed in a manner allowing for methodical and straightforward study, yet form a class rich enough to provide ample interesting examples.  As such, $S$-machines have proved to be a remarkable tool for the study of groups, answering many open questions and conjectures.  The following chronologically lists just a few examples of such achievements realized by means of this tool:

\begin{itemize}

\item The first version of the paper of Sapir, Birget, and Rips~\cite{SBR} introducing $S$-machines appeared in 1996. In it, the authors proved that for any $\alpha \geq 4$ that can be computed `reasonably fast', there is a finitely presented group with Dehn function equivalent to $n^\alpha$. 
\item In 1998, Birget, Ol'shanskii, Rips, and Sapir \cite{BORS} show that the word problem of a finitely generated group is NP solvable if and only if it can be embedded into a finitely presented group with polynomial Dehn function.

\item In 2001, Ol'shanskii and Sapir \cite{OS01} constructed an example of a non-amenable finitely presented group which is the extension of a finitely generated torsion group by a cyclic group.

\item In 2002, Ol'shanskii and Sapir \cite{OS02} constructed an analogue of the Higman embedding theorem which preserves the decidability (up to Turing degree) of the Conjugacy Problem.

\item In 2018, Ol'shanskii \cite{O18} gave a description of finitely presented groups with Dehn function equivalent to $n^\a$ for $\a\in(2,4)$ which, modulo the $\mathbf{P}=\mathbf{NP}$ conjecture, completes description of the isoperimetric spectrum \cite{Sapirnote}.

\item In 2019, Ol'shanskii \cite{O19} produced an analogue of the Higman embedding theorem when restricted to the Burnside variety of groups with sufficiently large exponent.

\item In 2020, Ol'shanskii and Sapir \cite{OS20} constructed a finitely presented group with quadratic Dehn function but unsolvable conjugacy problem.

\item In 2020, the second author \cite{W} produced examples of finitely presented groups with quadratic Dehn function which contain finitely generated torsion subgroups.  Indeed, it is shown that any free Burnside group of sufficiently large exponent can be quasi-isometrically embedded into a quadratic Dehn function group, an optimal such ambient group \cite{GdlH}.

\end{itemize}

In this manuscript, an alternative but equivalent construction is introduced called \textit{$S$-graphs}.  The reason for their introduction is twofold.

First, the complex technical nature of the definition of $S$-machine makes approaching them a formidable task.  While some attempts have been made to provide aids to advanced readers for understanding the makeup of particular machines (e.g see Figure 13 in \cite{OS01}), no existing tool seems sufficient for a general audience and arbitrary machine.  It is the authors' hope that $S$-graphs fill that role, providing a pictorial aid that fully describes a `machine' while also giving a simple way to interpret the computational structure.

Even more vitally, though, is that the pictorial nature of the definition facilitates a natural method of `composing' the machinery.  In and of itself, this is not an innovation, as it has been a common practice in previous manuscripts (e.g \cite{O18}, \cite{O19}, \cite{OS20}, \cite{W}) to combine and compose $S$-machines in various ways to produce desirable computational structures.  However, the established machinery has invariably necessitated elaborate definitions of such machines, obfuscating their makeup and purpose.  While the proof that the `composition' of $S$-graphs is a `well-defined' operation is itself quite convoluted (see \Cref{th-expansion-props}), the result is a simple coherent description of the process which allows for effective iterative construction of complex structures, something that would have been very difficult to describe in previous settings.   This is the principal innovation of the manuscript.

In particular, using the framework of $S$-graphs, the main result of this manuscript is a substantial improvement to the efficiency of an emulation of a Turing machine by an S-machine, the backbone of \cite{SBR}.  To state this effectively, some basic terminology is first summarized.  For full definitions of these terms, see \Cref{sec-intro} and \Cref{sec-complexity}.

Roughly speaking, an $S$-machine $\S$ is a special sort of (non-deterministic) finite rewriting system for group words.  The machine has various `states' that it can be in and various `tapes' on which group words over specified alphabets are written.  A `rule' of $\S$ looks at what state the machine is in and what words are written on the tapes (together, this information is called a `configuration' of $\S$) and makes sure they are `suitable', then switches the state and tape words in a prescribed way.  A `computation' of $\S$ is a finite sequence of rules applied consecutively.  The \emph{time} of a computation is the number of rules applied, while the \emph{space} is the maximal `size' of the configurations it passes through along the way.

Typically, $\S$ will be assumed to be \emph{recognizing}, i.e it has two distinguished states called the \emph{start} and \emph{end} states and a distinguished tape called the \emph{input} tape.  Given a group word $w$ over the alphabet corresponding to the input tape, the \emph{input configuration} $I(w)$ is the configuration where $\S$ is in the start state, the input tape has $w$ written on it, and all other tapes are `empty' (i.e have the trivial word written on it).  Similarly, the \emph{accept configuration} $W_0$ is the configuration where $\S$ is in the end state and all tapes are empty.  An `accepting computation' of $w$ is a computation between $I(w)$ and $W_0$, while the \emph{language} of $\S$, $L_\S$, is the set of words $w$ for which there exists an accepting computation.

Given $w\in L_\S$, the \emph{time} of $w$, $\tm(w)$, is the minimal time of an accepting computation.  Similarly, the \emph{space} of $w$, $\sp(w)$, is the minimal space of such a computation.  The \emph{time complexity} of the recognizing $S$-machine $\S$ is then defined to be the function $\TM_\S:\N\to\N$ given by $\TM_\S(n)=\max\{\tm(w)\mid w\in L_\S, \ |w|\leq n\}$.  The \emph{space complexity} is defined analogously.

Given a pair of non-decreasing functions $f,g:\N\to\N$, the recognizing $S$-machine $\S$ is said to have \emph{time-space complexity} bounded above by $(f,g)$, denoted $\TMSP_\S\leq(f,g)$, if for any $w\in L_\S$ there exists an accepting computation of $w$ with time at most $f(|w|)$ and space at most $g(|w|)$.

All non-decreasing functions on the natural numbers investigated in these settings are considered up to an equivalence $\sim$ induced by the preorder $\preceq$ defined by $f\preceq g$ if and only if there exists $C\in\N$ such that $f(n)\leq Cg(Cn)+C$ for all $n\in\N$.  Naturally, this relation is extended to tuples of nondecreasing functions.  For example, if $f'\sim f$ and $g'\sim g$, then $\S$ is said to have \emph{time-space complexity asymptotically} bounded above by $(f',g')$, denoted $\TMSP_\S\preceq(f',g')$.  It will often prove to be natural to consider also the equivalence relation $\sim_1$ on such non-decreasing functions induced by the preorder $\preceq_1$ given by $f\preceq_1 g$ if and only if $\max(f,n)\preceq\max(g,n)$.

Much of the same terminology is used for (multi-tape, non-deterministic) Turing machines.  See \Cref{sec-turing} for further discussion, or see \cite{SBR} or \cite{O19} for full background.  In particular, the \emph{language} recognized by the Turing machine $M$, denoted $L_M$, is the set of words $w$ over the prescribed alphabet for which there exists a computation taking a particular configuration prescribed by $w$ to a distinguished `accept' configuration.  The \emph{time complexity} of $M$ is then defined in a similar manner as for $S$-machines.

With this terminology at hand, the main theorem of this manuscript can be stated as follows:

\begin{theorem}\label{main-theorem}

Let $M$ be a multi-tape, non-deterministic Turing machine.  Then, there exist $S$-machines which `emulate' $M$ in quasilinear time, i.e for all $\eps>0$ there exists an $S$-machine $\S_\eps$ such that $L_{\S_\eps}=L_M$ and $\TMSP_{\S_\eps}\preceq_1(\TM_M^{1+\eps},\TM_M)$.

\end{theorem}

This statement can be contrasted with Proposition 4.1 of \cite{SBR}, where condition (4) implies that the corresponding emulating $S$-machine $\S(M)$ satisfies $\TMSP_{\S(M)}\preceq_1(\TM_M^3,\TM_M)$; in fact, it implies equivalence in that $(\TM_M^3,\TM_M)$ is the least upper bound on $\TMSP_{\S(M)}$ with respect to $\preceq_1$.
%
One can see that \Cref{main-theorem} provides a substantial improvement to this, and so can be employed to yield refinements of the various group theoretical results for which it serves as a backbone.  Some of these refinements will be discussed in \Cref{sec-implications} below.

This article is roughly divided into two parts. The first part is comprised of Sections \ref{sec-intro}--\ref{sec-complexity} and is dedicated to developing a primitive object-oriented language for $S$-machines, which will then be used in the second part to construct a simulation algorithm that produces an $S$-machine satisfying \Cref{main-theorem}.  The goal here is thus twofold: The language should be expressive enough to enable a concise description of this simulation, while keeping at bay the intricacies of transforming this algorithm into an S-machine. While overall the definitions and proofs in this part are bulky and technical, they pursue a straightforward goal and conceptually should not pose a problem to anyone familiar with any high-level programming language. A more detailed description of the first part follows:

In \Cref{sec-intro}, the concept of $S$-graph is introduced, culminating with the development of an explicit correspondence between $S$-graphs and $S$-machines in \Cref{th-MG}. \Cref{sec-relations} serves as a general-purpose reminder of some notions of relational algebra, tools which will prove valuable in all future proofs involving $S$-graphs.  \Cref{sec-operations} is dedicated to the formulation of objects and operations, which are simple syntactical embellishments of $S$-graphs that satisfy particular properties. These tools are then used to develop the `composition' of $S$-graphs in \Cref{sec-sg-ext}, enabling the construction of $S^*$-graphs as `$S$-graphs of higher degrees'. This construction produces mechanisms that resemble a standard programming language (although without recursion). In particular, in~\Cref{sec-star-computations}, we define a computation of an $S^*$-graph.   $S^*$-graphs can be transformed back into `regular' $S$-graphs and, consequently, into $S$-machines with a process called `expansion', described in~\Cref{sec-expansions}. Finally, \Cref{sec-complexity} discusses the time and space complexities of $S^*$-graphs. In order to be in line with the complexities of the corresponding S-machine, the complexities of an $S^*$-graph are defined in terms of the graph's expansion; however, many tools are developed to allow for the complexities to be estimated directly, bypassing any explicit reference to the graph's expansion.

In the second part of the paper, comprised of Sections \ref{sec-acceptors}--\ref{sec-final-proof}, a particular $S^*$-graph is constructed so that the $S$-machine corresponding to its expansion satisfies \Cref{main-theorem}.  In many ways, this construction is similar to that constructed in \cite{SBR}: First, a `counter' is defined in \Cref{sec-counters}, which enables one to tell if an integer, somehow encoded in the machine's tapes, is nonzero (a task notoriously nontrivial for S-machines).  Then, in \Cref{sec-positivity-check}, an $S^*$-graph is constructed which `accepts' the language of all positive group words, i.e. all words whose freely reduced form only contains positive powers of its letters. Then, in \Cref{sec-turing-tapes}, we construct an object, emulating one tape of a Turing machine.  Finally, in \Cref{sec-final-proof}, the preceding graphs are used to emulate an arbitrary Turing machine with an $S^*$-graph, thus finishing the proof of the main theorem. At the same time, unlike in \cite{SBR}, each one of these steps is `improved' so that the overall time complexity is arbitrarily close to the time function of the original Turing machine.  Although the details of addressing the complexity are different for all steps, they each use the recursive definition of the corresponding objects, with each iteration improving the bounds.  This is made substantially simpler, both in terms of presentation and calculation, by the recursive nature of $S^*$-graphs, justifying the tedious but dull task \cite{palfy2003} of defining them. The remaining sections are used to define what `accepting a language' means in terms of $S^*$ graphs (\Cref{sec-acceptors}), and to give a definition of a Turing machine which is tailored in a way that can be paralleled with an S-graph, and also is suitable for the emulation (\Cref{sec-turing}).

The following terminology and notation are used throughout this manuscript: 

For sets $X$ and $Y$, $X - Y$ denotes the set difference of $X$ and $Y$, i.e $x\in X-Y$ if and only if $x\in X$ and $x\notin Y$.

For $n \in \N-\{0\}$, the notation $\overline{n}$ represents the subset $\{1,\dots,n\}$ of $\N$.  Further, for any $n,m\in\Z$ such that $n \leq m$, $\overline{n, m}$ is defined to be $\{n, \dots, m\}$. These definitions are extended to all $n,m\in\Z$ by setting $\overline{n} = \emptyset$ and $\overline{n,m} = \emptyset$ for all $n,m$ for which these have not already been defined. 

For an \emph{index set} $I$, an \emph{$I$-tuple} $g = (g_i)_{i\in I}$ is a function with domain $I$ such that $g(i) = g_i$, for $i\in I$. For $n\in \N-\{0\}$, an $\overline{n}$-tuple is simply referred to as an $n$-tuple. For disjoint index sets $I$ and $J$, given an $I$-tuple $g$ and a $J$-tuple $h$, $g\sqcup h$ is the $(I\sqcup J)$-tuple given by $(g\sqcup h)(i) = g(i)$ for $i\in I$ and $(g\sqcup h)(j)=h(j)$ for $j\in J$. For an index set $I$ and an $I$-tuple $(G_i)_{i\in I}$ of sets, $\bigtimes_{i\in I} G_i$ is the set of all $I$-tuples $g$ such that $g(i) \in G_i$. For disjoint index sets $I$ and $J$, given a set $G$ of $I$-tuples and a $J$-tuple $h$, $G \sqcup h = \{g\sqcup h~|~g\in G\}$; $g\sqcup H$ and $G\sqcup H$ are defined analogously. For an $I$-tuple $g$ and $J\subseteq I$, $g(J)$ denotes a $J$-tuple $g|_{J}$; similarly, for a set of $I$-tuples $G$ and $J\subseteq I$, $G(J) = \{g(J)~|~g\in G\}$.

\bigskip

\textbf{Acknowledgment.} The authors are extremely grateful to Alexander Ol'shanskii and the late Mark Sapir, who introduced them to this deep and powerful theory and whose work inspired this investigation.

\section{Implications of \Cref{main-theorem}} \label{sec-implications}
As noted above, the emulation of Turing machines by $S$-machines demonstrated in \cite{SBR} is a central tool in the proof of several deep group theoretical results.  It is natural, then, to question whether the improvement to the complexity bounds of this emulation achieved in \Cref{main-theorem} is in and of itself sufficient to refine any of the aforementioned results.  This section serves to discuss what can be addressed along these lines.  Note that throughout the rest of this section, all notions of equivalence and asymptotic bound refer to those induced by the preorder $\preceq_1$.

First, consider the following two statements proved by Sapir, Birget, and Rips in \cite{SBR}:

\begin{enumerate}

\item \emph{(Theorem 1.3 of \cite{SBR})}: Let $L\subseteq X^+$ be a language accepted by a Turing machine $M$ with a time function $T(n)$ for which $T(n)^4$ is superadditive. Then there exists a finitely presented group $G(M) = \langle A\rangle$ with Dehn function equivalent to $T(n)^4$, the smallest isodiametric function equivalent to $T(n)^3$, and there exists an injective map $K : X^+ \rightarrow \left(A \cup A^{-1}\right)^+$ such that:
 \begin{enumerate}
 \item $u\in L$ if and only if $K(u) = 1$ in $G$;
 \item $K(u)$ has length $O(|u|)$ and is computable in time $O(|u|)$.
 \end{enumerate}
 
 \medskip
 
 \item \emph{(Corollary 1.3 of \cite{SBR})}:  Let $f(n) > n^4$ be a superadditive function such that the binary representation of $f(n)$ is computable in time $O\left(\sqrt[4]{f(n)} \right)$ by a Turing machine. Then $f(n)$ is equivalent to the Dehn function of a finitely presented group and the smallest isodiametric function of this group is equivalent to $f^{3/4}(n)$.
 
 \end{enumerate}

%
%

\bigskip

On the surface, the functions $T(n)^3$ and $T(n)^4$ featured in (1) may seem to arise from the time-space complexity of the emulating machine of \cite{SBR}.  Indeed, this is the case: The $T(n)^3$ term corresponds to the time complexity of the $S$-machine, while the $T(n)^4$ term corresponds to the \emph{area complexity} of the machine which is bounded by the product of the functions comprising any time-space complexity bound.  However, while the corresponding complexity bounds are improved in \Cref{main-theorem}, the nondominant terms defining the asymptotics of the Dehn function impede an improvement of the same degree.  Still, the following statement represents a refinement of (1) achieved through \Cref{main-theorem}:


\begin{theorem}\label{th-SBR-1-3-improved}
Let $L\subseteq X^+$ be a language accepted by a Turing machine $M$ with a time function $T(n)$ for which $T(n)^2$ is superadditive. Then, for any $\varepsilon >0$, there exists a finitely presented group $G_\eps(M) = \langle A\rangle$ with Dehn function bounded above by $T(n)^{2+\varepsilon} + n^4$, isodiametric function bounded above by $T(n)^{1+\varepsilon} + n^2$, and there exists an injective map $K : X^+ \rightarrow \left(A \cup A^{-1}\right)^+$ such that:
 \begin{enumerate}
 \item $u\in L$ if and only if $K(u) = 1$ in $G$;
 \item $K(u)$ has length $O(|u|)$ and is computable in time $O(|u|)$.
 \end{enumerate}
\end{theorem}
\begin{proof}
    The group $G(M)$ is the group $G_N(\S)$, constructed in Section 5 in~\cite{SBR}, where $\S = \S(M)$ is the $S$-machine emulating $M$.  One can use the same construction of the group, changing only the $S$-machine to the machine $\S_\eps$ given in \Cref{main-theorem}, producing the group $G_\eps(M)=G(\S_\eps)$.
    
    By Proposition 11.1 of \cite{SBR}, for any minimal diagram $\Delta$ over $G(M)$, 
    $$\Area(\Delta) \leq C_4 T(C_6 |\partial\Delta|)^4 + (C_3+C_9) |\partial\Delta|^4$$
where $C_3$, $C_4$, $C_6$, and $C_9$ are sufficiently large constants.

    The $C_4 T(C_6 |\partial\Delta|)^4$ term relies only on the superadditivity of $T(n)^4$ and on the area complexity of $\S(M)$.  The requirement of the superadditivity of $T(n)^2$ necessitates that $T(n)^{2+\eps}$ is also superadditive, while the area complexity of $\S_\eps$ is bounded above by $T(n)^{2+\eps}$.  Hence, identical arguments immediately imply that a minimal diagram $\Delta$ over $G_\eps(M)$ satisfies:
    $$\Area(\Delta)\leq C_4T(C_6|\partial\Delta|)^{2+\eps}+(C_3+C_9)|\partial\Delta|^4$$
    

    Similarly, the diameter $d(\Delta)$ of a minimal diagram $\Delta$ over $G(M)$ is estimated in Proposition 11.1 of \cite{SBR} by
    $$d(\Delta) \leq C_1'T(C_2'|\partial\Delta|)^3 + C_3'|\partial\Delta|^2$$
    for sufficiently large constants $C_1',C_2',C_3'$.
An identical argument to that above then shows that for any minimal diagram $\Delta$ over $G_\eps(M)$, $$d(\Delta) \leq C_1'T(C_2'|\partial\Delta|)^{1+\varepsilon} + C_3'|\partial\Delta|^2$$
\end{proof}

Note that while (1) gives equivalences for the Dehn function and the isodiametric function of $G(M)$, \Cref{th-SBR-1-3-improved} merely gives upper bounds for the corresponding functions of $G_\eps(M)$. One reason for this discrepancy is that \Cref{main-theorem} gives only an upper bound on the complexity of the emulation, whereas the bound in (1) is substantially worse but exact. This problem can be overcome, but is not done here for several reasons.  First, the construction of the S-machines exhibiting the bound would get more complicated.  Second, an explicit sequence of values $\varepsilon_k$ converging to $0$ would need to be constructed to assure the lower bound can be achieved.  Most importantly, though, is that the the nondominant term $n^4$ in the upper bound means that the equivalence will not be in terms of the time function unless $T(n)\succeq n^2$.  With this same reasoning, it can be seen that any refinement to (2) given by \Cref{main-theorem} would be minimal and that the main obstruction $f(n)>n^4$ would still be present.

This discussion is particularly important given that a principal goal of \cite{SBR} was to address the problem of describing the isoperimetric spectrum of groups, i.e the set of numbers $\alpha$ for which $n^\alpha$ is the Dehn function of some finitely presented group.  (2) enabled the authors to do so for $\alpha\geq 4$, and the problematic case $\alpha\in (2,4)$ remained open for some twenty years until it was characterized by Ol'shanskii in 2018~\cite{O18}. Although the paper of Ol'shanskii uses the framework of S-machines heavily, it also, as noted by its author, required substantial new ideas. It is thus natural to wonder if, with the improved simulation, this can be achieved using essentially the same, or perhaps even \emph{the} same, construction as in \cite{SBR}.  As noted, the main obstruction to it is the $n^4$ term in (1), and it is possible that it can be improved just by a more accurate estimate of the areas of the same diagrams.

\medskip

Next, in their follow-up manuscript, Birget, Ol'shanskii, Rips, and Sapir \cite{BORS} use the emulation of \cite{SBR} to address a group theoretical question that naturally arises from it, producing an analogue of the celebrated Higman embedding theorem \cite{Higman} which relates the complexity of the word problem to the Dehn function of the finitely presented group in which it embeds.  Specifically, the main theorem of \cite{BORS} (Theorem 1.1) states that if $G$ is finitely generated and has a word problem that is decidable by a nondeterministic Turing machine with time function $\leq T(n)$ where $T(n)^4$ is superadditive, then there exists an embedding of $G$ into a finitely presented group $H$ with isoperimetric function $n^2 T(n^2)^4$ and in such a way that $G$ has bounded distortion in $H$.

The following refinement is immediate from the improved complexities of \Cref{main-theorem}:


\begin{theorem}\label{th-BORS-1-1-improved}
Let $G$ be a finitely generated group with word problem solvable by a nondeterministic Turing machine with time function $\leq T(n)$ such that $T(n)^2$ is superadditive. Then, for any $\varepsilon > 0$, $G$ can be embedded into a finitely presented group $H_\eps$ with the Dehn function bounded above by $n^2T(n^2)^{2+\varepsilon}$ in such a way that $G$ has bounded distortion in $H_\eps$.

Moreover, the Dehn function of $G$ in $H_\eps$ is at most $T(n)^{2+\eps}$.  
\end{theorem}

Note that here, the terminology of the Dehn function of a finitely generated subgroup in a finitely presented subgroup aligns with that used in \cite{SBR}.  That is, given a finitely generated subgroup $G=\gen{A}$ of a finitely presented group $H$ with finite presentation $\mathcal{P}$, the Dehn function of $G$ in $H$ is the function $\delta_{G,H}:\N\to\N$ given by
$$\delta_{G,H}(n)=\max\{\Area_\mathcal{P}(w)\mid w\in(A\cup A^{-1})^*, \ w=_G1, \ |w|_A\leq n\}$$
As with Dehn functions of finitely presented groups, the terminology is justified by the observation that the equivalence class of the Dehn function of $G$ in $H$ is invariant with respect to changing the finite generating set $A$ or the finite presentation $\mathcal{P}$.


Further, note that the portion of the statement regarding the Dehn function of $G$ in $H$ is not formulated in \cite{BORS}.  However, such a statement is an immediate corollary of their construction, where it would be equivalent to $T(n)^4$.

\begin{proof}

In Theorem 1.1 of \cite{BORS}, the group $H$ is constructed as $H_N(\S)$, where $\S=\S(M)$ is the S-machine emulating the Turing machine $M$ accepting the language of $G$. $H_N(\S)$ is constructed in a very similar way to $G_N(\S)$ from \cite{SBR} referenced above, but with some extra relations to enforce the embedding.  The upper bound on its Dehn function is estimated in Section 6 of \cite{BORS} as $O(n^6) + O(T(O(n^2))^4) + O(n^2T(O(n^2))^4)$.  Here, the second term estimates the area coming from disks, the third estimates the area coming from $G_b$-cells, and the first estimates the area coming from all other cells.  

It is evident from these estimates given that the areas of the disks and $G_b$-cells are dependent on the machine $\S$.  Indeed, they arise from the area complexity of the machine.  As such, constructing $H_\eps=H_N(\S_\eps)$ in just the same way, the Dehn function is bounded above by $O(n^6)+O(T(O(n^2))^{2+\eps})+O(n^2T(O(n^2))^{2+\eps})$.  As $T(n)\succeq n$, the last term is the leading term, and so gives the required bound.


The fact that $G$ in $H_N(\S_\eps)$ has Dehn function at most $T(n)^{2+\eps}$ follows from the fact that every word that is trivial in $G$ corresponds to a $G_b$-cell, which can be filled by a diagram with at most $O(T(O(n))^{2+\eps})$ cells (compared with $O(T(O(n))^4)$ in $H$).
\end{proof}

The embedding of \cite{BORS} marked a major improvement to other analogues of the Higman embedding theorem, where the bounds would be exponential.  As such, \cite{BORS} inextricably links the class of groups whose word problem is NP solvable and the class of finitely presented groups with polynomially bounded Dehn function.  \Cref{th-BORS-1-1-improved} then serves as a major refinement of this, improving the bound to one that is in some cases quasi-optimal.  For example, if $G$ is a non-hyperbolic group with whose word problem is quasilinear (e.g a finitely generated infinite free Burnside group), then $G$ cannot be embedded with bounded distortion into a hyperbolic group, and so the isoperimetric gap \cite{gromov1987hyperbolic} \cite{O91} implies that the Dehn function of $G$ in any finitely presented group would be at least quadratic.


It is the authors' belief that \Cref{th-BORS-1-1-improved} can be improved even further using the techniques developed by the second author to embed the free Burnside group into a finitely generated group with quadratic Dehn function \cite{W}:

\begin{conjecture}\label{main-conjecture}

Let $G$ be a finitely generated group with word problem solvable by a nondeterministic Turing machine with time function $\leq T(n)$ such that $T(n)^2$ is superadditive.  
Then, for any $\varepsilon > 0$, $G$ can be embedded into a finitely presented group $H$ with the Dehn function bounded above by $T(n)^{2+\varepsilon}$ in such a way that $G$ has bounded distortion in $H$.
\end{conjecture}

The statement of \Cref{main-conjecture} will serve as a quasi-optimal general bound on the \emph{essential Dehn function} of a finitely generated group with decidable word problem, an invariant arising from this setting.  Further discussion is postponed to that manuscript.


\section{S-machines and S-graphs} \label{sec-intro}
There are many equivalent formulations of \emph{$S$-machine}. Following that given in \cite{SBR}, the following interpretation as a rewriting system of group words is adopted: 

\begin{definition}[$S$-machine] \label{def-smachine}
A \emph{hardware} $\H$ is a pair $(Y, Q)$ such that $Y = (Y_i)_{i\in \overline{N}}$ is an $N$-tuple of (not necessarily disjoint) finite sets $Y_i$ and $Q = (Q_i)_{i\in \overline{N+1}}$ is an $(N+1)$-tuple of finite pairwise disjoint sets $Q_i$ that are also disjoint from $\bigcup Y_i$. For notational ease, $Y$ is extended to include $Y_0=Y_{N+1}=\emptyset$.  The set $\bigcup Y_i$ is called \emph{the set of tape letters} and $\bigsqcup Q_i$ \emph{the set of state letters} of~$\H$.

The \emph{language of admissible words} over $\H$, denoted $L_\H$ (or simply $L$ if $\H$ is clear from context), is the set of words $L_\H = Q_1 F(Y_1) Q_2 \dots Q_N F(Y_N) Q_{N+1}$, where $F(Y_i)$ is the set of freely reduced group words over $Y_i \sqcup Y_i^{-1}$. A \emph{state} of $\H$ is an element of $\bigtimes_{i\in\overline{N+1}}Q_i$. Every admissible word has a state associated to it, obtained from it simply by removing every tape word.

For $l,r\in\overline{N+1}$ such that $l\leq r$, $L_\H^{l,r}$ is the set of subwords of words from $L_\H$ which start with a state letter from $Q_l$ and end with a state letter from $Q_r$.  As a result, if $r>l+1$, then $L_\H^{l,r}=Q_l F(Y_l) Q_{l+1} \dots F(Y_{r-1}) Q_r$.  This definition is extended to the set $\bigl(L_\H^{l,r}\bigr)^+$ of subwords of words from $L_\H$ starting with a word from $F(Y_{l-1})$ and ending with one from $F(Y_r)$, i.e $$\bigl(L_\H^{l,r}\bigr)^+=F(Y_{l-1})L_\H^{l,r}F(Y_r)$$

An \emph{S-rule} is a rewriting rule over $L_\H$ of the form $\tau = \left[U_1\rightarrow V_1,\dots, U_M\rightarrow V_M \right]$ such that:
\begin{enumerate}
	\item [(S1)] for each $i\in\overline{M}$, there exist $l(i),r(i)\in\overline{N+1}$ such that $U_i\in L_\H^{l(i), r(i)}$;
	\item [(S2)] $r(i) < l(j)$ for $i < j$;
	\item [(S3)] for each $i\in\overline{M}$, $V_i\in\bigl(L_\H^{l(i), r(i)}\bigr)^+$.
\end{enumerate}

Given an S-rule $\tau$ as above, the $S$-rule $[U_i\to V_i]$ is called a \emph{subrule} of $\tau$ for each $i\in\overline{M}$. For each $i$, let $V_i=x_iU_i'y_i$ such that $x_i\in F(Y_{l-1})$, $y_i\in F(Y_r)$, and $U_i'\in L_\H^{l(i),r(i)}$. Then the \emph{inverse} of $\tau$ is the S-rule $\tau^{-1}$ given by:
$$\tau^{-1} = \left[U_1'\rightarrow x_1^{-1} U_1 y_1^{-1},\dots, U_M'\rightarrow x_M^{-1} U_M y_M^{-1} \right]$$
Note that $\tau^{-1}$ is indeed an S-rule and that inversion is an involutive operation on the set of S-rules over $L_\H$.

An \emph{S-machine} $\S$ with hardware $\H$ (or over $\H$) is a rewriting system over $L_\H$ consisting of a \emph{symmetric} finite set of S-rules, i.e $\tau\in\S$ if and only if $\tau^{-1}\in\S$.
\end{definition}

An S-rule $\tau=[U_i \rightarrow V_i]_i$ is said to be \emph{applicable} to an admissible word $w\in L_\H$ if each $U_i$ is a subword of $w$.  Otherwise, $\tau$ is \emph{inapplicable} to $w$.  

Given that $\tau$ is applicable to $w$, the \emph{application} of $\tau$ to $w$ is the simultaneous replacement of every subword $U_i$ in $w$ with $V_i$, followed by the free reduction of the resulting word.  Note the following:

\begin{itemize}
    \item The word resulting from the application of $\tau$ to $w$, denoted $w \cdot \tau$, is an admissible word over $\H$.
    \item $\tau$ is applicable to $w$ if and only if $\tau^{-1}$ is applicable to $w \cdot \tau$; moreover, $(w\cdot \tau)\cdot\tau^{-1}=w$, justifying the use of the term `inverse'.
    \item The free reduction is not considered a separate step in the application of an S-rule to an admissible word.
\end{itemize}

A \emph{computation} of an S-machine $\S$ is a sequence
$\varepsilon = (w_0, \tau_1, w_1, \dots,  \tau_n, w_n)$ such that $n\geq 0$, each $w_i$ is an admissible word over $\H$, each $\tau_i$ is an S-rule of $\S$, and $w_i = w_{i-1} \cdot \tau_i$ for all $i\in \overline{n}$.
The sequence $\rho =\rho_\varepsilon = (\tau_1, \dots, \tau_n)$ is called the \emph{history} of $\varepsilon$ and is said to \emph{support} $\varepsilon$. 
The word $w_0$ is called the \emph{starting word} and $w_n$ the \emph{finishing word} of~$\varepsilon$; alternatively, $\varepsilon$ is said to be a computation \emph{between}~$w_0$ and~$w_n$. The number $n$ is the \emph{length} of the computation (and of the computation path), denoted $n = |\varepsilon| = |\rho|$.  The set of all computations between $w_0$ and $w_n$ is denoted $\varepsilon(w_0,w_n)$. 

For a fixed admissible word $w_0$ and sequence of S-rules $\rho=(\tau_1,\dots,\tau_n)$, there is at most one computation supported by $\rho$ and starting with $w_0$.  If such a computation exists, then it is denoted $\varepsilon(w_0, \rho) = (w_0, \tau_1,  w_1,  \dots, \tau_n, w_n)$. For simplicity, the notation $w_n = w_0 \cdot \rho$ is used to describe $\varepsilon(w_0, \rho)$.  Further, $w_n$ is called the \emph{result} of $\varepsilon(w_0,\rho)$ or of the \emph{application of $\rho$ to $w_0$}. 

\medskip

An alternative construction is now introduced, producing objects called S-graphs.  This construction bears a close resemblance to that of S-machines, which will be expanded upon in \Cref{th-MG}.

\begin{definition}[$S$-graph] \label{def-sgraph}
An \emph{S-graph hardware} $\H$ (or, with abuse of terminology, simply a \emph{hardware} when the context is clear) is a tuple $(\I, \A)$ such that $\I$ is a finite \emph{index set} and $\A=(A_i)_{i\in\I}$ is an $\I$-tuple of (not necessarily disjoint) finite \emph{alphabets}.  It is useful to note that the role of $\A$ in this setting is similar to that of $Y$ in the setting of S-machines. 

Let $L_i = F(A_i)$ and $L_\H= \bigtimes_{i\in\I} L_i$. Typically, the subscript $\H$ is omitted from $L_\H$ without causing ambiguity.  The elements of $L$ are then called \emph{L-words}; however, it is worth stating explicitly that an $L$-word $w\in L$ is not a word, but an $\I$-tuple of words $w_i \in L_i$, i.e $w = (w_i)_{i\in \I}$.

A \emph{transformation} $t$ over the hardware $\H$ is a function $t\colon \I \times\{l, r\} \rightarrow \bigcup L_i$ such that for all $i\in\I$, $t(i, l), t(i, r)\in L_i$. For simplicity, the values of $t$ are also denoted $t(i, l) = t_i^l$ and $t(i, r) = t_i^r$.  It is informative of the intended use to interpret $t$ as a function $t\colon L\rightarrow L$ such that for $w = (w_1, \dots, w_N)\in L$,
$$t(w) = \bigl(t_i^l w_i t_i^r\bigr)_{i\in \I},$$
where $t_i^l w_i t_i^r$ is the concatenation of the corresponding words followed by free reduction. 

A \emph{guard} $g$ over $\H$ is a function $g\colon \I \rightarrow \{\nan\}\cup\left(\bigcup L_i\right)$ such that $g(i)\in\{\nan\}\cup L_i$. As will be made clear in the definitions that follow, the value $\nan$ is a placeholder to indicate that the guard has `no value' in the corresponding index.

A pair $(t, g)$ of a transformation $t$ and a guard $g$ over the hardware $\H$ is called an \emph{action} over $\H$.  The set of all possible actions over $\H$ is denoted by $\ACT_\H$.

Given an action $(t,g)$ over $\H$, the \emph{inverse} action $(t,g)^{-1}=(t',g')$ over $\H$ is defined as follows:
\begin{align*}
	t'(i, l) &= \bigl(t(i, l)\bigr)^{-1} \\
	t'(i, r) &= \bigl(t(i, r)\bigr)^{-1} \\
	g'(i) &= \begin{cases}
			\nan & \text{ if $g(i) = \nan$,} \\
			t_i^l\cdot g(i)\cdot t_i^r & \text{ otherwise.}
		\end{cases}
\end{align*}

Let $\G=(S,E)$ be a finite directed graph, perhaps containing loops or multiple edges.  Suppose $\G$ is equipped with a labelling function $\lab\colon E\rightarrow \ACT_\H$, so that each edge is labelled with an action over $\H$.  Hence, each edge $e\in E$ can be identified with a quadruple $(s_1,s_2,t,g)$ where $s_1=T(e)\in S$ is the tail of $e$, $s_2=H(e)\in S$ is the head of $e$, and $(t,g)$ is the action $\lab(e)$.  Finally, suppose $\G$ is \emph{symmetric} in the sense that for each $e=(s_1,s_2,t,g)\in E$, there exists the \emph{inverse edge} $e'=(s_2,s_1,t',g')\in E$ such that $(t',g')=(t,g)^{-1}$.  Then $\G$ is called an \emph{S-graph} over~$\H$.

Given an S-graph $\G=(S,E)$, the elements of $S$ are called \emph{states} (in addition to being called vertices) to emphasize that S-graphs are a computational model.  As will be made clear in the correspondence established in \Cref{th-MG}, this convention is motivated by the fact that $S$ is the natural counterpart of the states of an S-machine (or of a Turing machine). 
\end{definition}

It is convenient to interpret actions in terms of defining pieces called \emph{fragments}. In particular, given a hardware $\H = (\I, \A)$, for each $i\in\I$, define the set of \emph{left transformation fragments} to be $(\FRAG_\H)_i^{lt} = \{[\i\rightarrow w\i]~|~w\in L_i\}$, the set of \emph{right transformation fragments} to be $(\FRAG_\H)_i^{rt} = \{[\i\rightarrow \i w]~|~w\in L_i\}$, and the set of \emph{guard fragments} to be $(\FRAG_\H)_i^{g} = \{[\i=w]~|~w\in L_i\}$, where all sets are disjoint.  The more general term \emph{fragment} is used to describe any element of $\bigsqcup_{i\in\I}\left((\FRAG_\H)_i^{lt}\sqcup(\FRAG_\H)_i^{rt}\sqcup(\FRAG_\H)_i^g\right)$.  For simplicity, the notation $[\i \rightarrow w_1 \i w_2]$ is permitted to represent the pair of fragments $[\i \rightarrow w_1 \i]\in(\FRAG_\H)_i^{lt}$ and $[\i\rightarrow \i w_2]\in(\FRAG_\H)_i^{rt}$.

Note that for fixed $i\in\I$, the sets $(\FRAG_\H)_i^{lt}$, $(\FRAG_\H)_i^{rt}$ and $(\FRAG_\H)_i^g$ can be identified with three (disjoint) copies of $L_i$, with some stylized notation used to distinguish the elements of each.  Hence, the set of fragments corresponds to the set of triples $(i, w, p)$ such that $i\in \I$, $w\in L_i$, and $p \in \{lt, rt, g\}$.

For $i\in\I$, any element of $(\FRAG_\H)_i^{lt} \sqcup(\FRAG_\H)_{i}^{rt} \sqcup(\FRAG_\H)_{i}^{g}$ is called a \emph{fragment mentioning $i$}.  The fragments $[\i\rightarrow \i1]\in(\FRAG_\H)_i^{rt}$ and $[\i\rightarrow 1\i]\in(\FRAG_\H)_i^{lt}$ are called \emph{trivial}, while all other fragments mentioning $\i$ are called \emph{nontrivial}; note that, following group-theoretic notation, here the symbol $1$ (rather than $\emptyset$) is used to denote the trivial  word.  The set of all fragments over $\H$ is denoted by $\FRAG_\H$.

A set of fragments $F \subseteq \FRAG_\H$ is said to be \emph{valid} if $|F \cap(\FRAG_\H)_i^p| \leq 1$ for every pair $(i,p)\in \I\times\{lt,rt,g\}$.  Further, two valid sets of fragments $F'$ and $F''$ are called \emph{equivalent} if they differ only on trivial fragments, i.e, for any fragment $f \in F'\Delta F''$, there exists $i\in\I$ such that either $f=[\mathbf{i} \rightarrow \mathbf{i}1]$ or $f=[\mathbf{i} \rightarrow 1\mathbf{i}]$. Note that this term defines an equivalence relation on the valid sets of fragments.

An action $A = (t, g)$ over the hardware $\H$ is identified with the associated valid set of fragments $\F_\H(A)\subseteq\FRAG_\H$ constructed by taking:
\begin{itemize}
    \item{\makebox[3cm]{$[\mathbf{i}\rightarrow w\mathbf{i}] \in \F_\H(A)$\hfill} if $t(i, l) = w$,}
    \item{\makebox[3cm]{$[\mathbf{i}\rightarrow \mathbf{i}w] \in \F_\H(A)$\hfill} if $t(i, r) = w$, and}
    \item{\makebox[3cm]{$[\mathbf{i} = w] \in \F_\H(A)$\hfill} if $g(i) = w$.}
\end{itemize}
Note that $\F_\H(A)$ is indeed a valid set of fragments.  Further, per this construction, if $g(i)=\nan$, then $\F_\H(A) \cap (\FRAG_\H)_i^g=\emptyset$. 
   
Conversely, a valid set of fragments $F\subseteq\FRAG_\H$ is identified with the action $\A_\H(F)= (t, g)$ over $\H$ defined by:
\begin{itemize}
\item $t(i,l)=
\begin{cases}
w & \text{if }[\i\to w\i]\in F\cap(\FRAG_\H)_i^{lt} \\
1 & \text{if }F\cap(\FRAG_\H)_i^{lt}=\emptyset
\end{cases}
$
\vspace{.25cm}

\item $t(i,r)=
\begin{cases}
w & \text{if }[\i\to \i w]\in F\cap(\FRAG_\H)_i^{rt} \\
1 & \text{if }F\cap(\FRAG_\H)_i^{rt}=\emptyset
\end{cases}
$
\vspace{.25cm}

\item $g(i)=
\begin{cases}
w & \text{if }[\i=w]\in F\cap(\FRAG_\H)_i^g \\
\nan & \text{if }F\cap(\FRAG_\H)_i^g=\emptyset
\end{cases}
$

\end{itemize}

It is straightforward to check that for any pair $F'$ and $F''$ of equivalent valid subsets of $\FRAG_\H$, then $\A_\H(F')=\A_\H(F'')$.  Moreover, if $A$ is an action over $\H$ and $F$ is a valid subset of $\FRAG_\H$, then $\A_\H(\F_\H(A))=A$ and $\F_\H(\A_\H(F))$ is equivalent to $F$.  

\vspace{.25cm}

Given this correspondence, actions and valid fragment sets will be discussed interchangeably hereafter.  As such, for any $S$-graph $\G$ over the hardware $\H$, if $e$ is an edge of $\G$ labelled by the action $\act_e$, then this label will typically be represented by a valid fragment set $F$ that is equivalent to $\F_\H(\act_e)$.  This will simplify the presentation of an $S$-graph, as it allows for the suppression of some of the action's data without the loss of any information.  Note that, in this case, $F\subseteq\F_\H(\act_e)$; hence, the notation $f\in e$ is used in place of $f\in\F_\H(\act_e)$ for simplicity.

Thus, the valid fragment set $F=\{[\mathbf{1}\rightarrow a\mathbf{1}b], [\mathbf{1}=a], [\mathbf{2} \rightarrow \mathbf{2}c]\}$ is identified with an action whose transformation $t$ satisfies $t(1, l)= a$, $t(1, r)= b$, $t(2, l) = 1$, and $t(2, r) = c$, and whose guard $g$ satisfies $g(1) = a$ and $g(2) = \nan$.  However, in order to state the action properly, the hardware $\H=(\I,\A)$ over which this action is being defined must also be described. One might infer that a natural candidate for this particular case would be $\I = \{1,2\}$, $A_1 = \{a, b\}$, and $A_2 = \{c\}$. It is possible, however, that $\I = \{1,2, e\}$, $A_1 = A_2 = \{a,b,c\}$, and $A_e = \{\delta\}$.  

To avoid such ambiguity, the definitions of fragment and action are slightly expanded.  In particular, fix countable sets $\I^*$ and $\A^*$ of all possible indices and all possible letters, respectively, and set $L^*=F(\A^*)$.  Then, it is henceforth understood that for a hardware $\H=(\I,\A)$, $\I$ is a (finite) subset of $\I^*$ and $\A$ is an $\I$-tuple of (finite) subsets of $\A^*$.  

Similarly, a fragment is understood to be a triple $(i, w, p)$ satisfying $i\in \I^*$, $w\in L^*$, and $p\in \{lt, rt, g\}$.  Letting $\FRAG$ be the set of all possible fragments, $\FRAG_i$ be the subset consisting of all fragments with fixed $i\in\I^*$, and $\FRAG_i^p$ be the subset of $\FRAG_i$ with fixed $p\in\{lt,rt,g\}$, a finite set of fragments $F$ is valid if $|F \cap \FRAG_i^{p}| \leq 1$ for all $i\in I^*$ and all $p\in \{lt, rt, g\}$.  As above, a fragment $(i,w,p)$ is called trivial if $w=1$ and $p\in\{lt,rt\}$ and two valid sets of fragments are then equivalent if they differ only on trivial fragments.

Let $\H^*$ be the collection of all possible $S$-graph hardwares, i.e $(\I,\A)\in\H^*$ if and only if $\I$ is a finite subset of $\I^*$ and $\A=(A_i)_{i\in\I}$ such that for all $i\in\I$, $A_i$ is a finite subset of $\A^*$.

Define the relation $\leq$ on $\H^*$ given by $\H_1\leq\H_2$ for $\H_j=(\I_j,\A_j)$ if and only if $\I_1\subseteq\I_2$ and $\A_1(i)\subseteq\A_2(i)$ for all $i\in\I_1$.  The relation $\leq$ is then a partial order on $\H^*$, and in fact a complete meet-semilattice.

Further, a hardware $\H = (\I, \A)\in\H^*$ is said to be \emph{compatible} with a valid fragment set $F$ if for all $(i,w,p)\in F$, $i\in\I$ and $w\in F(A_i)$.  Thus, for any valid fragment set $F$, there exists a unique minimal hardware $\H_F$ compatible with $F$.  In particular, a hardware $\H$ is compatible with $F$ if and only if $\H_F \leq \H$.

A valid fragment set $F$ can be interpreted as a fragment set (and thus an action) over $\H_F$. Thus, returning to the motivating example, for $F = \{[\mathbf{1}\rightarrow a\mathbf{1}b], [\mathbf{1}=a], [\mathbf{2} \rightarrow \mathbf{2}c]\}$, the implied hardware is indeed the natural candidate $\H_F= (\I, \A)$ where $\I = \{1,2\}$, $A_1 = \{a, b\}$, and $A_2 = \{c\}$.

However, there are contexts where a valid fragment set $F$ is used to define an action $\act$ over a hardware $\H$ that is not necessarily $\H_F$.  In this case, $\H_F\leq\H$ and $\act$ is interpreted as $\A_\H(F)$.

Note that if $F'$ and $F''$ are two equivalent valid sets of fragments, then it does not generally hold that $\H_{F'}=\H_{F''}$.  In fact, $\H_{F'}$ and $\H_{F''}$ may be incomparable under $\leq$.  However, for any two valid sets of fragments $F'$ and $F''$, there exists a unique minimal hardware $\H$ such that $\H_{F'}\leq\H$ and $\H_{F''}\leq\H$, so that $\A_\H(F')=\A_\H(F'')$.

\begin{definition}[Computation of an S-graph]\label{def-computation}
Let $(t,g)$ be an action over a hardware $\H=(\I,\A)$.  The guard $g$ \emph{accepts} a word $w=(w_i)_{i\in\I}\in L_\H$ if for all $i\in \I$, either $g(i) = \nan$ or $g(i)=w_i$.  

Given an $S$-graph $\G$ over the hardware $\H$, a pair $(w, s)\in L_\H\times S_\G$ is called a \emph{configuration} of $\G$.  An edge $e = (s_1, s_2, t, g)$ of $\G$ is said to be \emph{applicable} to the configuration $(w, s)$ if $s_1 = s$ and if $g$ accepts $w$.  In this case, the result of the \emph{application} of $e$ to $(w, s)$ is the configuration $e(w, s) = (t(w), s_2)$ (where, as mentioned above, $t(w)$ is defined by interpreting $t$ as a mapping $L_\H\to L_\H$).  Note that, per the definitions of application and inverse action, the edge $e$ is applicable to the configuration $(w, s)$ if and only if the inverse edge $e^{-1}$ is applicable to the configuration resulting from this application, and moreover $e^{-1}(e(w, s)) = (w, s)$.

A \emph{computation} of an S-graph $\G$ is then a sequence $$\eps=\bigl((w_0, s_0), e_1, (w_1, s_1), \dots, e_n, (w_n, s_n)\bigr)$$ where $n\geq 0$ and for all $i\in\overline{n}$, $e_i$ is an edge of $\G$ applicable to the configuration $(w_{i-1},s_{i-1})$ of $\G$ satisfying $e_i (w_{i-1}, s_{i-1})=(w_i,s_i)$.  

Informally, a computation starts with an $L$-word $w$ and a state $s\in S_\G$, traversing edges of the graph and changing both the word and the state accordingly.

Note that if $n=0$, $\eps$ is simply the single configuration $(w_0,s_0)$.  Moreover, per this construction, for every $i\in \overline{n}$, $s_{i-1} = T(e_i)$ and $s_i = H(e_i)$.  As a result, the computation $\eps$ corresponds naturally to a path $\rho_\varepsilon = e_1, \dots, e_n$ in $\G$, called the \emph{computation path} of $\eps$.  In this case, $\rho_\eps$ is said to \emph{support} $\eps$.

The configuration $(w_0, s_0)$ is called the \emph{starting configuration} and $(w_n, s_n)$ the \emph{finishing configuration} of $\varepsilon$; alternatively, $\varepsilon$ is called a computation \emph{between} $(w_0, s_0)$ and $(w_n, s_n)$. The number $n$ is called the \emph{length} of the computation $\eps$ (and of the computation path $\rho_\eps$), denoted $n = |\varepsilon| = |\rho_\eps|$. The notation $\varepsilon(w_0,s_0,w_n,s_n)$ is used to denote the set of all computations between $(w_0,s_0)$ and $(w_n,s_n)$. 

Notice that for any configuration $(w_0, s_0)$ and any path $\rho=e_1,\dots,e_n$ in $\G$, there is at most one computation $\varepsilon$ supported by $\rho$ and with starting configuration $(w_0, s_0)$.  Given the existence of such a computation, then it is of the form $$\varepsilon = \bigl((w_0, s_0), e_1, (w_1, s_1), \dots, e_n, (w_n, s_n)\bigr)$$ where for each $i\in\overline{n}$, $(w_i, s_i) = e_i \dots e_1 (w_0, s_0)$.  In this case, the finishing configuration $(w_n, s_n)$ of $\eps$ is called the \emph{result} of $\eps$; alternatively, $(w_n,s_n)$ is said to be the result of the application of $\rho$ to $(w_0, s_0)$ and is denoted $(w_n,s_n)=\rho(w_0,s_0)$ instead of $e_n\dots e_1(w_0,s_0)$.

\end{definition}

Two S-graphs $\G$ and $\G'$ are said to be \emph{isomorphic}, denoted $\G \cong \G'$, if one can pass from $\G$ to $\G'$ simply by renaming the states and by reindexing.  In other words, the graphs are isomorphic if there are bijections between the states of these graphs and between the index sets of the corresponding hardwares which make them identical. 

On the other hand, two S-machines $\S$ and $\S'$ are said to be \emph{isomorphic}, denoted $\S \cong \S'$, if one can pass from $\S$ to $\S'$ by simply renaming the set of states in the hardware of $\S$.  In other words, if $\S$ is an $S$-machine with hardware $\H=(Y,Q)$ with $Y=(Y_i)_{i\in\overline{N}}$ and $Q=(Q_i)_{i\in\overline{N+1}}$ and $\S'$ is an $S$-machine with hardware $\H'=(Y',Q')$ with $Y'=(Y_i')_{i\in\overline{N'}}$ and $Q'=(Q_i')_{i\in\overline{N'+1}}$, then the machines are isomorphic if $N=N'$ and there are bijections between $Q_i$ and $Q_i'$ which make $\S$ and $\S'$ identical.

As opposed to isomorphisms between S-graphs, isomorphisms between S-machines will not be of use.   This observation will be elaborated upon in the examples following \Cref{th-MG}.  Further, although it would seem natural to permit the renaming of the letters of the alphabets $A_i$ (or of the tape letters of the parts $Y_i$), this operation is not considered to preserve isomorphism.  This is because these letters will typically have a fixed role, so that allowing such a renaming would merely add a layer of convolution that would provide no seeming benefit.

\medskip

Let $\H=(Y,Q)$ be an $S$-machine hardware with $Y = (Y_i)_{i\in \overline{N}}$ and $Q = (Q_i)_{i\in \overline{N+1}}$.  Define the corresponding $S$-graph hardware $\G(\H)=(\I,\A)$ by setting $\I=\overline{N}$ and $\A=(A_i)_{i\in\overline{N}}$, where $A_i=Y_i$ for all $i\in\overline{N}$.  Further, define the set $S_\H$ by $S_\H=\bigtimes_{i\in\overline{N+1}}Q_i$.

Then, define the bijection let $\mu_\H:L_\H\to L_{\G(\H)}\times S_\H$ by
$$\mu_\H(q_1w_1q_2\dots w_Nq_{N+1})=\bigl((w_i)_{i\in\overline{N}},(q_i)_{i\in\overline{N+1}}\bigr)$$

Conversely, for an $S$-graph hardware $\H=(\I,\A)$ with $|\I|=N$, fix a bijection $\eta:\I\to\overline{N}$.   Let $\G=(S_\G,E_\G)$ be an $S$-graph over $\H$.  Then, define the $S$-machine hardware $\S_{\G,\eta}(\H)=(Y,Q)$ by setting $Y=(Y_i)_{i\in\overline{N}}$ and $Q=(Q_i)_{i\in\overline{N+1}}$, where $Y_i=A_{\eta^{-1}(i)}$ for all $i\in\overline{N}$, $Q_1=S_\G$, and $Q_j$ is the singleton $\{q_j^*\}$ for $j\in\overline{2,N+1}$.

Then, define the bijection $\mu_{\G,\eta}:L_{\S_{\G,\eta}(\H)}\to L_\H\times S_\G$ by
$$\mu_{\G,\eta}(q_1w_1q_2\dots w_Nq_{N+1})=\bigl((w_{\eta(i)})_{i\in\I},q_1\bigr)$$

\Cref{th-MG} below establishes a correspondence between S-machines and S-graphs.  Its statement is rather technical and perhaps can be better understood by looking at the examples that immediately follow it. 

\begin{theorem}\label{th-MG} \

\begin{enumerate}[label=({\alph*}), leftmargin=*]

\item \emph{(From $S$-machine to $S$-graph)}  Let $\S$ be an S-machine over the hardware $\H$.  Then there exists an $S$-graph $\G(\S)=(S_\H,E)$ over the hardware $\G(\H)$ such that $E=\bigsqcup_{\tau\in\S}E_\tau$ and which \emph{`simulates'} $\S$ in the following sense:

\renewcommand{\theenumii}{\roman{enumii}}

\begin{enumerate}

\item For any computation $\eps=(w_0,\tau_1,w_1,\dots,\tau_n,w_n)$ of $\S$, there exists a unique computation $$\varphi_{\G(\S)}(\eps)=\bigl(\mu_\H(w_0),e_1,\mu_\H(w_1),\dots,e_n,\mu_\H(w_n)\bigr)$$ of $\G(\S)$ such that $e_i\in E_{\tau_i}$ for all $i\in\overline{n}$

\item For any computation $\delta=\bigl((w_0,s_0),e_1,(w_1,s_1),\dots,e_n,(w_n,s_n)\bigr)$ of $\G(\S)$, there exists a computation $$ \ \ \ \ \  \varphi_\S(\delta)=\bigl(\mu_\H^{-1}(w_0,s_0),\tau_1,\mu_\H^{-1}(w_1,s_1),\dots,\tau_n,\mu_\H^{-1}(w_n,s_n)\bigr)$$
of $\S$ where $\tau_i$ is the (unique) $S$-rule of $\S$ satisfying $e_i\in E_{\tau_i}$.

\end{enumerate}

\vspace{.5cm}

\item \emph{(From $S$-graph to $S$-machine)} Let $\G=(S_\G,E_\G)$ be an $S$-graph over the hardware $\H=(\I,\A)$ and let $\eta:\I\to\overline{N}$ be a bijection.  Then there exists an $S$-machine $\S(\G,\eta)$ over the hardware $\S_{\G,\eta}(\H)$ which \emph{`simulates'} $\G$ in the following sense:

\begin{enumerate}

\item There is a bijection $\phi$ between the $S$-rules of $\S(\G,\eta)$ and the edges $E_\G$

\item For any computation $\eps=(w_0,\tau_1,w_1,\dots,\tau_n,w_n)$ of $\S(\G,\eta)$, there exists a computation $$\psi_{\G,\eta}(\eps)=\bigl(\mu_{\G,\eta}(w_0),\phi(\tau_1),\mu_{\G,\eta}(w_1),\dots,\phi(\tau_n),\mu_{\G,\eta}(w_n)\bigr)$$ of $\G$.

\item For any computation $\delta=\bigl((w_0,s_0),e_1,(w_1,s_1)\dots,e_n,(w_n,s_n)\bigr)$ of $\G$, there exists a computation $$\ \ \ \ \ \psi_{\S(\G,\eta)}(\delta)=\bigr(\mu_{\G,\eta}^{-1}(w_0,s_0),\tau_1,\mu_{\G,\eta}^{-1}(w_1,s_1),\dots,\tau_n,\mu_{\G,\eta}^{-1}(w_n,s_n)\bigr)$$ of $\S$ where $\tau_i=\phi^{-1}(e_i)$ for all $i$.

\end{enumerate}

\vspace{.5cm}

\item \emph{(Partial Duality)}  Let $\G'$ be an $S$-graph over the hardware $\H=(\I,\A)$.  Then for any enumeration $\eta:\I\to\overline{N}$, $\G(\S(\G',\eta))\cong\G'$.

\end{enumerate}

\end{theorem}

\bigskip

\begin{proof}

(a) Let $\tau = [U_t \rightarrow v_t^{l_t-1}V_tv_t^{r_t}]_{t\in \overline{M}}\in\S$, where for each $t\in\overline{M}$,
	\begin{align*}
	U_t &= q^{l_t}_t u_t^{l_t}  q^{l_t +1}_t \dots u^{r_t - 1}_t q_t^{r_t}, \\
	V_t &= p^{l_t}_t v_t^{l_t}  p^{l_t +1}_t \dots v^{r_t - 1}_t p_t^{r_t},
	\end{align*}
such that $t\in \overline{M}$, $q_t^j, p_t^j \in Q_j$, and $u_t^j, v_t^j \in F(Y_j)$.  Let $I_\tau = \bigsqcup_{i\in \overline{M}} \overline{l_t, r_t} \subseteq \overline{N+1}$ (noting that disjointness follows from condition (S2)) and let $J_\tau = \overline{N+1} - I_\tau$.  
	
	Notice that for any $i\in I_\tau$, there exists a unique $t(i)\in\overline{M}$ such that $i\in \overline{l_{t(i)}, r_{t(i)}}$; for simplicity, denote the corresponding $q_{t(i)}^i$ and $p_{t(i)}^i$ by $q_*^i$ and $p_*^i$, respectively.   
	
	Define the valid set of fragments $F_\tau\subseteq\FRAG_{\G(\H)}$ by
	\begin{align*}
		F_{\tau}^{lt}&= \bigl\{[\i \rightarrow (v_t^i) (u_t^i)^{-1} \i]  ~\bigl|~t\in \overline{M}, i\in \overline{l_t, r_t-1} \bigr\}\sqcup \bigl\{[\mathbf{r_t} \rightarrow v_t^{r_t}  \mathbf{r_t}] ~\bigl|~t\in \overline{M} \bigr\} \\
		F_{\tau}^{rt}&=\bigl\{[\mathbf{(l_t-1)} \rightarrow  \mathbf{(l_t-1)} v_t^{l_t-1}]~\bigl|~t\in \overline{M} \bigr\} \\
		F_{\tau}^g &= \bigl\{[\i = u_t^i] ~\bigl|~t\in \overline{M}, i\in \overline{l_t, r_t-1} \bigr\}
	\end{align*}	
	Then, let $(t_\tau,g_\tau)=\A(F_\tau)\in\ACT_{\G(\H)}$.
	
	Let $\Theta_\tau = \bigtimes_{j\in J_\tau} Q_j$ and $\sigma$ and $\rho$ be $I_\tau$-tuples $\sigma = (q_*^i)_{i\in I_\tau}$ and $\rho = (p_*^i)_{i\in I_\tau}$.  Then, for every $\theta\in\Theta_\tau$, define the quadruple $e_{\tau,\theta}=\bigl(\theta \sqcup \sigma, \theta \sqcup \rho,  t_{\tau}, g_{\tau}\bigr)$.
	
	Finally, define $E_\tau$ to be the set of quadruples $\left\{ e_{\tau, \theta} ~|~ \theta \in \Theta_\tau \right\}$ and define the labelled directed graph $\G(\S)=(S_\H,E)$ with $E=\bigsqcup_{\tau\in\S}E_\tau$.
	
	Let $e=(s_1,s_2,t_\tau,g_\tau)\in E_{\tau}$ where $\tau\in\S$ as above.  By the symmetry of the $S$-machine, $\tau^{-1}\in\S$, where $\tau^{-1}=[V_t\to(v_t^{l_t-1})^{-1}U_t(v_t^{r_t})^{-1}]_{t\in\overline{M}}$.  Further, fix $\theta\in\Theta_\tau$ such that $e=e_{\tau,\theta}$, so that $s_1=\theta\sqcup\sigma$ and $s_2\sqcup\rho$.  Then, there exists an edge $e'=e_{\tau^{-1},\theta}\in E$, so that $$e'=(\theta\sqcup\rho,\theta\sqcup\sigma,t_{\tau^{-1}},g_{\tau^{-1}})=(s_2,s_1,t_{\tau^{-1}},g_{\tau^{-1}})$$ where $(t_{\tau^{-1}},g_{\tau^{-1}})=\A(F_{\tau^{-1}})$.
	
	By construction, if $g_\tau(i)\neq\nan$ for $i\in\overline{N}$, then $g_\tau(i)=u_t^i$ for some $t\in\overline{M}$ such that $i\in\overline{l_t,r_t-1}$.  But then $[\i=v_t^i]\in F_{\tau^{-1}}^g$, so that $g_{\tau^{-1}}(i)=v_t^i=t_\tau(i,l)u_t^it_\tau(i,r)$.  Similarly, it follows from the construction that for all $i\in\overline{N}$ and $p\in\{l,r\}$, $t_{\tau^{-1}}(i,p)=t_\tau(i,p)^{-1}$.  
	
	Hence, $(t_{\tau^{-1}},g_{\tau^{-1}})=(t_\tau,g_\tau)^{-1}$, and so $e'=e^{-1}$.  Thus, $\G$ is indeed an $S$-graph.
	
	That $\G$ simulates $\S$ follows directly from the construction.
	
	\vspace{.25cm}
	
	(b)  For any edge $e=(s_1, s_2, t, g)\in E_\G$, define a \emph{maximal guarded interval} (with respect to $\eta$), or \emph{mgi} for short, as an ordered pair $(l,r)\in(\overline{N+1})^2$ such that $l\leq r$ and $\overline{l,r}$ is a maximal subset of $\overline{N+1}$ such that $g(\eta^{-1}(i))\neq\nan$ for all $i\in \overline{l, r-1}$.  Then, let $\MGI_\eta(e) = \{(l_i, r_i)\}_{i\in \overline{M}}$ be the set of all these intervals enumerated in the increasing order by $l_i$. 
	
	In this setting, note that, by construction, $r_i < l_{i+1}$ for all $i\in \overline{M-1}$.  In fact, as an mgi is permitted to be a pair of the form $(n,n)$, for any $i\in \overline{N+1}$ there exists a $j\in \overline{M}$ such that $i\in \overline{l_j, r_j}$; as a result, $r_{i} + 1 = l_{i+1}$, for $i\in \overline{M-1}$.
	
	Now, define the $S$-rule $\tau_e = [U_i \rightarrow (t_{l_i - 1}^r)V_i(t_{r_i}^l)]_{i\in\overline{M}}$ where
\begin{align*}
			U_i &= \gamma_{l_i}~g_{l_i}~\gamma_{l_i +1}~\dots~g_{r_i - 1}~\gamma_{r_i},  \\
			V_i &= \gamma'_{l_i}~t_{l_i}^l~g_{l_i}~t_{l_i}^r~\gamma'_{l_i +1}~t_{l_i+1}^l~g_{l_i+1}~\dots~g_{r_i-1}~t_{r_i -1}^r~\gamma'_{r_i},
		\end{align*}
such that $\gamma_1 = s_1$, $\gamma_1' = \ s_2$, $\gamma_j = \gamma_j' = q_j^*$ for $j \in \overline{2, N+1}$,  $g_j = g(\eta^{-1}(j))$, $t_j^l = t(\eta^{-1}(j), l)$, and $t_j^r = t(\eta^{-1}(j), r)$ for $j\in \overline{N}$.

	Then let $\S=\{\tau_e\mid e\in E_\G\}$.  Clearly, this establishes a bijection $\phi:\S\to E_\G$ given by $\phi(\tau_e)=e$.
	
	Let $e^{-1}=(s_2,s_1,t',g')$.  As $g(i)=\nan$ if and only if $g'(i)=\nan$, it follows immediately that $\MGI_\eta(e^{-1})=\MGI_\eta(e)$.  Further, $t'(i,p)=t(i,p)^{-1}$ for all $i\in\overline{N}$ and $p\in\{l,r\}$.  Hence, by construction, $\tau_e^{-1}=\bigl[V_i\to(t_{l_i-1}^r)^{-1}U_i(t_{r_i}^l)^{-1}\bigr]_{i\in\overline{M}}=\tau_{e^{-1}}\in\S$, and thus $\S$ is an $S$-machine.
	
	As in (a), that $\S$ simulates $\G$ follows directly from construction.
	
	\vspace{.25cm}
	
	(c) Given an $S$-machine $\S$ over a hardware $\H$, the restriction of $\mu_\H$ to the admissible words with trivial tape words induces a bijection between the states of $\S$ and the states of $\G(\S)$.
	
	Similarly, given an $S$-graph $\G$ over a hardware $\H=(\I,\A)$ and a bijection $\eta:\I\to\overline{N}$, the restriction of $\mu_{\G,\eta}$ to the admissible words over $\S_{\G,\eta}(\H)$ with trivial tape words induces a bijection between the states of $\S(\G,\eta)$ and those of $\G$.
	
	Hence, given an $S$-graph $\G'$ and an enumeration $\eta$ of the index set of its corresponding hardware, $\mu_{\S_{\G',\eta}(\H)}\circ\mu_{\G',\eta}^{-1}$ induces a bijection $\omega$ between the states of $\G'$ and those of $\G(\S(\G',\eta))$.
	
	Further, the index set of the hardware of $\G(\S(\G',\eta))$ is $\overline{N}$, so that $\eta$ induces a bijection between the index set $\I$ of the hardware of $\G'$ and that of the hardware of $\G(\S(\G',\eta))$. 
	
	That $\omega$ and $\eta$ induce an isomorphism of the $S$-graphs $\G'$ and $\G(\S(\G',\eta))$ is a direct consequence of the construction of the latter.

\end{proof}
	 
There are several notes to make about \Cref{th-MG}:

\setlength{\leftmargin}{0pt}

\begin{enumerate}[leftmargin=*]

\item Given an $S$-graph $\G$, the particular enumeration $\eta$ used to construct the $S$-machine $\S(\G,\eta)$ will typically be inconsequential.  Hence, the enumeration will often be disregarded in what follows, i.e the notation $\S(\G)$ will be used in place of $\S(\G,\eta)$, $\mu_\G$ in place of $\mu_{\G,\eta}$, etc.

\item In \Cref{th-MG}(a), $\varphi_\S$ and $\varphi_{\G(\S)}$ are natural inverses, and so are bijections between the computations of $\S$ and the computations of $\G(\S)$.

\noindent Similarly, in (b), $\psi_\G$ and $\psi_{\S(\G)}$ are inverses, and hence bijections between the computations of $\G$ and those of $\S(\G)$.  

\item Given an $S$-machine hardware $\H$, the \emph{natural correspondence} $\mu_\H$ does not depend on the particular $S$-machine $\S$ over $\H$ that is being considered.  

\noindent Conversely, in the $S$-graph setting, the construction of the \emph{natural correspondence} $\mu_{\G,\eta}$ differs for distinct $S$-graphs $\G$ over $\H$ and for distinct enumerations $\eta$ of the index set of $\H$.

\item While \Cref{th-MG}(c) exhibits the `Partial Duality' of this construction, \Cref{ex-G2S} below will show that `Duality' does not hold: If $\S'$ is an $S$-machine over the hardware $\H=(Y,Q)$ where $Y=(Y_i)_{i\in\overline{N}}$ and $\eta:\overline{N}\to\overline{N}$ is a bijection, then it does not necessarily hold that $\S' \cong\S(\G(\S'), \eta)$.

\end{enumerate}

\begin{example}[S-machine to S-graph] \label{ex-S2G}
	Let $\S$ be an S-machine consisting of the $S$-rules $\tau_1^{\pm1}$ and $\tau_2^{\pm1}$, where
	\begin{align*}
		\tau_1 =&[p_2'q_2' \rightarrow p_2'q_2',~p_1aq_1r_1 \rightarrow \delta^{-1}p_2a^2q_1r_1\delta],  \\
		\tau_2 =&[p_2 \rightarrow p_2a].
	\end{align*}
	
	The definition of these rules alone is not enough to distinguish the tape letters from the state letters.  In fact, if the hardware of $\S$ is $\H=(Y,Q)$ such that $Y=(Y_i)_{i\in\overline{N}}$ and $Q=(Q_i)_{i\in\overline{N+1}}$, then he value of $N$ cannot be determined, even assuming that for each $i\in\overline{N+1}$, a letter from $Q_i$ appears in one of the rules. 
	
	It can be discerned, though, that $p_2'$, $q_2'$, $p_1$, $p_2$, and $r_1$ are state letters, while $\delta$ and $a$ are tape letters.  Indeed, there exist $i,j,k\in\overline{N+1}$ satisfying $i+1<j<k<N$ such that $p_2' \in Q_i$, $q_2' \in Q_{i+1}$, $p_1, p_2 \in Q_j$, $r_1 \in Q_k$, $\delta \in Y_{j-1} \cap Y_k$, and $a\in Y_j$. However, despite the suggestive name, it cannot be determined whether $q_1$ is a state letter or a tape letter. In the former case, $q_1 \in Q_{j+1}$ and $k = j+2$; in the latter, $q_1 \in Y_j$ and $k = j+1$.
	
	To rectify this, suppose in this particular case that $N=5$, $Q_1 = \{p_2'\}$, $Q_2 = \{q_2'\}$, $Q_3 = \{p_1, p_2\}$, $Q_4 = \{q_1, q_2\}$,  $Q_5 = \{r_1\}$, $Q_6 = \{E\}$, $Y_1 = \emptyset$, $Y_2 = \{\delta\}$, $Y_3 = \{a\}$, $Y_4 = \{\delta\}$, and $Y_5 = \{\delta\}$.  Hence, using the terminology in the previous paragraph, $q_1$ is a state letter, $i=1$, $j=3$, and $k=5$.  
	 
	
	Note also that the makeup of the inverse rules $\tau_1^{-1}$ and $\tau_2^{-1}$ has not been explicitly stated, as will be a common practice hereafter.  For the sake of clarity, these inverse rules are:
	\begin{align*}
		\tau_1^{-1} =& [p_2'q_2' \rightarrow p_2'q_2',~p_2a^2q_1r_1\rightarrow \delta p_2aq_1r_1\delta^{-1}],  \\
		\tau_2^{-1} =&[p_2 \rightarrow p_2 a^{-1}].
	\end{align*}
	
	Now, define the $S$-graph hardware $\G(\H)=(\I,\A)$ by $\I = \overline{5} = \{1, 2,3,4,5\}$ and $\A=(A_i)_{i\in\overline{N}}$ where $A_1 = \emptyset$, $A_2 = \{\delta\}$, $A_3 = \{a\}$, $A_4 = \{\delta\}$, and $A_5 = \bigl\{\delta\}$.  Further, define $\S_\H$ by
	\begin{align*}
	S_\H = \{&(p_2', q_2', p_1, q_1, r_1, E),(p_2', q_2', p_1, q_2, r_1, E), \\ &(p_2', q_2', p_2, q_1, r_1, E),(p_2', q_2', p_2, q_2, r_1, E)\bigr\}.
	\end{align*}
	
	Now, per the construction of \Cref{th-MG}, define the $S$-graph $\G(\S)=(\S_\H,E)$ such that $E = E_{\tau_1} \sqcup E_{\tau_2} \sqcup E_{\tau_1^{-1}} \sqcup E_{\tau_2^{-1}}$, where these subsets are defined as follows.
	
	Because there exists a letter from each of $Q_1, \dots, Q_5$ in the definition of $\tau_1$, $J_{\tau_1}=\{6\}$. So, $|E_{\tau_1}| = |Q_6| = 1$, i.e there exists one edge $e_{\tau_1}$ corresponding to this rule.  By, construction, this edge is:
	\begin{align*}
		e_{\tau_1} = \Bigl(
			&(p_2', q_2', p_1, q_1, r_1, E), \\
			& (p_2', q_2', p_2, q_1, r_1, E), \\
			& \left\{ [\mathbf{2} \rightarrow \mathbf{2}\delta^{-1}],~ 
				[\mathbf{3} \rightarrow a\mathbf{3}],~
				[\mathbf{5} \rightarrow \delta\mathbf{5}]\right\}, \\
			& \left\{ [\mathbf{1} = 1],~[\mathbf{3} = a],~[\mathbf{4} = 1] \right\} 
		\Bigr).
	\end{align*}
	Here, the fragment $[\mathbf{3} \rightarrow a\mathbf{3}]$ is obtained by taking $u_2^3 = a$ and $v_2^3 = a^2$, where the lower index is $2$ because  $p_1aq_1r_1 \rightarrow \delta^{-1}p_2a^2q_1r_1\delta$ is the second subrule in $\tau_1$ and the upper index is $3$ because the corresponding words come from $Y_3$.
			
	Per the definition of the guard labelled on this edge, a computation of $\G$ can only traverse $e_{\tau_1}$ if the associated $L$-word $w=(w_i)_{i\in\overline{5}}$ satisfies $w_1=w_4=1$ and $w_3=5$.  In this case, by the definition of the transformation, if $\tau w=(w_i')_{i\in\overline{5}}$, then $w_2'=w_2\delta^{-1}$, $w_3'=aw_3=a^2$, and $w_5'=\delta w_5$. 

	Conversely, the definition of $\tau_2$ contains only state letters from $Q_3$, so that $J_{\tau_2}=\{1,2,4,5,6\}$.  As a result, 
		$|E_{\tau_2}| = |Q_1| \cdot |Q_2| \cdot |Q_4| \cdot |Q_5| \cdot |Q_6| = |Q_4| = 2$.
	
	The elements of $\Theta_{\tau_2}=\{(p_2',q_2',q_1,r_1,E),(p_2',q_2',q_2,r_1,E)\}$ can then be identified with the corresponding element of $Q_4$, i.e $\Theta_{\tau_2}=\{q_1,q_2\}$.  Following the notation of~\Cref{th-MG}, we call them $E_{\tau_2} = \{e_{2, q_1}, e_{2, q_2}\}$, where
	
	\vspace{-1cm}
	\begin{multicols}{2}
	\begin{align*}
		e_{\tau_2, q_1} = \Bigl(
			&(p_2', q_2', p_2, q_1, r_1, E), \\
			& (p_2', q_2', p_2, q_1, r_1, E), \\
			& \left\{[\mathbf{3} \rightarrow a\mathbf{3}]\right\}, 
			\emptyset 
		\Bigr),
	\end{align*}
	  \break
	\begin{align*}
		e_{\tau_2, q_2} = \Bigl(
			&(p_2', q_2', p_2, q_2, r_1, E), \\
			& (p_2', q_2', p_2, q_2, r_1, E), \\
			& \left\{[\mathbf{3} \rightarrow a\mathbf{3}]\right\}, 
			 \emptyset 	\Bigr).
	\end{align*}
	\end{multicols}
	\noindent Note that actions written on these edges are identical, as would be the case for any pair of edges constructed from the same rule.  Moreover, the guard of this action is empty since $\tau_2$ does not have any subrule with more than one state letter.
	
	The inverse edges are constructed similarly from the inverse rules, so that $E_{\tau_1^{-1}}=\{e_{\tau_1^{-1}}\}$ and $E_{\tau_2^{-1}}=\{e_{\tau_2^{-1},q_1},e_{\tau_2^{-1},q_2}\}$, where:

	\begin{align*}
		e_{\tau_1}^{-1} = e_{\tau_1^{-1}} = \Bigl(
			&(p_2', q_2', p_2, q_1, r_1, E), \\
			& (p_2', q_2', p_1, q_1, r_1, E), \\
			& \left\{ [\mathbf{2} \rightarrow \mathbf{2}\delta],~ 
				[\mathbf{3} \rightarrow a^{-1}\mathbf{3}],~
				[\mathbf{5} \rightarrow \delta^{-1}\mathbf{5}]\right\}, \\
			& \left\{ [\mathbf{1} = 1],~[\mathbf{3} = a^2],~[\mathbf{4} = 1] \right\} 
		\Bigr).
	\end{align*}
	\vspace{-1cm}
	\begin{multicols}{2}
	\begin{align*}
		e_{\tau_2, q_1}^{-1} = e_{\tau_2^{-1}, q_1} = \Bigl(
			&(p_2', q_2', p_2, q_1, r_1, E), \\
			& (p_2', q_2', p_2, q_1, r_1, E), \\
			& \left\{[\mathbf{3} \rightarrow a^{-1}\mathbf{3}]\right\}, 
			\emptyset 
		\Bigr),
	\end{align*}
	  \break
	\begin{align*}
		e_{\tau_2, q_2}^{-1} = e_{\tau_2^{-1}, q_2} = \Bigl(
			&(p_2', q_2', p_2, q_2, r_1, E), \\
			& (p_2', q_2', p_2, q_2, r_1, E), \\
			& \left\{[\mathbf{3} \rightarrow a^{-1}\mathbf{3}]\right\}, 
			 \emptyset 	\Bigr).
	\end{align*}
	\end{multicols}
	
	The edges $e_{\tau_1}$ and $e_{\tau_1}^{-1}$ illustrate a peculiarity of the construction of \Cref{th-MG}:  The fragments of $e_{\tau_1}$ which mention index $3$ are $[\mathbf{3} \rightarrow a\mathbf{3}]$ and $[\mathbf{3} = a]$, which effectively imply that an application of $e_{\tau_1}$ must change the value of the corresponding $L$-word from $a$ to $a^2$ in coordinate 3.  As mentioned above, the fragment $[\mathbf{3}\rightarrow a\mathbf{3}]$ of the transformation arises from the convention of taking $[\mathbf{i} \rightarrow (v_t^i)(u_t^i)^{-1}~\mathbf{i}]\in F_{\tau}^{lt}$ in the proof of \Cref{th-MG}.   However, a syntactically different but semantically identical fragment would have arisen from the convention $[\mathbf{i} \rightarrow \mathbf{i}(u_t^i)^{-1}(v_t^i)]\in F_\tau^{lt}$, or even $[\mathbf{i} \rightarrow (u_t^i)^{-1}\mathbf{i}(v_t^i)]\in F_\tau^{lt}$.  While adopting the former convention would have yielded a similar construction, adopting the latter would have produced an unwanted obstacle: The transformations corresponding to $e_{\tau_1}$ and $e_{\tau_1}^{-1}$ would contain the fragments $[\mathbf{3} \rightarrow a\mathbf{3}a^{-2}]$ and $[\mathbf{3} \rightarrow a^{-2}\mathbf{3}a]$, respectively, and so the actions labelling these edges would not be inverses.

	\Cref{fig-S2G} below represents the S-graph $\G(\S)$.  Note that the inverse edges $e_{\tau_1}^{-1}$, $e_{\tau_2,q_1}^{-1}$, and $e_{\tau_2,q_2}^{-1}$ are absent in this depiction; it is henceforth convention that only one of an edge or its inverse is depicted in such a figure, with the existence of its counterpart implied.
\begin{figure}[hbt]
	\centering
	\include{S2G}
	\caption{S-graph $\G(\S)$.}
	\label{fig-S2G}       
\end{figure} 
\end{example}

\begin{example}[S-graph to S-machine] \label{ex-G2S}
Now consider the S-graph $\G=(\S_\G,E_\G)$ depicted in \Cref{fig-G2S} taken over the $S$-graph hardware $\H=(\I,\A)$ with $\I=\{I,D_1,A,D_2\}$, $A_A=\{a\}$, $A_I=\emptyset$, and $A_{D_1}=A_{D_2}=\{\delta\}$.

\begin{figure}[hbt]
	\centering
	\include{G2S}
	\caption{S-graph $\G$.}
	\label{fig-G2S}       
\end{figure} 

Note the similarities between this $S$-graph and that constructed in \Cref{ex-S2G}, i.e $\G(\S)$ depicted in \Cref{fig-S2G}.  Indeed, $\G$ can be obtained from $\G(\S)$ by deleting the `isolated' state corresponding to $p_1$ and $q_2$, renaming the rest of the states, removing index 4, and renaming the other indices: $1$ corresponds to $I$, $2$ to $D_1$, $3$ to $A$, and $5$ to $D_2$.  Given the makeup of $\G(\S)$, such alterations may seem in some sense natural: Only trivial computations would contain the isolated state and any nontrivial computation would not alter the value of the $L$-word in the coordinate 4.  As $A_1=\emptyset$, it is also natural to consider the removal of index 1; however, it is kept here to demonstrate some nuances that arise in the construction of $\S(\G)$. 

Note that $\G$ and $\G(\S)$ are not isomorphic $S$-graphs, as there is no bijection between their states (or between the index sets of their corresponding hardware).  Hence, while renaming states and indices preserves isomorphism, removal does not.

When defining an S-graph by picture, it is assumed unless explicitly stated otherwise that all indices and all letters from the alphabets are mentioned in at least one of the fragments on the picture.  This defines a `minimal' hardware for the $S$-graph.  Thus, the makeup of $\H$ is implicitly defined in \Cref{fig-G2S}.  

For clarity in this discussion, \Cref{fig-G2S} explicitly names the edges $e_1$, $e_2$, and $e_3$; in practice, such names will be omitted.

Now, since $|\I|=4$, a bijection $\I\to\overline{4}$ must be specified in order to construct $\S(\G)$ as in the proof of \Cref{th-MG}.  

First, define $\eta \colon \I \rightarrow  \overline{4}$ so that the indices are in the `same order' as in \Cref{ex-S2G}, i.e $\eta(I) = 1$, $\eta(D_1) = 2$, $\eta(A) = 3$ and $\eta(D_2) = 4$, and let $\S_\eta = \S(\G, \eta)$.  So, $S_{\G,\eta}(\H)=(Y,Q)$ where $Y=(Y_i)_{i\in\overline{4}}$, $Q=(Q_i)_{i\in\overline{5}}$, $Y_1 = A_I = \emptyset$, $Y_2 = Y_4 = A_{D_1} = A_{D_2} = \{\delta\}$,  $Y_3 = A_A = \{a\}$, $Q_1 = S_\G = \{P1Q1, P2Q1, P2Q2\}$ and $Q_i = \{q_i^*\}$ for $i\in \overline{2, 5}$.

By construction, each S-rule of $\S_\eta$ corresponds to exactly one edge of $\G$. Thus, $$\S_\eta = \{\tau_{e_1}, \tau_{e_2}, \tau_{e_3}, \tau_{e_1^{-1}}, \tau_{e_2^{-1}}, \tau_{e_3^{-1}}\}$$ where $\tau_{e_i}^{-1} = \tau_{e_i^{-1}}$, for $i\in \overline{3}$.  To exhibit the construction of these rules, consider $\tau_{e_1}$:

Since the transformation labelling $e_1$ consists of guard fragments mentioning $I$ and $A$, its maximal guarded intervals with respect to $\eta$ are $$\MGI_\eta(e_1) = \{(1,2), (3,4), (5,5)\}$$ 

Each mgi corresponds to a subrule of $\tau_{e_1}$, yielding:
\begin{align*}
	\tau_{e_1} = \bigl[&P1Q1~q_2^*   \rightarrow P2Q1~q_2^*, \ \ \  q_3^*~a~q_4^* \rightarrow \delta^{-1}~q_3^*~a^2~q_4^*~\delta, \ \ \ q_5^* \rightarrow q_5^*  \bigr].
\end{align*}

Note that the last subrule can be omitted without altering $\S_\eta$, as it consists of only one state letter, $q_5^*\in Q_5$, and $Q_5$ is a singleton.  

As $e_2$ and $e_3$ contain no guard fragments, $\MGI_\eta(e_2)$ and $\MGI_\eta(e_3)$ consist of single-point intervals, i.e they are $\{(i,i)\}_{i\in\overline{5}}$.  The rules $\tau_{e_2}$ and $\tau_{e_3}$ are then defined as follows:

	\begin{multicols}{2}
		\begin{align*}
			\tau_{e_2} = \bigl[&P2Q1 \rightarrow P2Q1, \\
				& q_2^* \rightarrow q_2^*,~
				q_3^* \rightarrow q_3^*~a, \\
				& q_4^* \rightarrow q_4^*,~
				q_5^* \rightarrow q_5^* \bigr],
		\end{align*}
	  \break
		\begin{align*}
			\tau_{e_3} = \bigl[&P2Q2 \rightarrow P2Q2, \\
				& q_2^* \rightarrow q_2^*,~ 
				q_3^* \rightarrow q_3^*~a, \\
				& q_4^* \rightarrow q_4^*,~
				q_5^* \rightarrow q_5^* \bigr].
		\end{align*}
	\end{multicols}
	
As pointed out in the proof of \Cref{th-MG}, for $i\in\overline{3}$, the rule $\tau_{e_i^{-1}}$ constructed from the inverse edge corresponds to the inverse $S$-rule $\tau_{e_i}^{-1}$.
 
On the other hand, define the enumeration $\nu \colon \I_\G \rightarrow \overline{4}$ given by $\nu(D_1) = 1$, $\nu(D_2) = 2$, $\nu(I) = 3$, $\nu(A)=4$ and construct $\S_{\nu} = \S(\G, \nu)$.  For $e\in E_\G$, denote the corresponding $S$-rule of $\S_\nu$ by $\sigma_e$.

Then, in this case, $\MGI_\nu(e_1) = \{(1,1), (2,2), (3,5)\}$, $\MGI_\nu(e_2)=\MGI_\nu(e_3)=\{(i,i)\}_{i\in\overline{5}}$, and
\begin{align*}
	\sigma_{e_1} = \bigl[&P1Q1   \rightarrow P2Q1, \ \ \  q_2^* \rightarrow \delta^{-1}q_2^*\delta, \ \ \ q_3^*q_4^*aq_5^* \rightarrow q_3^*q_4^*a^2q_5^* \bigr],
\end{align*}

	\vspace{-1cm}
	\begin{multicols}{2}
		\begin{align*}
			\sigma_{e_2} = \bigl[&P2Q1 \rightarrow P2Q1, \\
				& q_2^* \rightarrow q_2^*,~
				q_3^* \rightarrow q_3^*, \\
				& q_4^* \rightarrow q_4^*~a,~
				q_5^* \rightarrow q_5^* \bigr],
		\end{align*}
	  \break
		\begin{align*}
			\sigma_{e_3} = \bigl[&P2Q2 \rightarrow P2Q2, \\
				& q_2^* \rightarrow q_2^*,~ 
				q_3^* \rightarrow q_3^*, \\
				& q_4^* \rightarrow q_4^*~a,~
				q_5^* \rightarrow q_5^* \bigr].
		\end{align*}
	\end{multicols}
	
As above, for all $i\in\overline{3}$, by construction $\sigma_{e_i^{-1}}=\sigma_{e_i}^{-1}$.

While $\sigma_2$ and $\sigma_3$ resemble $\tau_{e_2}$ and $\tau_{e_3}$, 
there are some major differences between $\sigma_{e_1}$ and $\tau_{e_1}$.  For one thing, though each rule consists of three subrules, the numbers of state letters comprising these subrules do not coincide.  In fact, two subrules of $\tau_{e_1}$ contain tape letters, while only one subrule of $\sigma_{e_1}$ does so. 

Evidently, $\S_\eta$ and $\S_\nu$ are not isomorphic.  To some extent, this observation elucidates the remark made earlier about the limited usefulness of $S$-machine isomorphisms.

\end{example}

\section{Relations and Relational Algebra} \label{sec-relations}

Many of the definitions and arguments in what follows are constructed using binary relations.  In this section, we recall the basic definitions, to provide useful tools of relational algebra that will be used in many of the proofs that follow, and to fix some notational conventions that will simplify the discussion in this setting.

In what follows, a relation $R$ is taken to be a binary relation on a set $X$ called the \emph{domain} of $R$, i.e $R\subseteq X^2$.  A relation $R$ is called: 

\begin{itemize}

\item \emph{reflexive} if $(x,x)\in R$ for all $x\in X$,

\item \emph{symmetric} if for all $(y,x)\in R$ whenever $(x,y)\in R$,

\item \emph{transitive} if $(x,z)\in R$ whenever $(x,y),(y,z)\in R$, and

\item \emph{an equivalence relation} if it is reflexive, symmetric, and transitive.

\end{itemize}

Given a relation $R$ on $X$, the notation $x R y$ is used to indicate $(x, y)\in R$.  The \emph{inverse} of $R$, denoted $R^{-1}$, is the relation on $X$ given by $R^{-1} = \left\{(y, x)\mid(x, y)\in R\right\}$.  Further, given relations $R$ and $S$ with the same domain $X$, the product of $R$ and $S$ is the relation over $X$ defined by 
$$R\circ S = \left\{(x, z)\mid x R y \text{ and } y S z \text{ for some } y\in X \right\}$$ 
As an operation on the set of relations with domain $X$, the product is associative but not commutative.  The powers of a relation $R$ on $X$ are defined in a standard fashion: 

\begin{itemize}

\item $R^n = \underbrace{R\circ R\circ \dots \circ R}_{n \text{ times}}$ if $n>0$,  

\item $R^n= \underbrace{R^{-1}\circ R^{-1}\circ \dots\circ R^{-1}}_{-n \text{ times}}$ if $n<0$, and 

\item $R^0 = \Id_X$

\end{itemize}
where the \emph{identity relation} $\Id_X$ is the smallest reflexive relation on $X$, i.e $$\Id_X = \{(x,x)\mid x\in X\}$$  In general, the identity relation is denoted $\Id$ when the domain is contextually clear.

If $R$ is a relation on $X$ and $A\subseteq X$, then $A$ is \emph{closed under $R$} if for all $(x_1,x_2)\in R$, $x_1\in A$ if and only if $x_2\in A$.

Given a relation $R$ on $X$, the \emph{reflexive symmetric transitive closure} of $R$, also simply called the \emph{rst-closure} and denoted $\RST R$ or $R^*$, is the smallest equivalence relation on $X$ containing $R$.  Note that this closure always exists, and indeed is given by $R^* = \bigcup_{i\in \Z}R^i$.

\begin{lemma}\label{lem-closure}
	Let $\{R_j\}_{j\in J}$ be a collection of relations on $X$ and suppose the subset $A\subseteq X$ is closed under $R_j$ for all $j\in J$.  Then for $R=\RST\bigcup_{j\in J} R_j$, $A$ is a disjoint union of equivalence classes of $R$.
\end{lemma}

\begin{proof}

By construction, for any $(x_1,x_2)\in R$, there exist $k\in\N$, $j_1,\dots,j_k\in J$, and $y_0,\dots,y_k\in X$ such that $y_0=x_1$, $y_k=x_2$, and for all $i\in\overline{k}$, $(y_{i-1},y_i)\in R_{j_i}$ or $(y_i,y_{i-1})\in R_{j_i}$.

Then for all $i\in\overline{k}$, since $A$ is closed under $R_j$ for all $j\in J$, $y_{i-1}\in A$ if and only if $y_i\in A$.  As a result, $x_1\in A$ if and only if $x_2\in A$.  In particular, for any $a\in A$, the equivalence class of $R$ containing $a$ must be a subset of $A$.

\end{proof}

If $R$ is a relation on $X\times Y$ and $A_1,A_2\subseteq X$, then the \emph{$X$-restriction} of $R$ to the pair $(A_1,A_2)$ is defined to be $$R\pr{A_1,A_2}=\left\{\bigl((x_1,y_1),(x_2,y_2)\bigr)\in R\mid x_1\in A_1 \text{ or } x_2\in A_2\right\}$$
Note that, per this definition, it is not generally the case that $R\pr{A_1,A_2}\subseteq(A_1\times Y)\times(A_2\times Y)$. If this does hold, then the $X$-restriction is called \emph{proper} and $R$ is said to be \emph{properly $X$-restricted} to the pair $(A_1,A_2)$.  Hence, $R$ is properly $X$-restricted to $(A_1,A_2)$ if and only if the conjunction `or' in the definition of the restriction can be replaced by `and'. 

If $A=A_1=A_2$, then $R\pr{A_1,A_2}$ is simply denoted $R\pr{A}$.  If the restriction $R\pr{A}$ is proper, the $R$ is simply said to be properly $X$-restricted to $A$.  Hence, $R$ is properly restricted to $A$ if and only if $A\times Y$ is closed under $R$.

On the other hand, if $R$ is a relation on $X\times Y$ and $B_1,B_2\subseteq Y$, then the \emph{$Y$-restriction} of $R$ to the pair $(B_1,B_2)$ is defined to be $$R^{B_1,B_2}=\left\{(x_1,y_1,x_2,y_2)\in R\mid y_1\in B_1 \text{ and } y_2\in B_2\right\}$$
Note the non-uniformity in the definitions of these restrictions, as a $Y$-restriction is necessarily `proper' in the sense of an $X$-restriction.  This disparity exists because the relevant relations considered in what follows will be constructed in such a way that restrictions to the second coordinate are only pertinent if they are `proper'.

As with $X$-restrictions, if $B=B_1=B_2$, then $R^{B_1,B_2}$ is simply denoted $R^B$. 

An $X$-restriction of $R^{B_1,B_2}$ is defined in much the same way as for $R$.  Hence, for $A_1,A_2\subseteq X$,
$$R^{B_1,B_2}\pr{A_1,A_2}=\{(x_1,y_1,x_2,y_2)\in R^{B_1,B_2}\mid x_1\in A_1 \text{ or } x_2\in A_2\}$$
As above, such a restriction is proper if the `or' can be replaced with `and', which is to say if $R^{B_1,B_2}\pr{A_1,A_2}$ is contained in $(A_1\times B_1)\times(A_2\times B_2)$.

If $Y_1$ and $Y_2$ are singletons with $Y_i=\{y_i\}$, then the $Y$-restriction of $R$ to $(y_1,y_2)$ can also be identified with a relation on $X$, i.e $$R^{y_1,y_2}=\left\{(x_1,x_2)\in X^2\mid (x_1,y_1,x_2,y_2)\in R\right\}$$
As such, $R^{y_1,y_2}$ is interpreted both as a relation on $X$ and as one on $X\times Y$, with context making it clear which interpretation is being used.  Thus, $$R=R^Y=\bigsqcup_{(y_1,y_2)\in Y^2}R^{y_1,y_2}$$

\section{Objects, Operations, and Counters} \label{sec-operations}


Let $\G=(S_\G,E_\G)$ be an $S$-graph over the hardware $\H$.  Then the \emph{reachability relation} of $\G$, denoted $\REACH_\G$, is the binary relation on the configurations of $\G$ given by relating the starting configuration of a computation of $\G$ to the ending configuration.  In other words, 
$$\REACH_\G= \left\{(w_1, s_1, w_2, s_2)~|~ \varepsilon(w_1, s_1, w_2, s_2) \neq \emptyset \right\} \subseteq (L_\H\times S_\G)^2$$
If $(w_1, s_1, w_2, s_2) \in \REACH_\G$, then the configuration $(w_2, s_2)$ \emph{can be reached} from $(w_1, s_1)$ (via a computation of $\G$).  By the definition of $S$-graph computations, $\REACH_\G$ is an equivalence relation on the set of configurations.

The notational conventions for relations on products described in \Cref{sec-relations} are adopted for $\REACH_\G$.  So, for any $W_1,W_2\subseteq L_\H$, the \emph{$L_\H$-restriction} of $\REACH_\G$ to the pair $(W_1, W_2)$ is the relation given by
$$\REACH_\G\pr{W_1, W_2} = \{(w_1, s_1, w_2, s_2)\in\REACH_\G \mid w_1 \in W_1 \text{ or } w_2\in W_2\}$$
which is \emph{proper} if it is contained in $(W_1\times S_\G)\times(W_2\times S_\G)$. 

Further, for $s_1,s_2\in S_\G$, the \emph{$S_\G$-restriction} of $\REACH_\G$ to $(s_1,s_2)$ is
$$\REACH_\G^{s_1, s_2} = \bigl\{(w_1, w_2) \mid (w_1, s_1, w_2, s_2)\in \REACH_\G  \bigr\}\subseteq L_\H\times L_\H$$
also identified with a relation on $L_\H\times S_\G$.  Hence, such an $S_\G$-restriction of $\G$ describes the computations supported by paths between particular states of $\G$.
 
Note that $\REACH_\G^{s_1,s_2}$ is not necessarily reflexive, transitive, or symmetric.

\begin{definition}\label{def-operation}
	Let $\G=(S,E)$ be an $S$-graph over the hardware $\H=(\I,\A)$ whose underlying directed graph is (strongly) connected.  Suppose there exists a partition $\I=\I^\ext \sqcup \I^\val \sqcup \I^\int$ and $s,f\in S$ satisfying the following condition:
	
	\begin{enumerate}
		\item [(O)] Let $\rho = e_1, \dots, e_N$ be a path in $\G$ from $s_1$ to $s_2$ where $s_1, s_2 \in  \{s,f\}$. For $i\in\I^\int$, suppose there exists a fragment $f$ mentioning $i$ such that $f\in e_j$ for some $j\in\overline{N}$.  Then there exist $l_1, l_2\in\overline{N}$ with $l_1\leq l_2$ such that $[\mathbf{i}=1]\in e_{l_1}\cap e_{l_2}^{-1}$ and for all $m\in\overline{N}-\overline{l_1,l_2}$, no fragment in $e_m$ mentions $i$.
	\end{enumerate}
Then, $\op=(\G,\I^\ext,\I^\val,\I^\int,s,f)$ is called an \emph{operation}.  

\end{definition}

For simplicity, the defining components of an operation are named in a way to indicate their relationship with the operation.  Hence, an operation $\op$ is denoted $\op=(\G_\op,\I_\op^\ext,\I_\op^\val,\I_\op^\int,\op_s,\op_f)$ where $\G_\op=(S_\op,E_\op)$ is an $S$-graph over the hardware $\H_\op=(\I_\op,\A_\op)$ with defining index set partition $\I_\op=\I_\op^\ext\sqcup\I_\op^\val\sqcup\I_\op^\int$.

Further, it is convenient to allow an operation $\op$ to be identified with the $S$-graph $\G_\op$.  So, for example, a computation of $\G_\op$ may be referred to as a computation of $\op$.

The elements of $\I_\op^\ext$ are called the \emph{external tapes} of $\op$, those of $\I_\op^\val$ the \emph{value tapes}, and those of $\I_\op^\int$ the \emph{internal tapes}.  

The \emph{external language of $\op$} is the set $L_\op^\ext = \bigtimes_{i\in\I_\op^\ext} L_i$.  Similarly, the \emph{value} and \emph{internal languages} of $\op$ are the sets $L_\op^\val = \bigtimes_{i\in\I_\op^\val}L_i$ and $L_\op^\int = \bigtimes_{i\in\I_\op^\int}L_i$, respectively.  Finally, the \emph{language of $\op$}, denoted $L_\op$, is the set of $L$-words $L_{\H_\op}$.  Thus, $L_\op = L_\op^\ext \times L_\op^\val \times L_\op^\int$.  

The specified states $\op_s$ and $\op_f$ are called the \emph{starting state} and the \emph{finishing state} of $\op$, respectively.  The subset $\{\op_s,\op_f\}\subseteq S_\op$ is denoted $\SF_\op$.  Hence, condition (O) can be restated as a restriction on paths in $\G_\op$ with endpoints in $\SF_\op$.

Note that condition (O) has a certain level of symmetry: Let $\rho=e_1,\dots,e_N$ be a path as in the statement of (O) and consider the inverse path $\rho^{-1} = e_N^{-1}, \dots, e_1^{-1}$.  As $\rho$ is a path from $s_1$ to $s_2$ for $s_1,s_2\in\SF_\op$, $\rho^{-1}$ is a path from $s_2$ to $s_1$, and so is a path subject to the condition.  For any $i\in\I_\op^\int$ and any edge $e$ of $\G_\op$, there exists a fragment $f$ mentioning $i$ such that $f\in e$ if and only if there exists a fragment $f'$ mentioning $i$ such that $f'\in e^{-1}$.  Hence, there exists a $j\in\overline{N}$ such that $e_j$ contains a fragment mentioning $i$ if and only if the same is true for $e_j^{-1}$, in which case the existence of the indices $l_1,l_2\in\overline{N}$ ensuring the condition for $\rho$ is equivalent to the existence of the analogous indices for $\rho^{-1}$.

This observation is useful when an operation is defined by picture, as it expedites the process of verifying condition (O).  For example, if one verifies that condition (O) holds for any path from $\op_s$ to $\op_f$, then it also holds for any path from $\op_f$ to $\op_s$.

Let $\eps$ be a computation of $\op$ between the configurations $(w_1,s_1)$ and $(w_2,s_2)$.  Suppose that $s_1,s_2\in\SF_\op$ and $w_1\bigl(\I_\op^\int\bigr)=w_2\bigl(\I_\op^\int\bigr)=(1)_{i\in\I_\op^\int}$.  Then $\eps$ is called an \emph{input-output computation}.

\begin{definition}\label{def-object}

Let $\op_1,\dots,\op_k$ be operations.  Suppose there exists a subset $\I^\val\subseteq\I^*$ such that $\I_{\op_m}^\val=\I^\val$ for all $m\in\overline{k}$.  Further, suppose that $\A_{\op_m}(i)=\A_{\op_n}(i)$ for all $i\in\I^\val$ and all $m,n\in\overline{k}$.  Then, $\Obj=(\I^\val,\op_1,\dots,\op_k)$ is called an \emph{object}.  

\end{definition}

Similar to the notational conventions for operations, the set of value tapes of any operation of $\Obj$ is denoted $\I_\Obj^\val$ to emphasize its relationship to the object.  Further, the notation $\op\in\Obj$ is used to indicate that $\op$ is one of the operations that define the object.

For clarity, it is assumed that any object contains a unique copy of each of its operations.  As such, for distinct objects $\Obj_1$ and $\Obj_2$, there is no operation $\op$ such that $\op\in\Obj_1$ and $\op\in\Obj_2$.

Note that an object with a single operation can be identified with its operation, while any single operation can be identified with such an object.  If the operation has no value tapes, then the corresponding object is called a \emph{proper operation}.

By definition, for any $\op,\op'\in\Obj$, $L_\op^\val=L_{\op'}^\val$.  Hence, the \emph{value language of $\Obj$}, denoted $L^\val_\Obj$, is defined as the value language of any of its operations.

Let the \emph{initial value} of $\Obj$ be the `trivial' tuple $1_\Obj^\val=(1)_{i\in \I^\val_\Obj} \in L^\val_\Obj$.  In this case and, more generally when the index set is clear, such a trivial tuple will simply be denoted 1.

To any object $\Obj$, associate an (undirected) graph $\G_\Obj^\val$ whose vertices are identified with the words of $L_\Obj^\val$ and such that there exists an edge connecting $v_1,v_2\in L_\Obj^\val$ if and only if there exists $\op\in\Obj$ and an input-output computation of $\op$ between a pair of configurations $(w_1,s_1),(w_2,s_2)\in L_\op\times\SF_\op$ such that $w_i\bigl(\I_\Obj^\val\bigr)=v_i$ for $i=1,2$.  Then, define the \emph{values} of $\Obj$, denoted $\VAL_\Obj$, as the connected component of $\G_\Obj^\val$ containing the initial value.  

\begin{proposition}\label{prop-nse} Let $\op\in \Obj$ be an operation and set $1^\int_\op=(1)_{i\in\I^\int_\op}\in L_\op^\int$.  Then $\REACH^{\SF_\op}_\op$ is properly restricted to $L_\op^\ext \times\VAL_\Obj \times 1^\int_\op$.
\end{proposition}

\begin{proof}

	Let $(w_1,s_1,w_2,s_2)\in\REACH_\op^{\SF_\op}\pr{V}$ where $V=L_\op^\ext\times\VAL_\Obj\times1^\int_\op$.  
	
	Assume that $w_2\in V$ and let $\eps$ be a computation of $\op$ between $(w_1,s_1)$ and $(w_2,s_2)$ with supporting path $\rho=e_1,\dots,e_k$.  Set $$\eps=\bigl((w_0',s_0'),e_1,(w_1',s_1'),e_2,\dots,(w_{k-1}',s_{k-1}'),e_k,(w_k',s_k')\bigr)$$ where $(w_1,s_1)=(w_0',s_0')$ and $(w_2,s_2)=(w_k',s_k')$.  
	
	Fix $i\in\I_\op^\int$.  Note that, in general, if an edge $e\in E_\op$ contains no fragment mentioning an index, then any application of $e$ leaves the $L$-word unchanged in the corresponding coordinate.  Hence, if $e_j$ contains no fragment mentioning $i$ for all $j\in\overline{k}$, then $w_1(i)=w_2(i)=1$.  
	
	Otherwise, condition (O) applies to $\rho$, i.e there exists $l\in\overline{n}$ such that $[\i=1]\in e_l$ and no fragment of $e_m$ mentions $i$ for all $m\in\overline{1,l-1}$.  So, $w_1(i)=w_{l-1}'(i)$ and $e_l$ is applicable to $(w_{l-1}',s_{l-1}')$.  But then the guard fragment of $e_l$ implies $w_{l-1}'(i)=1$, so that $w_1(i)=1$.
	
	Hence, $w_1\bigl(\I_\op^\int\bigr)=1^\int_\op$, and so $\eps$ is an input-output computation of $\op$.  By definition, this means there is an edge in $\G_\Obj^\val$ connecting $w_1\bigl(\I^\val_\Obj\bigr)$ with $w_2\bigl(\I^\val_\Obj\bigr)$.  Thus, $w_1\bigl(\I^\val_\Obj\bigr)\in\VAL_\Obj$ and so $w_1\in V$.
	
	If instead one assumes $w_1\in V$, then the analogous argument using the symmetry of condition (O) implies that $w_2\in V$.
	
\end{proof}

Given an object $\Obj$ and an operation $\op\in\Obj$, the \emph{natural domain} of $\op$ is defined to be the set $\ND_\op=L_\op^\ext \times \VAL_\Obj\subseteq L_\op^\ext\times L_\Obj^\val$.  Note that, though the notation does not suggest it, the natural domain depends on the makeup of the object being considered.  However, there is no ambiguity as $\Obj$ is assumed to be the unique object for which $\op\in\Obj$.

For notational purposes, it is sometimes convenient to expand the natural domain to $$\ND_\op^1 = \ND_\op \times 1^\int_\op  = L_\op^\ext \times \VAL_\Obj \times 1^\int_\op \subseteq L_\op$$ By \Cref{prop-nse}, $\REACH^{\SF_\op}_\op$ is properly restricted to $\ND_\op^1$.

Define the relation $\IO_\op^1 = \REACH^{\SF_\op}_\op\pr{\ND_\op^1} \subseteq \left(\ND_\op^1\times\SF_\op\right)^2$.  Here, $\IO$ stands for \emph{input-output} as its composition is determined by input-output computations.

Then, define the relation $\IO_\op$ on $\ND_\op\times\SF_\op$ given by $(w_1,s_1,w_2,s_2)\in\IO_\op$ if and only if there exist $w_1',w_2'\in\ND_\op^1$ such that $w_i'\bigl(\I^\ext_\op\sqcup\I_\Obj^\val\bigr)=w_i$ for $i=1,2$ and $(w_1',s_1,w_2',s_2)\in\IO_\op^1$.  Hence,
$$
	\IO_\op^{s_1,s_2} = \bigl\{(w_1,w_2)\in\bigl(\ND_\op\bigr)^2\mid(w_1\sqcup1_\op^\int,w_2\sqcup1_\op^\int)\in\REACH_\op^{s_1,s_2}\bigr\}
	$$
	for all $s_1,s_2\in\SF_\op$.

Note that, as $\REACH_\op$ is an equivalence relation on $L_\op\times S_\op$, $\IO_\op$ is an equivalence relation on $\ND_\op \times \SF_\op$.

As $\IO_\op$ is a relation on $\ND_\op\times\SF_\op$, an $\SF_\op$-restriction $\IO_\op^{s_1,s_2}$ for $s_1,s_2\in\SF_\op$ is defined as in \Cref{sec-relations}.  For simplicity of notation, this restriction is denoted $\IO_\op^{s,f}$ if $s_1=\op_s$ and $s_2=\op_f$.  Similarly, there are the $S$-restrictions $\IO_\op^{f,s}$, $\IO_\op^{s,s}$, and $\IO_\op^{f,f}$.

By the symmetry of $S$-graph computations, $(w_1,w_2)\in\IO_\op^{s,f}$ if and only if $(w_2,w_1)\in\IO_\op^{f,s}$.  Hence, to understand the relation $\IO_\op$, it suffices to understand just one of these two subsets.

Further, the existence of trivial computations implies that $\IO_\op^{s,s}$ and $\IO_\op^{f,f}$ both contain $\Id$.

For any object $\Obj$, any operation $\op \in \Obj$, and any external tape $a\in \I_\op^\ext$, then $a$ is called \emph{immutable} in $\op$ if $w_1(a) = w_2(a)$ whenever $(w_1,s_1,w_2,s_2)\in\IO_\op$.

\begin{example}[Copy operation]\label{ex-copy}
	For any finite subset $A\subseteq\A^*$, consider the object $\cpy^A$ identified with its single operation of the same name whose associated $S$-graph is depicted in \Cref{fig-COPY}.  When the set $A$ is contextually clear, $\cpy^A$ is simply denoted $\cpy$.
	
	\begin{figure}[hbt]
	\centering
	\include{COPY}
	\caption{Operation $\cpy(a,b)$} 
	\label{fig-COPY}       
\end{figure} 
	
	While it can be deduced from the contents of the figure that $\I_{\cpy}=\{a,b,c\}$, the defining partition $\I_{\cpy}=\I_{\cpy}^\ext\sqcup\I_{\cpy}^\val\sqcup\I_{\cpy}^\int$ is not immediately clear.  To rectify this, a notational convention is introduced in order to efficiently indicate these sets.
	
	Given an object $\Obj$ and an operation $\op\in\Obj$, a figure depicting the $S$-graph associated to $\op$ is titled $\Obj\bigl[a_1,\dots,a_k].\op\bigl(b_1,\dots,b_\ell\bigr)$, where $[a_1,\dots,a_k]$ is some enumeration of $\I_\Obj^\val$ and $(b_1,\dots,b_\ell)$ is some enumeration of $\I_\op^\ext$.  If $\I^\val_\Obj=\emptyset$, then the brackets are omitted, so that the figure is denoted $\Obj.\op(\I_\val^\ext)$.  Finally, if $\Obj$ is a proper operation, then reference to the object is dropped, instead identifying it with its single operation.
	
	So, as \Cref{fig-COPY} is titled $\cpy(a,b)$, it is a proper operation with $\I_\cpy^\ext=\{a,b\}$ and $\I_\cpy^\int=\{c\}$.
	
	As for $\A_\cpy=\{A^\cpy_a,A^\cpy_b,A^\cpy_c\}$, another piece of notation used in \Cref{fig-COPY} must be explained.  On each of the loops at the states $s_1$ and $s_2$, the phrase `for $x\in A$' is written on the label even though it is not a fragment of the corresponding action.  This expression indicates that the edge represents multiple edges with the same tail and head as the depicted edge; in this case, it stipulates that there is a correspondence between these edges and the elements of $A$, with the action built in the way indicated in the figure.  Of course, the graph also contains the implicitly constructed inverse edges for each of these edges.  Hence, despite the seeming simplicity of the diagram in the figure, $\G_\cpy$ consists of $4|A|+6$ edges.
	
	So, since every letter of $A$ makes an appearance in a fragment mentioning each of the indices of $\I_\cpy$, it follows that $A^\cpy_a=A^\cpy_b=A^\cpy_c=A$.
	
	The elements of $\SF_\cpy$, i.e the starting and finishing states $\cpy_s$ and $\cpy_f$, are denoted in \Cref{fig-COPY} simply $s$ and $f$, respectively.  This is the standard convention in what follows when defining an operation by a diagram, as the operation to which they belong will be clear.  However, this means that when multiple operations are present, multiple vertices will be labelled $s$ and multiple will be labelled $f$, so it is worth mentioning that such similarly labelled states are not identified.

	With these conventions, and so the makeup of $\cpy$, understood, it is now possible to verify that $\cpy$ satisfies condition (O) and is thus a well-defined operation.  Note that there are only two edges of $\cpy$ whose tails are in $\SF_\cpy$: The edge $e=(s,s_1,\act)$ depicted in \Cref{fig-COPY} and the edge $e'=(f,s_2,\act')$ whose inverse is depicted in \Cref{fig-COPY}.  So, for any path $\rho=e_1,\dots,e_N$ between elements of $\SF_\cpy$, it must hold that $e_1,e_N^{-1}\in\{e,e'\}$.  But it follows immediately from the definition of $e$ and $e'$ in the diagram that $[\mathbf{c}=1]\in e$ and $[\mathbf{c}=1]\in e'$.  Hence, (O) is satisfied.

	In the notation $\Obj[a_1,\dots,a_k].\op(b_1,\dots,b_\ell)$ introduced through this example, the enumeration of the value and external tapes allows for the efficient renaming of these tapes.  So, in the example of $\cpy$, given indices $i,j\in\I^*$, $\cpy^A(i,j)$ denotes the graph from \Cref{fig-COPY} in which $a$ is renamed to $i$, $b$ is renamed to $j$, and $c$ is renamed to $c'$ if $c=i$ or $c=j$ (in order to keep the internal and external tapes disjoint).  Moreover, if $i=x$ or $j=x$, then the indeterminate $x$ in \Cref{fig-COPY} used to represent any letter of $A$ may need to be renamed.  With this notation in hand, note that $\cpy(a, b)\neq \cpy(b, a)$; see \Cref{fig-COPY2} below for examples of $\cpy^A$.  
	
	\begin{figure}[hbt]
	\centering
	\include{COPY2}
	\caption{Operations $\cpy(x,y)$, $\cpy(b,a)$ and $\cpy(a,c)$.}
	\label{fig-COPY2}       
\end{figure}  
	
	Further, for simplicity, when referring to a word $v\in L^\val_\Obj$, the coordinates of $v$ will be listed in the order of the enumeration, i.e $v=(v_1,\dots,v_k)$ where $v_i=v(a_i)$ for all $i\in\overline{k}$.  Similarly, for any $w\in L_\op^\ext$, $w=(w_1,\dots,w_\ell)$ where $w_i=w(b_i)$ for all $i\in\overline{\ell}$.  Hence, in terms of $\cpy$, an element $w\in L_\cpy^\ext$ is denoted $w=(w_a,w_b)$ where $w_a=w(a)$ and $w_b=w(b)$.

\begin{lemma}\label{lem-reach-cpy}

Let $\eps$ be a computation of $\cpy=\cpy^A(a,b)$ between the configurations $(w_1,t_1)$ and $(w_2,t_2)$.  For $i=1,2$, set $w_i=(w_i^a,w_i^b,w_i^c)$, where $w_i^j=w_i(j)$ for $j\in\I_\cpy$.

\begin{enumerate}

\item Suppose $t_1=\cpy_s$.  If $w_1^b\neq1$ or $w_1^c\neq1$, then $w_1=w_2$ and $t_1=t_2$


\item $w_1^cw_1^a=w_2^cw_2^a$

\item If $t_1\in\{s_2,\cpy_f\}$ and $w_1^a\neq w_1^b$, then $t_2\in\{s_2,\cpy_f\}$ and $w_2^a\neq w_2^b$.  Moreover, in this case $\bigl(w_1^a\bigr)^{-1}w_1^b=\bigl(w_2^a\bigr)^{-1}w_2^b$.

\end{enumerate}

\end{lemma}

\begin{proof} \

\noindent (1) Set $n=|\eps|$ and suppose $n>0$.  Then, the computation path of $\eps$ is $\rho_\eps=e_1,\dots,e_n$, so that $e_1$ is applicable to $(w_1,t_1)$.  In particular, since $t_1=\cpy_s$, $e_1$ must be the only edge with tail $\cpy_s$.  But since this edge has guard fragments $[\mathbf{b}=1]$ and $[\mathbf{c}=1]$, this implies $w_1^b=w_1^c=1$.

Hence, if $w_1^b\neq1$ or $w_1^c\neq1$,then $n=0$ and so $\eps$ is a trivial computation.

\medskip

\noindent (2) Let $(v_1,p_1)$ be a configuration of $\cpy$ and let $e$ be an edge that is applicable to $(v_1,p_1)$.  Set $(v_2,p_2)=e(v_1,p_1)$ and $v_i=(v_i^a,v_i^b,v_i^c)$ as in the statement.  Then it suffices to show that $v_1^cv_1^a=v_2^cv_2^a$.

This is immediate if $e$ has no transformation fragments, so it suffices to assume that $e$ is a loop at $s_1$ or at $s_2$.  But then there exists $x\in A\cup A^{-1}$ such that $v_2^a=x^{-1}v_1^a$ and $v_2^c=v_1^cx$.  Hence, $v_2^cv_2^a=v_1^cv_1^a$.

\medskip

\noindent (3) As in (2), let $(v_1,p_1)$ be a configuration, let $e$ be an edge that is applicable to $(v_1,p_1)$, and set $(v_2,p_2)=e(v_1,p_1)$ with $v_i=(v_i^a,v_i^b,v_i^c)$ for $i=1,2$.  Suppose $p_1\in\{s_2,\cpy_f\}$.

If $e$ contains no transformation fragment, then $v_1^a=v_2^a$ and $v_1^b=v_2^b$, so that immediately $\bigl(v_1^a\bigr)^{-1}v_1^b=\bigl(v_2^a\bigr)^{-1}v_2^b$.  Otherwise, $e$ must be a loop at $s_2$; but then there exists $x\in A\cup A^{-1}$ such that $v_2^a=xv_1^a$ and $v_2^b=xv_1^b$, so that again $\bigl(v_2^a\bigr)^{-1}v_2^b=\bigl(v_1^a\bigr)^{-1}v_1^b$.

Denote $\eps=\bigl((u_0,q_0),e_1,(u_1,q_1),e_2,\dots,e_n,(u_n,q_n)\bigr)$ with $n\geq0$, $(u_0,q_0)=(w_1,t_1)$, and $(u_n,q_n)=(w_2,t_2)$.  Let $m\in\{0,\dots,n\}$ be the maximal index such that $q_m\in\{s_2,\cpy_f\}$.

Suppose $m<n$.  Then for all $i\in\overline{m+1}$, the tail of $e_i$ is $s_2$ or $\cpy_f$, so that the above argument yields $\bigl(u_i^a\bigr)^{-1}u_i^b=\bigl(w_1^a\bigr)^{-1}w_1^b$.  In particular, since $w_1^a\neq w_1^b$, then $u_m^a\neq u_m^b$.  But the only edge whose tail is in $\{s_2,\cpy_f\}$ and whose head is in $\{s_1,\cpy_s\}$ is the edge $e$ going from $s_2$ to $s_1$ which has guard fragments $[\mathbf{a}=1]$ and $[\mathbf{b}=1]$.  So, since $u_m^a\neq u_m^b$ and $q_m\in\{s_2,\cpy_f\}$, $e$ is not applicable to $(u_m,q_m)$.  Hence, $e_{m+1}\neq e$ and so $q_{m+1}\in\{s_2,\cpy_f\}$, contradicting the choice of $m$.  

Thus, $m=n$, and so $t_2\in\{s_2,\cpy_f\}$ and $\bigl(w_2^a\bigr)^{-1}w_2^b=\bigl(w_1^a\bigr)^{-1}w_1^b$, implying $w_2^a\neq w_2^b$.

\end{proof}

\begin{lemma}\label{lem-io-cpy}
Let $\cpy=\cpy^A(a,b)$.  Then the relation $\IO_\cpy$ is given by:
	\begin{align*}
		\IO^{s,s}_\cpy &= \IO^{f,f}_\cpy = \Id;\\
		\IO^{s,f}_\cpy &= \bigl\{\bigl((w, 1), (w, w)\bigr) \mid w\in F(A)\bigr\}.
	\end{align*}
In particular, the external tape $a$ is immutable in $\cpy(a,b)$. 
\end{lemma}

\begin{proof}

Let $(w_1,w_2)\in\IO_\cpy^{\sigma,\tau}$ for $\sigma,\tau\in\{s,f\}$ and denote $w_i=(w_i^a,w_i^b)$ for $i=1,2$, where $w_i^j=w_i(j)$ for $j\in\I_\cpy^\ext$.  Hence, there exists a computation $\eps$ of $\cpy$ between the configurations $(w_1^a,w_1^b,1,\cpy_\sigma)$ and $(w_2^a,w_2^b,1,\cpy_\tau)$.

By \Cref{lem-reach-cpy}(2), $w_1^a=w_2^a$, and hence $a$ is immutable.

Suppose $\sigma=\tau=s$.  If $w_1^b\neq1$ or $w_2^b\neq1$, then \Cref{lem-reach-cpy}(1) and the symmetry of $S$-graph computations imply that $w_1=w_2$.  Otherwise, $w_1^b=w_2^b=1$, and so again $w_1=w_2$.  Hence, $\IO_\cpy^{s,s}=\Id$.

Next, suppose $\sigma=\tau=f$.  If $w_i^b\neq w_i^a$ for some $i=1,2$, then \Cref{lem-reach-cpy}(3) implies that $\bigl(w_1^a\bigr)^{-1}w_1^b=\bigl(w_2^a\bigr)^{-1}w_2^b$; but $w_1^a=w_2^a$, which implies that $w_1^b=w_2^b$, and so $w_1=w_2$. Otherwise, $w_1^b=w_1^a=w_2^a=w_2^b$, and so again $w_1=w_2$.  Hence, $\IO_\cpy^{f,f}=\Id$.

Finally, suppose $\sigma=s$ and $\tau=f$.  By \Cref{lem-reach-cpy}(1) and (3), $w_1^b=1$ and $w_2^b=w_2^a$.  Hence, $\IO_\cpy^{s,f}\subseteq\bigl\{\bigl((w,1),(w,w)\bigr)\mid w\in F(A)\bigr\}$.

For the opposite inclusion, denote the edges of $\cpy^A(a,b)$ in the following way:

\begin{itemize}

\item $e_1$ is the edge from $\cpy_s$ to $s_1$

\item $e_{1,x}$ is the loop at $s_1$ corresponding to $x\in A$

\item $e_t$ is the edge from $s_1$ to $s_2$

\item $e_{2,x}$ is the loop at $s_2$ corresponding to $x\in A$

\item $e_2$ is the edge from $s_2$ to $\cpy_f$

\item All other edges are denoted as the inverse of one of those above

\end{itemize}

Fix $w\in F(A)$ and let $w=x_1^{\delta_1}\dots x_n^{\delta_n}$ such that $x_i\in A$ and $\delta_i\in\{\pm1\}$.  Then, consider the path in $\cpy$:
$$\rho_w=e_1,e_{1,x_1}^{\delta_1},\dots,e_{1,x_n}^{\delta_n},e_t,e_{2,x_n}^{\delta_n},\dots,e_{2,x_1}^{\delta_1},e_2$$
It then suffices to check that there exists a computation $\eps_w$ supported by $\rho_w$ with initial configuration $((w,1),\cpy_s)$ and terminal configuration $((w,w),\cpy_f)$.

\end{proof}

Note that $\IO_\cpy^{f,s}$ is not explicitly mentioned above.  This is because, as mentioned previously, the symmetry of the $\IO$ relation allows one to deduce $\IO_\op^{f,s}$ from $\IO_\op^{s,f}$ for any operation $\op$ of an object.  In this case, this is explicitly: $$\IO_\cpy^{f,s}=\bigl\{\bigl((w,w),(w,1)\bigr)\mid w\in F(A)\bigr\}$$

Informally, \Cref{lem-io-cpy} states that any computation from the starting to the finishing state of $\cpy(a, b)$ copies the contents of tape $a$ to the initially empty tape $b$.  Accordingly, any computation from the finishing to the starting state must begin with the contents of tapes $a$ and $b$ coinciding, in which case the computation erases the word on tape $b$. 

\end{example}

While $\cpy$ provides an example of a useful proper operation, it does not shed light on the intricacies of an object containing value tapes or multiple operations.  This is amended below, as a class of objects containing both value tapes and multiple operations is defined.  Objects of this form will be useful tools in later arguments, and will indeed play a major role in the central innovation that improves the bounds of \cite{SBR}.

\begin{definition}\label{def-counter}
Let $\Cntr=\bigl(\I_\Cntr^\val,\op_1,\op_2\bigr)$ be an object such that $\I_{\op_1}^\ext=\I_{\op_2}^\ext=\emptyset$.  Suppose there exists an enumeration of the values of $\Cntr$, $\VAL_\Cntr = \{v_i\}_{i=0}^\infty$, such that $v_0$ is the initial value and the relations $\IO_{\op_1}$ and $\IO_{\op_2}$ are defined as follows:
\begin{align*}
	\IO^{s,s}_{\op_j} &= \IO^{f,f}_{\op_j} = \Id;\\
	\IO^{s,f}_{\op_1} &= \bigl\{(v_i, v_{i+1})\mid i\geq 0\bigr\}; \\
	\IO^{s,f}_{\op_2} &= \bigl\{(v_i, v_i)\mid i\geq 1\bigr\}.
\end{align*}
Then, $\Cntr$ is called a \emph{counter}.  

In this case, the operations $\op_1$ and $\op_2$ are denoted $\inc$ and $\check$, respectively, and are called the \emph{increment} and \emph{positivity check} operations of the counter $\Cntr$.

\end{definition}

Informally, an input-output computation of $\inc$ whose computation path goes from $\inc_s$ to $\inc_f$ increases the value by one (as long as the contents of the value tapes initially represent a value).  Accordingly, such a computation is called an \emph{application of increment} or simply an \emph{increment}.   It follows that the inverse of such a computation essentially decreases the value by one, and so is called an \emph{application of decrement} or simply a \emph{decrement}.

Similarly, an input-output computation of $\check$ whose computation path is between $\check_s$ and $\check_f$ returns the same value as long as it is not the initial value.  Accordingly, such a computation is called a \emph{positivity check}.  As opposed to $\inc$, since the inverse computation serves the same essential purpose, no distinction is made in the terminology.

It is worth emphasizing that, despite the naming conventions, `$\inc$' and `$\check$' do not represent particular operations, but rather various pairs of operations that form an object whose values satisfy the conditions of the definition.  However, as no explicit construction of an example of such a pair is given, there is no reason a priori that the class of objects described by this definition is nonempty.

On the other hand, observe that, per their definitions, the set $\VAL_\Cntr$ is consistent with the $\IO$ relations in the following two senses:
\begin{enumerate}

\item Given $w_1,w_2\in L_\Cntr^\val$ and $p,q\in\{s,f\}$, $(w_1,w_2)\in\IO_\inc^{p,q}\cup\IO_\check^{p,q}$, $w_1\in\VAL_\Cntr$ if and only if $w_2\in\VAL_\Cntr$.

\item For any $n\in\N-\{0\}$, $v_n$ can be `obtained from $v_0$' through iterated applications of increment.

\end{enumerate}

Further, note that the relations $\IO_\inc$ and $\IO_\check$ are equivalence relations on $\VAL_\Cntr\times\SF_\inc$ and $\VAL_\Cntr\times\SF_\check$, respectively.  Hence, there is also no reason a priori that the class of objects described by the term `counter' is empty.

The following example begins to mend this ambiguity, as it presents a construction that will later be shown to be a counter.

\begin{example}[Counter]\label{ex-counter}
Consider the object $\Counter_0$ consisting of the two operations depicted in \Cref{fig-COUNTER1}.  With an abuse of terminology, these operations are named $\check$ and $\inc$; as one can infer, this is done because $\Counter_0$ is a counter, with each of the two operations functioning as its name suggests.

Notwithstanding this description introducing $\Counter_0$, it is worthwhile to check that it is indeed an object.

As is indicated by the notation $\Counter_0[a]$ presented in the graph, the object has a single value tape denoted $a$.  Moreover, the notation $\check()$ and $\inc()$ indicates that each of the operations has no external tapes.  Every other index used in these operations thus represents an internal tape, and so $\I_\check^\int=\I_\inc^\int=\{b,c,d\}$.

For $\op\in\{\check,\inc\}$, let $\A_\op=(A_i^\op)_{i\in\I_\op}$ be the set of alphabets of $\op$ given by the minimal hardware specified by \Cref{fig-COUNTER1}.  By inspection, in each graph, the only letter of $\A^*$ appearing in fragments is $\delta$, and this letter is used in at least one fragment mentioning each of the indices.  Hence, $A_a^\op=A_b^\op=A_c^\op=A_d^\op=\{\delta\}$.  In particular, $A_a^\check=A_a^\inc$, and thus $\Counter_0$ is an object if and only if each graph satisfies condition (O).

\begin{figure}[]
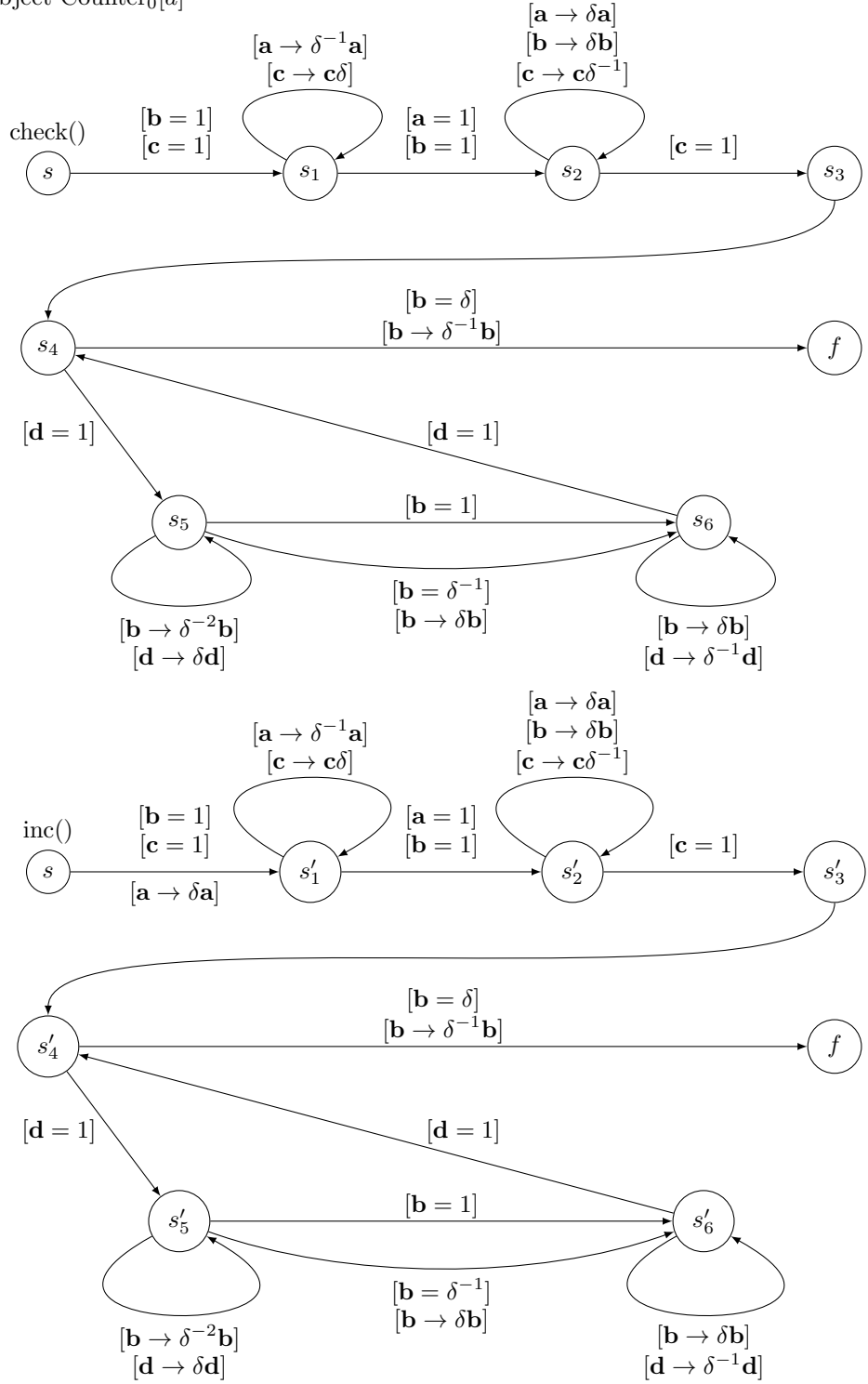

	\centering
	\include{COUNTER1}
	\caption{A counter object Counter$_0[a]$.}
	\label{fig-COUNTER1}       
\end{figure} 

\begin{lemma}\label{Counter1-O}

The $S$-graphs corresponding to $\check$ and $\inc$ satisfy condition (O).

\end{lemma}

\begin{proof}

Let $\rho=e_1,\dots,e_N$ be a path in $\check$ from $s_1$ to $s_2$ for $s_1,s_2\in\SF_\check$.  Further, let $i\in\{b,c,d\}$ and suppose there exists $j\in\overline{N}$ such that $e_j$ contains a fragment mentioning $i$.  By the symmetry of (O), it suffices to show that if $m\in\overline{N}$ is the minimal index such that $e_m$ contains a fragment mentioning $i$, then $e_m$ contains the fragment $[\mathbf{i}=1]$.

Suppose $s_1=\check_s$.  Then $e_1$ must be the unique edge whose tail is $\check_s$.  Hence, $e_1$ contains the guard fragments $[\mathbf{b}=1]$ and $[\mathbf{c}=1]$, so that it suffices to assume that $i=d$.  Note that any edge of $\check$ containing a fragment mentioning $d$ has tail $s_4$, $s_5$, or $s_6$.  The structure of the graph then implies that $e_m$ must have tail $s_4$; but then the only two edges satisfying this contain the guard fragment $[\mathbf{d}=1]$.

Conversely, suppose $s_1=\check_f$.  Then, the edge $e_1$ must contain the guard fragment $[\mathbf{b}=1]$, so that it suffices to assume that $i\in\{c,d\}$.  If $i=d$, then again $e_m$ must have tail $s_4$, and so must contain the guard fragment $[\mathbf{d}=1]$; if $i=c$, then the analogous argument implies that $e_m$ has tail $s_3$, and so contains the guard fragment $[\mathbf{c}=1]$.

Thus, $\check$ satisfies (O).  For $\inc$, note that the only difference between the two graphs is the transformation $[\mathbf{a}\to\delta\mathbf{a}]$ contained on the edge with tail $\inc_s$ (and the corresponding transformation fragment of the inverse edge), and so since $a$ is not internal, condition (O) is satisfied by the same argument.

\end{proof}

Now, while \Cref{Counter1-O} implies that $\Counter_0$ is an object, a full understanding of the $\IO$-relations of its operations is needed to verify that it is a counter.  To this end, using arguments similar to those used in the proofs of \Cref{lem-reach-cpy} and \Cref{lem-io-cpy}, it can be shown that $\VAL_{\Counter_0} = \left\{\delta^i~\mid~i\geq 0\right\}$ and that, setting $v_i=\delta^i$ for all $i$, $\IO_\check$ and $\IO_\inc$ satisfy the properties prescribed by \Cref{def-counter}.

However, notice that the subgraph of $\check$ induced by $\{\check_s,s_1,s_2,s_3\}$ is virtually identical to $\cpy^A(a,b)$ for $A=\{\delta\}$ in \Cref{fig-COPY}, while the analogous subgraph of $\inc$ differs only in a single fragment.  Similarly, the subgraph of $\check$ induced by $\{s_4,s_5,s_6,\check_f\}$ is virtually identical to the subgraph of $\inc$ induced by $\{s_4',s_5',s_6',\inc_f\}$, and these can be identified with some other proper operation.  Thus, a full understanding of the operations corresponding to these `complementary' subgraphs would effectively describe $\IO_\check$ and $\IO_\inc$.

These observations suggest that approaching the $\IO$-relations of these particular operations in the manner of \Cref{lem-reach-cpy} and \Cref{lem-io-cpy} may be superfluous.  As such, the proof is delayed until more efficient tools are available.  Instead, a generalization of $S$-graphs is first introduced, whose construction will evidently arise from this motivating example.  The terminology developed from these devices will allow for simpler representations of the objects of interest and more concise arguments, as will be seen later when revisiting the $\IO$-relations of operations like the ones in this example.

\end{example}

\section{$S^*$-graphs and Objects of Higher degree} \label{sec-sg-ext}

Recall that $\H^*$ denotes all possible $S$-graph hardwares and for any $\H\in\H^*$, $\ACT_\H$ denotes the collection of all actions over $\H$.  

Let $\ACT_0=\bigsqcup_{\H\in\H^*}\ACT_\H$ be the collection of all actions over all possible hardwares.  Similarly, let $\SGRAPH_0$ be the collection of all $S$-graphs, $\OP_0$ be the collection of all operations, and $\OBJ_0$ be the collection of all objects over all possible hardwares. 

Henceforth, elements of $\ACT_0$ are referred to as actions \emph{of degree $0$}.  Similarly, elements of $\SGRAPH_0$, $\OP_0$, and $\OBJ_0$ are referred to as $S$-graphs of degree 0, operations of degree 0, and objects of degree $0$, respectively.

Definitions \ref{def-instance}$-$\ref{def-s1-obj} below define \emph{instances}, actions, $S$-graphs, operations, and objects of arbitrary degree.  Accordingly, for any $n\geq0$, the set $\INST_n$ of instances of degree $n$ is defined.  Similarly, for $n\geq1$, the sets $\ACT_n$, $\SGRAPH_n$, $\OP_n$, and $\OBJ_n$ of all actions of degree $n$, $S$-graphs of degree $n$, operations of degree $n$, and objects of degree $n$, respectively, are defined.

Just as an element of $\OBJ_0$ is built out of elements of $\OP_0$, which are in turn built out of elements of $\SGRAPH_0$, which are in turn built out of elements of $\ACT_0$, so the definitions for arbitrary degree are developed iteratively.  To aid with this development, define $\INST_n'=\bigsqcup_{j\in\overline{0,n}}\INST_j$ for all $n\geq0$.  Hence, $\INST_n'$ is defined if and only if $\INST_j$ is defined for all $0\leq j\leq n$.  The sets $\ACT_n'$, $\SGRAPH_n'$, $\OP_n'$, and $\OBJ_n'$ are defined similarly.

\begin{definition}[Instances]\label{def-instance}
Fix $n\geq0$ and suppose that $\INST_{n-1}'$, $\ACT_n'$, $\SGRAPH_n'$, $\OP_n'$, and $\OBJ_n'$ are defined.

Let $\Obj\in\OBJ_n$ and suppose $\mu:\I_\Obj^\val\to\I^*$ is an injection.  Then, the pair $\inst=(\Obj,\mu)$ is called an \emph{instance of degree $n$}.  In this case, the injection $\mu$ is denoted $\mu_\inst^\val$ and is called the \emph{value tape renaming map} of $\inst$, while its image is denoted $\I_\inst^\val$.

Note that if $\Obj$ is an object of degree 0, then the index set $\I_\Obj^\val$ is defined in the previous section.  For higher degree objects, $\I_\Obj^\val$ plays an analogous role and is defined iteratively in \Cref{def-s1-obj} below.

\end{definition}

\begin{definition}[Actions]\label{def-s-action}
Fix $n\geq0$ and suppose that $\ACT_n'$, $\SGRAPH_n'$, $\OP_n'$, $\OBJ_n'$, and $\INST_n'$ are defined.

Let $\inst=(\Obj,\mu_\inst^\val)\in\INST_n$ and $\op\in\OP_m$ such that $\op\in\Obj$.  Suppose $\nu:\I_\op^\ext\to\I^*$ is an injection whose image is disjoint from $\I_\inst^\val$.  Then the triple $\act=(\inst,\op,\nu)$ is called an \emph{action of degree $n+1$ over the hardware $\H_\op$}.  In this case, the injection $\nu$ is denoted $\mu_\act^\ext$ and is called the \emph{external tape renaming map} of $\act$, while its image is denoted $\I_\act^\ext$.

Note that if $n=0$ (or indeed if $\Obj$ is an object of degree 0), then it is understood from the definitions of \Cref{sec-operations} that $\Obj$ consists of operations of degree 0, where any $\op\in\Obj$ is identified with an $S$-graph of degree 0 over the hardware $\H_\op=(\I_\op,\A_\op)$ such that $\I_\op$ can be partitioned into sets of external, value, and internal tapes.  Per Definitions \ref{def-s1-op} and \ref{def-s1-obj} below, higher degree objects and operations are constructed similarly, with an object of degree $n+1$ consisting of a set of operations of degree at most $n+1$ constructed from $S$-graphs of that degree.  As such, it is implicitly taken above that $m\leq n$.

Let $\act=(\inst,\op,\mu_\act^\ext)\in\ACT_{n+1}$ and $i\in\I^*$.  Then $\act$ is said to \emph{mention $i$} if $i\in\I_\inst^\val\sqcup\I_\act^\ext$.  If $i\in\I_\inst^\val$, then $i$ is \emph{mentioned as a value of $\inst$}; if $i\in\I_\act^\ext$, then $i$ is \emph{mentioned as a variable}.  Moreover, $i$ is \emph{mentioned as an immutable variable} if $\bigl(\mu_\op^\ext\bigr)^{-1}(i)$ is immutable in $\op$.  As above, an immutable external tape of an operation of degree 0 is defined in the previous section, while immutable external tapes of operations of higher degree will be defined in what follows.  

Observe that the concept of an action $\act\in\ACT_0$ mentioning an index $i$ is previously defined.  To match the terminology, in this case $\act$ is said to \emph{mention $i$ as a variable}, while $i$ is \emph{mentioned as an immutable variable} if any fragment of $\act$ mentioning $i$ is either a guard or a trivial fragment.

As with operations of degree 0, an operation $\op$ of arbitrary degree has an associated external language $L_\op^\ext=\bigtimes_{i\in\I_\op^\ext}F(\A_\op(i))$.  Similarly, an object $\Obj$ of higher degree has an associated tape language $L_\Obj^\val$ given by $\bigtimes_{i\in\I_\Obj^\val}F(\A_\op(i))$ for any $\op\in\Obj$.

Hence, for $\act = (\inst, \op, \mu^\ext_\act) \in \ACT_{n+1}$, the tape renaming maps are naturally extended to maps of languages, i.e injections $\mu_\act^\ext:L_\op^\ext\to\bigtimes_{i\in\I_\act^\ext}F(\A^*)$ and $\mu_\inst^\val:L_\Obj^\val\to\bigtimes_{i\in\I_\inst^\val}F(\A^*)$ given by reindexing.  The images of these maps are denoted $L_\act^\ext$ and $L_\inst^\val$, respectively.  

Thus, $L_\op^\ext=\bigtimes_{i\in\I_\act^\ext}F\left(\A_\op\bigl((\mu_\act^\ext)^{-1}(i)\bigr)\right)$ and similarly for $L_\inst^\val$.

\end{definition}

\begin{definition}[$S$-Graphs]\label{def-s1-graph}
Fix $n\geq0$ and suppose that $\ACT_{n+1}'$, $\SGRAPH_n'$, $\OP_n'$, $\OBJ_n'$, and $\INST_n'$ are defined.

Let $\H=(\I,\A)$ be an $S$-graph hardware and $\INST$ be a finite subset of $\INST_n'$ such that $\INST\cap\INST_n\neq\emptyset$.  Further, let $\G=(S,E)$ be a finite directed graph, perhaps containing parallel edges and loops, equipped with a labelling function $\lab:E\to\ACT_{n+1}'$.  Hence, every edge $e\in E$ can be identified with the triple $(s_1,s_2,\act)$, where $t(e)=s_1$, $h(e)=s_2$, and $\lab(e)=\act$.

The \emph{set of 0-actions of $\G$} is the set $\ACT_\G^0=\lab(E)\cap\ACT_0$, i.e the set of actions labelling the edges of $\G$ which are of degree 0.  Conversely, the \emph{set of positive actions of $\G$} is taken to be $\ACT_\G=\lab(E)-\ACT_\G^0$.

Suppose the following conditions are satisfied:
\begin{enumerate}

\item If $\act=(\inst,\op,\mu_\act^\ext)\in\ACT_\G$, then $\inst\in\INST$.

\item If $\inst\in\INST$, then $\I_\inst^\val\subseteq\I$.

\item For all $e\in E$, if $\lab(e)$ is an action over the hardware $\H_e\in\H^*$, then $\H_e\leq\H$, with equality if $\lab(e)\in\ACT_\G^0$.

\item If $e=(s_1,s_2,\act)\in E$ with $\act\in\ACT_\G^0$, then $e^{-1}=(s_2,s_1,\act^{-1})\in E$.  

\item If $\inst_1,\inst_2\in\INST$ such that $\inst_1\neq\inst_2$, then $\I_{\inst_1}^\val\cap\I_{\inst_2}^\val=\emptyset$.

\item For $\inst\in\INST$, $\act\in\ACT_\G\cup\ACT_\G^0$, and $i\in\I^*$, if $i\in\I_\inst^\val$ and $\act$ mentions $i$ as a variable, then $\act$ mentions $i$ as an immutable variable.
\end{enumerate}
Then $\G$ is called an \emph{$S$-graph of degree $n+1$} over $(\H,\INST)$.  In this case, the set of instances $\INST$ is denoted $\INST_\G$ to indicate its relationship with $\G$.

Note that the definition of an $S$-graph of degree 0 in \Cref{def-sgraph} matches that here if one permits $\INST_\G=\emptyset$.  That is, $\G=(S,E)$ is an $S$-graph (of degree 0) over the hardware $\H$ if and only if (3) and (4) are satisfied, while all other conditions are satisfied vacuously.  Hence, this is henceforth taken as the definition of an $S$-graph of degree 0 for uniformity.

As with the notation of previous sections, an edge $e$ of an $S$-graph of degree $n+1$ is said to \emph{mention} an index $i$ if the action labelling the edge mentions $i$.

Conditions (2) and (3) above describe a `compatibility' for a hardware with an $S$-graph of degree $n+1$ similar to the compatibility of a hardware with an action of degree 0 discussed in \Cref{ex-G2S}.  Thus, as $\H^*$ is a complete meet-semilattice with respect to the partial order $\leq$, for any $\G\in\SGRAPH_{n+1}$ there exists a minimal hardware $\H_\G=(\I_\G,\A_\G)$ compatible with $\ACT_\G\cup\ACT_\G^0\cup\INST_\G$, i.e $\G$ is an $S$-graph over $(\H,\INST_\G)$ if and only if $\H_\G\leq\H$.  When defining an $S$-graph $\G$ of arbitrary degree graphically, $\INST_\G$ is given explicitly while the hardware is implicitly taken as $\H_\G$.  However, note that this implicit definition is not as visually clear for $S$-graphs of higher degree as for those of degree 0 in previous sections, as some indices and letters comprising $\I_\G$ and $\A_\G$, respectively, may be hidden within an edge referencing a positive action of $\G$.

Notice that for any $\G\in\SGRAPH_{n+1}$, a single instance in $\INST_\G$ can be used to define multiple actions in $\ACT_\G$.  This simplifies the notation and presentation, allowing reference to a particular object in multiple settings without the need to rename it in a different way each time.

Moreover, an instance in $\INST_\G$ need not be used to define any action of $\ACT_\G$.  However, such an instance is still subject to conditions (2), (5), and (6).  Allowing $\INST_\G$ to contain such instances will prove useful for objects of higher degree (see \Cref{lem-nse-star}).

Finally, note that $S$-graphs of positive degree are no longer symmetric, as edges labelled by an action of positive degree do not have an inverse.  This is because, computationally speaking, this edge will serve as its own inverse.  Still, when an $S$-graph $\G$ of arbitrary degree is presented by picture, if an edge $e=(s_1,s_2,\act)$ with $\act\in\ACT_\G^0$ is depicted, then $e^{-1}$ is not depicted, as its presence is implicitly assumed by (4).

\end{definition}

\begin{definition}[Operations]\label{def-s1-op}
Fix $n\geq0$ and suppose that $\ACT_{n+1}'$, $\SGRAPH_{n+1}'$, $\OP_n'$, $\OBJ_n'$, and $\INST_n'$ are defined.

Let $\G=(S,E)$ be an $S$-graph of degree $n+1$ over $(\H,\INST_\G$) for $\H=(\I,\A)$ such that that the directed graph underlying $\G$ is connected.

For all $e=(s_1,s_2,\act)\in E$ such that $\act\in\ACT_\G$, define the \emph{formal inverse} of $e$ to be the triple $e^{-1}=(s_2,s_1,\act)$.  In this case, the \emph{tail} of $e^{-1}$ is $t(e^{-1})=s_2$ and the \emph{head} is $h(e^{-1})=s_1$.  Further, $e^{-1}$ \emph{mentions} the index $i$ if and only if $e$ does.

A \emph{flattened path} in $\G$ is a sequence of edges and formal inverses $\rho=e_1,\dots,e_N$ such that $h(e_i)=t(e_{i+1})$ for all $i\in\overline{N-1}$.  In this case, if $e_j$ is a formal inverse for $j\in\overline{N}$, set $e_j^{-1}$ as the unique edge in $E$ for which $e_j$ is the formal inverse; otherwise, define $e_j^{-1}$ in the natural way.  Then, $\rho^{-1}=e_N^{-1},\dots,e_1^{-1}$ is also a flattened path, called the \emph{inverse flattened path} of $\rho$.

Suppose there exist $s,f\in S$ and a partition $\I=\I^\ext\sqcup\I^\val\sqcup\I^\int$ satisfying the following condition:

\begin{enumerate}
		\item [(O)] Let $\rho=e_1,\dots,e_N$ be a flattened path in $\G$ from $s_1$ to $s_2$ where $s_1,s_2\in\{s,f\}$.  For $i\in\I^\int$, suppose there exists $j\in\overline{N}$ such that $e_j$ mentions $i$.  Then there exist $l_1,l_2\in\overline{N}$ with $l_1\leq l_2$ such that $\lab\bigl(e_{l_1}\bigr),\lab\bigl(e_{l_2}\bigr)\in\ACT_0$, $[\i=1]\in e_{l_1}\cap e_{l_2}^{-1}$, and $e_m$ does not mention $i$  for all $m\in\overline{N}-\overline{l_1,l_2}$.		
	\end{enumerate}
	Moreover, suppose that $\I_\inst^\val\subseteq\I^\val\sqcup\I^\int$ for all $\inst\in\INST_\G$.
	
	Then $\op=(\G,\I^\ext,\I^\val,\I^\int,s,f)$ is called an \emph{operation of degree $n+1$}.  
	
	As with operations of degree 0, the defining components of  $\op$ are named in a way to indicate their relationship to it, i.e $\op=(\G_\op,\I_\op^\ext,\I_\op^\val,\I_\op^\int,\op_s,\op_f)$ where $\G_\op=(S_\op,E_\op)$ is taken over $(\H_\op,\INST_\op)$ for $\H_\op=(\I_\op,\A_\op)$.
	
	The notational and naming conventions for operations of higher degree are mostly carried over from those defined for operations of degree 0; for example, the subset $\{\op_s,\op_f\}\subseteq S_\op$ is denoted $\SF_\op$, while the subsets $\I_\op^\ext$, $\I_\op^\val$, and $\I_\op^\int$ forming a partition of $\I_\op$ are called the \emph{external}, \emph{value}, and \emph{internal} tapes of $\op$.  
	
	One notable exception, however, is the definition of an immutable external tape.  As with operations of degree 0, the definition of these tapes relies on computations of the corresponding $S$-graph of degree $n+1$.  Hence, the discussion of which external tapes are immutable is delayed until the next section (see \Cref{def-immutable-star}).  
	
	However, condition (6) of the definition of $S$-graphs of higher degree relies on immutable tapes, and so their definition is necessary for the iterative step of the main definitions of this section.  For the moment, this is rectified simply by taking the set of immutable tapes of an operation $\op$ of degree $n+1$ to be some arbitrary subset $\I_\op^\ext$; while this conceals the significance of immutable tapes, it is sufficient for the iterative step and will be remedied in short order.
	
	Finally, note the differences between the definition of condition (O) here and that in \Cref{def-operation}.  First, the concept of an edge mentioning an index is generalized in this setting, as actions of positive degree may label the edges.  Further, it is required here that the actions labelling $e_{l_1}$ and $e_{l_2}$ are of degree 0, a requirement that goes without saying for operations of degree 0.  Finally, the `paths' of the graph in this setting are flattened paths, technical generalizations of actual paths whose use will become apparent in \Cref{def-generic}.  
	
	All that being said, the condition (O) given in \Cref{def-operation} can be viewed as a particular case of the one given in this setting, and hence the conditions are given the same name.  Moreover, the ability to invert flattened paths allows for the same level of `symmetry' in this setting as was established in the previous section.

\end{definition}

\begin{definition}[Objects]\label{def-s1-obj}
	Fix $n\geq0$ and suppose that $\ACT_{n+1}'$, $\SGRAPH_{n+1}'$, $\OP_{n+1}'$, $\OBJ_n'$, and $\INST_n'$ are defined.
	
	Let $\op_1,\dots,\op_k\in\OP_{n+1}'$ such that $\op_m\in\OP_{n+1}$ for some $m\in\overline{k}$.  Suppose there exist subsets $\INST\subseteq\INST_n'$ and $\I^\val\subseteq\I^*$ such that $\INST_{\op_j}=\INST$ and $\I_{\op_j}^\val=\I^\val$ for all $j\in\overline{k}$.  Further, suppose that for all $i\in\I^\val$, $\A_{\op_m}(i)=\A_{\op_\ell}(i)$ for all $m,\ell\in\overline{k}$.  
	
	Then $\Obj=(\I^\val,\INST,\op_1,\dots,\op_k)$ is called an \emph{object of degree $n+1$}.
	
	As with objects of degree 0, the notation $\op\in\Obj$ is used to denote that $\op$ is one of the operations comprising $\Obj$. Moreover, the \emph{set of value tapes} of the object $\Obj$, in this case $\I^\val$, is denoted $\I_\Obj^\val$ to emphasize its relationship with the object.  Similarly, the collection of instances $\INST$ defining the object $\Obj$ is denoted $\INST_\Obj$. 
	
	Note that the definition of an object of degree 0 given in \Cref{def-object} matches that here if one permits $\INST_\Obj=\emptyset$.  Hence, as with $S$-graphs of degree 0, this is henceforth taken as the definition of an object of degree 0 for uniformity.
	
	Further, just as for objects of degree 0, an object of degree at $n+1$ is assumed to contain a unique copy of each of its operations.  Hence, for distinct objects $\Obj_1,\Obj_2\in\OBJ_{n+1}'$, there is no $\op\in\OP_{n+1}'$ such that $\op\in\Obj_1$ and $\op\in\Obj_2$.
	
	An object $\Obj$ consisting of a single operation is identified with that operation and is called a \emph{proper operation} if, in addition, $\I_\Obj^\val=\emptyset$.
\end{definition}

Finally, define $\INST^* = \bigsqcup \INST_n$, $\ACT^* = \bigsqcup\ACT_n$, $\SGRAPH^* = \bigsqcup\SGRAPH_n$, $\OP^*=\bigsqcup\OP_n$, and $\OBJ^* = \bigsqcup\OBJ_n$.  In particular, any $\G\in\SGRAPH^*$ is called an \emph{$S^*$-graph}.  Conversely, with an abuse of terminology, elements of $\ACT^*$, $\OP^*$, $\OBJ^*$, and $\INST^*$ are simply called \emph{actions}, \emph{operations}, \emph{objects}, and \emph{instances}, respectively.

Note that the degree of an $S^*$-graph $\G$ is one more than the maximal degree of the instances of $\INST_\G$.  Similarly, the degree of an object is the maximal degree of one of its operations, which is the same as the degree of the $S^*$-graph defining that operation.  However, while the degree of an instance is the same as that of the object defining it, the degree of an action is one more than that of the instance it is constructed from.  So, if $\Obj\in\OBJ_n$, $\inst=(\Obj,\mu_\inst^\val)$, and $\act=(\inst,\op,\mu_\act^\ext)$ for some $\op\in\Obj$, then $\act\in\ACT_{n+1}$ even if $\op\in\OP_m$ for $m<n$.

The precise degree of an $S^*$-graph, operation, or object is generally insignificant and exists simply to aid with the iterative definitions.

Naming conventions analogous to those used for objects of degree 0 are instituted to simplify the process of defining an $S^*$-graph by picture.

First, for any object $\Obj\in\OBJ^*$, the notation $\Obj[a_1,\dots,a_k]$ indicates a fixed enumeration of the elements of $\I_\Obj^\val$.  Further, for any operation $\op\in\OP^*$, the notation $\op(c_1,\dots,c_\ell)$ indicates a fixed enumeration of the elements of $\I_\op^\ext$.

Then, when defining any $S^*$-graph $\G$ by picture, the elements of $\INST_\G$ are explicitly listed above the graph, with a square bracket indicating the definition of the associated value tape renaming maps.  Hence, for any object $\Obj=\Obj[a_1,\dots,a_k]$, the notation `$\inst \text{ is }\Obj[b_1,\dots,b_k]$' appears above the graph to indicate that $\inst=\bigl(\Obj,\mu_\inst^\val\bigr)\in\INST_\G$ such that $\mu_\inst^\val(a_i)=b_i$ for all $i\in\overline{k}$.  

Additionally, for any operation $\op=\op(c_1,\dots,c_\ell)$ of $\Obj$ with $\inst=(\Obj,\mu_\inst^\val)$, the notation $\inst.\op(d_1,\dots,d_\ell)$ is used to denote an action $\act=(\inst,\op,\mu_\act^\ext)$ given by $\mu_\act^\ext(c_j)=d_j$ for all $j\in\overline{\ell}$.

To simplify the notation slightly in a particular case, if $\Obj$ is a proper operation given by the operation $\op$ and there exists $\inst\in\INST_\G$ such that $\inst=(\Obj,\mu_\inst^\val)$, then $\Obj$ may be omitted from the declaration of the elements of $\INST_\G$ if it is used to construct an action of $\ACT_\G$.  Further, the corresponding action can be simply denoted $\op(d_1,\dots,d_\ell)$, omitting reference to the instance.

Similar conventions are instituted when defining an object $\Obj\in\OBJ^*$ by picture.  For example, the elements of $\INST_\Obj$ (perhaps with some proper operations omitted) are listed in much the same manner above the graphs defining the operations of $\Obj$.

An example of an object defined by picture is presented in \Cref{ex-silly-counter} below.  To be precise, an object $\Obj_0$ of degree 0 is constructed and used to construct an object $\Obj_1$ of degree 1 consisting of a single operation.  Hence, as immutable external tapes of operations of degree 0 are understood from the previous section, no ambiguity is created by the lack of specificity in defining immutable external tapes in \Cref{def-s1-op}, and thus the example is self-contained.

\begin{example}[Object of higher degree]\label{ex-silly-counter}
\Cref{fig-SILLYCOUNTER} below presents the objects $\Obj_0[a]$ and $\Obj_1[d']$ introduced above.

\begin{figure}[hbt]
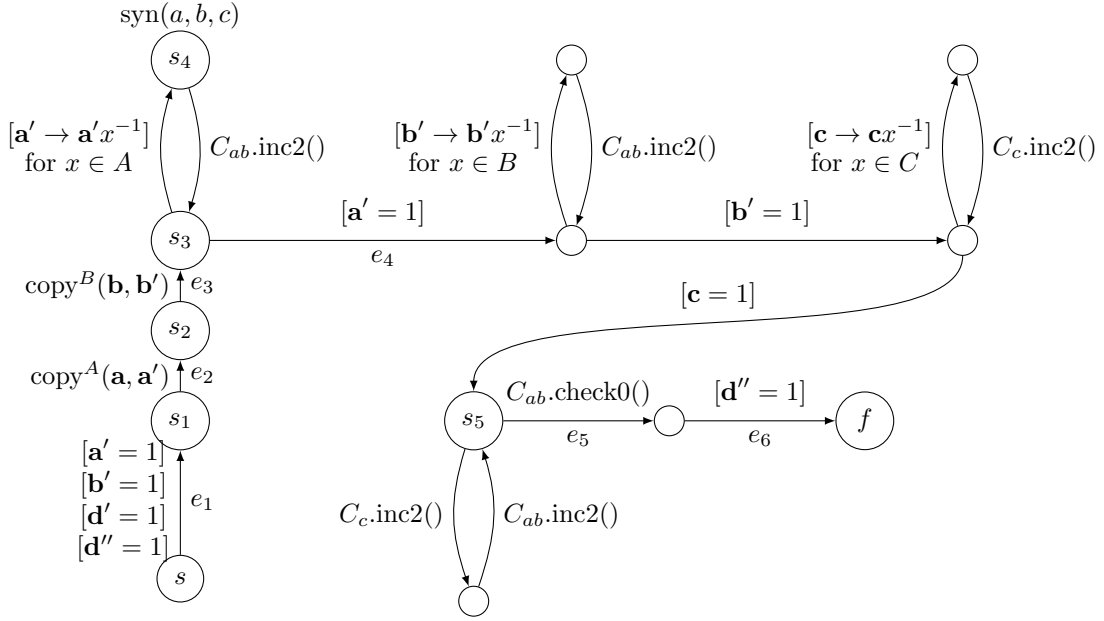

	\centering
	\include{SILLYCOUNTER}
	\caption{Objects $\Obj_0[a]$ and $\Obj_1[d']$ consisting of operations $\check0()$, $\inc2()$, and syn$(a,b,c)$.}
	\label{fig-SILLYCOUNTER}       
\end{figure} 

The object $\Obj_0$ consists of two operations $\check0$ and $\inc2$, with one edge of each pictured.  The notation $\check0()$ and $\inc2()$ suggests that each of these operations consists of no external tapes.  So, since an action labelling any edge of these operations consists only of a single fragment mentioning the value tape, $\I_{\check0}=\I_{\inc2}=\I_{\Obj_0}^\val=\{a\}$.  Hence, condition (O) is trivially satisfied for each, and so each is indeed an operation of degree 0.

More subtle, though, is the makeup of the alphabet corresponding to these operations.  The minimal hardwares of the individual graphs would produce the alphabets $\A_{\check0}(a)=\emptyset$ and $\A_{\inc2}(a)=\{\delta\}$; however, the definition of an object requires that these alphabets be the same.  To rectify this, for an object presented by picture, it is generally understood that the alphabet corresponding to a value tape is the minimal set satisfying the compatibility with the graphs of all defining operations.  Hence, $\A_{\check0}(a)=\{\delta\}$ despite the fact that $\delta$ is not mentioned by any of its edges.

The simple makeup of these graphs allows for straightforward computation of the $\IO$-relation of $\Obj_0$, yielding $\VAL_{\Obj_0}=\{\delta^{2k}\mid k\in\Z\}\subsetneq L_{\Obj_0}^\val$.

The object $\Obj_1[d']$ consists of the single operation syn$(a, b, c)$, whose name is meant to indicate that the purpose of this example is simply to illustrate the syntax.  Some edges of $\G_\syn$ are named in \Cref{fig-SILLYCOUNTER} for ease of reference.  By convention, $\G_\syn$ is taken over the hardware $\H_\syn=(\I_\syn,\A_\syn)$ where $\I_\syn = \{a, b, c, a', b', d', d''\}$.  Moreover, the naming indicates that $\I_\syn=\I_\syn^\ext\sqcup\I_{\Obj_1}^\val\sqcup\I_\syn^\int$ with $\I_\syn^\ext=\{a,b,c\}$, $\I_{\Obj_1}^\val=\{d'\}$, and $\I_\syn^\int=\{a',b',d''\}$.

The set $\INST_{\Obj_1}$ consists of four instances, all of degree 0: Two constructed from versions of $\cpy$ and two constructed from $\Obj_0$.  As $\cpy$ is a proper operation for any defining finite alphabet, though, only the latter instances are listed in the declaration.  As indicated there, $C_{ab}$ is formed by renaming the value tape of $\Obj_0$ as the value tape $d'$, while $C_c$ is formed by renaming it $d''$.  So, condition (6) of \Cref{def-s1-graph} requires that any edge of $\G_\syn$ whose action is not formed from one of these instances either does not mention $d'$ or $d''$ or mentions it as an immutable variable.  But the only such edges are $e_1^{\pm1}$ and $e_6^{\pm1}$, each of which is labelled by an action of degree 0 consisting only of guard fragments.  Hence, as condition (5) is also satisfied, $\G_\syn$ is indeed an $S$-graph of degree  1.

Next, note that $d'$ is a value and $d''$ is an internal tape of $\Obj_1$.  Thus, as required by \Cref{def-s1-op}, the images of the value tape renaming maps, $\I_{C_{ab}}^\val=\{d'\}$ and $\I_{C_c}^\val=\{d''\}$, are contained in $\I_{\Obj_1}^\val\sqcup\I_\syn^\int$.  Hence, to verify that $\syn$ is indeed an operation, it suffices to show that condition (O) is satisfied.  

As in the analogous discussion in \Cref{ex-copy}, let $\rho=e_1',\dots,e_N'$ be a flattened path with $t(e_1')\in\SF_\syn$ and assume that there exists $j\in\overline{N}$ such that $e_j'$ mentions $i\in\I_\syn^\int$.  Then, it suffices to show that for the minimal $m\in\overline{N}$ such that $e_m'$ mentions $i$, $e_m'$ is labelled by an action of degree 0 containing the fragment $[\i=1]$. 

If $t(e_1')=\syn_s$, then $e_1'=e_1$; but then for every $i\in\I_\syn^\int$, $e_1$ mentions $i$ and is labelled by an action of degree 0 containing the fragment $[\i=1]$.  Similarly, if $t(e_1')=\syn_f$, then $e_1'=e_6^{-1}$, so that the same conclusion is reached if $i=d''$.  Now suppose $t(e_1')=\syn_f$ and $i=a'$.  Setting $S'=\{\syn_s,s_1,s_2,s_3,s_4\}$, let $\G'$ be the subgraph of $\G_\syn$ induced by $S_\syn-S'$.  As no edge of $\G'$ mentions $a'$, $\rho$ must pass through a vertex of $S'$.  But then $e_m'=e_4^{-1}$, which again is labelled by an action of degree 0 containing the fragment $[\mathbf{a}'=1]$.  A similar argument implies the analogue for $i=b'$.

Thus, $\syn$ is an operation of degree 1, and so $\Obj_1$ is an object of degree 1.  As $\Obj_1$ consists of a single operation, it can be identified with $\syn$; however, since it contains a value tape, $\Obj_1$ is not a proper operation.

By the makeup of the edges of $\G_\syn$ that mention $d'$ discussed above, $\A_\syn(d')$ must be the same as the corresponding alphabets of the operations of $\Obj_0$, i.e $$\A_\syn(d')=\A_{\check0}(a)=\A_{\inc2}(a)=\{\delta\}$$  Similarly, $\A_\syn(d'')=\{\delta\}$.

Further analysis of the edges pertaining to the remaining indices yields that the alphabets are $\A_\syn(a)=\A_\syn(a')=A$, $\A_\syn(b)=\A_\syn(b')=B$, and $\A_\syn(c)=C$.  However, the makeup of the sets $A$, $B$, and $C$ is not described.  Indeed, while it may be inferred that each is a finite subset of $\A^*$, it follows that different choices of such subsets would yield different graphs.  As with $\cpy$ operations, this can be rectified by adding the defining sets to the name of the operation, e.g $\syn^{A,B,C}$.  As this is merely an example built to demonstrate syntax, though, it is taken here that $A$, $B$, and $C$ are some fixed arbitrary subsets.

Now, a few features of $\G_\syn$ demonstrate some redundancy in condition (O), suggesting it may be weakened:

\begin{itemize}

\item First, the edge $e_5$ fundamentally represents an action of degree 0 consisting only of the guard fragment $[\mathbf{d}'=1]$.  Given the same fragment included on $e_1$, $d'$ essentially satisfies (O) and functions in the same way as an internal tape.  However, as condition (O) requires the actual fragment to be written on the edge, $d'$ cannot be included as an internal tape of $\syn$. 

\item Second, while the guard $[\mathbf{a}'=1]$ included in the action labelling $e_1$ is necessary to enforce (O), its presence seems superfluous since the operation $\cpy^A(a, a')$ labelling $e_2$ requires $a'=1$ for an input-output computation.  Indeed, the same can be said for the guard $[\mathbf{b}'=1]$ given the label of $e_3$.  In practice, this is true: Omitting these guards does not significantly alter the possible computations of $\G_\syn$ (see \Cref{sec-star-computations}).

\end{itemize}
Despite these redundancies, though, the form of (O) presented in \Cref{def-s1-op} is chosen for its ease in verification.

On a similar note, notice that the actions of $e_2$ and $e_3$ are essentially independent of one another.  Hence, it seems a natural extension to generalize actions and $S^*$-graphs of higher degrees so that both can be written on a single edge, thus simplifying the graph. This, however, would make the corresponding definitions more cumbersome without providing enough benefit to justify the rigor.

Finally, while computations of $S^*$-graphs are as yet undefined, some mention of the $\IO$-relation for $\syn$ is given here for two reasons: First, it motivates the definition of computations that follows; and second, it gives an example to practice with after the definitions that follow.

Informally, while traversing from $\syn_s$ to $s_5$, the operation counts the total degree of the words written on tapes $a$, $b$, and $c$ while leaving the words written on the first two unchanged throughout and erasing the word written on the last.  The computation can then proceed to $\syn_f$ only if the sum of the first two is equal to the last one.

Precisely, $\ND_\syn$ is the set of quadruples $w=(w(a),w(b),w(c),1)$ such that $w(a)\in F(A)$, $w(b)\in F(B)$, $w(c)\in F(C)$, and $w(d')=1$.  Now, define $\sigma:F(\A^*)\to\Z$ by setting $\sigma(w)=\sum\eps_i$ for $w=a_1^{\eps_1}\dots a_n^{\eps_n}$ such that $a_i\in\A^*$ and $\eps_i\in\{\pm1\}$.  Then, a pair $(w_1,w_2)\in(\ND_\syn)^2$ is an element of $\IO_\syn^{s,f}$ if and only if:

\begin{itemize}

\item $w_2(c)=1$,

\item $w_1(a)=w_2(a)$,

\item $w_1(b)=w_2(b)$, and

\item $\sigma(w_1(a))+\sigma(w_1(b))=\sigma(w_1(c))$

\end{itemize}
Moreover, $a$ and $b$ are immutable external tapes of $\syn$ while $c$ is not.
\end{example}

\section{S$^*$-graph computations}\label{sec-star-computations}

\newcommand{\s}{\hat{s}}
\newcommand{\f}{\hat{f}}

Before approaching the definition of computations of $S^*$-graphs, another generalization of paths in a directed graph is first introduced.  These new paths will provide the support needed for computations of $S$-graphs of degree greater than 0.

\begin{definition}[Generic paths]\label{def-generic}
	Let $\G=(S_\G,E_\G)$ be an $S^*$-graph and $\s$, $\f$ be fixed constants. A \emph{generic edge} $e'$ is a triple $(\sf_1, \sf_2, e)$ such that $\sf_1, \sf_2 \in \{\s,\f\}$ and $e=(s_1, s_2, \act) \in E_\G$.  
	
	In this case, the \emph{tail} of $e'$, denoted $t(e')$, is taken to be $s_1$ if $\sf_1=\s$ and is taken to be $s_2$ if $\sf_1=\f$.  Similarly, the \emph{head} of $e'$, $h(e')$, is $s_1$ if $\sf_2=\s$ and $s_2$ if $\sf_2=\f$.
	
	A \emph{generic path} in $\G$ is a sequence of generic edges $\rho=e'_1, \dots, e'_k$ such that $h(e'_i) = t(e'_{i+1})$, for $i\in \overline{k-1}$.

Informally, a generic edge is an edge of $\G$ with a direction in which it is traversed: The pair $(\s,\f)$ represents traversing from tail to head, $(\f,\s)$ represents head to tail, $(\s,\s)$ represents a loop at the tail, and $(\f,\f)$ represents a loop at the head. 

\end{definition}

If $\G=(S_\G,E_\G)$ is an $S$-graph of degree 0 over the hardware $\H$, then the \emph{graph language} of $\G$, denoted $\GL_\G$, is taken to be the language $L_\H$.  With this definition, the set of configurations of $\G$ can be identified with $\GL_\G\times S_\G$. 

Let $n\geq0$.  Then for all $\G\in\SGRAPH_n$ and $\Obj\in\OBJ_n$, consider the following notions:
\medskip

\begin{addmargin}[3em]{0em}
\begin{enumerate}[label=({\Alph*[$n$]})]

\item The graph language $\GL_\G$,

\item A computation of $\G$,

\item The equivalence relation $\REACH_\G$ on $\GL_\G\times S_\G$,

\item The values $\VAL_\Obj\subseteq L_\Obj^\val$,

\item The natural domain $\ND_\op$ for any $\op\in\Obj$,

\item The equivalence relation $\IO_\op$ on $\ND_\op\times\SF_\op$ for any $\op\in\Obj$, and

\item An immutable external tape of $\op$ for any $\op\in\Obj$.

\end{enumerate}
\end{addmargin}
\medskip

Note that (A[0]) is defined above, while (B[0])$-$(G[0]) are defined in Sections \ref{sec-intro} and \ref{sec-operations}.

Definitions \ref{def-graph-language}$-$\ref{def-immutable-star} extend these to arbitrary degrees, building them iteratively in much the same way as the main definitions of \Cref{sec-sg-ext}.  

\begin{definition}[Graph Languages]\label{def-graph-language}

Fix $n\geq0$ and suppose that (A[$m$])$-$(G[$m$]) have been defined for all $m\leq n$.

Let $\G=(S_\G,E_\G)$ be an $S$-graph of degree $n+1$ over $(\H,\INST_\G)$ for $\H=(\I,\A)$.  Then, define $\I_\G^*=\I-\bigsqcup_{\inst\in\INST_\G}\I^\val_\inst$ and $L^*_\G = \bigtimes_{i\in\I_\G^*}F(\A(i))$.  Further, for any $\inst=(\Obj,\mu_\inst^\val)\in\INST_\G$, let $\VAL_\inst = \mu^\val_\inst\left(\VAL_\Obj \right)$, where $\VAL_\Obj$ is the set of values of $\Obj\in\OBJ_n'$ as defined iteratively in \Cref{def-val-star}.

Then, define the  \emph{graph language} of $\G$, $\GL_\G$, as the subset of $L_\H$ consisting of all words $w$ such that for any $\inst\in\INST_\G$, $w(\I_\inst^\val)\in\VAL_\inst$, i.e
	$$\GL_\G = L^*_\G \times \bigtimes_{\inst\in\INST_\G} \VAL_\inst$$
In other words, the disjoint images of the value tape renaming maps are identified with the value tapes of the corresponding objects and attention is restricted to $L$-words whose contents in these coordinates comprise values.  As will be elaborated on in forthcoming definitions, this essentially ensures that the words in the graph language correspond to words in the natural domain of the operations used to construct the actions labelling the edges of $\G$.

The elements of $\GL_\G\times S_\G$ are called the \emph{configurations} of $\G$. 
\end{definition}

\begin{definition}[Computations]\label{def-computation-star}

Fix $n\geq0$ and suppose that (A[$m$])$-$(G[$m$]) have been defined for all $m\leq n$.  Note that \Cref{def-graph-language} then defines (A[$n+1$]).

Let $\G=(S_\G,E_\G)$ be an $S$-graph of degree $n+1$ over $(\H,\INST_\G)$ for $\H=(\I,\A)$.  For any $\act\in\ACT_\G\cup\ACT_\G^0$, define the \emph{language input-output relation} $\mathrm{LIO}_\act$ on $L_\H\times \{\s,\f\}$ as follows:
\begin{itemize}
	\item If $\act \in \ACT_0$, then for $\act=(t,g)$,
	\begin{align*}
	\mathrm{L}\IO_\act^{\s,\s} &= \mathrm{L}\IO_\act^{\f,\f} = \Id, \\  
	\mathrm{L}\IO_\act^{\s,\f} &= \left\{(w, tw)~|~w\in L_\H,~\textrm{$g$ accepts $w$}\right\}, \\
	\mathrm{L}\IO_\act^{\f,\s} & = \left\{(w, t'w)~|~w\in L_\H,~  \textrm{$g'$ accepts $w$}\right\},
	\end{align*}
	for $(t', g') = (t, g)^{-1}$.
	
	\item If $\act\in\ACT_\G$, then let $\act=(\inst,\op,\mu_\act^\ext)$ with $\inst=(\Obj,\mu_\inst^\val)\in\INST_\G$.  Set $\I_\act=\I_\act^\ext\sqcup\I_\inst^\val$ and define $\mu_\act:L_\op^\ext\times L_\Obj^\val\to L_\H(\I_\act)$ as the bijection given by $\mu_\act=\mu_\act^\ext\sqcup\mu_\inst^\val$.  Further, let $\lambda_\op:\{\s,\f\}\to\SF_\op$ be the bijection given by $\lambda_\op(\s)=\op_s$ and $\lambda_\op(\f)=\op_f$.  Then, for any $w_1,w_2\in L_\H$ and $p,q\in\{\s,\f\}$, define $(w_1,w_2)\in\mathrm{L}\IO_\act^{p,q}$ if and only if::
	\begin{enumerate}[label=({\roman*)}]
	\item $w_1\bigl(\I_\G-\I_\act\bigr)=w_2\bigl(\I_\G-\I_\act\bigr)$
	
	\item $\left(\mu_\act^{-1}\bigl(w_1(\I_\act)\bigr),\mu_\act^{-1}\bigl(w_2(\I_\act)\bigr)\right)\in\IO_\op^{\lambda_\op(p),\lambda_\op(q)}$
	\end{enumerate}
\end{itemize}

\begin{lemma}\label{lem-lio-io} Let $\act\in\ACT_\G\cup\ACT_\G^0$.

\begin{enumerate}

\item The relation $\mathrm{LIO}_\act$ is properly $L_\H$-restricted to $\GL_\G$.

\item For all $w_1,w_2\in L_\H$, $(w_1,w_2)\in\mathrm{LIO}_\act^{\s,\f}$ if and only if $(w_2,w_1)\in\mathrm{LIO}_\act^{\f,\s}$.

\item $\mathrm{LIO}_\act$ is transitive.

\end{enumerate}

\end{lemma}

\begin{proof} 
(1) Let $(w_1,w_2)\in\mathrm{LIO}_\act^{p,q}$ with $p,q\in\{\s,\f\}$.  It suffices to show that for any $\inst\in\INST_\G$, $w_1(\I_\inst^\val)\in\VAL_\inst$ if and only if $w_2(\I_\inst^\val)\in\VAL_\inst$.

Suppose $\act\in\ACT_0$.  Then for any $\inst\in\INST_\G$ and $i\in\I_\inst^\val$, $\act$ must mention $i$ as an immutable variable.  Hence, $w_1(i)=w_2(i)$ and so $w_1\bigl(\I_\inst^\val\bigr)=w_2\bigl(\I_\inst^\val\bigr)$.

Conversely, suppose $\act\in\ACT_\G$ and let $\inst_0\in\INST_\G$ such that $\act=(\inst_0,\op,\mu_\act^\ext)$ with $\inst_0=(\Obj,\mu_{\inst_0}^\val)$.  Then $\inst_0\in\INST_n'$ so that $\Obj\in\OBJ_m$ for some $m\in\overline{n}$.

Let $\inst\in\INST_\G-\{\inst_0\}$ and $i\in\I_\inst^\val$.  If $\act$ does not mention $i$, then condition (i) above implies that $w_1(i)=w_2(i)$.  Otherwise, condition (6) of \Cref{def-s1-graph} implies that $\act$ mentions $i$ as an immutable variable, i.e $(\mu_\act^\ext)^{-1}(i)$ is immutable in $\op$.  So, the definition of (G[$m$]) in \Cref{def-immutable-star} implies that $w_1(i)=w_2(i)$.  Hence, $w_1\bigl(\I_\inst^\val\bigr)=w_2\bigl(\I_\inst^\val\bigr)$.

Further, (ii) implies the existence of an edge in $\G_\Obj^\val$ between the vertices corresponding to $(\mu_{\inst_0}^\val)^{-1}(w_1(\I_{\inst_0}^\val))$ and $(\mu_{\inst_0}^\val)^{-1}(w_2(\I_{\inst_0}^\val))$ (see \Cref{def-val-star}), and so $w_1\bigl(\I_{\inst_0}^\val\bigr)\in\VAL_{\inst_0}$ if and only if $w_2\bigl(\I_{\inst_0}^\val\bigr)\in\VAL_{\inst_0}$.

\medskip

\noindent
(2)  If $\act\in\ACT_0$, then the statement follows from the definition of the inverse action $(t',g')$.

Otherwise, there exist $\inst=(\Obj,\mu_\inst^\val)\in\INST_\G$ and $\op\in\Obj$ such that $\act=(\inst,\op,\mu_\act^\ext)$.  As in (1), there exists $m\in\overline{n}$ such that $\Obj\in\OBJ_m$ with $\op\in\Obj$.  Thus, the statement follows from \Cref{def-io-star} for (F[$m$]), where $\IO_\op$ is shown to be an equivalence relation.

\medskip

\noindent
(3) follows similarly.

\end{proof}

In view of \Cref{lem-lio-io}(1), the proper restriction $\mathrm{LIO}_\act\pr{\GL_\G}$ is denoted $\IO_\act$.

Now, for any generic edge $e' = (\sf_1, \sf_2, e)$ of $\G$ with $e=(s_1, s_2, \act)\in E_\G$, define $\IO_{e'}$ as the relation on $\GL_\G \times \{t(e'), h(e')\}$ given by
	$$\IO_{e'} = \bigl\{(w_1, t(e'), w_2, h(e'))\mid (w_1, \sf_1, w_2, \sf_2)\in\IO_\act\bigr\}.$$

A \emph{generic computation} of $\G$ is then a sequence $$\eps = \bigl((w_0, s_0), e'_1, (w_1, s_1), e_2', \dots, e'_m, (w_m,  s_m)\bigr)$$ of alternating configurations $(w_i, s_i) \in \GL_\G\times S_\G$ and generic edges $e'_i$ such that for all $i\in\overline{m}$, $t(e_i')=s_{i-1}$, $h(e_i')=s_i$, and $(w_{i-1}, s_{i-1}, w_i, s_i) \in \IO_{e_i'}$.

As with computations for $S$-graphs of degree 0, $m$ is permitted to be 0, in which case the computation is merely the configuration $(w_0,s_0)$.

Informally, a generic computation functions in a similar manner to a computation as defined in \Cref{sec-intro}, except that some internal computation is performed upon traversing an edge labelled by a positive action. This internal computation does not necessarily go straight from the tail to head, but rather may be traversed in the opposite direction or in a loop about one of the endpoints. Hence, the necessity of generic edges becomes more apparent, as for an arbitrary operation $\op$, $\IO_\op^{s,s}$ may be strictly bigger than $\IO_\op^{s,f} \circ\IO_\op^{f,s}$.

The term `generic' is presently employed merely to differentiate from the computations of $S$-graphs of degree 0 considered in previous sections.  To keep the notation uniform, a generic computation of an $S$-graph of degree 0 is defined in an analogous way to how it is defined for an $S$-graph of degree $n+1$.  However, any computation of an $S$-graph of degree 0 can be identified with a generic computation by simply replacing any edge $e$ in the computation with the generic edge $(\s,\f,e)$.

\begin{lemma} \label{lem-comp-star}

Let $\G$ be an $S$-graph of degree 0.  Then for every generic computation $$\eps=\bigl((w_0,s_0),e_1',(w_1,s_1),e_2',\dots,e_m',(w_m,s_m)\bigr)$$ of $\G$, there exists a computation $\delta=\bigl((v_0,t_0),e_1'',(v_1,t_1),e_2'',\dots,e_k'',(v_k,t_k)\bigr)$ of $\G$ such that:

\begin{enumerate}

\item $k\leq m$, 

\item $(v_0,t_0)=(w_0,s_0)$ and $(v_k,t_k)=(w_m,s_m)$, and 

\item for all $i\in\overline{m}$, there exists a $j\in\overline{k}$ such that $(w_i,s_i)=(v_j,t_j)$.

\end{enumerate}

\end{lemma}

\begin{proof}

For all $i\in\overline{m}$, let $e_i'=(\sf_i,\sf_i',e_i)$ with $\sf_i,\sf_i'\in\{\s,\f\}$ and $e_i\in E_\G$.

Suppose there exists $i$ such that $\sf_i=\sf_i'=\s$.  Then $s_{i-1}=t(e_i')=h(e_i')=s_i$ and $(w_{i-1},w_i)\in\IO_\act^{\s,\s}=\Id$, so that $(w_{i-1},s_{i-1})=(w_i,s_i)$.  As a result, removing $e_i'$ and identifying these configurations reduces $m$ without affecting the makeup of the rest of $\eps$.  If there exists $i$ such that $\sf_i=\sf_i'=\f$, then a similar removal is possible.

Hence, it may be assumed without loss of generality that $\sf_i\neq\sf_i'$ for all $i\in\overline{m}$.

Now, for all $i\in\overline{m}$, set $\a_i=1$ if $\sf_i=\s$ and $\a_i=-1$ if $\sf_i=\f$.  Then, consider the sequence $\delta=\bigl((w_0,s_0),e_1^{\a_1},(w_2,s_2),e_2^{\a_2},\dots,e_k^{\a_k},(w_k,s_k)\bigr)$.  By the definition of the relation $\IO_\act$ for $\act\in\ACT_0$, $\delta$ is a computation of $\G$ satisfying the statement.

\end{proof}

In light of \Cref{lem-comp-star}, generic computations of an $S$-graph of degree 0 are identified with computations of that graph.  Thus, the distinction `generic' is largely omitted in what follows.

Given a computation $\eps=\bigl((w_0, s_0), e'_1, (w_1, s_1), e_2', \dots, e'_m, (w_m,  s_m)\bigr)$, the following terminology is adopted from previous sections for analogous purposes: 

\begin{itemize}

\item The generic path $\rho = e_1', \dots, e_m'$ in $\G$ is called the \emph{computation path} of $\eps$.

\item The computation path $\rho$ is said to \emph{support} $\eps$.

\item $(w_0, s_0)$ is called the \emph{starting configuration} and $(w_m, s_m)$ the \emph{finishing configuration} of $\eps$.

\item $\eps$ is called a computation \emph{between} $(w_0,s_0)$ and $(w_m,s_m)$.

\item The \emph{length} of $\eps$ (and of the computation path $\rho$) is $m$, with the notation $m = |\varepsilon| = |\rho|$ representing this. 

\end{itemize}
Further, the set of all computations between two configurations $(w_1,s_1)$ and $(w_2,s_2)$ is denoted $\eps(w_1,s_1,w_2,s_2)$.

However, unlike for $S$-graphs of degree 0, given a configuration $(w_0, s_0)$ and a generic path $\rho$ in $\G$, there are generally multiple distinct computations supported by $\rho$ with starting configuration $(w_0, s_0)$.  Indeed, no two such computations need to have the same finishing configuration.

Given an $\op\in\OP_{n+1}'$, a computation of $\op$ (i.e a computation of $\G_\op$) is called an \emph{input-output computation} if it is between configurations $(w_1,s_1)$ and $(w_2,s_2)$ such that $s_1,s_2\in\SF_\op$ and $w_1(\I_\op^\int)=w_2(\I_\op^\int)=1_\op^\int$.

\end{definition}

\begin{definition}[$\REACH$-relation]\label{def-reach-star}

Fix $n\geq0$ and suppose that (A$[m]$)$-$(G$[m]$) have been defined for all $m\leq n$.  Note that Definitions \ref{def-graph-language} and \ref{def-computation-star} then define the notions (A$[n+1]$) and (B$[n+1]$).

Let $\G=(S_\G,E_\G)$ be an $S$-graph of degree $n+1$.  Then the \emph{reachability relation} of $\G$, denoted $\REACH_\G$, is the relation on the configurations of $\G$ which relates the starting configuration of any computation to the finishing configuration.  Hence, 
$$\REACH_\G = \left\{(w_1, s_1, w_2, s_2)\in(\GL_\G\times S_\G)^2 \mid \eps(w_1, s_1, w_2, s_2) \neq \emptyset \right\}$$ 
If $(w_1, s_1, w_2, s_2) \in \REACH_\G$, then $(w_2, s_2)$ \emph{can be reached} (via a computation of $\G$) from $(w_1, s_1)$.

\begin{lemma}\label{lem-reach-star}

For $\G\in\SGRAPH_{n+1}$, $\REACH_\G$ is an equivalence relation on $\GL_\G\times S_\G$.

\end{lemma}

\begin{proof}

Note that reflexivity follows from computations of length 0 and transitivity is clear, so it suffices to prove symmetry.

Let $\eps=\bigl((w_0,s_0),e_1',(w_1,s_1),e_2',\dots,e_m',(w_m,s_m)\bigr)$ be any computation of $\G$. 
For every $i\in\overline{m}$, let $e_i'=(\sf_i,\sf_i',e_i)$ with $\sf_i,\sf_i'\in\{\s,\f\}$ and $e_i\in E_\G$.  Then, define $e_i''=(\sf_i',\sf_i,e_i)$.  \Cref{lem-lio-io}(2) then implies that $(w_i,s_i,w_{i-1},s_{i-1})\in\IO_{e_i''}$ for all $i$, so that there exists an `inverse computation'
$$\eps^{-1}=\bigl((w_m,s_m),e_m'',\dots,e_2'',(w_1,s_1),e_1'',(w_0,s_0)\bigr)$$
Hence, $(w_0,s_0,w_m,s_m)\in\REACH_\G$ implies $(w_m,s_m,w_0,s_0)\in\REACH_\G$, and so $\REACH_\G$ is symmetric.

\end{proof}

Much of the rest of the terminology of \Cref{sec-operations} is adopted for the $\REACH$ relation of an $S$-graph of degree $n+1$.
 
 \end{definition}

Note that Definitions \ref{def-val-star}$-$\ref{def-io-star} below are virtually identical to their counterparts in \Cref{sec-intro}, where analogous terminology was defined for operations and objects of degree $0$.

\begin{definition}[Values of an object]\label{def-val-star}

Fix $n\geq0$ and suppose that (A$[m]$)$-$(G$[m]$) have been defined for all $m\leq n$.  Note that Definitions \ref{def-graph-language}$-$\ref{def-reach-star} then define the notions (A$[n+1]$)$-$(C$[n+1]$).

Let $\Obj\in\OBJ_{n+1}$ and define the \emph{initial value} of $\Obj$ as $1_\Obj^\val=(1)_{i\in\I_\Obj^\val}\in L_\Obj^\val$.

Let $\G_\Obj^\val$ be the (undirected) graph whose vertices are identified with the words of $L_\Obj^\val$ and such that there exists an edge connecting $v_1,v_2\in L_\Obj^\val$ if and only if there exists $\op\in\Obj$ and an input-output computation of $\op$ between the configurations $(w_1,s_1),(w_2,s_2)\in\GL_\op\times\SF_\op$ such that $w_i(\I_\Obj^\val)=v_i$ for $i=1,2$.  Then define the \emph{values} of $\Obj$, denoted $\VAL_\Obj$, as the connected component of $\G_\Obj^\val$ containing the initial value.

\end{definition}

\begin{definition}[Natural domain of an operation]\label{def-ND-star}

Fix $n\geq0$ and suppose that (A$[m]$)$-$(G$[m]$) have been defined for all $m\leq n$.  Note that Definitions \ref{def-graph-language}$-$\ref{def-val-star} then extend this to define the notions (A$[n+1]$)$-$(D$[n+1]$).

Then for any $\Obj\in\OBJ_{n+1}$ and any operation $\op\in\Obj$, the \emph{natural domain} of $\op$ is given by $\ND_\op= L_\op^\ext\times\VAL_\Obj$.

As in \Cref{sec-operations}, the notation $\ND_\op$ conceals the fact that the domain's composition depends on the object $\Obj$.  However, as in that setting, any potential ambiguity is avoided by the hypothesis that $\Obj$ is the unique object such that $\op\in\Obj$.

For notational ease, the \emph{extended natural domain} is $\ND_\op^1=\ND_\op\times1^\int_\op\subseteq L_\op$.

\begin{lemma}\label{lem-nse-star}

Let $\Obj\in\OBJ_{n+1}$ and $\op\in\Obj$.

\begin{enumerate}

\item $\ND_\op^1\subseteq\GL_\op$

\item $\REACH_\op^{\SF_\op}$ is properly restricted to $\ND_\op^1$.

\end{enumerate}

\end{lemma}

\begin{proof}

(1) First, note that $\op\in\OP_{n+1}'$, so that $\GL_\op=\GL_{\G_\op}$ is understood from \Cref{def-graph-language}.  

Let $w\in\ND_\op^1$ and fix an instance $\inst\in\INST_\Obj$.  Then it suffices to show $w\bigl(\I_\inst^\val\bigr)\in\VAL_\inst$.

As $w\bigl(\I_\Obj^\val\bigr)\in\VAL_\Obj$, there exists $(v_0, \eps_1, v_1, \eps_2, \dots, \eps_k,v_k)$ with $v_0=1_\Obj^\val$ and $v_k=w\bigl(\I_\Obj^\val\bigr)$ such that for all $i\in\overline{k}$, $\eps_i$ is an input-output computation of $\op_i\in\Obj$ between configurations $(u_i,s_i)$ and $(w_i, t_i)$ satisfying $u_i\bigl(\I_\Obj^\val\bigr)=v_{i-1}$ and $w_i\bigl(\I_\Obj^\val\bigr)=v_i$.

Let $\I_\inst^{\val\val}=\I_\inst^\val\cap\I_\Obj^\val$.  Then, since $\inst\in\INST_\Obj$, $\inst\in\INST_{\op_0}$ for all $\op_0\in\Obj$ and so $\I_\inst^{\val\int}=\I_\inst^\val-\I_\inst^{\val\val}\subseteq\I_{\op_0}^\int$.

Since $\eps_k$ is an input-output computation, $w_k\in\GL_{\op_k}$ and $w_k\bigl(\I_{\op_k}^\int\bigr)=1_{\op_k}^\int$.  In particular, this implies that $w_k\bigl(\I_\inst^\val\bigr)\in\VAL_\inst$ and $w_k\bigl(\I_\inst^{\val\int}\bigr)=(1)_{i\in\I_\inst^{\val\int}}$.

Further, $w_k\bigl(\I_\Obj^\val\bigr)=v_k=w\bigl(\I_\Obj^\val\bigr)$, so that $w_k\bigl(\I_\inst^{\val\val}\bigr)=w\bigl(\I_\inst^{\val\val}\bigr)$.  But $w\in\ND_\op^1$, so that also $w\bigl(\I_\op^\int\bigr)=1_\op^\int$, i.e $w\bigl(\I_\inst^{\val\int}\bigr)=(1)_{i\in\I_\inst^{\val\int}}=w_k\bigl(\I_\inst^{\val\int}\bigr)$.  Thus, $$w\bigl(\I_\inst^\val\bigr)=w\bigl(\I_\inst^{\val\val}\sqcup\I_\inst^{\val\int}\bigr)=w_k\bigl(\I_\inst^{\val\val}\sqcup\I_\inst^{\val\int}\bigr)=w_k\bigl(\I_\inst^\val\bigr)\in\VAL_\inst$$

\medskip

\noindent
(2) Let $\eps=\bigl((w_0,s_0),e_1',(w_1,s_1),e_2',\dots,e_m',(w_m,s_m)\bigr)$ be a computation of $\op$ with $s_0,s_m\in\SF_\op$.  Denote the generic path supporting $\eps$ by $\rho'=e_1',e_2',\dots,e_m'$.

For all $i\in\overline{m}$, let $e_i'=(\sf_i,\sf_i',e_i)$ for $e_i\in E_\op$.  If $\lab(e_i)\in\ACT_0$, then there exists an inverse edge $e_i^{-1}\in E_\op$; otherwise, let $e_i^{-1}$ be the formal inverse of $e_i$.  Then, define the path $q_i$ in $\G_\op$ as follows:

\begin{itemize}

\item If $\sf_i=\sf_i'=\s$, then let $q_i=e_i,e_i^{-1}$

\item If $\sf_i=\sf_i'=\f$, then let $q_i=e_i^{-1},e_i$

\item If $\sf_i=\s$ and $\sf_i'=\f$, then let $q_i=e_i$

\item If $\sf_i=\f$ and $\sf_i'=\s$, then let $q_i=e_i^{-1}$

\end{itemize}
With this, $fl(\rho')=q_1,\dots,q_m$ is a flattened path in $\op$ from $s_0$ to $s_m$ called the \emph{flattening} of $\rho$.  Hence, $fl(\rho')$ is subject to the restrictions of condition (O).

Thus, as in the proof of \Cref{prop-nse}, it follows that $w_0\in\ND_\op^1$ if and only if $w_m\in\ND_\op^1$.

\end{proof}

\end{definition}

\begin{definition}[$\IO$-relation]\label{def-io-star}
Fix $n\geq0$ and suppose that (A$[m]$)$-$(G$[m]$) have been defined for all $m\leq n$.  Note that Definitions \ref{def-graph-language}$-$\ref{def-ND-star} then define the notions (A$[n+1]$)$-$(E$[n+1]$).

Let $\Obj\in\OBJ_{n+1}$ and $\op\in\Obj$.  Then, let $\IO_\op^1=\REACH_\op^{\SF_\op}\pr{\ND_\op^1}$.  By \Cref{lem-nse-star}, $\IO_\op^1$ is a relation on $\ND_\op^1\times\SF_\op$.  

Define the relation $\IO_\op$ on $\ND_\op\times\SF_\op$ by setting $(w_1,s_1,w_2,s_2)\in\IO_\op$ if and only if there exist $w_1',w_2'\in\ND_\op^1$ such that $w_i'\bigl(\I_\op^\ext\sqcup\I_\Obj^\val\bigr)=w_i$ for $i=1,2$ and $(w_1',s_1,w_2',s_2)\in\IO_\op^1$.  Hence, for any $s_1,s_2\in\SF_\op$,
$$\IO_\op^{s_1,s_2}=\bigl\{(w_1,w_2)\in\bigl(\ND_\op\bigr)^2\mid(w_1\sqcup1_\op^\int,w_2\sqcup1_\op^\int)\in\REACH_{\op}^{s_1,s_2}\bigr\}$$ Thus, the following statement follows immediately from \Cref{lem-reach-star}:
\begin{lemma}\label{lem-io-star}

For any $\Obj\in\OBJ_{n+1}$ and $\op\in\Obj$, $\IO_\op$ is an equivalence relation on $\ND_\op\times\SF_\op$.

\end{lemma}

\end{definition}

Finally, recall that the definition of immutable external tapes given in the definitions of \Cref{sec-sg-ext} was of an arbitrary nature.  This is now clarified using the machinery developed in this section.

\begin{definition}[Immutable tape]\label{def-immutable-star}
Fix $n\geq0$ and suppose that (A$[m]$)$-$(G$[m]$) have been defined for all $m\leq n$.  Note that Definitions \ref{def-graph-language}$-$\ref{def-io-star} then define the notions (A$[n+1]$)$-$(F$[n+1]$).

For any object $\Obj\in\OBJ_{n+1}$ and any operation $\op \in \Obj$, an external tape $a\in\I_\op^\ext$ is called \emph{immutable} in $\op$ if for any $(w_1,s_1,w_2,s_2)\in\IO_\op$, $w_1(a) = w_2(a)$. 
\end{definition}

The goal of the rest of this section is to illustrate these definitions by calculating the values and $\IO$-relations for some example $S^*$-graphs.  First, though, some machinery is developed that will allow for such arguments (both here and in the sections that follow) in an efficient manner.

Let $\G=(S_\G,E_\G)$ be an $S^*$-graph.  For any $e\in E_\G$, let $g(e)$ be the set consisting of the four generic edges constructed from $e$, i.e $g(e)=\bigl\{(\sf_1,\sf_2,e)\mid\sf_1,\sf_2\in\{\s,\f\}\bigr\}$.  Then, let $\IO_e$ be the relation on $\GL_\G\times S_\G$ given by $\IO_e=\bigcup_{e'\in g(e)}\IO_{e'}$.

 \begin{lemma}\label{lem-reach-rst}
	For any S$^*$-graph $\G=(S_\G,E_\G)$, $\REACH_\G$ is the rst-closure of $\bigcup\limits_{e\in E_\G}\IO_e$.
\end{lemma}
\begin{proof}

Let $U=\bigcup\IO_e$ and $U^*$ be the rst-closure of $U$.

For any $e\in E_\G$, $(w_1,s_1,w_2,s_2)\in\IO_e$ if and only if there exists a generic edge $e'\in g(e)$ such that $(w_1,s_1,w_2,s_2)\in\IO_{e'}$.  But this is the case if and only if $\bigl((w_1,s_1),e',(w_2,s_2)\bigr)$ is a computation of $\G$, i.e $(w_1,s_1,w_2,s_2)\in\REACH_\G$.  Hence, $U\subseteq\REACH_\G$, so that \Cref{lem-reach-star} implies that $U^*\subseteq\REACH_\G$.

Conversely, let $\eps=\bigl((w_0,s_0),e_1',(w_1,s_1),e_2',\dots,e_m',(w_m,s_m)\bigr)$ be a computation of $\G$.  Then, for all $i\in\overline{m}$, let $e_i\in E_\G$ and $\sf_i,\sf_i'\in\{\s,\f\}$ such that $e_i'=(\sf_i,\sf_i',e_i)$.  So, $(w_{i-1},s_{i-1},w_i,s_i)\in\IO_{e_i}$, i.e $(w_{i-1},s_{i-1},w_i,s_i)\in U$ for all $i$.  Hence, by transitivity, $(w_0,s_0,w_m,s_m)\in U^*$.

\end{proof}

Let $\Obj\in\OBJ^*$, $\op\in\Obj$, and $C$ be any set of configurations of $\op$.  For any $T\subseteq S_\op$, define the set of configurations $C^T=C\cap(\GL_\op\times T)$.  When $T$ is a singleton in $\SF_\op$, then the simpler notation $C^s$ or $C^f$ is used accordingly.  

Then, let $C_\IO^1=C^{\SF_\op}\cap\bigl(\ND_\op^1\times S_\op\bigr)$ and define $C_\IO\subseteq\ND_\op\times\SF_\op$ by $(w,s)\in C_\IO$ if and only if $(w\sqcup1_\op^\int,s)\in C_\IO^1$.

\begin{lemma}\label{lem-io-class}

Let $\Obj\in\OBJ^*$ and $\op\in\Obj$.

\begin{enumerate}

\item For any equivalence class $C$ of $\REACH_\op$, $C_\IO$ is either empty or an equivalence class of $\IO_\op$.

\item For any equivalence class $Q$ of $\IO_\op$, there exists a unique equivalence class $C$ of $\REACH_\op$ such that $Q=C_\IO$.

\end{enumerate}

\end{lemma}

\begin{proof}

(1) Suppose there exists $(w_1,s_1)\in C_\IO$.  Then $(w_1\sqcup1_\op^\int,s_1)\in C_\IO^1\subseteq C$.

Then, for any $(w_2,s_2)\in\ND_\op\times\SF_\op$, \Cref{lem-nse-star} implies that $(w_1,w_2)\in\IO_\op^{s_1,s_2}$ if and only if $$(w_1\sqcup1_\op^\int,w_2\sqcup1_\op^\int)\in\REACH_\op^{s_1,s_2}$$

But since $C$ is an equivalence class of $\REACH_\op$, $(w_1\sqcup1_\op^\int,w_2\sqcup1_\op^\int)\in\REACH_\op^{s_1,s_2}$ if and only if $(w_2\sqcup1_\op^\int,s_2)\in C$, which is true if and only if $(w_2,s_2)\in C_\IO$.

Thus, $C_\IO=\{(w_2,s_2)\in\ND_\op\times\SF_\op\mid(w_1,w_2)\in\IO_\op^{s_1,s_2}\}$, i.e the equivalence class of $\IO_\op$ containing $(w_1,s_1)$.

\bigskip

\noindent
(2)  Let $(w_1,s_1)\in Q$ and let $C\subseteq\GL_\op\times S_\op$ be the equivalence class of $\REACH_\op$ containing $(w_1\sqcup1_\op^\int,s_1)$.

Then $(w_2,s_2)\in C_\IO$ if and only if $(w_2\sqcup1_\op^\int,s_2)\in C$, and so if and only if $$(w_1\sqcup1_\op^\int,w_2\sqcup1_\op^\int)\in\REACH_\op^{s_1,s_2}$$ 
But as above, this last condition is true if and only if $(w_1,w_2)\in\IO_\op^{s_1,s_2}$, and so if and only if $(w_2,s_2)\in Q$.  Thus, $C_\IO=Q$.

Moreover, for any equivalence class $D$ of $\REACH_\op$ satisfying $D_\IO=Q$, then $(w_1,s_1)\in D_\IO$ implying $(w_1\sqcup1_\op^\int,s_1)\in D$.  But then $D\cap C\neq\emptyset$, and so $D=C$.

\end{proof}

\begin{proposition}\label{prop-io-classes}
	Let $\Obj \in \OBJ^*$ and $\op\in \Obj$.  Suppose $\{C_j\}_{j\in J}$ is a family of (not necessarily disjoint) subsets on $\GL_\op\times S_\op$ such that
	\begin{enumerate}
		\item For all $j\in J$, $C_j$ is closed under $\IO_e$ for every $e\in E_\op$
		\item $\bigcup_{j\in J} C_j \supseteq \ND_\op^1 \times\SF_\op$.
	\end{enumerate}
	Then $\IO_\op \subseteq \bigcup_{j\in J}\bigl(C_{j,\IO}\bigr)^2$, where $C_{j,\IO}=\bigl(C_j\bigr)_\IO$ for all $j\in J$.
	\end{proposition}
	
\begin{proof}

Suppose $(u_1,u_2)\in\IO_\op^{s_1,s_2}$ for some $s_1,s_2\in\SF_\op$.  Then, for $w_i=u_i\sqcup1_\op^\int$ for $i=1,2$, there exists an input-output computation of $\op$ between $(w_1,s_1)$ and $(w_2,s_2)$.  Denote such a computation $\bigl((w_0',s_0'),e_1',(w_1',s_1'),e_2',\dots,e_m',(w_m',s_m')\bigr)$, i.e so that $(w_0',s_0')=(w_1,s_1)$ and $(w_m',s_m')=(w_2,s_2)$.

For all $i\in\overline{m}$, let $\sf_i,\sf_i'\in\{\s,\f\}$ and let $e_i\in E_\op$ be such that $e_i'=(\sf_i,\sf_i',e_i)$.  Then, $(w_{i-1}',s_{i-1}',w_i',s_i')\in\IO_{e_i}$ for all $i\in\overline{m}$.  Hence, (1) implies that for any $i\in\overline{m}$ and $j\in J$, $(w_{i-1}',s_{i-1}')\in C_j$ if and only if $(w_i',s_i')\in C_j$.

Hence, for any $j\in J$, $(w_1,s_1)\in C_j$ if and only if $(w_2,s_2)\in C_j$.  Moreover, since by construction $w_1\in\ND_\op^1$, hypothesis (2) implies the existence of $k\in J$ such that $(w_1,s_1)\in C_k$.  As a result, $(w_1,s_1),(w_2,s_2)\in C_k$.  But $w_i\in\ND_\op^1$ and $s_i\in\SF_\op$ for $i=1,2$, so that $(w_1,s_1),(w_2,s_2)\in C_{k,\IO}^1$.  Thus, $(u_1,s_1),(u_2,s_2)\in C_{k,\IO}$.

\end{proof}

\begin{proposition}\label{prop-io-classes2}
	Let $\Obj \in \OBJ^*$ and $V\subseteq L^\val_\Obj$.  Suppose that for all $\op\in\Obj$, there exists a family of subsets of $ \GL_\op\times S_\op$, $\C_\op = \{C_{j, \op}\}_{j\in J_\op}$, such that:
	\begin{enumerate}
		\item For all $j\in J_\op$, $C_{j, \op}$ is closed under $\IO_e$ for every $e\in E_\op$
		\item $\bigcup_{j\in J_\op} C^{\SF_\op}_{j, \op}= \widehat{V}_\op\times \SF_\op$, where $\widehat{V}_\op=L_\op^\ext \times V \times 1^\int_\op \subseteq L_\op$.
	\end{enumerate}
	Then either $\VAL_\Obj\cap V=\emptyset$ or $\VAL_\Obj \subseteq V$.
\end{proposition}

\begin{proof}

For any $v_1,v_2\in L_\Obj^\val$, there exists an edge between $v_1$ and $v_2$ in $\G_\Obj^\val$ if and only if there exists $\op\in\Obj$ and an input-output computation of $\op$ between $(w_1,s_1)$ and $(w_2,s_2)$ such that $w_i\bigl(\I_\Obj^\val\bigr)=v_i$ for $i=1,2$.

As in the proof of \Cref{prop-io-classes}, (1) implies that for any $j\in J_\op$, $(w_1,s_1)\in C_{j,\op}$ if and only if $(w_2,s_2)\in C_{j,\op}$.  But then by (2), $(w_i,s_i)\in C_{j,\op}$ if and only if $w_i\in\widehat{V}_\op$, i.e if and only if $v_i\in V$.

Hence, $V$ is the disjoint union of connected components of $\G_\Obj^\val$, so that the statement follows from the definition of $\VAL_\Obj$.

\end{proof}

Recall that in \Cref{def-counter}, an object (of degree 0) was defined to be a \emph{counter} if it contains two operations (of degree 0) such that the values of the object and the $\IO$-relations of its operations satisfy certain restrictions.  This definition is adapted naturally for objects of arbitrary degree:

\begin{definition}\label{def-counter-star}

Let $\Obj\in\Obj^*$ be an object consisting of two operations, denoted $\inc$ and $\check$, with no external tapes.  Suppose there exists an enumeration $\VAL_\Obj=\{v_i\}_{i=0}^\infty$ with $v_0=1_\Obj^\val$ such that:
\begin{align*}
		\IO^{s,s}_\inc &= \IO^{f,f}_\inc = \Id;\\
		\IO^{s,s}_\check &= \IO^{f,f}_\check = \Id; \\
		\IO^{s,f}_\inc &= \bigl\{(v_i, v_{i+1})~|~i\geq 0\bigr\}; \\
		\IO^{s,f}_\check &= \bigl\{(v_i, v_i)~|~i\geq 1\bigr\}.
	\end{align*}
	Then $\Obj$ is called a \emph{counter}.

\end{definition}

Notice that if $\Obj\in\OBJ_0$, then this definition agrees with that of \Cref{def-counter}.

Recall that in \Cref{ex-counter}, an object of degree 0, denoted Counter$_0[a]$, was presented and studied.  There, it was claimed without proof that the object is a counter.  While Propositions \ref{prop-io-classes} and \ref{prop-io-classes2} make the proof of this claim feasible, it is omitted.

Instead, a new object is constructed, denoted Counter$_1[a]$.  This object will be proved to be a counter (see \Cref{lem-counter-is-counter}), but has a much simpler design.  However, as will be expanded upon in the next section, this object will virtually correspond to Counter$_0$, with the concise presentation illustrating the advantages of $S^*$-graphs.

\begin{example}\label{ex-counter1}

Let Counter$_1[a]$ be the object depicted in \Cref{fig-COUNTER2}.  As is customary (and with abuse of notation), its two operations are denoted $\check$ and $\inc$, each of which has no external tapes.

\begin{figure}[hbt]
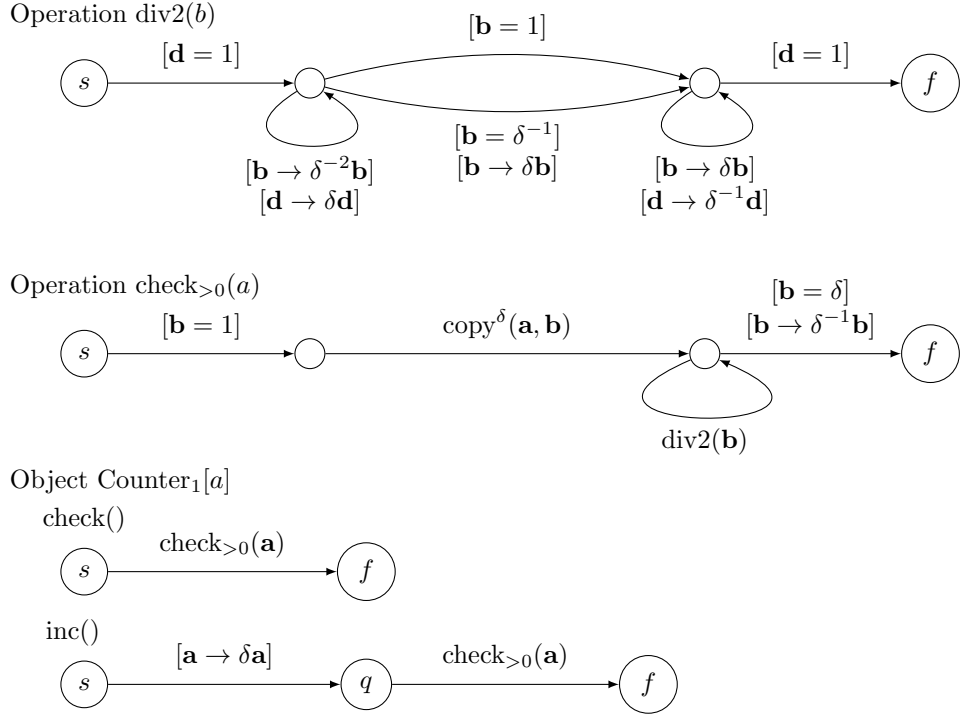

	\centering
	\include{COUNTER2}
	\caption{A Counter$_1[a]$ object is a counter.}
	\label{fig-COUNTER2}       
\end{figure} 

As no instances are declared before presenting the operations, $\INST_{\textrm{Counter}_1}$ consists only of the proper operations labelling the edges of the graphs of $\check$ and $\inc$.  Here, that is $\check_{>0}$, a proper operation also presented in the figure.

Naturally, to understand Counter$_1$, attention is first focused on the proper operation $\check_{>0}$.  As with Counter$_1$, $\INST_{\check_{>0}}$ consists of the proper operations labelling its edges, i.e $\cpy$ and $\divtwo$.  However, while $\cpy$ has already been studied (see \Cref{ex-copy}), $\divtwo$ has not and is again presented in the figure.  Hence, the discussion begins with the operation $\divtwo$.

As the action labelling every edge is constructed from fragments, $\divtwo$ is a proper operation of degree 0.  The notation $\divtwo(b)$ indicates that $\I_{\divtwo}^\ext=\{b\}$ and that the remaining indices used in actions labelling the edges must be internal.  Hence, $\I_{\divtwo}^\int=\{d\}$.  Finally, as the only letter of $\A^*$ appearing in fragments on the graph is $\delta$, and moreover this letter appears both in fragments mentioning $b$ and those mentioning $d$, $\A_\divtwo(b)=\A_\divtwo(d)=\{\delta\}$.

Condition (O) is easily verified for $\divtwo$, as the only two edges whose tails are $\divtwo_s$ or $\divtwo_f$ contain the guard fragment $[\mathbf{d}=1]$.

Finally, for clarity, denote the elements of $\GL_\divtwo=L_{\H_{\divtwo}}$ by $w=(w(b),w(d))$, i.e adopting the convention that the first coordinate of the pair represents the $b$-coordinate.  Hence, for all $s_1,s_2\in\SF_\divtwo$, the relation $\IO_\divtwo^{s_1,s_2}$ consists of the pairs $(w_1,w_2)\in F(\delta)^2$ such that there exists a computation of $\divtwo$ between the configurations $((w_1,1),s_1)$ and $((w_2,1),s_2)$.

Note that since $\dt$ is an $S$-graph of degree 0, $\IO_\dt$ can be calculated in much the same way as in previous sections (e.g as in \Cref{lem-io-cpy}).  However, it is approached below instead using the tools developed in this section, demonstrating their efficacy and developing a blueprint for similar arguments related to $S^*$-graphs that follow.

\begin{lemma}\label{lem-ndr-div2}
	The relation $\IO_{\dt}$ is given by:
	\begin{align*}
		\IO^{s,s}_\dt &= \bigr\{(\delta^{2j - \alpha}, \delta^{2j-\beta})\mid j\in \Z, \ \alpha,\beta \in \{0,1\} \bigr\};\\
		\IO^{s,f}_\dt &= \bigr\{(\delta^{2j - \alpha}, \delta^{j})\mid j\in \Z, \ \alpha \in \{0,1\} \bigr\};\\
		\IO^{f,f}_\dt &= \bigr\{(\delta^{j}, \delta^{j})~|~j\in \Z\bigr\},
	\end{align*}
\end{lemma}
\begin{proof} 
For ease of reference, the graph of $\dt$ is copied below with some vertices and edges labelled.

\begin{center}
	\input{DIV2}
\end{center}

For all $j\in\Z$, define $C_j\subseteq\GL_\dt\times S_\dt$ by:
\begin{align*}
	C_j^s &= \left\{(\delta^{2j-\alpha}, 1) \mid \alpha\in \{0,1\}\right\}, \\
	C_j^{q_1} &= \left\{(\delta^{2j-2i-\alpha}, \delta^i) \mid i\in \Z, \ \alpha\in \{0,1\}\right\}, \\
	C_j^{q_2} &= \left\{(\delta^{j-i}, \delta^i) \mid i\in \Z\right\}, \\
	C_j^f &= \left\{(\delta^{j}, 1)\right\}.
\end{align*}
For any $j\in\Z$ and $k\in\overline{6}$, it is straightforward to check that $C_j$ is closed under $\IO_{e_k}$:  

\begin{enumerate}

\item Suppose $(w_1,w_2)\in\IO_{e_1}^{s_1,s_2}$.  If $s_1=s_2$, then $(w_1,s_1)=(w_2,s_2)$.  Otherwise, suppose without loss of generality that $s_1=\dt_s$, i.e $(w_1,w_2)\in\IO_{e_1}^{s,q_1}$.  Then for $(t,g)=\lab(e_1)$, $g$ accepts $w_1$ and $w_2=tw_1$.  But then $w_1(d)=1$ and $w_2=w_1$.  Thus, $(w_2,s_2)\in C_j$ if and only if $w_1(b)=w_2(b)=\delta^{2j-\a}$ for some $\a\in\{0,1\}$, which is true if and only if $(w_1,s_1)\in C_j$.

\item Suppose $(w_1,w_2)\in\IO_{e_2}^{s_1,s_2}$.  Then $s_1=s_2=q_1$.  Moreover, if $w_1=(\delta^m,\delta^n)$, then there exists $\ell\in\{-1,0,1\}$ such that $w_2=(\delta^{m-2\ell},\delta^{n+\ell})$.  But then $w_1\in C_j^{q_1}$ if and only if $w_2\in C_j^{q_1}$.

\item Suppose $(w_1,w_2)\in\IO_{e_3}^{s_1,s_2}$.  As in (1), it can be assumed without loss of generality that $s_1=q_1$ and $s_2=q_2$.  Then, the definition of $\lab(e_3)$ implies that $w_1(b)=1$ and $w_2=w_1$.  As a result, $(w_2,s_2)\in C_j$ if and only if $w_1(d)=w_2(d)=j$, which is true if and only if $(w_1,s_1)\in C_j$.

\item Suppose $(w_1,w_2)\in\IO_{e_4}^{s_1,s_2}$.  As in (3), it may be assumed that $s_1=q_1$ and $s_2=q_2$. Then, $w_1(b)=\delta^{-1}$, $w_2(b)=1$, and $w_1(d)=w_2(d)$.  But again this implies that $(w_1,s_1)\in C_j$ if and only if $(w_2,s_2)\in C_j$.

\item Suppose $(w_1,w_2)\in\IO_{e_5}^{s_1,s_2}$.  As in (2), this implies that $s_1=s_2=q_2$ and if $w_1=(\delta^m,\delta^n)$, then $w_2=(\delta^{m-\ell},\delta^{n+\ell})$ for some $\ell\in\{-1,0,1\}$.  But $(\delta^\a,\delta^\beta)\in C_j^{q_2}$ if and only if $\a+\beta=j$, so that $(w_1,s_1)\in C_j$ if and only if $(w_2,s_2)\in C_j$.

\item Suppose $(w_1,w_2)\in\IO_{e_6}^{s_1,s_2}$.  As in (1), it can be assumed without loss of generality that $s_1=q_2$ and $s_2=\dt_f$, in which case the action labelling $e_6$ implies that $w_1(d)=1$ and that $w_2=w_1$.  So, $(w_2,s_2)\in C_j$ if and only if $w_1(b)=w_2(b)=\delta^j$, which is true if and only if $(w_1,s_1)\in C_j$.

\end{enumerate}

As all edges of $\dt$ are labelled with an action of degree 0, there is an implicit inverse edge to each not pictured.  However, as with any edge labelled with an action of degree $0$, $\IO_{e_k^{-1}}=\IO_{e_k}^{-1}$ for all $k\in\overline{6}$, and hence $C_j$ is also closed under $\IO_{e_k^{-1}}$ for all $j\in\Z$.

Moreover, note that $$\ND_\dt^1\times\SF_\dt=\bigl\{((\delta^i,1),\sf)\mid i\in\Z, \ \sf\in\{\dt_s,\dt_f\}\bigr\}=\bigcup_{j\in\Z}(C_j^s\cup C_j^f)$$

Thus, the family of subsets $\{C_j\}_{j\in \Z}$ satisfies the hypotheses laid out in \Cref{prop-io-classes}, and so $\IO_\dt \subseteq\bigcup_{j\in\Z} \bigl(C_{j,\IO}\bigr)^2$.  Note that as constructed, the set $U=\bigcup_{j\in\Z} \bigl(C_{j,\IO}\bigr)^2$ consists of exactly the pairs described in the statement.  Hence, it suffices to exhibit input-output computations of $\dt$ which imply the reverse containment.

\Cref{lem-io-star} implies that $\IO_\dt$ is an equivalence relation on $\ND_\dt\times\SF_\dt$.  So, since any $(w_1,\dt_f,w_2,\dt_f)\in U$ satisfies $w_1=w_2$, $(w_1,w_2)\in\IO_\dt^{f,f}$ by reflexivity.  

Fix $n\in\N$ and $\tau\in\{\pm1\}$.  Then the path $e_1,(e_2^\tau)^n,e_3,(e_5^\tau)^n,e_6$ in $\dt$ (where $e^n$ represents a subpath consisting of a loop $e$ repeated $n$ times) supports a computation $\eps_{n,\tau}$ between $((\delta^{2\tau n},1),\dt_s)$ and $((\delta^{\tau n},1),\dt_f)$.  Similarly, the path $e_1,(e_2^\tau)^n,e_4,(e_5^\tau)^n,e_6$ supports a computation $\eps_{n,\tau}'$ between $((\delta^{2\tau n-1},1),\dt_s)$ and $((\delta^{\tau n},1),\dt_f)$.  

Hence, $(\delta^{2j-\a},\delta^j)\in\IO_\dt^{s,f}$ for all $j\in\Z$ and $\a\in\{0,1\}$.  By symmetry, any $(w_1,s_1,w_2,s_2)\in U$ such that $s_1\neq s_2$ satisfies $(w_1,w_2)\in\IO_\dt^{s_1,s_2}$.

Finally, for $(w_1,\dt_s,w_2,\dt_s)\in U$, there exist $j\in\Z$ and $\a,\beta\in\{0,1\}$ such that $w_1=\delta^{2j-\a}$ and $w_2=\delta^{2j-\beta}$.  But then $(w_1,\delta^j)\in\IO_\dt^{s,f}$ and $(\delta^j,w_2)\in\IO_\dt^{f,s}$, so that $(w_1,w_2)\in\IO_\dt^{s,s}$ by transitivity.

\end{proof}

The proof of \Cref{lem-ndr-div2} is a prototypical example of the proofs of the $\IO$-relations of operations that follow: Methodically choose sets $C_j$ which satisfy the hypotheses of \Cref{prop-io-classes}, then exhibit input-output computations which imply the statement.  As this proof is the template, particular care was given to outlining the details, especially in checking that each $C_j$ was closed under the $\IO$-relation of each edge.  As checking such closure is generally straightforward, though, it is common in what follows for the details to be omitted.

With $\IO_\dt$ now calculated, attention now returns to the proper operation $\check_{>0}$.  For presentational purposes, the notation ch\textgreater0 represents the operation when it is used in a subscript, i.e $\G_{\ch>0}$ is an $S^*$-graph over $(\H_{\ch>0},\INST_{\ch>0})$ where $\H_{\ch>0}=(\I_{\ch>0},\A_{\ch>0})$.

The notation $\check_{>0}(a)$ indicates that $\I_{\ch>0}^\ext=\{a\}$ and that all remaining indices are internal, i.e $\I_{\ch>0}^\int=\{b\}$.  Note that the internal indices of both $\dt$ and $\cpy$ do not appear in $\I_{\ch>0}$ even though $\dt$ and $\cpy$ label edges.

The fragments comprising the actions of degree 0 labelling edges of $\check_{>0}$ immediately imply that $\delta\in\A_{\ch>0}(b)$.  As also $\A_\dt(b)=\{\delta\}$ and $\A_\cpy(a)=\A_\cpy(b)=\{\delta\}$, it follows that $\A_{\ch>0}(a)=\A_{\ch>0}(b)=\{\delta\}$.

Again, condition (O) is easily verified for $\check_{>0}$ since any edge with tail in $\SF_{\ch>0}$ is labelled with an action of degree 0 containing the guard fragment $[\mathbf{b}=1]$.

Further, similar to the conventions adopted for $\dt$, denote the elements of $\GL_{\ch>0}=L_{\H_{\ch>0}}$ by $w=(w(a),w(b))$, i.e so that the first coordinate represents the $a$-coordinate.  Hence, for all $s_1,s_2\in\SF_{\ch>0}$, the elements of $\IO_{\ch>0}^{s_1,s_2}$ are the pairs $(w_1,w_2)\in F(\delta)^2$ such that there exists a computation of $\check_{>0}$ between $((w_1,1),s_1)$ and $((w_2,1),s_2)$.

\begin{lemma}\label{lem-ndr-checkstar}
	The relation $\IO_{\ch>0}$ is given by:
	\begin{align*}
		\IO^{s,s}_{\ch>0} &= \IO^{f,f}_{\ch>0} = \Id;\\
		\IO^{s,f}_{\ch>0} &= \bigr\{(\delta^{i}, \delta^{i})~|~i \geq 1\bigr\},
	\end{align*}
	In particular, $a$ is immutable in $\check_{>0}(a)$.
\end{lemma}
\begin{proof} 
As in the proof of \Cref{lem-ndr-div2}, the graph of $\check_{>0}$ is reproduced below with some extra labels to aid in the proof.
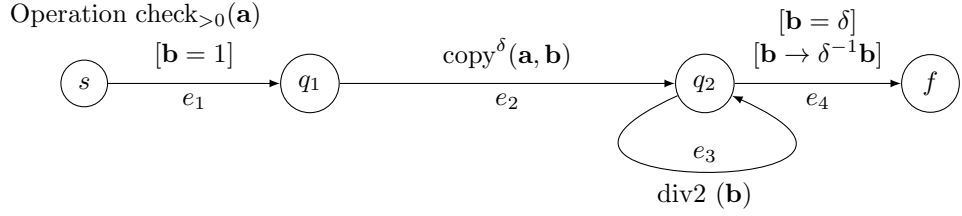
\begin{center}
	\input{CHECKSTAR}
\end{center}

Define the subsets $C_1,C_2,C_3\subseteq\GL_{\ch>0}\times S_{\ch>0}$ by:
\begin{align*}
	C_1^s=C_1^{q_1}=C_1^f&=\{(\delta^i,1)\mid i\geq1\} \\
	C_1^{q_2}&=\{(\delta^i,\delta^j)\mid i,j\geq1\} \\
	C_2^s = C_2^{q_1} &= \emptyset, \\
	C_2^{q_2} &= \left\{(\delta^i, \delta^j)~|~i \leq 0, j  \geq 1\right\}, \\
	C_2^f &= \left\{(\delta^i, 1)~|~i \leq 0\right\} \\
	C_3^s=C_3^{q_1}&=\{(\delta^i,1)\mid i\leq0\} \\
	C_3^{q_2}&=\{(\delta^i,\delta^j)\mid i,j\leq0\} \\
	C_3^f&=\emptyset
\end{align*}

As in the proof of \Cref{lem-ndr-div2}, to show that each $C_j$ is closed under $\IO_e$ for all $e\in E_{\ch>0}$, it suffices to show that it is closed under $\IO_{e_k}$ for all $k\in\overline{4}$.  

For all $m\in\Z$, let $\widehat{m}\in\Z$ and $r_m\in\{0,1\}$ be the unique integers satisfying $m=2\widehat{m}-r_m$.  Then, define the set $D(m)=\bigl\{\widehat{m},2\widehat{m}-\a,4\widehat{m}-2r_m-\a\mid \a\in\{0,1\}\bigr\}$.  Note that $D(m)\subseteq\N-\{0\}$ if and only if $m\in\N-\{0\}$.

By the definition of the $\IO$-relation on edges, $(w_1,w_2)\in\IO_{e_3}^{s_1,s_2}$ if and only if $s_1=s_2=q_2$, $w_1(a)=w_2(a)$, and $(w_1(b),w_2(b))\in\IO_\dt^{\sf_1,\sf_2}$ for some $\sf_1,\sf_2\in\{s,f\}$.  Hence, for $w_1=(\delta^n,\delta^m)$, \Cref{lem-ndr-div2} implies that $(w_1,w_2)\in\IO_{e_3}^{q_2,q_2}$ if and only if $w_2=(\delta^n,\delta^\ell)$ for some $\ell\in D(m)$.  Thus, each $C_j$ is closed under $\IO_{e_3}$.

The closure of each $C_j$ under $\IO_{e_k}$ for all other $k$ follows from analogous arguments.  For example, for $k=2$, appealing to \Cref{lem-io-cpy} in place of \Cref{lem-ndr-div2} produces a similar proof.

Moreover, note that $$\ND^1_{\ch>0}\times\SF_{\ch>0}=\{((\delta^i,1),\sf)\mid i\in\Z, \ \sf\in\SF_{\ch>0}\}=\bigcup_{j\in\overline{3}}(C_j^s\cup C_j^f)$$

Hence, \Cref{prop-io-classes} can be applied to $\{C_j\}_{j=1}^3$, so that $\IO_{\ch>0}\subseteq\bigcup_{j\in\overline{3}}\bigl(C_{j,\IO}\bigr)^2$.  As in the proof of \Cref{lem-ndr-div2}, the set $U=\bigcup_{j\in\overline{3}}\bigl(C_{j,\IO}\bigr)^2$ consists exactly of the pairs described in the statement, so that it suffices to exhibit input-output computations of $\check_{>0}$ which imply the reverse containment.

Note that for all $(w_1,s_1,w_2,s_2)\in U$ such that $s_1=s_2$, it holds that $w_1=w_2$, so that $(w_1,w_2)\in\IO_{\ch>0}^{s_1,s_2}$ by reflexivity.

Fix $m\geq1$ and let $\rho_k=e_1',e_2',(e_3')^k$ be the generic path in $\check_{>0}$ given by taking $e_i'=(\s,\f,e_i)$ and $(e_3')^k$ represents the generic path $\underbrace{e_3',\dots,e_3'}_{k\text{ times}}$.  

Then, by Lemmas \ref{lem-io-cpy} and \ref{lem-ndr-div2}, $\rho_k$ supports a computation of the form:
\begin{align*}
	\bigl((\delta^m, 1, (\ch_{>0})_s), e_1', (\delta^m, 1, q_1), &e_2' ,(\delta^m, \delta^m, q_2), e_3', (\delta^m, \delta^{\ell_1}, q_2), e_3', \dots e_3', (\delta^m, \delta^{\ell_k}, q_2)\bigr)
\end{align*}
where, setting $\ell_0=m$, $\ell_i= \lceil \ell_{i-1}/2 \rceil$ for all $i\in\overline{k}$.  

So, setting $r=\lceil\log_2(m) \rceil$, $1 \leq \ell_i \leq 2^{\max(r - i, 0)}$ for all $i\in\overline{k}$. In particular, assuming $k\geq r$, $\rho_k$ supports a computation between $(\delta^m,1,(\check_{>0})_s)$ and $(\delta^m,\delta,q_2)$.  

But $((\delta^m,\delta),(\delta^m,1))\in\IO_{e_4}^{q_2,f}$.  Hence, for $e_4'=(\s,\f,e_4)$, the generic path $\rho_r,e_4'$ supports an input-output computation $\eps_m$ between $(\delta^m,1,(\check_{>0})_s)$ and $(\delta^m,1,(\check_{>0})_f)$.  

Thus, $(\delta^m,\delta^m)\in\IO_{\ch>0}^{s,f}$, so that $U\subseteq\IO_{\ch>0}$ by symmetry.

\end{proof}

While the proof of \Cref{lem-ndr-checkstar} presented above is similar to that presented for \Cref{lem-ndr-div2}, it can be streamlined with a simple observation:

\begin{lemma}\label{lem-immutable}

Let $\Obj\in\OBJ^*$, $\op\in\Obj$, and $i\in\I_\op^\ext$.  Suppose that for all $\act\in\ACT_\op\cup\ACT_\op^0$, $\act$ either does not mention $i$ or mentions $i$ as an immutable variable.  Then $i$ is immutable in $\op$.

\end{lemma}

\begin{proof}

By hypothesis, for any generic edge $e'$ of $\op$ and $t_1,t_2\in\{t(e'),h(e')\}$, $(u_1',u_2')\in\IO_{e'}^{t_1,t_2}$ implies that $u_1'(i)=u_2'(i)$.  So, for any computation of $\op$ between two configurations $(w_1',s_1')$ and $(w_2',s_2')$, $w_1'(i)=w_2'(i)$.

Thus, since the relation $(w_1,w_2)\in\IO_\op^{s_1,s_2}$ implies the existence of a computation between $(w_1\sqcup1_\op^\int,s_1)$ and $(w_2\sqcup1_\op^\int,s_2)$, $i$ must be immutable in $\op$.

\end{proof}

So, as the only edge of $\check_{>0}$ mentioning $a$ is $e_2$ and \Cref{lem-io-cpy} implies that $\lab(e_2)$ mentions $a$ as an immutable variable, $\IO_{\ch>0}^{s_1,s_2}\subseteq\Id$ for all $s_1,s_2\in\SF_{\ch>0}$.  Reflexivity then immediately implies that $\IO_{\ch>0}^{s,s}=\IO_{\ch>0}^{f,f}=\Id$, the existence of $\eps_m$ implies that $(\delta^m,\delta^m)\in\IO_{\ch>0}^{s,f}$ for all $m\geq1$, and the closure of $C_2$ under $\IO_{e_k}$ for all $k\in\overline{4}$ implies that $(\delta^m,\delta^m)\notin\IO_{\ch>0}^{s,f}$ for all $m\leq0$.

Such an argument is useful for streamlining the calculation of the $\IO$-relation of an operation, though can only be employed in certain circumstances.

Finally, with $\check_{>0}$ and $\dt$ fully understood, attention returns to the object Counter$_1$ depicted in \Cref{fig-COUNTER2}.  Similar to the presentational convention for $\check_{>0}$, when used as a subscript, Counter$_1$ is denoted $\Cntr_1$.

The notation Counter$_1[a]$ indicates that $\I_{\Cntr_1}^\val=\{a\}$.  Further, the notations $\check()$ and $\inc()$ indicate that $\I_\check^\ext=\I_\inc^\ext=\emptyset$. Finally, the only index mentioned by an edge of either $\check$ or $\inc$ is $a$, so that $\I_\check^\int=\I_\inc^\int=\emptyset$.  Hence, condition (O) is satisfied trivially by both $\inc$ and $\check$.

As $\A_{\ch>0}(a)=\{\delta\}$, the only letter of an alphabet in $\A^*$ mentioned on an edge of $\check$ or $\inc$ is $\delta$.  Thus, $\A_{\Cntr_1}(a)=\A_\check(a)=\A_\inc(a)=\{\delta\}$.

\begin{lemma}\label{lem-counter-is-counter}
	The object $\textrm{Counter}_1[a]$ is a counter.
	
\end{lemma}

\begin{proof}

	Let $V=\{\delta^i\}_{i=0}^\infty\subseteq F(\A_{\Cntr_1}(a))$. So, $\widehat{V}_\inc=\widehat{V}_\check=V$.
	
	Define the sets configurations $C_{i,\check}\subseteq\GL_\check\times S_\check$ for $i\in\overline{3}$ by:
	\begin{align*}
	C_{1,\check}^s=C_{1,\check}^f&=\{\delta^i\}_{i=1}^\infty \\
	C_{2,\check}^s=C_{3,\check}^f&=\emptyset \\
	C_{2,\check}^f=C_{3,\check}^s&=\{1\} \\
	\end{align*}
	By \Cref{lem-ndr-checkstar}, $C_{i,\check}$ is closed under $\IO_e$ for the unique $e\in E_\check$.  Moreover, 
	$$\bigcup_{i\in\overline{3}}C_{i,\check}^{\SF_\check}=\{(\delta^i,\sf)\mid i\geq0, \ \sf\in\SF_\check\}=\widehat{V}_\check\times\SF_\check$$
	
	Next, define the sets of configurations $C_{j,\inc}\subseteq\GL_\inc\times S_\inc$ for $j\geq0$ by:
	\begin{align*}
	C_{j,\inc}^s&=\{\delta^j\} \\
	C_{j,\inc}^q=C_{j,\inc}^f&=\{\delta^{j+1}\}
	\end{align*}
	Further, define $C_{-1,\inc}=\{(1,\inc_f)\}\subseteq\GL_\inc\times S_\inc$.
	
	Then, \Cref{lem-ndr-checkstar} again makes it straightforward to verify that for all $j\geq-1$, $C_{j,\inc}$ is closed under $\IO_e$ for all edges $e\in E_\inc$.  What's more, again
	\begin{align*}
	\bigcup_{j=-1}^\infty C_{j,\inc}^{\SF_\inc}&=\{(\delta^i,\sf)\mid i\geq0, \ \sf\in\SF_\inc\}=\widehat{V}_\inc\times\SF_\inc
	\end{align*}
	Hence, the hypotheses of \Cref{prop-io-classes2} are satisfied.  So, since $\delta^0=1_{\Cntr_1}^\val\in\VAL_{\Cntr_1}\cap V$, $\VAL_{\Cntr_1}\subseteq V$.
	
	Now, let $e_1$ and $e_2$ be the edges of $\inc$ such that $t(e_1)=\inc_s$, $h(e_1)=t(e_2)=q$, and $h(e_2)=\inc_f$.  Then, let $e_k'=(\s,\f,e_k)$ for $k=1,2$.
	
	It follows that for any $i\geq0$, the generic path $e_1',e_2'$ in $\inc$ supports an input-output computation $\eps_i$ of $\inc$ between $(\delta^i,\inc_s)$ and $(\delta^{i+1},\inc_f)$.  So, there exists an edge of $\G_{\Cntr_1}^\val$ between $\delta^i$ and $\delta^{i+1}$ corresponding to $\eps_i$.  As a result, for any $\delta^n\in V$ with $n\geq1$, there exists a path in $\G_{\Cntr_1}^\val$ between $1_{\Cntr_1}^\val$ and $\delta^n$ with path corresponding to $\eps_0,\dots,\eps_{n-1}$.  Thus, $V=\VAL_{\Cntr_1}$.
	
	Now, fix the natural enumeration $\VAL_{\Cntr_1}=\{v_i\}_{i=0}^\infty$ given by $v_i=\delta^i$ for all $i\in\N$.  By \Cref{lem-ndr-checkstar}, it follows immediately that $\IO_\check^{s,s}=\IO_\check^{f,f}=\Id$ and $\IO_\check^{s,f}=\{(v_i,v_i)\mid i\geq1\}$.
	
	Further, note that $$\ND_\inc\times\SF_\inc=\VAL_\inc\times\SF_\inc=\bigcup_{j=-1}^\infty C_{j,\inc}^{\SF_\inc}=\bigcup_{j=-1}^\infty\bigl(C_{j,\inc}\bigr)_\IO$$
	So, since each $C_{j,\inc}$ is closed under $\IO_e$ for each $e\in E_\inc$, $\{C_{j,\inc}\}$ satisfies the hypotheses of \Cref{prop-io-classes}, i.e $\IO_\inc\subseteq\bigcup\bigl(C_{j,\inc}\bigr)_\IO^2=\bigcup\bigl(C_{j,\inc}^{\SF_\inc}\bigr)^2$.
	
	For any $j\geq-1$, $C_{j,\inc}^s$ and $C_{j,\inc}^f$ are at most singletons.  Hence, by reflexivity, $\IO_\inc^{s,s}=\IO_\inc^{f,f}=\Id$.
	
	What's more, for $j\geq0$, $C_{j,\inc}^s\times C_{j,\inc}^f=\{(v_j,v_{j+1})\}$, while $C_{-1,\inc}^s\times C_{-1,\inc}^f=\emptyset$.  Hence, the computations $\eps_i$ described above imply $\IO_\inc^{s,f}=\{(v_i,v_{i+1})\mid i\geq0\}$.
	
	Thus, the values of Counter$_1$ and the $\IO$-relations of $\check$ and $\inc$ match the requirements for the definition of a counter.
		
\end{proof}

\end{example}

\section{Expansions}\label{sec-expansions}

As the proof of \Cref{main-theorem} requires the construction of an $S$-machine satisfying certain properties, it can be inferred from \Cref{th-MG} that it suffices to construct an $S^*$-graph $\G$ satisfying some other properties (these will be discussed in detail in \Cref{sec-acceptors}).  However, in order to apply the correspondence of \Cref{th-MG}, $\G$ must be an $S$-graph of degree 0.  Hence, while $S$-graphs of higher degree can be understood as efficient tools, some connection must be established between them and $S$-graphs of degree 0.  This gap is bridged by a process called \emph{expansion}.

The expansion of $S^*$-graphs, operations, and objects are defined iteratively with respect to degree.  To aid with the terminology of this construction, a more general process is first defined.

\begin{definition}\label{def-E-mapping}

Fix $n\geq0$ and suppose that the triple of mappings $\E^n_G:\SGRAPH_n'\to\SGRAPH_0$, $\E_\OP^n:\OP_n'\to\OP_0$, and $\E_\OBJ^n:\OBJ_n'\to\OBJ_0$ satisfy the following properties:

\begin{enumerate}

\item Let $\G=(S_\G,E_\G)$ be an $S$-graph of degree at most $n$ over $(\H_\G,\INST_\G)$ for $\H_\G=(\I_\G,\A_\G)$.  Then $\E^n_G(\G)=(S_{\G_\E},E_{\G_\E})$ is an $S$-graph (of degree 0) over $\H_{\G_\E}=(\I_{\G_\E},\A_{\G_\E})$ such that $S_\G\subseteq S_{\G_\E}$, $\H_\G\leq\H_{\G_\E}$, and $\A_{\G_\E}(i)=\A_\G(i)$ for all $i\in\I_\G$.

\item For every operation $\op$ of degree at most $n$, $\E_\OP^n(\op)=\op_\E$ is an operation of degree 0 such that $\G_{\op_\E}=\E_G^n(\G_\op)$, $\I_{\op_\E}^\ext=\I_\op^\ext$, $\I_{\op_\E}^\val=\I_\op^\val$, $(\op_\E)_s=\op_s$, and $(\op_\E)_f=\op_f$.

\item For every object $\Obj=(\I_\Obj^\val,\INST_\Obj,\op_1,\dots,\op_k)$ of degree at most $n$, $\E_\OBJ^n(\Obj)$ is an object of degree 0 such that $\E_\OBJ^n(\Obj)=(\I_\Obj^\val,\E_\OP^n(\op_1),\dots,\E_\OP^n(\op_k))$.

\end{enumerate}
Then the triple $(\mathcal{E}_G^n,\E_\OP^n,\E_\OBJ^n)$ is called a \emph{generalized $n$-expansion mapping}.

\end{definition}

Now, define $\EXP_G^0$, $\EXP_\OP^0$, and $\EXP_\OBJ^0$ as the identity maps on $\SGRAPH_0$, $\OP_0$, and $\OBJ_0$, respectively.  Note that $(\EXP_G^0,\EXP_\OP^0,\EXP_\OBJ^0)$ is a generalized $0$-expansion mapping.

\begin{definition}\label{def-graph-expansion}

Fix $n\geq0$ and suppose that $(\EXP_G^n,\EXP_\OP^n,\EXP_\OBJ^n)$ is a generalized $n$-expansion mapping.  

Let $\G=(S_\G,E_\G)$ be an $S$-graph of degree $n+1$ over $(\H_\G,\INST_\G)$ with $\H_\G=(\I_\G,\A_\G)$.  Define $E_\G^0$ as the subset of edges of $E_\G$ given by $e\in E_\G^0$ if and only if $\lab(e)\in\ACT_\G^0$.  

For $e\in E_\G-E_\G^0$, set $\lab(e)=(\inst_e,\op_e,\mu_{\lab(e)}^\ext)$.  Then, since $\op_e\in\OP_n'$, there exists an operation $\op_e^\exp=\EXP_\OP^n(\op_e)$.  Fix an injection $\mu_e^\int:\I_{\op_e^\exp}^\int\to\I^*$ whose image $\I_e^\int$ is disjoint from $\I_\G$.  Further, suppose that $\I_{e_1}^\int\cap\I_{e_2}^\int=\emptyset$ for distinct $e_1,e_2\in E_\G-E_\G^0$.  Then, $\mu_e^\int$ is called the \emph{internal tape renaming map} of $e$.  

Let $\I_\G^\inv=\bigsqcup\limits_{e\in E_\G-E_\G^0}\I_e^\int$ and define the hardware $\H_\G^\exp=(\I_\G^\exp,\A_\G^\exp)$ given by $\I_\G^\exp=\I_\G\sqcup\I_\G^\inv$, $\A_\G^\exp(i)=\A_\G(i)$ for all $i\in\I_\G$, and $\A_\G^\exp(i)=\A_{\op_e^\exp}((\mu_e^\int)^{-1}(i))$ for all $i\in\I_e^\int$ with $e\in E_\G-E_\G^0$.  Note that $\H_\G\leq\H_\G^\exp$.  The subset $\I_\G^\inv=\I_\G^\exp-\I_\G$ is called the set of \emph{invisible tapes} of $\H_\G^\exp$.

To keep the notation uniform, for any $\G_0\in\SGRAPH_0$, the set of invisible tapes of $\G_0$ is taken to be $\I_{\G_0}^\inv=\emptyset$.

Now, define the \emph{edge tape renaming map} of $e$ to be $\mu_e=\mu_{\lab(e)}^\ext\sqcup\mu_{\inst_e}^\val\sqcup\mu_e^\int$.  Let $\H_e=(\I_e,\A_e)$ be the hardware satisfying $\H_e\leq\H_\G^\exp$ and $\I_e=\I_{\lab(e)}\sqcup\I_e^\int$, i.e so that $\I_e$ is the image of $\mu_e$.  Then, let $\G_e$ be the copy of $\G_{\op_e^\exp}$ obtained by applying $\mu_e$, i.e by altering the fragments defining the action labelling any edge by replacing any index $i$ with $\mu_e(i)$.  Hence, $\G_e$ is an $S$-graph (of degree 0) over the hardware $\H_e$ whose vertex set is identified with that of $\op_e^\exp$.  However, as $\H_e\leq\H_\G^\exp$, $\G_e$ is interpreted as an $S$-graph over $\H_\G^\exp$.

Construct the $S$-graph $\G_\exp=(S_\G^\exp,E_\G^\exp)$ (of degree 0) over $\H_\G^\exp$ as follows:

\begin{itemize}

\item $S_\G\subseteq S_\G^\exp$

\item For any $e=(s_1,s_2,\act)\in E_\G^0$, construct the edge $e'=(s_1,s_2,\act')\in E_\G^\exp$, where $\act'$ is the interpretation of $\act$ as an action over $\H_\G^\exp$.

\item For any $e=(s_1,s_2,\act)\in E_\G-E_\G^0$, paste $\G_e$ into the graph by identifying the starting and finishing states with $s_1$ and $s_2$, respectively.  The subgraph of $\G_\exp$ obtained from this pasting is denoted $\G_e'$.

\end{itemize}
Then define $\EXP_G^{n+1}(\G)=\G_\exp$.

Finally, extend this to $\EXP_G^{n+1}:\SGRAPH_{n+1}'\to\SGRAPH_0$ by setting $\EXP_G^{n+1}\pr{\SGRAPH_n'}=\EXP_G^n$.

\end{definition}

From the construction of $\G_\exp=\EXP_G^{n+1}(\G)$ for $\G\in\SGRAPH_{n+1}$, each edge of $E_\G^0$ corresponds to an edge of $E_\G^\exp$.  This corresponding edge of $E_\G^\exp$ is called a \emph{proper} edge of $\G_\exp$.  For any edge $e'\in E_\G^\exp$ which is not proper, there exists $e\in E_\G-E_\G^0$ such that $e'$ is an edge of $\G_e'$.  In this case, $e'$ is called an \emph{expanded edge of $e$}.

Note that the previous observation verifies the claim that $\G_\exp\in\SGRAPH_0$:  For any edge $e'=(s_1,s_2,\act)\in E_\G^\exp$,

\begin{itemize}

\item $\act\in\ACT_{\H_\G^\exp}$ by construction;

\item if $e'$ is proper, then there exists a proper edge $(s_2,s_1,\act^{-1})\in E_\G^\exp$ by condition (4) of \Cref{def-s1-graph}; and

\item if $e'\in\G_e'$, then it arises from an edge of $\G_e$ which must have an inverse in $\G_e$ by \Cref{def-sgraph}, and so $(e')^{-1}\in E_\G^\exp$.

\end{itemize}

For any $s\in S_\G^\exp-S_\G$, there exists $e\in E_\G-E_\G^0$ such that $s$ arises from a state of $\G_e$.  In this case, $s$ is called an \emph{expanded state of $e$}.  All other elements of $S_\G^\exp$ are called \emph{proper states}. 

Let $\G'$ be the graph obtained from $\G_\exp$ by removing all proper edges.  Then $\G'$ consists of isolated proper states and the connected subgraphs $\G_e'$ which are in correspondence with $E_\G-E_\G^0$.  Note that for any distinct $e_1,e_2\in E_\G-E_\G^0$, $\G_{e_1}'\cap\G_{e_2}'\subseteq S_\G$.

Moreover, if $e=(s_0,s_0,\act)\in E_\G-E_\G^0$ is a loop, then $\G_e'$ is not isomorphic to $\G_e$, as the vertices corresponding to the starting and finishing states of $\op_e$ are pasted together.  Then, there exists a surjection $\pi_e:\G_e\to\G_e'$ preserving states and edges whose restriction to $\G_e-\SF_{\op_e}$ is a label-preserving bijection satisfying $\pi_e((\op_e)_s)=\pi_e((\op_e)_f)=s_0$.

For uniformity, if $e=(s_1,s_2,\act)\in E_\G-E_\G^0$, then define $\pi_e:\G_e\to\G_e'$ as the label-preserving bijection given by sending the starting and finishing vertices of $\op_e$ to $s_1$ and $s_2$, respectively.

\begin{lemma}\label{lem-op-e-paths}

Let $\G\in\SGRAPH_{n+1}$ and set $\G_\exp=\EXP_G^{n+1}(\G)$.  Let $e=(s_1',s_2',\act_e)\in E_\G-E_\G^0$ with $\act_e=(\inst_e,\op_e,\mu_{\act_e}^\ext)$.  Then for any path $\rho'=e_1',\dots,e_m'$ in $\G_e'$ such that $h(e_i')\notin\{s_1',s_2'\}$ for all $i\in\overline{m-1}$, the sequence of edges $\rho=e_1,\dots,e_m$ in $\G_e$ such that $\pi_e(e_j)=e_j'$ for all $j\in\overline{m}$ is a path such that $h(e_i)\notin\SF_{\op_e}$ for all $i\in\overline{m-1}$.  Moreover, $t(e_1)\in\SF_{\op_e}$ if and only if $t(e_1')\in\{s_1',s_2'\}$, $h(e_m)\in\SF_{\op_e}$ if and only if $h(e_m')\in\{s_1',s_2'\}$, and $\lab(e_j)=\lab(e_j')$ for all $j\in\overline{m}$.

In this case, $\rho$ is called the \emph{lift} of $\rho'$.

\end{lemma}

\begin{proof}

As $\pi_e$ is a label-preserving bijection between the edge sets of $\G_e$ and $\G_e'$, then letting $e_j=\pi_e^{-1}(e_j')$ for each $j\in\overline{m}$, $\rho=e_1,\dots,e_m$ is a path in $\G_e$ with $\lab(e_j)=\lab(e_j')$ for all $j\in\overline{m}$.

Further, for any vertex $s$ of $\G_e$, $\pi_e(s)\in\{s_1',s_2'\}$ if and only if $s\in\SF_{\op_e}$.  

Hence, for $i\in\overline{m-1}$, $h(e_i)\notin\SF_{\op_e}$ since $\pi_e(h(e_i))=h(e_i')$.  Moreover, $t(e_1)\in \SF_{\op_e}$ if and only if $\pi_e(t(e_1))=t(e_1')\in\{s_1',s_2'\}$, and similarly for $h(e_m)$ and $h(e_m')$.

\end{proof}

Let $\rho=e_1,\dots,e_N$ be a path in $\G_\exp$ with $s_0=t(e_1)$ and $s_i=h(e_i)$ for all $i\in\overline{N}$.  Suppose $s_0$ and $s_N$ are proper states.  Then a maximal subset $\overline{k,\ell}\subseteq\overline{N}$ such that $s_i$ is not proper for each $i\in\overline{k,\ell-1}$ is called a \emph{maximal expansion interval}.  There exists an ordering of the maximal expansion intervals $(\overline{k_1,\ell_1},\dots,\overline{k_t,\ell_t})$ such that $k_1=1$, $\ell_t=N$, and $k_{i+1}=\ell_i+1$ for all $i\in\overline{t-1}$.  Further, for each $n\in\overline{N}$, there exists a unique $j(n)\in\overline{t}$ such that $n\in\overline{k_{j(n)},\ell_{j(n)}}$.  By construction, for $m,n\in\overline{N}$, $m\leq n$ implies $j(m)\leq j(n)$.

For $i\in\overline{t}$, if $e_{k_i}$ is a proper edge, then $k_i=\ell_i$.  In this case, for $e$ the edge of $\G$ from which $e_{k_i}$ is constructed, define the generic edge of $\G$ $e_{k_i,\ell_i}'=(\s,\f,e)$.

Otherwise, there exists $e=(s_1',s_2',\act_e)\in E_\G-E_\G^0$ such that $e_n$ is an expanded edge of $e$ for each $n\in\overline{k_i,\ell_i}$.  The subpath $\rho_{k_i,\ell_i}=e_{k_i},\dots,e_{\ell_i}$ of $\rho$ is then called an \emph{expanded subpath} of $\rho$ and is a path in $\G_e'$.  Hence, $s_{k_i-1},s_{\ell_i}\in\{s_1',s_2'\}$.  If $s_1'\neq s_2'$, then define the generic edge $e_{k_i,\ell_i}'=(\sf_1,\sf_2,e)$ of $\G$ given by $\sf_1=\s$ if and only if $s_{k_i-1}=s_1'$ and $\sf_2=\s$ if and only if $s_{\ell_i}=s_1'$.  Conversely, if $s_1'=s_2'$, then let $e_{k_i,\ell_i}'=(\sf_1,\sf_2,e)$ such that $\sf_1$ and $\sf_2$ are any choices of $\s,\f$.

Then, the sequence of generic edges $\rho'=e_{k_1,\ell_1}',e_{k_2,\ell_2}',\dots,e_{k_{t-1},\ell_{t-1}}',e_{k_t,\ell_t}'$ is a generic path in $\G$ with $t(e_{k_i,\ell_i}')=s_{k_i-1}$ and $h(e_{k_i,\ell_i}')=s_{\ell_i}$ for all $i\in\overline{t}$.  In this case, $\rho'$ is called a \emph{generic reduction} of $\rho$.  Note that generic reductions are not unique: If $e$ is a loop in $\G$ and $\rho$ contains an expanded subpath $\rho_{k,\ell}$ contained in $\G_e'$, then the generic edge $e_{k,\ell}'$ of $\rho'$ can be given any orientation.

\begin{definition}\label{def-op-expansion}

Fix $n\geq0$ and suppose that $(\EXP_G^n,\EXP_\OP^n,\EXP_\OBJ^n)$ is a generalized $n$-expansion mapping.

For $\op\in\OP_{n+1}$, define $\G_\op^\exp=\EXP_G^{n+1}(\G_\op)$ as constructed in \Cref{def-graph-expansion}.  In particular, then $\G_\op^\exp=(S_\op^\exp,E_\op^\exp)$ is an $S$-graph of degree 0 over the hardware $\H_\op^\exp=(\I_\op^\exp,\A_\op^\exp)$ with $S_\op\subseteq S_\op^\exp$ and $\H_\op\leq\H_\op^\exp$.

Let $\I^\ext=\I_\op^\ext$, $\I^\val=\I_\op^\val$, and $\I^\int=\I_\op^\exp-(\I^\ext\sqcup\I^\val)=\I_\op^\int\sqcup\I_\op^\inv$.  Hence, $\I^\int$ contains all new (invisible) tapes of $\G_\op^\exp$.

\begin{lemma}\label{lem-op-expansion}

For $\op\in\OP_{n+1}$ as above, $\op_\exp=(\G_\op^\exp,\I^\ext,\I^\val,\I^\int,\op_s,\op_f)$ is an operation of degree 0.

\end{lemma}

\begin{proof}

As the underlying graphs of $\G_\op$ and of $\G_e'$ for each $e\in E_\op-E_\op^0$ are connected by hypothesis, $\G_\op^\exp$ is (strongly) connected by construction.  Hence, it suffices to show that condition (O) is satisfied.

Let $\rho=e_1,\dots,e_N$ be a path in $\G_\op^\exp$ between $s_0,s_N\in\SF_\op$ and fix $i\in\I^\int$.  As in the proofs of \Cref{prop-nse} and \Cref{lem-nse-star}, it suffices to assume that there exists $m\in\overline{N}$ such that $e_m$ mentions $i$ and to show that for $m$ minimal, $[\i=1]\in e_m$.

Let $\rho'=e_{k_1,\ell_1}',\dots,e_{k_t,\ell_t}'$ be any generic reduction of $\rho$, setting $e_{k_j,\ell_j}'=(\sf_j,\sf_j',e_j')$ for all $j\in\overline{t}$.  Then, set $\rho''=fl(\rho')$, i.e $\rho''$ is the flattening of $\rho'$ (see the proof of \Cref{lem-nse-star}).  So, $\rho''=q_1,\dots,q_t$ where each $q_j$ is a subpath of $\rho''$ consisting only of the (flattened) edges $e_j'$ and $(e_j')^{-1}$.  

Suppose $i\in\I_\op^\int$.  Then for any $n\in\overline{N}$ such that $e_n$ mentions $i$, either:
\begin{itemize}

\item $e_n$ is a proper edge, in which case $e_{j(n)}'$ is the edge of $E_\op^0$ from which $e_n$ is constructed, or

\item $e_n$ is an expanded edge of $e_{j(n)}'\in E_\op-E_\op^0$.

\end{itemize}
In the first case, the same fragments are used to define $\lab(e_n)$ and $\lab(e_{j(n)}')$, so that $e_{j(n)}'$ also mentions $i$.  In the second case, $i\in\I_{e_{j(n)}'}$ and $\I_\op^\inv\cap\I_\op^\int=\emptyset$, so that $e_{j(n)}'$ mentions $i$.

In particular, since an edge mentions an index if and only if its (formal) inverse does, $\rho''$ contains an edge that mentions $i$.  Let $r\in\overline{t}$ be the minimal index for which $(e_r')^{\pm1}$ mentions $i$.  Since $(e_{j(m)}')^{\pm1}$ mentions $i$, it follows that $r\leq j(m)$.

As $\rho''$ is a flattened path in $\G_\op$ between $s_0$ and $s_N$, it is subject to the constraints of condition (O) in \Cref{def-s1-op}.  So, the first edge of $q_r$ is labelled by an action of degree 0 and contains $[\i=1]$.  Hence, $e_r'\in E_\op^0$.  By the construction of the generic reduction, there exists $n\in\overline{N}$ such that $k_r=\ell_r=n$ and $e_{k_r,\ell_r}'=(\s,\f,e_r')$.  The construction of the flattening then implies that $q_r=e_r'$.  In particular, $j(n)=r$ and $e_n$ is constructed from $e_r'$, and so mentions $i$.  But then $m\leq n$ implies $j(m)\leq r$, and so $m=n$.  Hence, $[\i=1]\in e_m$.

Thus, it suffices to assume that $i\in\I_e^\int$ for some expanded edge $e=(s_1',s_2',\act_e)$.  By the disjointness of the images of the internal tape renaming maps, an edge of $\G_\op^\exp$ that mentions $i$ must be contained in $\G_e'$.  Hence, $e_j'=e$ for $j=j(m)$.

Then, $\rho_{k_j,\ell_j}$ is a path in $\G_e'$ with $t(e_{k_j}),h(e_{\ell_j})\in\{s_1',s_2'\}$ and $h(e_r)\notin\{s_1',s_2'\}$ for all $r\in\overline{k_j,\ell_j-1}$.  So, \Cref{lem-op-e-paths} yields a lift $\rho_j=e_{k_j}'',\dots,e_{\ell_j}''$ of $\rho_{k_j,\ell_j}$ in $\G_e$ such that $t(e_{k_j}''),h(e_{\ell_j}'')\in\SF_{\op_e}$.

But applying $\mu_e^{-1}$, $\rho_j$ corresponds to a path in $\op_e$ between elements of $\SF_{\op_e}$.  So, since $\inst_e\in\INST_n'$, condition (O) can be applied to $\op_e$, so that $[\mathbf{\mu_e^{-1}(i)}=1]\in e_m''$ since $i\in\I_e^\int$.  Hence, $[\i=1]\in e_m$.

\end{proof}

Similar to previous conventions, the defining partition of $\I_\op^\exp$ is denoted $\I_{\op_\exp}^\ext\sqcup\I_{\op_\exp}^\val\sqcup\I_{\op_\exp}^\int$.  Other terms are represented analogously; for example, the set of invisible tapes $\I_\op^\exp-\I_\op$ is denoted $\I_{\op}^\inv$.

Define $\EXP_\OP^{n+1}(\op)=\op_\exp$ for all $\op\in\OP_{n+1}$ as above.  Finally, extend this definition to a map $\EXP_\OP^{n+1}:\OP_{n+1}'\to\OP_0$ by setting $\EXP_\OP^{n+1}\pr{\OP_n'}=\EXP_\OP^n$.

\end{definition}

\begin{definition}\label{def-obj-expansion}

Fix $n\geq0$ and suppose that $(\EXP_G^n,\EXP_\OP^n,\EXP_\OBJ^n)$ is a generalized $n$-expansion mapping.

For $\Obj=(\I_\Obj^\val,\INST_\Obj,\op_1,\dots,\op_k)\in\OBJ_{n+1}$, define $$\EXP_\OBJ^{n+1}(\Obj)=(\I_\Obj^\val,\EXP_\OP^{n+1}(\op_1),\dots,\EXP_\OP^{n+1}(\op_k))$$
By \Cref{lem-op-expansion}, $\EXP_\OBJ^{n+1}(\Obj)\in\OBJ_0$. 

Finally, extend this to $\EXP_\OBJ^{n+1}:\OBJ_{n+1}'\to\OBJ_0$ given by setting $\EXP_\OBJ^{n+1}\pr{\OBJ_n'}=\EXP_\OBJ^n$.

\end{definition}

The following statement is a direct consequence of the definitions and \Cref{lem-op-expansion}:

\begin{proposition}\label{prop-exp}

For any $n\geq0$, the triple $(\EXP_G^{n+1},\EXP_\OP^{n+1},\EXP_\OBJ^{n+1})$ is a generalized $(n+1)$-expansion mapping.

\end{proposition}

In light of \Cref{prop-exp}, there exists a map $\EXP_G:\SGRAPH^*\to\SGRAPH_0$ given by $\EXP_G(\G)=\EXP_G^n(\G)$ for any $n\geq0$ such that $\G\in\SGRAPH_n$.  Hence, by the construction of the maps, $\EXP_G(\G)=\EXP_G^m(\G)$ for all $m\geq n$.  Similarly, there exist analogous maps $\EXP_\OP:\OP^*\to\OP_0$ and $\EXP_\OBJ:\OBJ^*\to\OBJ_0$.

With an abuse of notation, each element of the triple $(\EXP_G,\EXP_\OP,\EXP_\OBJ)$ is simply denoted $\EXP$ and referred to as an \emph{expansion map}, with context making it clear which is being referenced.

\begin{example}\label{ex-silly-counter-expanded}
\Cref{fig-SILLYCOUNTER-EXPANDED} below shows the expansion $\G_\exp=\EXP(\G)$ of the $S$-graph of degree 1 $\G$ of the operation $\syn(a,b,c)$ from \Cref{fig-SILLYCOUNTER}. The four expanded states are shaded for reference.

\begin{figure}[hbt]
	\centering
	\include{SILLYCOUNTER-EXPANDED}
	\caption{The expansion $\G_\exp$ of the S$^1$-graph $\G$ defined in~\Cref{fig-SILLYCOUNTER}.}
	\label{fig-SILLYCOUNTER-EXPANDED}       
\end{figure} 

Note that for any $\op\in\Obj_0$, $\I_\op^\int=\emptyset$ and $S_\op=\SF_\op$.  Hence, the expansion of such an edge neither introduces new tapes nor new states.  

Conversely, expanding any edge corresponding to the copy operation produces both expanded states and invisible tapes.  Indeed, because $|S_\cpy-\SF_\cpy|=2$ and $|\I_\cpy^\int|=1$, each such edge adds two new states and one new tape.  As a result, $$\I_{\G} =  \{a, b, c, a', b', d', d''\} \subsetneq \I_\G^\exp =  \{a, b, c, a', b', d', d'', c_a, c_b\}$$

This example can also be understood as the expansion of the object $\Obj_1[d']$, which is the object $\EXP(\Obj_1)$ with a single operation $\EXP(\syn)$. The tapes of $\I_\G^\exp$ arising from $\I_\G$ are classified in the same way they were in $\Obj_1$, while any new tapes (here, the invisible tapes $c_a$ and $c_b$) are necessarily internal.
\end{example}

Note that the expansion of the object of degree 2 Counter$_1[a]$ (see \Cref{ex-counter1}) resembles the object of degree 0 Counter$_0[a]$ (see \Cref{ex-counter}).  However, there are several minor differences.  For example, the unique edge $e$ of the operation $\check_{>0}$ in \Cref{ex-counter} satisfying $t(e)=(\check_{>0})_s$ contains the guard fragments $[\mathbf{b}=1]$ and $[\mathbf{c}=1]$, while these fragments appear on distinct edges in the expansion of the analogous operation of \Cref{ex-counter1}.  Despite such discrepancies, though, the graphs are related in a particular sense.

To describe this relationship in a more fundamental setting, \Cref{th-expansion-props} below draws a correspondence between the computations of an $S^*$-graph and those of its expansion.  First, \Cref{lem-op-e} deals with a basic case.

\begin{lemma}\label{lem-op-e}  Let $\G\in\SGRAPH^*$, $e=(s_1,s_2,\act_e)\in E_\G-E_\G^0$, and $\act_e=(\inst_e,\op_e,\mu_{\act_e}^\ext)$.

\begin{enumerate}

\item Let $\eps=\bigl((w_0,t_0),e_1,(w_1,t_1),e_2,\dots,e_m,(w_m,t_m)\bigr)$ be a computation of $\EXP(\op_e)$.  Then for every $w\in L_{\H_\G^\exp}$, there exists a computation
$$\sigma_w(\eps)=\bigl((w_0',\pi_e(t_0)),e_1',(w_1',\pi_e(t_1)),e_2',\dots,e_m',(w_m',\pi_e(t_m))\bigr)$$
of $\EXP(\G)$ such that $w_j'=\mu_e(w_j)\sqcup w(\I_\G^\exp-\I_e)$ for all $j\in\overline{0,m}$.

\item Let $\delta=\bigl((w_0',t_0'),e_1',(w_1',t_1'),e_2',\dots,e_m',(w_m',t_m')\bigr)$ be a computation of $\EXP(\G)$ such that the computation path $\rho'=e_1',\dots,e_m'$ is a path in $\G_e'$ with $h(e_i')\notin S_\G$ for $i\in\overline{m-1}$.  Then there exists a computation
$$\tau(\delta)=\bigl((w_0,t_0),e_1,(w_1,t_1),e_2,\dots,e_m,(w_m,t_m)\bigr)$$
of $\EXP(\op_e)$ such that $w_j'(\I_e)=\mu_e(w_j)$, $w_i'(\I_\G^\exp-\I_e)=w_j'(\I_\G^\exp-\I_e)$, and $\pi_e(t_j)=t_j'$ for all $i,j\in\overline{0,m}$.

\end{enumerate}

\end{lemma}

\begin{proof}

(1) By construction, there exists a bijection $\zeta_e$ between the set of edges of $\EXP(\op_e)$ and that of $\G_e$ given by applying $\mu_e$ to any fragment and then interpreting the resulting action so that it is over $\H_\G^\exp$.  Hence, for any $v,v'\in L_{\H_{\op_e}^\exp}$, there exists a computation $\bigl((v,s),e_0,(v',s')\bigr)$ of $\EXP(\op_e)$ if and only if for any $w\in L_{\H_\G^\exp}$, there exists a computation 
$$\bigl((\mu_e(v)\sqcup w(\I_\G^\exp-\I_e),s),\zeta_e(e_0),(\mu_e(v')\sqcup w(\I_\G^\exp-\I_e),s')\bigr)$$ 
of $\EXP(\G)$.  Moreover, for any computation $\bigl((u,t),e_0',(u',t')\bigr)$ of $\G_e$, there exists a computation $$\bigl((u,\pi_e(t)),\pi_e(e_0'),(u',\pi_e(t'))\bigr)$$ of $\G_e'$.

Hence, letting $e_i'=\pi_e(\zeta_e(e_i))$ for all $i\in\overline{m}$, there exists a computation $\sigma_w(\eps)$ of $\EXP(\G)$ satisfying the statement.

\medskip

(2) As $h(e_i')\notin S_\G$ for all $i\in\overline{m-1}$, by \Cref{lem-op-e-paths} there exists a lift $\rho''=e_1'',\dots,e_m''$ of $\rho'$.  So, $\rho''$ is a path in $\G_e$ with $\lab(e_j'')=\lab(e_j')$ for all $j\in\overline{m}$.  Hence, there exists a computation
$$\tau'(\delta)=\bigl((w_0',t_0),e_1'',(w_1',t_1),e_2'',\dots,e_m'',(w_m',t_m)\bigr)$$ 
of $\G_e$ such that $\pi_e(t_i)=t_i'$ for all $i\in\overline{0,m}$.  

Since any edge of $\G_e$ is labelled by an action constructed from fragments of $\FRAG_{\H_e}$, it follows that $w_i'(\I_\G^\exp-\I_e)=w_j'(\I_\G^\exp-\I_e)$ for all $i,j\in\overline{0,m}$.  Thus, letting $w_i=\mu_e^{-1}(w_i'(\I_e))$ and $e_j=\zeta_e^{-1}(e_j'')$ for all $i\in\overline{0,m}$ and $j\in\overline{m}$, there exists a computation $\tau(\delta)$ satisfying the statement.

\end{proof}

For any $\G\in\SGRAPH^*$, define the \emph{trivial invisible word} by $1_\G^\inv=(1)_{i\in\I_\G^\inv}$.  Similarly, for any $\op\in\OP^*$, let $1_\op^\inv=(1)_{i\in\I_\op^\inv}$.

\begin{theorem}\label{th-expansion-props}

Let $\G\in\SGRAPH^*$ and $\Obj\in\OBJ^*$.  Set $\G_\exp=\EXP(\G)$.  Then:

\begin{enumerate}[label=({\alph*})]

\item $\REACH_{\G_\exp}^{S_\G}$ is properly $L_{\H_\G^\exp}$-restricted to $\GL_\G\times1_\G^\inv$.

\item For any computation $\delta=\bigl((w_0,s_0),e_1,(w_1,s_1),e_2,\dots,e_m,(w_m,s_m)\bigr)$ of $\G_\exp$ such that $(w_0,s_0),(w_m,s_m)\in(\GL_\G\times1_\G^\inv)\times S_\G$, there exists a computation $$\eta(\delta)=\bigl((w_0',s_0'),e_1',(w_1',s_1'),e_2',\dots,e_\ell',(w_\ell',s_\ell')\bigr)$$ of $\G$ such that $w_0=w_0'\sqcup1_\G^\inv$, $w_m=w_\ell'\sqcup1_\G^\inv$, $s_0=s_0'$, $s_m=s_\ell'$, and $\rho'=e_1',\dots,e_\ell'$ is a generic reduction of $\rho=e_1,\dots,e_m$.

\item Suppose $\delta=\bigl((w_0',s_0'),e_1',(w_1',s_1')\bigr)$ is a computation of $\G$ such that $e_1'=(\sf_0,\sf_1,e)$ for some $\sf_0,\sf_1\in\{\s,\f\}$ and $e=(t_0,t_1,\act_e)\in E_\G-E_\G^0$.  Then for $\act_e=(\inst_e,\op_e,\mu_{\act_e}^\ext)$,
$$(\mu_{\act_e}^{-1}(w_0'(\I_{\act_e})),\mu_{\act_e}^{-1}(w_1'(\I_{\act_e})))\in\IO_{\EXP(\op_e)}^{\lambda_{e}(\sf_0),\lambda_e(\sf_1)}$$ 
where $\lambda_e=\lambda_{\op_e}$.  Moreover, for any input-output computation 
$$\eps=\bigl((w_0'',s_0''),e_1'',(w_1'',s_1''),e_2'',\dots,e_m'',(w_m'',s_m'')\bigr)$$ 
of $\EXP(\op_e)$ such that 
\begin{align*}
(w_0'',s_0'')&=(\mu_{\act_e}^{-1}(w_0'(\I_{\act_e}))\sqcup1_{\EXP(\op_e)}^\int,\lambda_e(\sf_0)) \\
(w_m'',s_m'')&=(\mu_{\act_e}^{-1}(w_1'(\I_{\act_e}))\sqcup1_{\EXP(\op_e)}^\int,\lambda_e(\sf_1))
\end{align*}
 there exists a computation 
$$\kappa_\delta(\eps)=\bigl((w_0,s_0),e_1,(w_1,s_1),e_2,\dots,e_m,(w_m,s_m)\bigr)$$
of $\G_\exp$ such that $(w_0,s_0)=(w_0'\sqcup1_\G^\inv,s_0')$, $(w_m,s_m)=(w_1'\sqcup1_\G^\inv,s_1')$, and $$w_j=\mu_e(w_j'')\sqcup w_0'(\I_\G-\I_{\act_e})\sqcup (1)_{i\in\I_\G^\inv-\I_e^\int}$$ for all $j\in\overline{0,m}$.

\item For any computation $\delta=\bigl((w_0',s_0'),e_1',(w_1',s_1'),e_2',\dots,e_\ell',(w_\ell',s_\ell')\bigr)$ of $\G$, there exists a computation $$\nu(\delta)=\bigl((w_0,s_0),e_1,(w_1,s_1),e_2,\dots,e_m,(w_m,s_m)\bigr)$$ of $\G_\exp$ such that $(w_0,s_0),(w_m,s_m)\in(\GL_\G\times1_\G^\inv)\times S_\G$, $w_0=w_0'\sqcup1_\G^\inv$, $w_m=w_\ell'\sqcup1_\G^\inv$, $s_0=s_0'$, and $s_m=s_\ell'$.  

\item $\VAL_{\EXP(\Obj)}=\VAL_\Obj$.

\item For any $\op\in\Obj$, $\REACH_{\EXP(\op)}^{\SF_\op}$ is properly $L_{\H_\op^\exp}$-restricted to $\ND_\op^1\times1_\op^\inv$.

\item For any $\op\in\Obj$, $\IO_{\EXP(\op)}=\IO_\op$.

\end{enumerate}

\end{theorem}

\begin{proof}

Note that if $\G\in\SGRAPH_0$, then (a)$-$(d) follow by construction and by \Cref{lem-lio-io}.  Similarly, if $\Obj\in\OBJ_0$, then (e)$-$(g) follow by construction and by \Cref{lem-nse-star}.

Now fix $n\geq0$, suppose the statements hold for all $S$-graphs and objects of degree at most $n$, and let $\G\in\SGRAPH_{n+1}$ and $\Obj\in\OBJ_{n+1}$.

\bigskip

\noindent
(a) Let $\delta=\bigl((w_0,s_0),e_1,\dots,e_m,(w_m,s_m)\bigr)$ be a computation of $\G_\exp$ such that $s_0,s_m\in S_\G$.  As $\REACH_{\G_\exp}$ is an equivalence relation by \Cref{lem-reach-star}, it may be assumed without loss of generality that $s_i\notin S_\G$ for all $i\in\overline{m-1}$.

Suppose $e_1$ is a proper edge.  Then, there exists $e_1'=(s_0,s_1,\act)\in E_\G^0$ such that $e_1$ is constructed from $e_1'$.  As such $s_1\in S_\G$, so that $m=1$.  Further, 
by construction, $w_0(\I_\G^\inv)=w_1(\I_\G^\inv)$ and $(w_0(\I_\G),w_1(\I_\G))\in\mathrm{LIO}_{\act}^{\s,\f}$.  Hence, by \Cref{lem-lio-io}, $w_0(\I_\G)\in\GL_\G$ if and only if $w_1(\I_\G)\in\GL_\G$, and thus $w_0\in\GL_\G\times1_\G^\inv$ if and only if $w_1\in\GL_\G\times1_\G^\inv$.

Otherwise, $\rho=e_1,\dots,e_m$ is a path in $\G_e'$ for some $e=(s_1',s_2',\act_e)\in E_\G-E_\G^0$ such that $s_0,s_m\in\{s_1',s_2'\}$.  Let $\act_e=(\inst_e,\op_e,\mu_{\act_e}^\ext)$.

By \Cref{lem-op-e}(2), there exists a computation $\tau(\delta)$ of $\EXP(\op_e)$ between the configurations $(\mu_e^{-1}(w_0(\I_e)),t_0)$ and $(\mu_e^{-1}(w_m(\I_e)),t_m)$ such that $\pi_e(t_j)=s_j$ for $j\in\{0,m\}$.  As $s_j\in\{s_1',s_2'\}$, it follows that $t_j\in\SF_{\op_e}$, and hence 
$$(\mu_e^{-1}(w_0(\I_e)),\mu_e^{-1}(w_m(\I_e)))\in\REACH_{\EXP(\op_e)}^{\SF_{\op_e}}$$
Then, since $\inst_e\in\INST_n'$, the inductive hypothesis may be applied to it, and so (f) implies that $\mu_e^{-1}(w_0(\I_e))\in\ND_{\op_e}^1\times1_{\op_e}^\inv$ if and only if $\mu_e^{-1}(w_m(\I_e))\in\ND_{\op_e}^1\times1_{\op_e}^\inv$.

Let $\GL_e=L_{\act_e}^\ext\times\VAL_{\inst_e}\times1_e^\int$, where $1_e^\int=(1)_{i\in\I_e^\int}$.  Then, the restriction of $\mu_e$ to $\ND_{\op_e}^1\times1_{\op_e}^\inv$ defines a bijection with $\GL_e$.  Hence, $w_0(\I_e)\in\GL_e$ if and only if $w_m(\I_e)\in\GL_e$.  But since any action labelling an edge of $\G_e'$ is constructed from elements of $\FRAG_{\H_e}$, $$w_0(\I_\G^\exp-\I_e)=w_m(\I_\G^\exp-\I_e)$$
Thus, $w_0\in\GL_\G\times1_\G^\inv$ if and only if $w_m\in\GL_\G\times1_\G^\inv$.

\bigskip

\noindent
(b)  Suppose that $s_i\notin S_\G$ for all $i\in\overline{m-1}$.

If $e_1$ is a proper edge, then $s_1\in S_\G$ so that $m=1$.  Moreover, for $e_1'\in E_\G^0$ the edge from which $e_1$ is constructed, $t(e_1')=s_0$, $h(e_1')=s_1$, and $(w_0(\I_\G),w_1(\I_\G))\in\IO_{e_1'}^{s_0,s_1}$.  

Hence, $\eta(\delta)=\bigl((w_0(\I_\G),s_0),(\s,\f,e_1'),(w_1(\I_\G),s_1)\bigr)$ is a computation of $\G$ satisfying the statement.

Otherwise, $\rho$ is a path in $\G_e'$ for some $e=(s_1'',s_2'',\act_e)\in E_\G-E_\G^0$.  As in (a), the construction of $\G_e'$ implies $w_0(\I_\G^\exp-\I_e)=w_m(\I_\G^\exp-\I_e)$.  Moreover, by hypothesis $w_0(\I_e^\int)=1_e^\int=w_m(\I_e^\int)$, and so $w_0(\I_\G-\I_{\act_e})=w_m(\I_\G-\I_{\act_e})$.

Letting $\act_e=(\inst_e,\op_e,\mu_{\act_e}^\ext)$, \Cref{lem-op-e}(2) produces a computation $\tau(\delta)$ of $\EXP(\op_e)$ between the configurations $(\mu_e^{-1}(w_0(\I_e)),t_0)$ and $(\mu_e^{-1}(w_m(\I_e)),t_m)$ such that $\pi_e(t_j)=s_j$ for $j\in\{0,m\}$.  Further, by construction $s_0,s_m\in\{s_1'',s_2''\}$, so that $t_0,t_m\in\SF_{\op_e}$.  Since $w_j(\I_e^\int)=1_e^\int$ for $j\in\{0,m\}$, it follows that $\tau(\delta)$ is an input-output computation of $\EXP(\op_e)$.  

But $w_j\in\GL_\G\times1_\G^\inv$ for $j\in\{0,m\}$, so that $w_j(\I_e)\in \GL_e$.  Further, as $\inst_e\in\INST_n'$, the inductive hypothesis may be applied, so that (e) yields $\mu_e^{-1}(w_j(\I_e))\in\ND_{\op_e}^1\times1_\G^\inv=\ND_{\EXP(\op_e)}^1$.  Hence, by the existence of $\tau(\delta)$ and (g),
$$(\mu_{\act_e}^{-1}(w_0(\I_{\act_e})),\mu_{\act_e}^{-1}(w_m(\I_{\act_e})))\in\IO_{\EXP(\op_e)}^{t_0,t_m}=\IO_{\op_e}^{t_0,t_m}$$
By definition, this implies that $(w_0(\I_\G),w_m(\I_\G))\in\IO_{\act_e}^{\lambda_e^{-1}(t_0),\lambda_e^{-1}(t_m)}$.  

So, letting $e'=(\lambda_e^{-1}(t_0),\lambda_e^{-1}(t_m),e)$, then $\eta(\delta)=\bigl((w_0(\I_\G),t(e')),e',(w_m(\I_\G),h(e'))\bigr)$ is a computation of $\G$. 

If $e$ is a loop, then $s_0=s_1''=s_2''=s_m$.  So, $s_0=t(e')=h(e')=s_m$ and $e'$ is a generic reduction of $\rho$.

Conversely, suppose $s_1''\neq s_2''$.  Then, $\pi_e$ is bijective, so that $t_j=\pi_e^{-1}(s_j)$.  If $s_0=s_1''$, then $t_0=(\op_e)_s$, and so $\lambda_e^{-1}(t_0)=\s$; but then $t(e')=s_0$.  Similarly, if $s_0=s_2''$, then $t_0=(\op_e)_f$, $\lambda_e^{-1}(t_0)=\f$, and $t(e')=s_0$.  A similar argument implies that $h(e')=s_m$ and that $e'$ is a generic reduction of $\rho$.

Thus, $\eta(\delta)$ satisfies the statement.

Now suppose there exists $i\in\overline{m-1}$ such that $s_i\in S_\G$.  Letting $k_1$ be the minimal such index, the above arguments produce the construction of a computation $\delta_1'=\bigl((w_0(\I_\G),s_0),e_1',(w_{k_1}(\I_\G),s_{k_1})\bigr)$ of $\G$ such that $e_1'$ is a generic reduction of the subpath $e_1,\dots,e_{k_1}$ of $\rho$.  If there then exists $i\in\overline{k_1+1,m-1}$ such that $s_i\in S_\G$, then letting $k_2$ be the minimal such index, there exists a computation $\delta_2'=\bigl((w_{k_1}(\I_\G),s_{k_1}),e_2',(w_{k_2}(\I_\G),s_{k_2})\bigr)$ of $\G$ such that $e_2'$ is a generic reduction of $e_{k_1+1},\dots,e_{k_2}$.  Iterating, there exists $\ell\geq1$ such that $s_i\notin S_\G$ for any $i\in\overline{k_{\ell}+1,m-1}$.  But then 
$$\eta(\delta)=\delta_1',\dots,\delta_\ell'=\bigl((w_0(\I_\G),s_0),e_1',(w_{k_1}(\I_\G),s_{k_1}),e_2',\dots,e_\ell',(w_m(\I_\G),s_m)\bigr)$$ is a computation of $\G$ satisfying the statement.

\bigskip

\noindent
(c) By the definition of computation, $(w_0',s_0',w_1',s_1')\in\IO_{e_1'}$.  

So, $s_0'=t(e_1')$, $s_1'=h(e_1')$, and $(w_0',w_1')\in\IO_{\act_e}^{\sf_0,\sf_1}$.  Hence, $w_0'(\I_\G-\I_{\act_e})=w_1'(\I_\G-\I_{\act_e})$ and
$$(\mu_{\act_e}^{-1}(w_0'(\I_{\act_e})),\mu_{\act_e}^{-1}(w_1'(\I_{\act_e})))\in\IO_{\op_e}^{\lambda_e(\sf_0),\lambda_e(\sf_1)}$$ 
As $\inst_e\in\INST_n'$, the inductive hypothesis can be applied, and so (g) implies $\IO_{\op_e}=\IO_{\EXP(\op_e)}$.  Thus, $(\mu_{\act_e}^{-1}(w_0'(\I_{\act_e})),\mu_{\act_e}^{-1}(w_1'(\I_{\act_e})))\in\IO_{\EXP(\op_e)}^{\lambda_e(\sf_0),\lambda_e(\sf_1)}$.

Now, let $\eps$ be an input-output computation as in the statement.  

Note that $\mu_{\act_e}^{-1}(w_j'(\I_{\act_e}))\sqcup1_{\EXP(\op_e)}^\int=\mu_e^{-1}(w_j'(\I_{\act_e})\sqcup1_e^\int)$ for $j\in\{0,1\}$.  

As a result, letting $w=w_0'\sqcup1_\G^\inv$, \Cref{lem-op-e}(1) yields a computation of $\G_\exp$
$$\sigma_w(\eps)=\bigl((w_0,s_0),e_1,(w_1,s_1),e_2,\dots,e_m,(w_m,s_m)\bigr)$$
such that $s_0=\pi_e(\lambda(\sf_0))$, $s_m=\pi_e(\lambda(\sf_1)$, and $w_j=\mu_e(w_j'')\sqcup w(\I_\G^\exp-\I_e)$ for all $j\in\overline{0,m}$.  

As $w(\I_\G^\inv)=1_\G^\inv$ and $w(\I_\G)=w_0'$, $w(\I_\G^\exp-\I_e)=w_0'(\I_\G-\I_{\act_e})\sqcup(1)_{i\in\I_\G^\inv-\I_e^\int}$.  So, $$w_j=\mu_e(w_j'')\sqcup w_0'(\I_\G-\I_{\act_e})\sqcup(1)_{i\in\I_\G^\inv-\I_e^\int}$$ for all $j\in\overline{0,m}$.  In particular, since $w_0'(\I_\G-\I_{\act_e})=w_1'(\I_\G-\I_{\act_e})$, 
\begin{align*}
w_0&=\mu_e(w_0'')\sqcup w_0'(\I_\G-\I_{\act_e})\sqcup(1)_{i\in\I_\G^\inv-\I_e^\int} \\
w_m&=\mu_e(w_m'')\sqcup w_1'(\I_\G-\I_{\act_e})\sqcup(1)_{i\in\I_\G^\inv-\I_e^\int}
\end{align*}

Note that $w_0''=\mu_{\act_e}^{-1}(w_0'(\I_{\act_e}))\sqcup1_{\EXP(\op_e)}^\int=\mu_e^{-1}(w_0'(\I_{\act_e})\sqcup 1_e^\int)$, and so as a consequence $\mu_e(w_0'')=w_0'(\I_{\act_e})\sqcup1_e^\int$.  Similarly, $\mu_e(w_m'')=w_1'(\I_{\act_e})\sqcup1_e^\int$.  

Hence, since $(1)_{i\in\I_\G^\inv-\I_e^\int}\sqcup1_e^\int=1_\G^\inv$, $w_0=w_0'\sqcup1_\G^\inv$ and $w_m=w_1'\sqcup1_\G^\inv$.

Now suppose $\sf_0=\s$.  Then $s_0'=t(e_1')=t_0$ and $\lambda_e(\sf_0)=(\op_e)_s$.  So, $\pi_e(\lambda_e(\sf_0))=t_0$, and so $\pi_e(\lambda_e(\sf_0))=s_0'$.  Similarly, if $\sf_0=\f$, then $s_0'=t(e_1')=t_1$ and $\lambda_e(\sf_0)=(\op_e)_f$, so that $\pi_e(\lambda_e(\sf_0))=t_1=s_0'$.  An analogous argument implies that $\pi_e(\lambda_e(\sf_1))=s_1'$.

Thus, $\kappa_\delta(\eps)=\sigma_w(\eps)$ satisfies the statement

\bigskip

\noindent
(d) Suppose that $\ell=1$ and let $e_1'=(\sf_0,\sf_1,e)$ for $\sf_0,\sf_1\in\{\s,\f\}$ and $e\in E_\G$.

If $e\in E_\G-E_\G^0$, then for $e=(t_0,t_1,\act_e)$ with $\act_e=(\inst_e,\op_e,\mu_{\act_e}^\ext)$, it follows from (c) that there exists a computation $\eps$ of $\EXP(\op_e)$ between $(\mu_e^{-1}(w_0'(\I_{\act_e})\sqcup1_e^\int),\lambda_{\op_e}(\sf_0))$ and $(\mu_e^{-1}(w_1'(\I_{\act_e})\sqcup1_e^\int),\lambda_{\op_e}(\sf_1))$.  But then $\kappa_\delta(\eps)$ is a computation of $\G_\exp$ satisfying the statement.

Otherwise, $e\in E_\G^0$.  Let $e_1\in E_\G^\exp$ be the proper edge arising from $e$.  If $\sf_0=\sf_1$, then $w_0'=w_1'$, so that $\nu(\delta)$ can be taken to be an empty configuration. Conversely, if $\sf_0\neq\sf_1$, then let $\a\in\{\pm1\}$ with $\a=1$ if and only if $\sf_0=\s$ and define $\nu(\delta)=((w_0'\sqcup1_\G^\inv,s_0'),e_1^\a,(w_1'\sqcup1_\G^\inv,s_1'))$.  Either way, $\nu(\delta)$ is a computation of $\G_\exp$ satisfying the statement.

Now suppose $\ell\geq1$.  Then for each $i\in\overline{\ell}$, the above arguments produce a computation $\eps_i$ of $\G_\exp$ between $(w_{i-1}'\sqcup1_\G^\inv,s_{i-1}')$ and $(w_i'\sqcup1_\G^\inv,s_i')$.  Letting $\rho_i$ be the path supporting $\eps_i$, it then follows that $\rho=\rho_1,\dots,\rho_\ell$ supports a computation $\nu(\delta)$ satisfying the statement.

\bigskip

\noindent
(e) Note that $\I_{\EXP(\Obj)}^\val=\I_\Obj^\val$ and $\A_{\EXP(\op)}(i)=\A_\op(i)$ for any $\op\in\Obj$ and $i\in\I_\Obj^\val$.  Hence, there is a natural bijection between the vertex sets of $\G_{\EXP(\Obj)}^\val$ and $\G_\Obj^\val$ given by sending the vertex of $\G_{\EXP(\Obj)}^\val$ corresponding to $v\in L_\Obj^\val$ to its analogue in $\G_\Obj^\val$.

For any edge $(v_1,v_2)$ of $\G_\Obj^\val$, there exists $\op\in\Obj$ and an input-output computation $\delta$ between configurations $(w_1',s_1')$ and $(w_2',s_2')$ of $\op$ such that $w_i'(\I_\Obj^\val)=v_i$ for $i=1,2$.  But then (d) yields an input-output computation $\nu(\delta)$ of $\EXP(\op)$ between configurations $(w_1,s_1)$ and $(w_2,s_2)$ such that $s_i=s_i'$ and $w_i(\I_\Obj^\val)=v_i$ for $i=1,2$.  Hence, $(v_1,v_2)$ is an edge of $\G_{\EXP(\Obj)}^\val$.

Similarly, for any edge $(v_1,v_2)$ of $\G_{\EXP(\Obj)}^\val$, then (b) implies that $(v_1,v_2)$ is an edge of $\G_\Obj^\val$.

Thus, $\G_\Obj^\val$ and $\G_{\EXP(\Obj)}^\val$ are isomorphic graphs, and so the connected component containing $1_\Obj^\val$ is the same in both.

\bigskip

\noindent
(f) Let $\delta$ be a computation of $\EXP(\op)$ between $(w_1,s_1)$ and $(w_2,s_2)$ such that $s_1,s_2\in\SF_\op$.  Then, (a) implies that $w_1\in\GL_\op\times1_\op^\inv$ if and only if $w_2\in\GL_\op\times1_\op^\inv$.  As a result, since \Cref{lem-nse-star}(1) implies that $\ND_\op^1\subseteq\GL_\op$, it suffices to assume that $w_1,w_2\in\GL_\op\times1_\op^\inv$ and show that $w_1(\I_\op)\in\ND_\op^1$ if and only if $w_2(\I_\op)\in\ND_\op^1$.

But by (b), there exists an input-output computation $\eta(\delta)$ of $\op$ between the configurations $(w_1(\I_\op),s_1)$ and $(w_2(\I_\op),s_2)$, so that the statement follows from \Cref{lem-nse-star}(2).

\bigskip

\noindent
(g) By (e), $\ND_{\EXP(\op)}=\ND_\op$.  What's more, note that $1_\op^\int\sqcup1_\op^\inv=1_{\EXP(\op)}^\int$.

For any $(w_1,w_2)\in\IO_\op^{s,f}$, then by definition there exists a computation $\delta$ of $\op$ between the configurations $(w_1\sqcup1_\op^\int,\op_s)$ and $(w_2\sqcup1_\op^\int,\op_f)$.  By (d), there then exists a computation $\nu(\delta)$ of $\EXP(\op)$ between the configurations $(w_1\sqcup1_{\EXP(\op)}^\int,\op_s)$ and $(w_2\sqcup1_{\EXP(\op)}^\int,\op_f)$.  Hence, $(w_1,w_2)\in\IO_{\EXP(\op)}^{s,f}$.  Analogous arguments for the other possible endpoints of a path supporting an input-output computation then imply that $\IO_\op\subseteq\IO_{\EXP(\op)}$.

The reverse containment then follows from a parallel argument, applying (b) in place of (d).

\end{proof}

\section{Complexity} \label{sec-complexity}

The goal of \Cref{main-theorem} is to exhibit an $S$-machine that is in some way `efficient' in its computations.  This concept of efficiency is quantified by the complexity functions of the machine.  In this section, the definitions of these functions are recalled and then extended to $S$-graphs of degree 0 and finally to $S^*$-graphs. 

First, though, multiple binary relations are introduced for a particular class of functions containing all of those of interest.

\begin{definition} \label{def-equivalences}

A nondecreasing function $f\colon \N\rightarrow [0,\infty)$ is called a \emph{growth function}.  Define the binary relation $\preceq$ on growth functions by $f\preceq g$ if and only if there exist constants $a,c>0$ and $b\in\N$ such that $f(n) \leq a g(b n) + c$ for all $n\in \N$.  Note that this condition is equivalent to the existence of a single constant $C\in\N$ such that $f(n)\leq Cg(Cn)+C$ for all $n\in\N$.

Then, $\preceq$ is a preorder on the set of growth functions, and so there exists an equivalence relation $\sim$ given by $f\sim g$ if and only if $f\preceq g$ and $g\preceq f$.  The growth functions $f$ and $g$ are called \emph{$\sim$-equivalent} if $f\sim g$.

Thus, $n^3 \sim 5 n^3 + 3 n^2 + 7$ and $e^n \sim 2^{10 n}$, while $n^2 \precneq n^2 \log n \precneq n^{2.01}$ and $2^n \precneq n^n \precneq 2^{n^2}$. 

Note that this definition departs slightly from the definition of equivalence from \cite{SBR}, where an arbitrary linear term was also included in the right hand side of the definition of $\preceq$.  This preorder is denoted here by $\preceq_1$ and the corresponding equivalence by $\sim_1$. 
Thus, $0$, $n^{1/2}$ and $n$ are not $\sim$-equivalent, but are $\sim_1$-equivalent. 

However, note that $\sim$ and $\sim_1$ coincide for the set of growth functions $f$ satisfying $n\preceq f$.  Moreover, for any growth functions $f$ and $g$, $f\preceq_1 g$ if and only if $\max(f, n)\preceq \max(g, n)$. 

These preorders and equivalences are extended in a natural way to finite tuples of growth functions.  For example, for pairs of growth functions $(f_1,f_2)$ and $(g_1,g_2)$, then $(f_1,f_2)\preceq(g_1,g_2)$ if and only if $f_1\preceq g_1$ and $f_2\preceq g_2$.

\end{definition}

\begin{definition}[$S$-machine complexity]  Let $\S$ be an S-machine over the hardware $\H = (Y, Q)$, where $Y = (Y_i)_{i\in \overline{N}}$ and $Q = (Q_i)_{i\in \overline{N+1}}$.  For all $i\in\overline{N+1}$, fix elements $q_i^s,q_i^f\in Q_i$ such that there exists $j\in\overline{N+1}$ such that $q_j^s\neq q_j^f$.  Further, fix an index $i^*\in\overline{N}$.  Then the quadruple $(\S,i^*,(q_i^s)_{i\in\overline{N+1}},(q_i^f)_{i\in\overline{N+1}})$ is called a \emph{recognizing $S$-machine}.

In this case, the states $(q_i^s)_{i\in\overline{N+1}}$ and $(q_i^f)_{i\in\overline{N+1}}$ are called the \emph{starting} and the \emph{finishing states} of the recognizing $S$-machine.  Further, $i^*$ is called its \emph{input tape}.

Typically, the recognizing $S$-machine is identified with the $S$-machine defining it, with the choices of the starting state, finishing state, and input tape understood.  Moreover, all $S$-machines from this point forward will be taken to be recognizing.

For any $w\in L(Y_{i^*})$, the \emph{input word of $\S$ corresponding to $w$} is the admissible word $$I(w)=q_1^sw_1q_2^sw_2\dots w_Nq_{N+1}^s$$ such that $w_i=w$ if $i=i^*$ and $w_i=1$ otherwise.  Further, the \emph{accept word} of $\S$ is the admissible word $W_{ac}=q_1^fq_2^f\dots q_{N+1}^f$.

The \emph{language recognized by $\S$}, denoted $L_\S$, is the subset of $F(Y_{i^*})$ such that $w\in L_\S$ if and only if there exists a computation $\eps_w$ of $\S$ between the admissible words
$I(w)$ and $W_{ac}$.  In this case, $\eps_w$ is called an \emph{accepting computation} of $w$.  

For any computation $\eps = (w_0, \tau_1, w_1, \dots, \tau_n, w_n)$ of an $S$-machine, denote $\left[\varepsilon\right]$ as the sequence $(w_0, \dots, w_n)$.  The notation $w'\in[\eps]$ indicates that $w'=w_i$ for some $i\in\overline{0,n}$.  Then, define:

\begin{itemize}

\item the \emph{time} of $\eps$ as $\tm(\eps)=|\eps|=n$,

\item the space of $\eps$ as $\sp(\eps)=\max\{|w'|\mid w'\in[\eps]\}$, and

\item the \emph{area} of $\eps$ as $\ar(\varepsilon)=\sum\limits_{w'\in[\eps]} |w'|$

\end{itemize} 
For $w\in L_\S$, the \emph{time of $w$}, denoted $\tm_\S(w)$, is defined to be the minimal time of any accepting computation of $w$.  Similarly, the \emph{space} $\sp_\S(w)$ and \emph{area} $\ar_\S(w)$ of $w$ are the minimal space and area, respectively, of accepting computations of $w$.  Note that an accepting computation of $w$ achieving any one of these values need not achieve the others.

Then, the \emph{time complexity $\TM_\S \colon \N\rightarrow \N$ of $\S$} is defined by $$\TM_\S(n) = \max\bigl\{\tm_\S(w)\mid w\in L_\S, \ |w|\leq n\bigr\}$$where the maximum of the empty set is taken to be $0$.  The \emph{space complexity $\SP_\S$} and \emph{area complexity $\AR_\S$} of $\S$ are defined analogously.  By definition, all three of these are growth functions.

Finally, given two growth functions $f$ and $g$, $\S$ is said to have \emph{time-space complexity at most $(f,g)$} if for all $w\in L_\S$ satisfying $|w|\leq n$, there exists an accepting computation $\eps$ of $w$ such that $\tm_\S(\eps)\leq f(n)$ and $\sp_\S(\eps)\leq g(n)$.  In this case, $(f,g)$ is called an \emph{upper bound on $\TMSP_\S$}, indicated by the notation $\TMSP_\S\leq(f,g)$.  

Moreover, if $f'$ and $g'$ are growth functions satisfying $(f,g)\sim(f',g')$, then $(f',g')$ is called an \emph{asymptotic upper bound on $\TMSP_\S$}, denoted $\TMSP_\S\preceq(f',g')$.  Similarly, if $\TMSP_\S\leq(f,g)$ and $(f,g)\sim_1(f',g')$, then $\TMSP_\S\preceq_1(f',g')$.

\end{definition}

\begin{definition}[Object Complexity]  Let $\G\in\SGRAPH^*$ be an $S$-graph over $(\INST_\G,\H_\G)$ with $\H_\G=(\I_\G,\A_\G)$.  Then, let $$\eps=\bigl((w_0,s_0),e_1,(w_1,s_1),e_2,\dots,e_m,(w_m,s_m)\bigr)$$ be a computation of $\EXP(\G)$.  Similar to the analogue for $S$-machine computations, set $[\eps]$ as the sequence of $L_{\H_\G^\exp}$-words $(w_0,w_1,\dots,w_m)$ and let the notation $w'\in[\eps]$ indicate that $w'=w_i$ for some $i\in\overline{0,m}$.  

Then, letting $|w'|=\sum\limits_{i\in\I_\G^\exp}|w'(i)|$ for any $L_{\H_\G^\exp}$-word $w'$, define:

\begin{itemize}

\item the \emph{time} of $\eps$ as $\tm_\G(\eps)=|\eps|=n$,

\item the \emph{space} of $\eps$ by $\sp_\G(\eps)=\max\{|w'|\mid w'\in[\eps]\}$, and 

\item the \emph{area} of $\eps$ by $\ar_\G(\eps)=\sum\limits_{w'\in[\eps]}|w'|$.

\end{itemize}

Let $\Obj\in\OBJ^*$ and $\op\in\Obj$.  Then, $\eps$ is said to \emph{realize} $(w_1,w_2)\in\IO_\op^{s_1,s_2}$ if it is an input-output computation between $(w_1\sqcup1_\op^\int,s_1)$ and $(w_2\sqcup1_\op^\int,s_2)$.  Similarly, $\eps$ is said to \emph{$\EXP$-realize} $(w_1,w_2)\in\IO_\op^{s_1,s_2}$ if it is an input-output computation of $\EXP(\op)$ between $(w_1\sqcup1_{\EXP(\op)}^\int,s_1)$ and $(w_2\sqcup1_{\EXP(\op)}^\int,s_2)$.  Note that $\IO_\op=\IO_{\EXP(\op)}$ by \Cref{th-expansion-props}(g), and so the relation $(w_1,w_2)\in\IO_\op^{s_1,s_2}$ holds if and only if there exists a computation realizing it, and further if and only if there exists a computation $\EXP$-realizing it.

The set of all computations $\EXP$-realizing $(w_1,w_2)\in\IO_\op^{s_1,s_2}$ is denoted $\EXP_\op^\IO(w_1,s_1,w_2,s_2)$.  Note that, per previously defined notation, 
$$\EXP_\op^\IO(w_1,s_1,w_2,s_2)=\eps_{\EXP(\op)}(w_1\sqcup1_{\EXP(\op)}^\int,s_1,w_2\sqcup1_{\EXP(\op)}^\int,s_2)$$

Then, for $(w_1,w_2)\in\IO_\op^{s_1,s_2}$, the \emph{time} of $(w_1,s_1,w_2,s_2)$ is defined as $$\tm_\op(w_1,s_1,w_2,s_2)=\min\{\tm_\op(\eps)\mid\eps\in\EXP_\op^\IO(w_1,s_1,w_2,s_2)\}$$ 
The \emph{time complexity} of $\op$ is then the function $\TM_\op:\N\to\N$ given by $$\TM_\op(n)=\max\{\tm_\op(w_1,s_1,w_2,s_2)\mid (w_1,w_2)\in\IO_\op^{s_1,s_2}, \ |w_1|,|w_2|\leq n\}$$ 
The \emph{space complexity} and \emph{area complexity} of $\op$, $\SP_\op$ and $\AR_\op$, are defined analogously.  By definition, each of the three functions $\TM_\op$, $\SP_\op$, and $\AR_\op$ is a growth function.

Given a pair of growth functions $f$ and $g$, $\op$ has \emph{time-space complexity at most $(f,g)$} if for any $(w_1,w_2)\in\IO_\op^{s_1,s_2}$ such that $|w_1|,|w_2|\leq n$, there exists $\eps\in\EXP_\op^\IO(w_1,s_1,w_2,s_2)$ satisfying $\tm_\op(\eps)\leq f(n)$ and $\sp_\op(\eps)\leq g(n)$.  In this case, $(f,g)$ is called an \emph{upper bound} on $\TMSP_\op$, which is indicated by the notation $\TMSP_\op\leq(f,g)$.  

Moreover, for growth functions $f'$ and $g'$, the notation $\TMSP_\op\preceq(f',g')$ is used to indicate that there exist growth functions $f$ and $g$ such that $\TMSP_\op\leq(f,g)$ and $(f,g)\sim(f',g')$.  Analogously, $\TMSP_\op\preceq_1(f',g')$ indicates the existence of growth functions $f$ and $g$ such that $\TMSP_\op\leq(f,g)$ and $(f,g)\sim_1(f',g')$.

If $\TMSP_\op\leq(f,g)$ for all $\op\in\Obj$, then $(f,g)$ is called an \emph{upper bound on $\TMSP_\Obj$}, denoted $\TMSP_\Obj\leq(f,g)$.  As above, the notation $\TMSP_\Obj\preceq(f',g')$ (respectively $\TMSP_\Obj\preceq_1(f',g')$) indicates the existence of an upper bound $(f,g)$ on $\TMSP_\Obj$ such that $(f',g')\sim(f,g)$ (respectively $(f',g')\sim_1(f,g)$).

\end{definition}

Given an object $\Obj\in\OBJ^*$ and $\op\in\Obj$, a quadruple $(w_1,s_1,w_2,s_2)\in\IO_\op$ is called \emph{essential} if either:
\begin{itemize}

\item $s_1=\op_s$ and $s_2=\op_f$, or

\item $s_1=s_2$ and $w_1\neq w_2$,

\end{itemize}
The following statement reduces the study of the time-space complexity of an object to the complexity of computations that $\EXP$-realize essential quadruples.

\begin{lemma}\label{lem-tmsp-basics}

Let $\Obj\in\OBJ^*$ and $f,g$ be growth functions such that $g(n)\geq n$ for all $n\in\N$.  Then $\TMSP_\Obj\leq(f,g)$ if and only if for any $\op\in\Obj$ and essential $(w_1,s_1,w_2,s_2)\in\IO_\op$, there exists an $\EXP$-realizing computation $\eps\in\EXP_\op^\IO(w_1,s_1,w_2,s_2)$ satisfying $\tm_\op(\eps)\leq f(n)$ and $\sp_\op(\eps)\leq g(n)$ for $n=\max(|w_1|,|w_2|)$.

\end{lemma}

\begin{proof}

The sufficiency of the statement is immediate by definition.

For necessity, let $(w_1,w_2)\in\IO_\op^{s_1,s_2}$ such that neither (1) nor (2) hold.  

If $s_1=s_2$, then $w_1=w_2$, and hence there exists an empty computation $\eps\in\EXP_\op^\IO(w_1,s_1,w_2,s_2)$.  As a result, $\tm_\op(\eps)=0\leq f(n)$ and $\sp_\op(\eps)=n\leq g(n)$ for $n=\max(|w_1|,|w_2|)$.

Otherwise, $s_1=\op_f$ and $s_2=\op_s$.  Then, by \Cref{lem-lio-io}(2), $(w_2,w_1)\in\IO_\op^{s,f}$.  By hypothesis, there then exists a computation $\eps\in\EXP_\op^\IO(w_2,s_2,w_1,s_1)$ satisfying the bounds $\tm_\op(\eps)\leq f(n)$ and $\sp_\op(\eps)\leq g(n)$ for $n=\max(|w_2|,|w_1|)$.  

But then letting $\eps=\bigl((w_0',s_0'),e_1,(w_1',s_1'),e_2,\dots,e_m,(w_m',s_m')\bigr)$, it follows that $$\eps^{-1}=\bigl((w_m',s_m'),e_m^{-1},\dots,e_2^{-1},(w_1',s_1'),e_1,(w_0',s_0')\bigr)$$ is a computation $\EXP$-realizing $(w_1,w_2)\in\IO_\op^{s_1,s_2}$ and satisfying $\tm_\op(\eps^{-1})=\tm_\op(\eps)\leq f(n)$ and $\sp_\op(\eps^{-1})=\sp_\op(\eps)\leq g(n)$.

\end{proof}

For a practical example, \Cref{lem-basic-0bounds} presents two particular calculations of the complexity bounds of specific (proper) operations that were constructed and studied in previous sections.

\begin{lemma}\label{lem-basic-0bounds} \

\begin{enumerate}

\item $\TMSP_\dt\preceq(n,n)$.

\item For any $A\subseteq\A^*$, $\TMSP_\cpy\preceq(n,n)$ for $\cpy=\cpy^A$.

\end{enumerate}

\end{lemma}

\begin{proof}

Recall that the expansion of an operation of degree 0 is itself.  In particular, realizing computations and $\EXP$-realizing computations are equivalent for operations of degree 0.  Hence, by \Cref{lem-tmsp-basics} and the definition of the preorder, in order to show that $\TMSP_\op\preceq(n,n)$, it suffices to find $C>0$ such that for any essential quadruple $(w_1,s_1,w_2,s_2)\in\IO_\op$, there exists a computation $\eps$ of $\op$ between the configurations $(w_1\sqcup1_\op^\int,s_1)$ and $(w_2\sqcup1_\op^\int,s_2)$ satisfying $\tm_\op(\eps),\sp_\op(\eps)\leq C\max(|w_1|,|w_2|)+C$.

\bigskip

\noindent
(1)  Let $n\in\N$ and $\tau\in\{\pm1\}$.  Then, the computation $\eps_{n,\tau}$ constructed in the proof of \Cref{lem-ndr-div2} is an input-output computation between $((\delta^{2\tau n},1),\dt_s)$ and $((\delta^{\tau n},1),\dt_f)$.  By construction,
\begin{align*}
\eps_{n,\tau}=\bigl((&u_0,\dt_s),e_1,(u_{1,0},q_1),e_2^\tau,(u_{1,1},q_1),e_2^\tau,\dots,e_2^\tau,(u_{1,n},q_1), \\
&e_3,(u_{2,0},q_2),e_5^\tau,(u_{2,1},q_2),e_5^\tau,\dots,e_5^\tau,(u_{2,n},q_2),e_6,(u_3,\dt_f)\bigr)
\end{align*}
Hence, $\tm_\dt(\eps_{n,\tau})=2n+3$.

By the definition of $\lab(e_1)$, $u_{1,0}=u_0=(\delta^{2\tau n},1)$.  The definition of the action $\lab(e_2)$ then implies that $u_{1,i}=(\delta^{2\tau(n-i)},\delta^{\tau i})$ for all $i\in\overline{n}$, and so $|u_{1,i}|=2n-i\leq 2n$.  In particular, $u_{1,n}=(1,\delta^{\tau n})$, and so $u_{2,0}=u_{1,n}$ by the definition of $\lab(e_3)$.  Then, $u_{2,i}=(\delta^{\tau i},\delta^{\tau(n-i)})$ for all $i\in\overline{n}$, and so $|u_{2,i}|=n\leq 2n$.  Finally, the definition of $\lab(e_3)$ yields $u_3=u_{2,n}=(\delta^{\tau n},1)$.  Hence, $\sp_\dt(\eps_{n,\tau})=2n$.

Similarly, for all $n\in\N$ and $\tau\in\{\pm1\}$, the computation $\eps_{n,\tau}'$ is an input-output computation between the configurations $((\delta^{2\tau n-1},1),\dt_s)$ and $((\delta^{\tau n},1),\dt_f)$.  By construction, 
\begin{align*}
\eps_{n,\tau}'=\bigl((&u_0',\dt_s),e_1,(u_{1,0}',q_1),e_2^\tau,(u_{1,1}',q_1),e_2^\tau,\dots,e_2^\tau,(u_{1,n}',q_1), \\
&e_4,(u_{2,0}',q_2),e_5^\tau,(u_{2,1}',q_2),e_5^\tau,\dots,e_5^\tau,(u_{2,n}',q_2),e_6,(u_3',\dt_f)\bigr)
\end{align*}
Hence, $\tm_\dt(\eps_{n,\tau}')=2n+3$.

As above, the definition of the actions $\lab(e_1)$ and $\lab(e_2)$ imply that $u_{1,0}'=u_0'=(\delta^{2\tau n-1},1)$ and $u_{1,i}'=(\delta^{2\tau(n-i)-1},\delta^i)$ for all $i\in\overline{n}$.  In particular, since $u_{1,n}'=(\delta^{-1},\delta^{\tau n})$, the definition of $\lab(e_4)$ yields $u_{2,0}'=(1,\delta^{\tau n})$.  As above, this yields $u_{2,i}'=(\delta^{\tau i},\delta^{\tau(n-i)})$ for $i\in\overline{n}$, and so $u_{2,n}'=u_3'=(\delta^{\tau n},1)$.

So, for any $i\in\overline{n}$, $|u_{1,i}'|=|2(n-i)-1|+i\leq 2(n-i)+1+i=2n-i+1$.  Hence, since $|u_{2,i}'|=n$ for all $i\in\overline{n}$, it follows that $\sp_\dt(\eps_{n,\tau}')\leq 2n+1$.

Now, suppose $(w_1,w_2)\in\IO_\dt^{s,f}$ such that $w_1=\delta^m$ for $m\in\Z$.  Then there exists $n\in\N$ and $\tau\in\{\pm1\}$ such that either:
\begin{itemize}

\item $m=2\tau n$, in which case $w_2=\delta^{\tau n}$; or

\item $m=2\tau n-1$, in which case $w_2=\delta^{\tau n}$.

\end{itemize}

In the first case, $|m|=2n$, and so $\eps_{n,\tau}\in\eps_\dt^\IO(w_1,\dt_s,w_2,\dt_f)$ and satisfies 
\begin{align*}
\tm_\dt(\eps_{n,\tau})&=2n+3=|m|+3 \\
\sp_\dt(\eps_{n,\tau})&=2n=|m|
\end{align*}
In the second case, $|m|\geq2n-1$, and so $\eps_{n,\tau}'\in\eps_\dt^\IO(w_1,\dt_s,w_2,\dt_f)$ and satisfies 
\begin{align*}
\tm_\dt(\eps_{n,\tau}')&=2n+3\leq |m|+4 \\
\sp_\dt(\eps_{n,\tau}')&\leq2n+1\leq|m|+2
\end{align*}

Hence, for any $(w_1,w_2)\in\IO_\dt^{s,f}$, there exists a computation $\eps$ of $\dt$ between the configurations $(w_1\sqcup1_\dt^\int,\dt_s)$ and $(w_2\sqcup1_\dt^\int,\dt_f)$ such that $\tm_\dt(\eps)\leq |w_1|+4$ and $\sp_\dt(\eps)\leq |w_1|+2$.

Now suppose $(w_1,w_2)\in\IO_\dt^{s_1,s_2}$ such that $s_1=s_2$ and $w_1\neq w_2$.
By \Cref{lem-ndr-div2}, $s_1=s_2=\dt_s$ and $\{w_1,w_2\}=\{\delta^{2\tau n},\delta^{2\tau n-1}\}$ for some $n\in\N$ and $\tau\in\{\pm1\}$.  But then $(w_i,\delta^{\tau n})\in\IO_\dt^{s,f}$ for $i\in\{1,2\}$, so that as above there exist input-output computations $\eps_i$ between $(w_i\sqcup1_\dt^\int,\dt_s)$ and $(\delta^{\tau n}\sqcup1_\dt^\int,\dt_f)$ satisfying $\tm_\dt(\eps_i)\leq|w_i|+4$ and $\sp_\dt(\eps_i)\leq|w_i|+2$.  

Hence, letting $\rho_i$ be the path supporting $\eps_i$, the concatenated path $\rho_1,\rho_2^{-1}$ supports a computation $\eps\in\eps_\dt^\IO(w_1,\dt_s,w_2,\dt_s)$ such that: 
\begin{align*}
\tm_\dt(\eps)&=\tm_\dt(\eps_1)+\tm_\dt(\eps_2)\leq2\max(|w_1|,|w_2|)+8 \\ 
\sp_\dt(\eps)&=\max(\sp_\dt(\eps_1),\sp_\dt(\eps_2))\leq\max(|w_1|,|w_2|)+2
\end{align*}

Thus, the parameter $C=8$ implies $\TMSP_\dt\preceq(n,n)$.

\bigskip

\noindent
(2)  As  \Cref{lem-io-cpy} indicates that $\IO_\cpy^{s,s}=\IO_\cpy^{f,f}=\Id$, any essential quadruple must be of the form $(w_1,\cpy_s,w_2,\cpy_f)$.  Hence, it suffices to find $C>0$ such that for any $(w_1,w_2)\in\IO_\cpy^{s,f}$, there exists a computation $\eps$ of $\cpy$ between $(w_1\sqcup1_\cpy^\int,\cpy_s)$ and $(w_2\sqcup1_\cpy^\int,\cpy_f)$ satisfying $$\tm_\cpy(\eps),\sp_\cpy(\eps)\leq C\max(|w_1|,|w_2|)+C$$

By \Cref{lem-io-cpy}, for any $(w_1,w_2)\in\IO_\cpy^{s,f}$, there exists $w\in F(A)$ such that $w_1=(w,1)$ and $w_2=(w,w)$.  Let $w=a_1^{\delta_1}\dots a_n^{\delta_n}$ such that $a_i\in\A$ and $\delta_i\in\{\pm1\}$.  Then, as in the proof of \Cref{lem-io-cpy}, there exists a realizing computation $\eps_w$ of the form:
\begin{align*}
\eps_w=\bigl((&u_0,\cpy_s),e_1,(u_{1,0},s_1),e_{1,x_1}^{\delta_1},(u_{1,1},s_1),e_{1,x_2}^{\delta_2},\dots,e_{1,x_n}^{\delta_n},(u_{1,n},s_1), \\
&e_t,(u_{2,0},s_2),e_{2,x_n}^{\delta_n},(u_{2,1},s_2),\dots,e_{2,x_1}^{\delta_1},(u_{2,n},s_2),e_2,(u_3,\cpy_f)\bigr)
\end{align*}
Hence, $\tm_\cpy(\eps_w)=2n+3$.

By the definition of $\lab(e_1)$, $u_{1,0}=u_0=(w,1,1)$.  The makeup of $\lab(e_{1,x})$ for $x\in\A$ then implies that $u_{1,j}=(v_j',1,v_j)$ for all $j\in\overline{n-1}$, where $v_j=x_1^{\delta_1}\dots x_j^{\delta_j}$ and $v_j'=x_{j+1}^{\delta_{j+1}}\dots x_n^{\delta_n}$.  In particular, $u_{1,n}=(1,1,w)$, and so $u_{2,0}=(1,1,w)$ by the definition of $\lab(e_t)$.  Then, the definition of $\lab(e_{2,x})$ yields $u_{2,j}=(v_{n-j}',v_{n-j}',v_{n-j})$ for $j\in\overline{n-1}$, so that $u_{2,n}=(w,w,1)$.  Hence, $\lab(e_2)$ implies $u_3=(w,w,1)$.

As $|v_j|+|v_j'|=n$ for all $j\in\overline{n-1}$, $|u_0|=|u_{1,i}|=|u_{2,0}|=n$ for all $i\in\overline{n}$.  What's more, $|u_{2,j}|=n+|v_{n-j}'|\leq2n$ for all $j\in\overline{n-1}$.  Hence, since $|u_{2,n}|=|u_3|=2n$, it follows that $\sp_\cpy(\eps_w)=2n$.

Thus, the statement is proved by the parameter choice $C=3$.

\end{proof}

There are other examples of objects and operations constructed in previous sections, e.g $\check_{>0}$ and $\inc()$ constructed in \Cref{ex-counter1}, whose complexity bounds will be useful for later constructions.  However, the proof of these bounds is delayed slightly.

This is because these examples are objects and operations of positive degree, and so their complexity bounds depend in some way on those of objects and operations of lower degree.  For example, $\check_{>0}$ is constructed from $\dt$ and $\cpy$, using them as instances, and thus its expansion contains subgraphs arising from the graphs of $\dt$ and $\cpy$.

Hence, estimating functions are first constructed, providing a link between the complexity of an object and that of its instances that will enable future arguments to bypass expansions.

\begin{definition}[Estimated Complexity]\label{def-complexity-est}

Let $\Obj\in\OBJ^*$.  For each $\inst\in\INST_\Obj$, fix a pair of growth functions $\F_\inst=(f_\inst,g_\inst)$.  Then, the tuple of pairs $\F=(\F_\inst)_{\inst\in\INST_\Obj}$ is called an \emph{estimating set} for $\Obj$.

Let $\delta$ be a computation of $\op\in\Obj$ supported by a generic path consisting of a single (generic) edge.  Denote $\delta=\bigl((w_0,s_0),e',(w_1,s_1)\bigr)$, where $e'=(\sf_0,\sf_1,e)$ with $e=(s_0',s_1',\act_e)\in E_\op$.  Then define the \emph{estimated time} and \emph{estimated space} of $\delta$ with respect to $\F$, denoted $\tm_\op^\F(\delta)$ and $\sp_\op^\F(\delta)$, respectively, by:
\begin{itemize}

\item If $\act_e\in\ACT_0$, then $\tm_\op^\F(\delta)=1$ and $\sp_\op^\F(\delta)=\max(|w_0|,|w_1|)$.

\item If $\act_e\in\ACT_\op$, then for $\act_e=(\inst_e,\op_e,\mu_{\act_e}^\ext)$,
\begin{align*}
\tm_\op^\F(\delta)&=f_{\inst_e}(M_e) \\
\sp_\op^\F(\delta)&=g_{\inst_e}(M_e)+|w_0(\I_\op-\I_{\act_e})|
\end{align*}
where $M_e=\max(|w_0(\I_{\act_e})|,|w_1(\I_{\act_e})|)$.

\end{itemize}

Now let $\eps=\bigl((w_0,s_0),e_1',(w_1,s_1),e_2',\dots,e_m',(w_m,s_m)\bigr)$ be an arbitrary nonempty computation of $\op$.  The computation $\eps(i)=\bigl((w_{i-1},s_{i-1}),e_i',(w_i,s_i)\bigr)$ of $\op$ is then called the \emph{$i$-th subcomputation of $\eps$}.  Note that as $|\eps(i)|=1$ for each $i$, $\tm_\op^\F(\eps(i))$ and $\sp_\op^\F(\eps(i))$ are defined above.  With this, the \emph{estimated time} and \emph{estimated space} of $\eps$ with respect to $\F$ are then defined to be $\tm_\op^\F(\eps)=\sum_{i\in\overline{m}}\tm_\op^\F(\eps(i))$ and $\sp_\op^\F(\eps)=\max_{i\in\overline{m}}\sp_\op^\F(\eps(i))$, respectively.

Finally, given an empty computation $\eps=(w_0,s_0)$ of $\op$, define $\tm_\op^\F(\eps)=0$ and $\sp_\op^\F(\eps)=|w_0|$.  Note that in this case, $\eps$ can be viewed as a computation of $\EXP(\op)$, and so $\tm_\op^\F(\eps)=\tm_\op(\eps)$ and $\sp_\op^\F(\eps)=\sp_\op(\eps)$.

The \emph{estimated time} and \emph{estimated space complexity functions} of $\op$ with respect to the estimating set $\F$ are defined in much the same way as for the actual complexity functions.  In particular, $\op$ has \emph{estimated time-space complexity with respect to $\F$ at most $(f,g)$} if $f$ and $g$ are a pair of growth functions such that for all $(w_1,s_1,w_2,s_2)\in\IO_\op$ with $|w_1|,|w_2|\leq n$, there exists a computation $\eps$ of $\op$ realizing $(w_1,w_2)\in\IO_\op^{s_1,s_2}$ and satisfying $\tm_\op^\F(\eps)\leq f(n)$ and $\sp_\op^\F(\eps)\leq g(n)$.  In this case, $(f,g)$ is called an \emph{upper bound} on $\TMSP_\op^\F$, as indicated by the notation $\TMSP_\op^\F\leq(f,g)$.  Further, the pair of growth functions $(f',g')$ is called an \emph{asymptotic upper bound on $\TMSP_\op^\F$}, denoted $\TMSP_\op^\F\preceq(f',g')$, if there exists an upper bound $(f,g)$ on $\TMSP_\op^\F$ such that $(f,g)\sim(f',g')$.

Finally, a pair of growth functions $(f,g)$ is called an upper bound on $\TMSP_\Obj^\F$, denoted $\TMSP_\Obj^\F\leq(f,g)$, if $\TMSP_\op^\F\leq(f,g)$ for all $\op\in\Obj$.  Similarly, $(f',g')$ is an asymptotic upper bound on $\TMSP_\Obj^\F$, denoted $\TMSP_\Obj^\F\preceq(f',g')$, if $\TMSP_\op^\F\preceq(f',g')$ for all $\op\in\Obj$.

\end{definition}

\begin{proposition}\label{prop-complexity-estimate} Let $\F=(\F_\inst)_{\inst\in\INST_\Obj}$ be an estimating set for $\Obj\in\OBJ^*$.  For all $\inst\in\INST_\Obj$, let $\inst=(\Obj_\inst,\mu_\inst^\val)$, let $\F_\inst=(f_\inst,g_\inst)$, and abbreviate $\TMSP_{\Obj_\inst}$ by $\TMSP_\inst$.

\begin{enumerate}

\item Suppose $\F_\inst$ is an upper bound on $\TMSP_\inst$ for all $\inst\in\INST_\Obj$.  Then for any computation $\delta=\bigl((w_0,s_0),e',(w_1,s_1)\bigr)$ of $\op\in\Obj$, there exists a computation $\eps$ of $\EXP(\op)$ between $(w_0\sqcup1_\op^\inv,s_0)$ and $(w_1\sqcup1_\op^\inv,s_1)$ satisfying $\tm_\op(\eps)\leq\tm_\op^\F(\delta)$ and $\sp_\op(\eps)\leq\sp_\op^\F(\delta)$.

\item If $\F_\inst$ is an upper bound on $\TMSP_{\inst}$ for all $\inst\in\INST_\Obj$, then $\TMSP_\Obj^\F\leq(f,g)$ implies $\TMSP_\Obj\leq(f,g)$.
	 
\item Suppose $\TMSP_{\inst} \preceq \F_\inst$ and $f_\inst\neq0$ for all $\inst\in\INST_\Obj$.  Further, suppose there exist growth functions $\{\phi_\inst,\psi_\inst\}_{\inst\in\INST_\Obj}$ such that $\psi_\inst\neq0$, $f_\inst(c n) \leq \phi_\inst(c) f_\inst(n)$, and $g_\inst(cn)\leq \psi_\inst(c)g_\inst(n)$ for any $c,n\in\N$.  Then $\TMSP_\Obj^\F \preceq (f, g)$ implies that $\TMSP_\Obj \preceq (f, g)$.

\end{enumerate}
	 
\end{proposition}
\begin{proof}

(1) Let $e'=(\sf_0,\sf_1,e)$ for $e=(t_0,t_1,\act)\in E_\op$.

Suppose $e\in E_\op^0$.  If $\sf_0=\sf_1$, then $w_0=w_1$ and $s_0=s_1$, so that there exists an empty computation $\eps$ of $\EXP(\op)$ between $(w_0\sqcup1_\op^\inv,s_0)$ and $(w_1\sqcup1_\op^\inv,s_1)$.  In this case, $\tm_\op(\eps)=0<1=\tm_\op^\F(\delta)$ and $\sp_\op(\eps)=|w_0\sqcup1_\op^\inv|=|w_0|=\sp_\op^\F(\delta)$.  Otherwise, set $\a=1$ if $\sf_0=\s$ and $\a=-1$ if $\sf_0=\f$.  Then, by construction, there exists a computation
$$\eps=\bigl((w_0\sqcup1_\op^\inv,s_0),(e'')^\a,(w_1\sqcup1_\op^\inv,s_1)\bigr)$$
of $\EXP(\op)$, where $e''$ is the proper edge constructed from $e$.  In this case,
\begin{align*}
\tm_{\op}(\eps)&=1=\tm_\op^\F(\delta) \\ 
\sp_{\op}(\eps)&=\max(|w_0\sqcup1_\op^\inv|,|w_1\sqcup1_\op^\inv|)=\max(|w_0|,|w_1|)=\sp_\op^\F(\delta)
\end{align*}

Hence, it suffices to assume that $e\in E_\op-E_\op^0$.  Letting $\act=(\inst,\op_e,\mu_\act^\ext)$, \Cref{th-expansion-props}(c) then implies that
$$(\mu_\act^{-1}(w_0(\I_\act)),\mu_\act^{-1}(w_1(\I_\act)))\in\IO_{\EXP(\op_e)}^{\lambda_{\op_e}(\sf_0),\lambda_{\op_e}(\sf_1)}$$  
So, as $\F_\inst$ is an upper bound of $\TMSP_\inst$ by hypothesis, there exists a computation of $\EXP(\op_e)$ 
$$\gamma=\bigl((w_0',s_0'),e_1'',(w_1',s_1'),e_2'',\dots,e_m'',(w_m',s_m')\bigr)$$
between the configuration $(w_0',s_0')=\bigl(\mu_\act^{-1}(w_0(\I_\act))\sqcup1_{\EXP(\op_e)}^\int,\lambda_{\op_e}(\sf_0)\bigr)$ and the configuration $(w_m',s_m')=\bigl(\mu_\act^{-1}(w_1(\I_\act))\sqcup1_{\EXP(\op_e)}^\int,\lambda_{\op_e}(\sf_1)\bigr)$ which satisfies $\tm_{\op_e}(\gamma)\leq f_\inst(M)$ and $\sp_{\op_e}(\gamma)\leq g_\inst(M)$ for $M=\max(|w_0'|,|w_m'|)$.  

Note that $|w_0'|=|\mu_\act^{-1}(w_0(\I_\act))|=|w_0(\I_\act)|$ and, similarly, $|w_m'|=|w_1(\I_\act)|$.  Hence, it follows that $M=\max(|w_0(\I_\act)|,|w_1(\I_\act)|)$.

\Cref{th-expansion-props}(c) then yields a computation 
$$\eps=\kappa_\delta(\gamma)=\bigl((u_0,\sigma_0),e_1''',(u_1,\sigma_1),e_2''',\dots,e_m''',(u_m,\sigma_m)\bigr)$$
of $\EXP(\op)$ such that $(u_0,\sigma_0)=(w_0\sqcup1_\op^\inv,s_0)$, $(u_m,\sigma_m)=(w_1\sqcup1_\op^\inv,s_1)$, and $$u_j=\mu_e(w_j')\sqcup w_0(\I_\op-\I_\act)\sqcup (1)_{i\in\I_\op^\inv-\I_e^\int}$$ for all $j\in\overline{0,m}$.  

Note that $|u_j|=|\mu_e(w_j')|+|w_0(\I_\op-\I_\act)|=|w_j'|+|w_1(\I_\op-\I_\act)|$ for all $j\in\overline{0,m}$.  Thus,
\begin{align*}
\tm_{\op}(\eps)&=m=\tm_{\op_e}(\gamma)\leq f_\inst(M)=\tm_\op^\F(\delta) \\
\sp_{\op}(\eps)&=\max_{j\in\overline{0,m}}|u_j|\leq g_\inst(M)+|w_1(\I_\op-\I_{\act_i})|=\sp_\op^\F(\delta)
\end{align*}

\bigskip

\noindent
(2) Let $(w_1,s_1,w_2,s_2)\in\IO_\op$ for some $\op\in\Obj$.  Then, for $n=\max(|w_1|,|w_2|)$, there exists a computation 
$$\delta=\bigl((w_0',s_0'),e_1,(w_1',s_1'),e_2,\dots,e_m,(w_m',s_m')\bigr)$$
realizing $(w_1,w_2)\in\IO_\op^{s_1,s_2}$ such that $\tm_\op^\F(\delta)\leq f(n)$, and $\sp_\op^\F(\delta)\leq g(n)$.

If $\delta$ is empty, then $\tm_\op(\delta)\leq f(n)$ and $\sp_\op(\delta)\leq g(n)$.  So, one can assume that $\delta$ is nonempty.

By (1), for all $i\in\overline{m}$ there exists a computation $\eps_i$ of $\EXP(\op)$ between $(w_{i-1}'\sqcup1_\op^\inv,s_{i-1}')$ and $(w_i'\sqcup1_\op^\inv,s_i')$ such that $\tm_\op(\eps_i)\leq\tm_\op^\F(\delta(i))$ and $\sp_\op(\eps_i)\leq\sp_\op^\F(\delta(i))$.

Let $\rho_i$ be the path in $\EXP(\op)$ supporting $\eps_i$.  Then, $\rho_1,\dots,\rho_m$ is a path supporting a computation $\eps$ of $\EXP(\op)$ between $(w_1\sqcup1_{\EXP(\op)}^\int,s_1)$ and $(w_2\sqcup1_{\EXP(\op)}^\int,s_2)$ and satisfying
\begin{align*}
\tm_\op(\eps)&=\sum_{i\in\overline{m}}\tm_\op(\eps_i)\leq\sum_{i\in\overline{m}}\tm_\op^\F(\delta(i))=\tm_\op^\F(\delta)\leq f(n) \\
\sp_\op(\eps)&=\max_{i\in\overline{m}}\sp_\op(\eps_i)\leq\max_{i\in\overline{m}}\sp_\op^\F(\delta(i))=\sp_\op^\F(\delta)\leq g(n)
\end{align*}
Hence, $\TMSP_\op\leq(f,g)$.

\bigskip

\noindent
(3)  For all $\inst\in\INST_\Obj$, let $\F_\inst'=(f_\inst',g_\inst')$ be a pair of growth functions such that $\TMSP_\inst\leq\F_\inst'$ and $\F_\inst'\sim\F_\inst$.  Then, $\F'=(\F_\inst')_{\inst\in\INST_\Obj}$ is an estimating set for $\Obj$.

Similarly, let $f'$ and $g'$ be growth functions such that $\TMSP_\Obj^\F\leq(f',g')$ and $(f',g')\sim(f,g)$.

By the definition of the preorder and the hypotheses on $f_\inst$ and $\psi_\inst$, there exists a constant $C\in\N$ such that for all $\inst\in\INST_\Obj$,

\begin{itemize}

\item $f_\inst'(n)\leq Cf_\inst(Cn)+C$ for all $n\in\N$,

\item $g_\inst'(n)\leq Cg_\inst(Cn)+C$ for all $n\in\N$,

\item $f_\inst(n)>0$ for all $n\geq C$, and

\item $\psi_\inst(n)>0$ for all $n\geq C$.

\end{itemize}

Moreover, as increasing the value of $C$ would not affect whether these conditions hold, it may be assumed without loss of generality that $C\geq1$ and $\psi_\inst(n)\geq C^{-1}$ for all $n\geq C$.

Let $\delta=\bigl((w_0,s_0),e',(w_1,s_1)\bigr)$ be a computation of some $\op\in\Obj$.  Further, let $e'=(\sf_0,\sf_1,e)$ for $e=(t_0,t_1,\act)\in E_\op$.

If $e\in E_\op^0$, then $\tm_\op^\F(\delta)=1=\tm_\op^{\F'}(\delta)$ and $\sp_\op^\F(\delta)=\max(|w_0|,|w_1|)=\sp_\op^{\F'}(\delta)$.  

Otherwise, let $\act=(\inst,\op_e,\mu_\act^\ext)$.  Then, for $M=\max(|w_0(\I_\act)|,|w_1(\I_\act)|)$, 
\begin{align*}
\tm_\op^{\F'}(\delta)&=f_\inst'(M)\leq Cf_\inst(CM)+C\leq C\phi_\inst(C)f_\inst(M)+C \\
&\leq C\phi_\inst(C)\tm_\op^\F(\delta)+C\\
\\
\sp_\op^{\F'}(\delta)&=g_\inst'(M)+|w_0(\I_\op-\I_\act)|\leq Cg_\inst(CM)+C+|w_0(\I_\op-\I_\act)| \\
&\leq C\psi_\inst(C)\bigl(g_\inst(M)+|w_0(\I_\op-\I_\act)|\bigr)+C\leq C\psi_\inst(C)\sp_\op^\F(\delta)+C
\end{align*}

Hence, for $D=C\max_{\inst\in\INST_\Obj}(\phi_\inst(C),\psi_\inst(C),1)$, then $\tm_\op^{\F'}(\delta)\leq D\tm_\op^\F(\delta)+D$ and $\sp_\op^{\F'}(\delta)\leq D\sp_\op^\F(\delta)+D$.

Now, let $(w_0,s_0,w_1,s_1)\in\IO_\op$ for some $\op\in\Obj$.  Then, since $\TMSP_\Obj^\F\leq(f',g')$, there exists a computation 
$$\eps=\bigl((w_0',s_0'),e_1',(w_1',s_1'),e_2',\dots,e_m',(w_m',s_m')\bigr)$$ 
realizing $(w_0,w_1)\in\IO_\op^{s_0,s_1}$ which satisfies the bounds $\tm_\op^\F(\eps)\leq f'(n)$ and $\sp_\op^\F(\eps)\leq g'(n)$ for $n=\max(|w_1|,|w_2|)$

Suppose there exist $i,j\in\overline{0,m}$ such that $(w_i',s_i')=(w_j',s_j')$.  Assuming without loss of generality that $i\leq j$, then there exists a computation 
$$\eps'=\bigl((w_0',s_0'),e_1',\dots,(w_i',s_i'),e_{j+1}',(w_{j+1}',s_{j+1}'),\dots,(w_m',s_m')\bigr)$$
between $(w_0',s_0')$ and $(w_m',s_m')$ such that $\tm_\op^\F(\eps')\leq\tm_\op^\F(\eps)$ and $\sp_\op^\F(\eps')\leq\sp_\op^\F(\eps)$.  Hence, it may be assumed without loss of generality that the configurations of $\eps$ are distinct.

Then, the bounds above imply that:

\begin{align}\label{tmF}
\tm_\op^{\F'}(\eps)=\sum_{i\in\overline{m}}\tm_\op^{\F'}(\eps(i))\leq\sum_{i\in\overline{m}}(D\tm_\op^\F(\eps(i))+D)= D\tm_\op^\F(\eps)+Dm
\end{align}
\begin{equation}\label{spF}
\sp_\op^{\F'}(\eps)=\max_{i\in\overline{m}}\sp_\op^{\F'}(\eps(i))\leq\max_{i\in\overline{m}}(D\sp_\op^\F(\eps(i))+D)=D\sp_\op^\F(\eps)+D
\end{equation}

For each $i\in\I_\op$, let $L_i^C=\{w\in F(\A_\op(i)) : |w|<C\}$.  As each alphabet is finite, $L_i^C$ is finite for each $i$.  Set $K_i=|L_i^C|$ and $K=|S_\op|\cdot\prod_{i\in\I_\op}K_i$.

Now, for each $\ell,r\in\overline{0,m}$ with $\ell<r$, let $\gamma_{\ell,r}=\bigl((w_\ell',s_\ell'),e_{\ell+1}',\dots,e_r',(w_r',s_r')\bigr)$.  So, the subcomputations of $\gamma_{\ell,r}$ coincide with the subcomputations $\eps(\ell+1),\dots,\eps(r)$ of $\eps$.

Suppose $r-\ell\geq K$.  For each $j\in\overline{\ell+1,r}$, let $e_j'=(\sf_j,\sf_j',e_j)$ where $e_j=(t_j,t_j',\act_j)\in E_\op$.

If $e_j\in E_\op^0$ for some $j$, then $\tm_\op^\F(\gamma_{\ell,r})\geq\tm_\op^\F(\eps(j))=1$.  

Otherwise, suppose $\act_j=(\inst_j,\op_j,\mu_{\act_j}^\ext)\in\ACT_\op$ for each $j$ and let $\I_{\ell,r}=\bigcup_{j\in\overline{\ell+1,r}}\I_{\act_j}$.

For any $i\in\I_\op-\I_{\ell,r}$ and $j\in\overline{\ell+1,r}$, $i\notin\I_{\act_j}$.  Hence, for any $j\in\overline{\ell,r}$, $w_j(i)=w_0(i)$, and so $w_j(\I_\op-\I_{\ell,r})=w_0(\I_\op-\I_{\ell,r})$.

Next, suppose that for each $i\in\I_{\ell,r}$ and $j\in\overline{\ell,r}$, $w_j(i)\in L_i^C$.  Then, for each $j\in\overline{\ell,r}$, $w_j\in\bigl(\bigtimes_{i\in\I_{\ell,r}}L_i^C\bigr)\sqcup w_0(\I_\op-\I_{\ell,r})$.  As each configuration of $\eps$ is distinct, it follows that $r-\ell+1\leq K$, yielding a contradiction.

So, there exists $i\in\I_{\ell,r}$ and $j\in\overline{\ell,r}$ such that $|w_j(i)|\geq C$.  As a result, there exists $k\in\overline{\ell+1,r}$ such that $i\in\I_{\act_k}$.  Assuming that $k>j$, let $p\in\overline{j+1,k}$ be the minimal index such that $i\in\I_{\act_p}$.  Then $w_{p-1}(i)=w_j(i)$, so that $|w_{p-1}(\I_{\act_p})|\geq C$.  In this case, $\tm_\op^\F(\eps(p))\geq f_{\inst_p}(C)$.  Similarly, if $k\leq j$, then letting $q\in\overline{k,j}$ be the maximal index such that $i\in\I_{\act_q}$, then $w_q(i)=w_j(i)$, and so again $\tm_\op^\F(\eps(q))\geq f_{\inst_q}(C)$.

Thus, for $\omega=\min_{\inst\in\INST_\Obj}(f_\inst(C),1)$, $\tm_\op^\F(\gamma_{\ell,r})\geq\omega$ whenever $r-\ell\geq K$.  As a result, letting $y=\lfloor m/K \rfloor$, for any $z\in\overline{y}$, $\tm_\op^\F(\gamma_{(z-1)K,zK})\geq\omega$.  Hence, $\tm_\op^\F(\eps)\geq\omega y\geq\omega(m/K-1)$.

So, $m\leq\omega^{-1}K\tm_\op^\F(\eps)+K$.  Hence, combining this with (\ref{tmF}), there exists a sufficiently large $E\in\N$ such that
$$\tm_\op^{\F'}(\eps)\leq E\tm_\op^\F(\eps)+E\leq Ef'(n)+E$$
This bound and that of (\ref{spF}) then imply that $\TMSP_\Obj^{\F'}\leq(f'',g'')$, where the growth functions $f''$ and $g''$ are given by $f''(n)=Ef'(n)+E$ and $g''(n)=Dg'(n)+D$ for all $n\in\N$.  By (2), this implies that $\TMSP_\Obj\leq(f'',g'')$. 

But $(f'',g'')\sim(f',g')\sim(f,g)$, and thus $\TMSP_\Obj\preceq(f,g)$.

\end{proof}

In practical terms, part (3) of \Cref{prop-complexity-estimate} is useful for calculating an asymptotic upper bound on $\TMSP_\Obj$, as it reduces the process to the calculation of estimated complexities of computations of each operation (rather than its expansion) using a simple estimating set which is an asymptotic upper bound on the time-space complexity of each instance.  The following example demonstrates this use.

\begin{lemma}\label{lem-basic-star-bounds} \

\begin{enumerate}[label=({\alph*})]

\item $\TMSP_{\ch>0}\preceq(n,n)$.

\item $\TMSP_{\textrm{Counter}_1}\preceq(n,n)$.

\end{enumerate}

\end{lemma}

\begin{proof}

(a) Note that $\INST_{\ch>0}$ consists of one instance of $\cpy^\delta$ and one instance of $\dt$.  Simply naming these instances $\cpy$ and $\dt$, respectively, let $\F$ be the estimating set for $\check_{>0}$ given by $\F_\cpy=\F_\dt=(n,n)$.  
By \Cref{lem-basic-0bounds}, $\TMSP_\inst\preceq\F_\inst$ for $\inst\in\INST_{\ch>0}$.  Further, taking $\phi_\inst(n)=\psi_\inst(n)=n$ for all $n\in\N$, it immediately follows that the hypotheses of \Cref{prop-complexity-estimate}(3) are satisfied.  Thus, it suffices to show that $\TMSP_{\ch>0}^\F\preceq(n,n)$.

Fix $n\in\N$ and let $(w_1,w_2)\in\IO_{\ch>0}^{s_1,s_2}$ such that $|w_1|,|w_2|\leq n$.

If $s_1=s_2$, then there exists an empty computation $\eps$ realizing $(w_1,w_2)\in\IO_{\ch>0}^{s_1,s_2}$, so that the trivial bounds $\tm_{\ch>0}^\F(\eps)=0\leq n$ and $\sp_{\ch>0}^\F(\eps)=|w_1|\leq n$ hold.

Now, suppose $s_1=(\check_{>0})_s$ and $s_2=(\check_{>0})_f$.  By \Cref{lem-ndr-checkstar}, there exists $m\in\N-\{0\}$ such that $w_1=\delta^m=w_2$.  Moreover, there exists a computation $\eps_m$ realizing $(w_1,w_2)\in\IO_{\ch>0}^{s_1,s_2}$ such that:
\begin{align*}
\eps_m=\bigl((\delta^m, 1, (\ch_{>0})_s), &e_1', (\delta^m, 1, q_1), e_2' ,(\delta^m, \delta^m, q_2), e_3', (\delta^m, \delta^{\ell_1}, q_2), e_3', \dots  \\
&\dots,e_3', (\delta^m, \delta^{\ell_{r_m-1}}, q_2),e_3',(\delta^m,\delta,q_2),e_4',(\delta^m,1,(\ch_{>0})_f)\bigr) 
\end{align*}
where $e_4'=(\s,\f,e_4)$, $r_m=\lceil \log_2(m) \rceil$, and, setting $\ell_0=m$ and $\ell_{r_m}=1$, $\ell_i= \lceil \ell_{i-1}/2 \rceil$ for all $i\in\overline{r_m}$.

As $e_1,e_4\in E_{\ch>0}^0$ it follows that:

\begin{itemize}

\item $\tm_{\ch>0}^\F(\eps_m(1))=\tm_{\ch>0}^\F(\eps_m(r+3))=1$, 

\item $\sp_{\ch>0}^\F(\eps_m(1))=m$, and 

\item$\sp_{\ch>0}^\F(\eps_m(r+3))=m+1$.

\end{itemize}

Next, for $\act_2=\lab(e_2)$, note that $\I_{\act_2}=\I_{\ch>0}$.  So, $\tm_{\ch>0}^\F(\eps_m(2))=\sp_{\ch>0}^\F(\eps_m(2))=2m$.  Moreover, for $\act_3=\lab(e_3)$, $\I_{\act_3}=\{b\}$.  So, since $\ell_i\leq\ell_{i-1}$ for all $i\in\overline{r_m}$, $\tm_{\ch>0}^\F(\eps_m(j))=\ell_{j-3}$ and $\sp_{\ch>0}^\F(\eps_m(j))=\ell_{j-3}+m$ for all $j\in\overline{3,r_m+2}$. 

Hence, combining these yields:
\begin{align*}\tm_{\ch>0}^\F(\eps_m)&=2m+2+\sum_{i\in\overline{r_m}}\ell_{i-1}\leq 2m+2+\sum_{i\in\overline{r_m}}\left(\frac{m}{2^{i-1}}+1\right)\leq 5n+2 \\
\sp_{\ch>0}^\F(\eps_m)&\leq2m\leq 2n
\end{align*}

Finally, if $s_1=(\check_{>0})_f$ and $s_2=(\check_{>0})_s$, then again by \Cref{lem-ndr-checkstar} there exists an $m\in\N-\{0\}$ such that $w_1=w_2=\delta^m$.  But then $\eps_m^{-1}$ is a computation realizing $(w_1,w_2)\in\IO_{\ch>0}^{s_1,s_2}$ and satisfying the bounds $\tm_{\ch>0}^\F(\eps_m^{-1})\leq 5n+2$ and $\sp_{\ch>0}^\F(\eps_m^{-1})\leq 2n$ as above.

Thus, $\TMSP_{\ch>0}^\F\preceq(n,n)$.

\bigskip

\noindent
(b)  By construction, $\INST_{\Counter_1}$ consists of one instance of $\check_{>0}$.  Naming this instance $\check_{>0}$, it follows from (a) that $\F=(\F_{\ch>0})$ with $\F_{\ch>0}=(n,n)$ is an estimating set for $\Counter_1$ satisfying $\TMSP_\inst\preceq\F_\inst$ for all $\inst$.  Moreover, setting $\phi_\inst(n)=\psi_\inst(n)=n$ for all $n$, it follows that the hypotheses of \Cref{prop-complexity-estimate}(3) are satisfied, and so it suffices to check that $(n,n)$ is an asymptotic upper bound of $\TMSP_{\check}^\F$ and $\TMSP_\inc^\F$.  By construction, it follows immediately that $\TMSP_{\check}^\F\leq(n,n)$.

Now, fix $n\in\N$ and let $(w_1,s_1,w_2,s_2)\in\IO_\inc$ such that $|w_1|,|w_2|\leq n$.  

If $s_1=s_2$, then an empty computation $\eps$ realizes $(w_1,w_2)\in\IO_\inc^{s_1,s_2}$, so that again trivial bounds hold.

Suppose $s_1=\inc_s$ and $s_2=\inc_f$.  Then by \Cref{lem-counter-is-counter}, there exists $j\in\N-\{0\}$ such that $w_1=\delta^{j-1}$ and $w_2=\delta^j$.  Moreover, there exists a computation of $\inc$ realizing $(w_1,w_2)\in\IO_\inc^{s_1,s_2}$ of the form:
$$\eps_j=\bigl((\delta^{j-1},\inc_s),e_1',(\delta^j,q),e_2',(\delta^j,\inc_f)\bigr)$$
where $e_k'=(\s,\f,e_k)$ for $i=1,2$.

As $e_1\in E_\inc^0$, $\tm_\inc^\F(\eps_j(1))=1$ and $\sp_j^\F(\eps_j(1))=j$.  Further, by the definition of the estimating set, $\tm_\inc^\F(\eps_j(2))=j$ and $\sp_j^\F(\eps_j(2))=2j$.  Hence, $\tm_\inc^\F(\eps_j)=j+1\leq n+1$ and $\sp_j^\F(\eps_j)=2j\leq 2n$.

Finally, if $s_1=\inc_f$ and $s_2=\inc_s$, then again by \Cref{lem-counter-is-counter}, there exists an integer $j\in\N-\{0\}$ such that $w_1=\delta^j$ and $w_2=\delta^{j-1}$.  But then $\eps_j^{-1}$ is a computation realizing $(w_1,w_2)\in\IO_\inc^{s_1,s_2}$ and again satisfying $\tm_\inc^\F(\eps_j^{-1})\leq n+1$ and $\sp_j^\F(\eps_j^{-1})\leq2n$.

Thus, $\TMSP_\inc^\F\preceq(n,n)$.

\end{proof}

While \Cref{prop-complexity-estimate}(3) is helpful for arguments as above, note that it is restrictive in the possible pairs of asymptotic upper bounds being studied.  Moreover, the calculation of the estimated complexity of a computation requires one to calculate the size of each of the words comprising the configurations of the computation.  

Such restrictions and complications can be overcome if each element of $\IO_\op$ for $\op\in\Obj$ can be realized by a computation that `does not grow too large'.

\begin{definition}[Linearly bounded objects]

Let $\Obj\in\OBJ^*$.  For any computation $$\eps=\bigl((w_0',s_0'),e_1',(w_1',s_1'),e_2',\dots,e_m',(w_m',s_m')\bigr)$$ of an operation $\op\in\Obj$, define the \emph{reduced space} of $\eps$ as $\sph_\op(\eps)=\max_{j\in\overline{0,m}}|w_j'|$ and the \emph{estimated reduced space} of $\eps$ as $M_\eps=\max(|w_0'|,|w_m'|)$.

Then for $K\in\N$, $\Obj$ is called \emph{linearly $K$-bounded} if for all $\op\in\Obj$ and $(w_1,s_1,w_2,s_2)\in\IO_\op$, there exists a computation $\eps$ of $\op$ realizing $(w_1,w_2)\in\IO_\op^{s_1,s_2}$ such that $\sph_\op(\eps)\leq KM_\eps$.  In this case, $\eps$ is called a \emph{linearly $K$-bounded computation}.

\end{definition}

\begin{definition}[Linearly approximated complexity]

Let $\Obj\in\OBJ^*$ be a linearly $K$-bounded object and $\F=(\F_\inst)_{\inst\in\INST_\Obj}$ be an estimating set for $\Obj$, where $\F_\inst=(f_\inst,g_\inst)$ for each $\inst\in\INST_\Obj$.

Let $\eps=\bigl((w_0,s_0),e_1',(w_1,s_1),e_2',\dots,e_m',(w_m,s_m)\bigr)$ be a computation of some $\op\in\Obj$ with $m\geq1$.  For all $i\in\overline{m}$, let $e_i'=(\sf_i,\sf_i',e_i)$ with $e_i=(s_{i-1}',s_i',\act_i)\in E_\op$.  Then, define the \emph{linearly approximated time} and \emph{linearly approximated space} of $\eps(i)$ with respect to $(\F,K)$, denoted $\tmt_\op^\F(\eps(i),K)$ and $\spt_\op^\F(\eps(i),K)$ respectively, by:

\begin{itemize}

\item If $\act_i=\ACT_0$, then $\tmt_\op^\F(\eps(i),K)=1$ and $\spt_\op^\F(\eps(i),K)=KM_\eps$.

\item If $\act_i\in\ACT_\op$, then for $\act_i=(\inst_i,\op_i,\mu_{\act_i}^\ext)$,
\begin{align*}
\tmt_\op^\F(\eps(i),K)&=f_{\inst_i}(KM_\eps) \\
\spt_\op^\F(\eps(i),K)&=g_{\inst_i}(KM_\eps)+KM_\eps
\end{align*}
\end{itemize}
The \emph{linearly approximated time} and \emph{linearly approximated space} of $\eps$ with respect to $\F$ are then given by:
\begin{itemize}

\item $\tmt_\op^\F(\eps,K)=\sum_{i\in\overline{m}}\tmt_\op^\F(\eps(i),K)$

\item $\spt_\op^\F(\eps,K)=\max_{i\in\overline{m}}\bigl(\spt_\op^\F(\eps(i),K)\bigr)$

\end{itemize}

Finally, if $\eps=(w_0,s_0)$ is an empty computation of $\op$, then, as with the estimated complexity, define $\tmt_\op^\F(\eps,K)=0=\tm_\op(\eps)$ and $\spt_\op^\F(\eps,K)=|w_0|=\sp_\op(\eps)$.

The operation $\op$ is said to have \emph{linearly approximated time-space complexity at most $(f,g)$ with respect to $\F$ and $K$}, denoted $\TMSPT_\op^\F(K)\leq(f,g)$, if $f$ and $g$ are a pair of growth functions such that for all $(w_1,s_1,w_2,s_2)\in\IO_\op$ with $|w_1|,|w_2|\leq n$, there exists a linearly $K$-bounded computation $\eps$ of $\op$ realizing $(w_1,w_2)\in\IO_\op^{s_1,s_2}$ satisfying the bounds $\tmt_\op^\F(\eps,K)\leq f(n)$ and $\spt_\op^\F(\eps,K)\leq g(n)$.  Further, $(f',g')$ is called an \emph{asymptotic upper bound} on $\TMSPT_\op^\F(K)$, denoted $\TMSPT_\op^\F(K)\preceq(f',g')$, if there exists an upper bound $(f,g)$ on $\TMSPT_\op^\F(K)$ with $(f,g)\sim(f',g')$.

Finally, analogous to the definitions in previous settings, $(f,g)$ is called an \emph{upper bound} on $\TMSPT_\Obj^\F(K)$, denoted $\TMSPT_\Obj^\F(K)\leq(f,g)$, if $\TMSPT_\op^\F(K)\leq(f,g)$ for all $\op\in\Obj$.  Similarly, $\TMSPT_\Obj^\F(K)\preceq(f',g')$ if and only if $\TMSPT_\op^\F(K)\preceq(f',g')$ for all $\op\in\Obj$.

\end{definition}

\begin{lemma}\label{lem-K-bounded}

Let $\Obj$ be a linearly $K$-bounded object for some $K\in\N$.  Further, let $\F=(\F_\inst)_{\inst\in\INST_\Obj}$ be an estimating set for $\Obj$, where $\F_\inst=(f_\inst,g_\inst)$ for each $\inst\in\INST_\Obj$.  Suppose that $\TMSP_\inst\defeq\TMSP_{\Obj_\inst}\preceq\F_\inst$ for all $\inst\in\INST_\Obj$.  Then $\TMSPT_\Obj^\F(K)\preceq(f,g)$ implies $\TMSP_\Obj\preceq(f,g)$.

\end{lemma}

\begin{proof}

Let $(f',g')$ be an upper bound on $\TMSPT_\Obj^\F(K)$ such that $(f',g')\sim(f,g)$.  Then, for any $\op\in\Obj$ and $(w_1,s_1,w_2,s_2)\in\IO_\op$ with $|w_1|,|w_2|\leq n$, let 
$$\eps=\bigl((w_0',s_0'),e_1',(w_1',s_1'),e_2',\dots,e_m',(w_m',s_m')\bigr)$$ 
be a linearly $K$-bounded computation of $\op$ realizing $(w_1,w_2)\in\IO_\op^{s_1,s_2}$ and satisfying the bounds $\tmt_\op^\F(\eps,K)\leq f'(n)$ and $\spt_\op^\F(\eps,K)\leq g'(n)$.

For each $i\in\overline{m}$, let $e_i'=(\sf_i,\sf_i',e_i)$ with $e_i=(t_i,t_i',\act_i)\in E_\op$.  Then:
\begin{itemize}

\item If $\act_i\in\ACT_\op^0$, then $\tm_\op^\F(\eps(i))=1=\tmt_\op^\F(\eps(i),K)$ and 
$$\sp_\op^\F(\eps(i))=\max(|w_{i-1}'|,|w_i'|)\leq KM_\eps=\spt_\op^\F(\eps(i),K)$$

\item If $\act_i\in\ACT_\op$, then for $\act_i=(\inst_i,\op_i,\mu_{\act_i}^\ext)$,
\begin{align*}
\tm_\op^\F(\eps(i))&=f_{\inst_i}(\max(|w_{i-1}'(\I_{\act_i})|,|w_i'(\I_{\act_i})|)) \\
&\leq f_{\inst_i}(\max(|w_{i-1}'|,|w_i'|))\\
&\leq f_{\inst_i}(KM_\eps)=\tmt_\op^\F(\eps(i),K)
\end{align*}

\begin{align*}
\sp_\op^\F(\eps(i))&=g_{\inst_i}(\max(|w_{i-1}'(\I_{\act_i})|,|w_i'(\I_{\act_i})|))+|w_i'(\I_\op-\I_{\act_i})| \\
&\leq g_{\inst_i}(\max(|w_{i-1}'|,|w_i'|))+|w_i'| \\
&\leq g_{\inst_i}(KM_\eps)+KM_\eps=\spt_\op^\F(\eps(i),K)
\end{align*}

\end{itemize}

Hence, $\tm_\op^\F(\eps)\leq \tmt_\op^\F(\eps,K)\leq f'(n)$ and $\sp_\op^\F(\eps)\leq \spt_\op^\F(\eps,K)\leq g'(n)$, and so it follows that $\TMSP_\Obj^\F\leq(f',g')$.

Thus, $\TMSP_\Obj\preceq(f,g)$ by \Cref{prop-complexity-estimate}(2).

\end{proof}

\section{Acceptors, Weak Acceptors, and Checkers} \label{sec-acceptors}

While the definitions of the complexity functions of objects and those of $S$-machines have obvious similarities, there are fundamental differences between them.  In particular, the complexity of an object is given by all possible input-output computations of its operations, while that of an $S$-machine is given only by a particular class of its computations which is restricted not only in the possible states of its starting and finishing configurations but also in the tape letters comprising them.
To draw a complete analogy, a particular class of proper operation is defined below.

\begin{definition}[Acceptor]\label{def-acceptor}
	Let $A\subseteq\A^*$ and $L\subseteq F(A)$.  Then an \emph{acceptor of $L$} is a proper operation $\accept_L\in\OBJ^*$ such that $\I_\acc^\ext=\{a\}$, $\A_\acc(a) = A$, and $\IO_\acc$ is given by:
	\begin{align*}
		\IO^{s,s}_\acc &= \Id \cup \{(w_1, w_2)~|~w_1, w_2\in L\};\\
		\IO^{s,f}_\acc &= \bigr\{(w, 1)~|~w\in L\bigr\}; \\
		\IO^{f,f}_\acc &= \Id.
	\end{align*}
	For $w\in L$, a computation $\eps$ of $\accept_L$ realizing $(w,1)\in\IO_\acc^{s,f}$ is called an \emph{accepting computation} of $w$.   Analogous to the terminology of realizing computations, a computation $\widetilde{\eps}_w$ of $\EXP(\accept_L)$ between $(w\sqcup1_\acc^\int\sqcup1_\acc^\inv,\acc_s)$ and $(1\sqcup1_\acc^\int\sqcup1_\acc^\inv,\acc_f)$ is called an $\EXP$-accepting computation of $w$.  Note that in this case, $\widetilde{\eps}_w$ realizes $(w,1)\in\IO_{\EXP(\acc)}^{s,f}=\IO_\acc^{s,f}$.

\end{definition}

\begin{proposition}\label{prop-machine-acc} For any acceptor $\accept_L(a)\in\OBJ^*$, there exists a recognizing $S$-machine $\S$ with input tape $i^*$ such that $Y_{i^*}=\A_\acc(a)$, $L_\S=L$, $\TM_\S\sim\TM_\acc$, and $\SP_\S\sim_1\SP_\acc$.  Moreover, for growth functions $f$ and $g$, $\TMSP_\S\preceq_1(f,g)$ if and only if $\TMSP_\acc\preceq_1(f,g)$.

\end{proposition}

\begin{proof}

It follows from \Cref{th-expansion-props}(g) that $\IO_{\EXP(\acc)}=\IO_\acc$, and so $\EXP(\accept_L)$ is also an acceptor of $L$.  Hence, since the complexity functions of proper operations are defined in terms of their expansions, it may be assumed without loss of generality that $\accept_L\in\OBJ_0$.  Hence, accepting computations and $\EXP$-accepting computations are equivalent in $\accept_L$.

Now, let $\G_\acc$ be the $S$-graph of degree 0 corresponding to $\accept_L$ and fix an enumeration $\eta:\I_\acc\to\overline{N}$.  Then, let $\S=\S(\G_\acc,\eta)$ be the $S$-machine over $\H_\S=\S_{\G_\acc,\eta}(\H_\acc)$ constructed in \Cref{th-MG}(b).

Let $i^*=\eta(a)$, $q_1^s=\acc_s$, and $q_1^f=\acc_f$.  Then, letting $q_j^s=q_j^f=q_j^*$ for all $j\in\overline{2,N+1}$, define $q^s,q^f\in S_{\H_\S}$ as the states $q^s=(q_i^s)_{i\in\overline{N+1}}$ and $q^f=(q_i^f)_{i\in\overline{N+1}}$.  

In accordance with the notation, the machine $\S$ is identified with the recognizing $S$-machine $(\S,i^*,q^s,q^f)$.  Note that $Y_{i^*}=\A_\acc(a)$ by construction.

Suppose $w\in L_\S$ and let $\eps=(w_0,\tau_1,w_1,\dots,\tau_n,w_n)$ be any computation of $\S$ between $I(w)$ and $W_{ac}$.  Then, by part (ii) of \Cref{th-MG}(b), there exists a computation of $\accept_L$
$$\psi_{\acc,\eta}(\eps)=\bigl(\mu_{\acc,\eta}(w_0),\phi(\tau_1),\mu_{\acc,\eta}(w_1),\dots,\phi(\tau_n),\mu_{\acc,\eta}(w_n)\bigr)$$ 
For all $j\in\overline{0,n}$, letting $w_j=q_1w_1'q_2\dots w_N'q_{N+1}$, by construction $\mu_{\acc,\eta}(w_j)=\bigl((w_{\eta(i)}')_{i\in\I_{\acc}},q_1\bigr)$.  In particular, $\mu_{\acc,\eta}(w_0)=(w\sqcup1_\acc^\int,\acc_s)$, $\mu_{\acc,\eta}(w_n)=(1,\acc_f)$, and $|\mu_{\acc,\eta}(w_j)|=|w_j|-(N+1)$ for each $j$.  Hence, $w\in L$ and $\psi_{\acc,\eta}(\eps)$ is an input-output computation of $\accept_L$ realizing $(w,1)\in\IO_\acc^{s,f}$ that satisfies: 

\begin{itemize}

\item $\tm_\acc(\psi_{\acc,\eta}(\eps))=\tm(\eps)$ 

\item $\sp_\acc(\psi_{\acc,\eta}(\eps))=\sp(\eps)-(N+1)$

\end{itemize}

Conversely, let $w\in L$ and $\eps_w=\bigl((w_0,s_0),e_1,(w_1,s_1),\dots,e_n,(w_n,s_n)\bigr)$ be an input-output computation of $\accept_L$ realizing $(w,1)\in\IO_\acc^{s,f}$.  Then, by part (iii) of \Cref{th-MG}(b), there exists a computation 
$$\psi_\S(\eps_w)=\bigl((\mu_{\acc,\eta}^{-1}(w_0,s_0)),\tau_1,\mu_{\acc,\eta}^{-1}(w_1,s_1),\dots,\tau_n,\mu_{\acc,\eta}^{-1}(w_n,s_n)\bigr)$$
of $\S$, where $\tau_i=\phi^{-1}(e_i)$ for all $i$.  As $(w_0,s_0)=(w\sqcup1_\acc^\int,\acc_s)$ and $(w_n,s_n)=(1,\acc_f)$, by construction $\mu_{\acc,\eta}^{-1}(w_0,s_0)=I(w)$ and $\mu_{\acc,\eta}^{-1}(w_n,s_n)=W_{ac}$.  In particular, $\psi_\S(\eps_w)$ is an accepting computation of $w$, so that $w\in L_\S$.  Moreover, $|\mu_{\acc,\eta}^{-1}(w_j,s_j)|=|w_j|+N+1$ for all $j$.  Hence, $\tm(\psi_\S(\eps_w))=\tm_\acc(\eps_w)$ and $\sp(\psi_\S(\eps_w))=\sp_\acc(\eps_w)+N+1$.

Thus, $L_\S=L$.

Now fix $n\in\N$ and $w\in L_\S$ such that $|w|\leq n$.  Then $w\in L$, so that $(w,1)\in\IO_\acc^{s,f}$.

As a result, there exist input-output computations $\eps_w^\tm,\eps_w^\sp$ of $\accept_L$ realizing $(w,1)\in\IO_\acc^{s,f}$ such that $\tm_\acc(\eps_w^\tm)\leq\TM_\acc(n)$ and $\sp_\acc(\eps_w^\sp)\leq\SP_\acc(n)$.  But then as above $\psi_\S(\eps_w^\tm)$ and $\psi_\S(\eps_w^\sp)$ are computations of $\S$ accepting $w$ and satisfying the bounds $\tm(\psi_\S(\eps_w^\tm))\leq\TM_\acc(n)$ and $\sp(\psi_\S(\eps_w^\sp))\leq\SP_\acc(n)+N+1$.  Hence, $\TM_\S(n)\leq\TM_\acc(n)$ and $\SP_\S(n)\leq\SP_\acc(n)+N+1$, i.e $\TM_\S\leq\TM_\acc$ and $\SP_\S\preceq\SP_\acc$.

Further, if $\TMSP_\acc\preceq(f,g)$, then for growth functions $f'$ and $g'$ satisfying $(f,g)\sim(f',g')$ and $\TMSP_\acc\leq(f',g')$, there exists a computation $\eps_w'$ of $\accept_L$ realizing $(w,1)\in\IO_\acc^{s,f}$ such that $\tm_\acc(\eps_w')\leq f'(n)$ and $\sp_\acc(\eps_w')\leq g'(n)$.  But then $\psi_\S(\eps_w')$ is a computation of $\S$ accepting $w$ and satisfying $\tm(\psi_\S(\eps_w'))\leq f'(n)$ and $\sp(\psi_\S(\eps_w'))\leq g'(n)+N+1$.  So, for $g''(n)=g'(n)+N+1$, $\TMSP_\S\leq(f',g'')$.  But $g''\sim g'$, so that $\TMSP_\S\preceq(f,g)$.

Conversely, let $(w_1,s_1,w_2,s_2)\in\IO_\acc$ such that $|w_1|,|w_2|\leq n$.

Suppose $s_1\neq s_2$ and fix $i\in\{1,2\}$ such that $s_i=\acc_s$.  Then, $w_i\in L=L_\S$, so that there exist computations $\eps^\tm$ and $\eps^\sp$ of $\S$ accepting $w_i$ and satisfying the bounds $\tm(\eps^\tm)\leq\TM_\S(n)$ and $\sp(\eps^\sp)\leq\SP_\S(n)$.  As above, there then exist computations $\psi_{\acc,\eta}(\eps^\tm)$ and $\psi_{\acc,\eta}(\eps^\sp)$ between $(w_i,\acc_s)$ and $(1,\acc_f)$ and such that $\tm_\acc(\psi_{\acc,\eta}(\eps^\tm))\leq\TM_\S(n)$ and $\sp_\acc(\psi_{\acc,\eta}(\eps^\sp))\leq\SP_\S(n)-(N+1)$.  If $i=1$ then these computations realize $(w_1,w_2)\in\IO_\acc^{s,f}$; if $i=2$, then their inverses realize $(w_1,w_2)\in\IO_\acc^{s,f}$ and satisfy the same bounds.

Otherwise, suppose $s_1=s_2$.  If $w_1\neq w_2$, then $w_1,w_2\in L$ and $s_1=s_2=\acc_s$.  For $i\in\{1,2\}$, let $\eps_i^\tm$ and $\eps_i^\sp$ be computations of $\S$ accepting $w_i$ and satisfying $\tm(\eps_i^\tm)\leq\TM_\S(n)$ and $\sp(\eps_i^\sp)\leq\SP_\S(n)$.  Then, concatenation yields computations $\delta^\tm=\psi_{\acc,\eta}(\eps_1^\tm)\psi_{\acc,\eta}(\eps_2^\tm)^{-1}$ and $\delta^\sp=\psi_{\acc,\eta}(\eps_1^\sp)\psi_{\acc,\eta}(\eps_2^\sp)^{-1}$ realizing $(w_1,w_2)\in\IO_\acc^{s_1,s_2}$ and satisfying $\tm_\acc(\delta^\tm)\leq2\TM_\S(n)$ and $\sp_\acc(\delta^\sp)\leq\SP_\S(n)-(N+1)$.

Finally, if $(w_1,s_1)=(w_2,s_2)$, then an empty computation $\eps$ realizes $(w_1,w_2)\in\IO_\acc^{s_1,s_2}$ and satisfies $\tm(\eps)=0\leq\TM_\S(n)$ and $\sp(\eps)\leq n$.

Hence, $\TM_\acc\preceq\TM_\S$ and $\SP_\acc\preceq_1\SP_S$.

In much the same way, $\TMSP_\S\preceq(f,g)$ implies $\TMSP_\acc\preceq_1(f,g)$.

\end{proof}

Note that \Cref{prop-machine-acc} only establishes a correspondence between the complexities of accepting languages in one direction: For any acceptor, there exists a corresponding $S$-machine satisfying similar complexity bounds.  A converse to this is possible and can be proved in much the same way as \Cref{prop-machine-acc} (but using part (a) of \Cref{th-MG} in place of (b)).  However, such a converse is not necessary for the purposes of this manuscript, and so is omitted. 

The following example illustrates the necessity of using the relation $\sim_1$ in \Cref{prop-machine-acc} instead of simply using the stronger relation $\sim$:

\begin{example}\label{ex-bad-complexities}

Let $\acc=\accept_L$ be an acceptor of degree 0, where $L$ is some infinite language over an alphabet $A$ such that:
\begin{enumerate}
	\item $\left\{|w|~|~w\in L\right\} = P = \left\{2^{2^i}~|~i\in\N\right\}$;
	\item There exist positive constants $C,K$ such that for each $w\in L$, there is a computation $\eps_w$ realizing $(w,1)\in\IO_\acc^{s,f}$ and satisfying 
	$$C|w|\leq\tm_\acc(\varepsilon_w),\sp_\acc(\eps_w)\leq K|w|$$
\end{enumerate}

In view of condition (2), it is tempting to assume that $\TM_\acc \sim \SP_\acc$.  However, this is not true:

Let $\id_P$ be the growth function given by $\id_P(n)=\max\{p\in P\mid p\leq n\}$.  Note that $\id_P(n)\leq n$ for all $n\in\N$.  

Suppose there exists a positive constant $L$ such that $n\leq L\id_P(n)$ for all $n\in\N$.  Then, for any $m\in\N$, $2^{2^m}-1\leq L\id_P(2^{2^m}-1)=L\cdot 2^{2^{m-1}}$.  Since $2^{2^m}\geq2$ for each $m\in\N$, $2^{2^m}-1\geq2^{2^m-1}$.  But then $L\geq2^{2^{m-1}-1}$ for all $m\in\N$, which is impossible for fixed $L$.  Hence, $\id_P\precneq n$.

Next, note that for any $n\in\N$ and $w\in L$, $|w|\leq n$ implies that $|w|\leq\id_P(n)$.  So, for any $w\in L$ with $|w|\leq n$, condition (2) can be rewritten as:
$$C~\id_P(n)\leq\tm_\acc(\eps_w),\sp_\acc(\eps_w)\leq K~\id_P(n)$$

Hence, it follows that for any essential $(w_1,s_1,w_2,s_2)\in\IO_\acc$ with $|w_1|,|w_2|\leq n$, there exists a computation realizing $(w_1,w_2)\in\IO_\acc^{s_1,s_2}$ with time and space bounded above and below by linear functions of $\id_P(n)$.  The symmetry of $\IO_\acc$ further implies similar bounds for all $(w_1,\acc_f,w_2,\acc_s)\in\IO_\acc$.

However, for `trivial' quadruples, i.e., those of the form $(w_1,s_1,w_1,s_1)\in\IO_\acc$, any computation $\eps$ realizing $(w_1,w_1)\in\IO_\acc^{s_1,s_1}$ has $\sp_\acc(\eps)\geq|w_1|$.  As a result, $n\preceq\SP_\acc$ (indeed, for any operation with a nonempty alphabet).  Of course, though, such a relation is realized by an empty computation with time 0.  So, $\TM_\acc\sim\id_P$ and $\SP_\acc\sim n$.

Meanwhile, for the recognizing $S$-machine $\S$ constructed from $\acc$ as in \Cref{prop-machine-acc}, the computations determining the time and space function are simply those going `from start to end.'  As such, any consideration of the trivial computations is moot in this case, so that condition (2) assures that $\TM_\S\sim\SP_\S\sim\id_P$.

While this example shows some pathological behavior, it should be noted that such discussion proves rather inessential, as the equivalence $\sim_1$ is adequate for most purposes studied.  

It should be noted, though, that one can remove this consideration by removing the `trivial' quadruples from the definition of the $\IO$-relation.  Such an alteration would not affect the key results of \Cref{th-expansion-props}, but would add some complications that potentially outweigh any benefits.

\end{example}

It is convenient in what follows to slightly alter the definition of acceptors, producing proper operations that function in essentially the same manner and with the same complexity.  These come in the form of \emph{checkers} and \emph{weak acceptors}.
\begin{definition}\label{def-acceptor2}
Let $A$ be an alphabet and $L\subseteq F(A)$. 

A proper operation $\check_L(a)$ is called a \emph{checker of $L$} if $\A_\ch(a) = A$ and its $\IO$-relation $\IO_{\ch}$ is given by:
	\begin{align*}
		\IO^{s,s}_\ch &= \IO^{f,f}_\check = \Id;\\
		\IO^{s,f}_\ch &= \bigr\{(w, w)~|~w\in L\bigr\}.
	\end{align*}

A proper operation $\waccept_L(a)$ is called a \emph{weak acceptor of $L$} if $\A_\ch(a)=A$ and its $\IO$-relation $\IO_{\wacc}$ satisfies:
$$\IO_\wacc^{s,f}\cap\{(w,1)\mid w\in F(A)\} = \{(w,1)\mid w\in L\}$$
	Hence, an acceptor is a weak acceptor. 
	
Given that $L$ is nonempty, note that the operation $\wacc$ is a weak acceptor of $L$ if and only if $\IO_\wacc^{s,f}\supseteq\{(w,1)\mid w\in L\}$ and $\IO_\wacc^{s,s}\subseteq L^2\cup(F(A)-L)^2$, i.e $(w_1,w_2)\in\IO_\wacc^{s,s}$ implies $w_1,w_2\in L$ or $w_1,w_2\notin L$.
	
\end{definition}

Note that the term `checker' is consistent with that of previous sections.  For example, per \Cref{lem-ndr-checkstar}, the proper operation $\check_{>0}$ is a checker of the language $L_{>0} = \{\delta^i~|~i>0\}$ over the alphabet $\{\delta\}$.

While the time-space complexity of an operation depends on all quadruples defining its input-output quadruples, many such quadruples will prove to be in some sense irrelevant in the study of weak acceptors (see \Cref{prop-acc-ch-wacc}).  So, to simplify many of the arguments to come, a loosened version of the complexity of a weak acceptor is considered.

\begin{definition}[Weak-time-space complexity]

Let $\wacc=\waccept_L(a)$ be a weak acceptor of a language $L$ and $f,g$ be growth functions.  

Suppose that for all $w\in L$, there exists a computation $\eps$ of $\EXP(\wacc)$ which $\EXP$-realizes $(w,1)\in\IO_\wacc^{s,f}$ and satisfies $\tm_\wacc(\eps)\leq f(|w|)$ and $\sp_\wacc(\eps)\leq g(|w|)$.  Then $(f,g)$ is called an \emph{upper bound on the weak-time-space complexity of $\wacc$}, denoted $\WTMSP_\wacc\leq(f,g)$.  Naturally, if $f'$ and $g'$ are growth functions satisfying $(f',g')\sim(f,g)$, then $(f',g')$ is called an \emph{asymptotic upper bound on the weak-time-space complexity of $\wacc$}, denoted $\WTMSP_\wacc\preceq(f',g')$.

Further, let $\F$ be an estimating set for $\wacc$ and suppose there exists $K\in\N$ such that for any $w\in L$ there exists a computation $\eps$ of $\wacc$ realizing $(w,1)\in\IO_\wacc^{s,f}$ and satisfying the bounds $\tm_\wacc^\F(\eps)\leq Kf(K|w|)+K$ and $\sp_\wacc^\F(\eps)\leq Kg(K|w|)+K$.  Then $(f,g)$ is called an \emph{upper bound on the estimated weak-time-space complexity of $\wacc$}, denoted $\WTMSP_\wacc^\F\leq(f,g)$.  Similarly, in this case $(f',g')$ is an \emph{asymptotic upper bound on the weak-time-space complexity of $\wacc$} if $(f,g)\sim(f',g')$, denoted $\WTMSP_\wacc^\F\preceq(f',g')$.

\end{definition}

The following analogue of \Cref{prop-complexity-estimate}(3) can be proved in much the same way, using \Cref{prop-complexity-estimate}(1) for the relevant computations:

\begin{lemma}\label{lem-wtmsp}

Let $\F=(\F_\inst)_{\inst\in\INST_\wacc}$ be an estimating set for a weak acceptor $\wacc$ of a language $L$.  For all $\inst\in\INST_\wacc$, let $\inst=(\Obj_\inst,\mu_\inst^\val)$ and $\F_\inst=(f_\inst,g_\inst)$.  Suppose $\TMSP_\inst\preceq\F_\inst$ and $f_\inst\neq0$ for all $\inst\in\INST_\wacc$.  Moreover, suppose there exist growth functions $\phi_\inst$ and $\psi_\inst$ such that $\psi_\inst\neq0$, $f_\inst(cn)\leq\phi_\inst(c)f_\inst(n)$, and $g_\inst(cn)\leq\psi_\inst(c)g_\inst(n)$ for all $c,n\in\N$.  Then  $\WTMSP_\wacc^\F\preceq(f,g)$ implies $\WTMSP_\wacc\preceq(f,g)$.

\end{lemma}

\begin{proposition}\label{prop-acc-ch-wacc}
	For growth functions $f(n), g(n) \succeq n$ and a language $L$ over the alphabet $A$, the following are equivalent:
	\begin{enumerate}
		\item There is an acceptor $\acc = \accept_L$ for $L$ 
		with $\TMSP_\acc \preceq (f, g)$;
		\item There is a weak acceptor $\wacc = \waccept_L$ for $L$ with $\WTMSP_\wacc \preceq (f, g)$;
		\item There is a checker $\ch = \check_L$ for $L$ with $\TMSP_\ch \preceq (f, g)$.
	\end{enumerate}
\end{proposition}
\begin{proof} Because an acceptor is a weak acceptor, the existence of (1) implies the existence of a weak acceptor such that $(f,g)$ is an asymptotic upper bound on its time-space complexity.  Then, letting $(f',g')$ be an upper bound on its time-space complexity such that $(f',g')\sim(f,g)$, it is immediate from the definition that $(f',g')$ is an upper bound on the weak-time-space complexity of the weak acceptor.  Hence, (1) implies (2).

Given the existence of a weak acceptor $\waccept_L$ satisfying $\WTMSP_\wacc\preceq(f,g)$, \Cref{fig-ACC-CH-WACC} shows how a checker $\check_L$ can be constructed from $\waccept_L$, and how in turn an acceptor $\accept_L$ can be constructed from $\check_L$.  It then suffices to prove the $\IO$-relations and show that $\TMSP_\ch,\TMSP_\acc\preceq(f,g)$.  

First, $\I_\op^\int=\{b\}$, $\I_\op^\ext=\{a\}$, and $\A_\op(a)=\A_\op(b)=A$ for any $\op\in\{\check_L,\accept_L\}$.  So, since the only edges with tails in $\SF_\ch$ or $\SF_\acc$ contain the guard fragment $[b=1]$, both operations satisfy condition (O).

\begin{figure}[hbt]
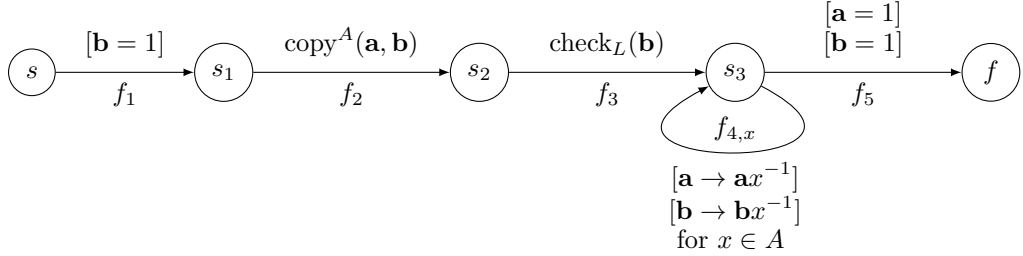

	\centering
	\include{ACC-CH-WACC}
	\caption{Constructing a checker from a weak acceptor, and an acceptor from a checker.}
	\label{fig-ACC-CH-WACC}       
\end{figure} 

For $w\in L$, define the subset $C_w\subseteq\GL_\ch\times S_\ch$ by:
\begin{align*}
C_w^s=C_w^{q_1}=C_w^f&=\{(w,1)\} \\
C_w^{q_2}&=\{(w,v)\mid v\in L\} \\
C_w^{q_3}&=\{(w,u)\mid (u,1)\in \IO_\wacc^{f,f}\}
\end{align*}
It is straightforward to check that $C_w$ is closed under $\IO_e$ for each $e\in E_\ch$.  For example, suppose $(w_1,s_1,w_2,s_2)\in\IO_{e_3}$ with $(w_1,s_1)\in C_w$.  If $s_1=s_2=q_2$, then $w_1=(w,v)$ for some $v\in L$.  Then $w_2=(w,v')$ for some $v'\in F(A)$ such that $(v,v')\in \IO_\wacc^{s,s}$; but then $v'\in L$ by the definition of a weak acceptor.  All other choices for $s_1,s_2$ follow similarly.

Further, for $w\notin L$, define the subsets $C_{1,w},C_{2,w}\subseteq\GL_\ch\times S_\ch$ by:

\begin{align*}
C_{1,w}^s=C_{1,w}^{q_1}=C_{2,w}^f&=\{(w,1)\} \\
C_{1,w}^{q_2}&=\{(w,v)\mid v\notin L\} \\
C_{1,w}^{q_3}&=\{(w,u)\mid (u,1)\notin\IO_\wacc^{f,f}\} \\
C_{1,w}^f=C_{2,w}^s=C_{2,w}^{q_1}&=\emptyset \\
C_{2,w}^{q_2}&=\{(w,v)\mid v\in L\} \\
C_{2,w}^{q_3}&=\{(w,u)\mid (u,1)\in\IO_\wacc^{f,f}\}
\end{align*}
Again, it is straightforward to check that each of $C_{1,w}$ and $C_{2,w}$ is closed under $\IO_e$ for all $e\in E_\ch$.  Hence, it follows from \Cref{prop-io-classes} that $\IO_\ch^{s,s},\IO_\ch^{f,f}\subseteq \Id$ and $\IO_\ch^{s,f}\subseteq\{(w,w)\mid w\in L\}$, and thus $\check_L$ is indeed a checker of $L$.

Of course, for any $w\in F(A)$ with $|w|=n$, $(w,w)\in\IO_\ch^{s,s}$ and $(w,w)\in\IO_\ch^{f,f}$ for all $w\in F(A)$, where either can be realized by an empty computation $\eps$ satisfying $\tm_\ch(\eps)=0\leq f(n)$ and $\sp_\ch(\eps)=|w|\leq g(n)$.

Now, by Lemmas  \ref{lem-io-cpy} and \ref{lem-basic-0bounds} and the assumption that $\WTMSP_\wacc\preceq(f,g)$, there exists $K\in\N$ such that for any $w\in L$, there exist computations $\gamma_w$ and $\delta_w$ of $\EXP(\check_L)$ such that:

\begin{itemize}

\item $\gamma_w$ is a computation between the configurations $((w,1)\sqcup1_\ch^\inv,q_1)$ and $((w,w)\sqcup1_\ch^\inv,q_2)$ and satisfies $\tm_\ch(\gamma_w)\leq K|w|+K$ and $\sp_\ch(\gamma_w)\leq K|w|+K$

\item $\delta_w$ is a computation between the configurations $((w,w)\sqcup1_\ch^\inv,q_2)$ and $((w,1)\sqcup1_\ch^\inv,q_3)$ and satisfies $\tm_\ch(\delta_w)\leq Kf(K|w|)+K$ and $\sp_\ch(\delta_w)\leq Kg(K|w|)+K+|w|$.

\end{itemize}

Then, if $\rho$ is the path supporting $\gamma_w$ and $\rho'$ be the path supporting $\delta_w$, the path $e_1,\rho,\rho',e_4$ supports a computation $\eps_w$ of $\EXP(\check_L)$ realizing $(w,w)\in\IO_\ch^{s,f}$ and satisfying:
\begin{align*}
\tm_\ch(\eps_w)&=1+\tm_\ch(\gamma_w)+\tm_\ch(\delta_w)+1\leq Kf(K|w|)+K|w|+2K+2 \\
\sp_\ch(\eps_w)&=\max(|w|,\sp_\ch(\gamma_w),\sp_\ch(\delta_w))\leq\max(K|w|+K,Kg(K|w|)+K+|w|)
\end{align*}
Thus, as $f(n), g(n) \succeq n$, $\TMSP_\ch\preceq(f,g)$.

Now, for $w\in L$, define the subset $D_w\subseteq\GL_\acc\times S_\acc$ by:
\begin{align*}
D_w^s=D_w^{s_1}&=\{(w,1)\} \\
D_w^{s_2}&=\{(w,w)\} \\
D_w^{s_3}&=\{(v,v)\mid v\in L\} \\
D_w^f&=\{(1,1)\}
\end{align*}

Further, for $w\notin L$, define the subsets $D_{1,w},D_{2,w}\subseteq\GL_\acc\times S_\acc$ by:
\begin{align*}
D_{1,w}^s=D_{2,w}^f=D_{1,w}^{s_1}&=\{(w,1)\} \\
D_{1,w}^{s_2}&=\{(w,w)\} \\
D_{1,w}^{s_3}=D_{1,w}^f=D_{2,w}^{s_j}=D_{2,w}^s&=\emptyset
\end{align*}
As in the previous case, it is straightforward to check that these subsets are closed under $\IO_e$ for all $e\in E_\acc$, and so $\accept_L$ is indeed an acceptor of $L$  by \Cref{prop-io-classes}.

Let $\F$ be the estimating set for $\accept_L$ given by $\F_\cpy=(n,n)$ and $\F_\ch=(f,g)$.  By \Cref{lem-basic-0bounds}, $\TMSP_\inst\preceq\F_\inst$ for both $\inst\in\INST_\acc$.  However, note that \Cref{prop-complexity-estimate}(3) cannot be applied to this situation, as the existence of suitable growth functions $\phi_\ch$ and $\psi_\ch$ is a priori unknown.

Let $w\in L$ and set $w\equiv x_1^{\delta_1}\dots x_k^{\delta_k}$, where $x_i\in A$ and $\delta_i\in\{\pm1\}$ for $k\in\N$.  Then, let $\rho_w'=e_1',\dots,e_k'$ be the generic path given by setting $e_i'=(\s,\f,f_{4,x_i}^{\delta_i})$ for all $i$.  It then follows that $\rho_w'$ supports a computation $\gamma_w$ of $\accept_L$ between $((w,w),s_3)$ and $((1,1),s_3)$ satisfying $\tm_\acc^\F(\gamma_w)=k=|w|$ and $\sp_\acc^\F(\gamma_w)=2k=2|w|$.

So, letting $f_i'=(\s,\f,f_i)$ for all $i$, as for all $w\in L$ there exists a computation $$\delta_w=\bigl(((w,1),\acc_s),f_1',((w,1),s_1),f_2',((w,w),s_2),f_3',((w,w),s_3)\bigr)$$ and there exists a computation $\delta'=\bigl(((1,1),s_3),f_5',((1,1),\acc_f)\bigr)$ of $\accept_L$, there exists a computation $\eps_w$ realizing $(w,1)\in\IO_\acc^{s,f}$ formed by concatenating $\delta_w$, $\gamma_w$, and $\delta'$.  By the construction of $\gamma_w$, it follows that $\eps_w$ is a linearly 2-bounded computation, and so $\accept_L$ is a linearly 2-bounded proper operation.

But for all $w\in L$, 
\begin{align*}
\tmt_\acc^\F(\eps_w,2)&=1+2|w|+f(2|w|)+|w|+1\leq4f(2|w|)+1 \\
\spt_\acc(\eps_w,2)&=\max(4|w|,g(2|w|)+2|w|)
\end{align*}
Hence, since $f,g\succeq n$, it follows that $\TMSPT_\acc^\F(2)\preceq(f,g)$.  Thus, by \Cref{lem-K-bounded}, $\TMSP_\acc\preceq(f,g)$.

\end{proof}

Thus, \Cref{prop-machine-acc} and \Cref{prop-acc-ch-wacc} reduce the proof of \Cref{main-theorem} to constructing one of $\accept_L$, $\waccept_L$ or $\check_L$ for a given language $L$ satisfying certain complexity (or weak complexity) bounds.

\section{Improved counters}\label{sec-counters}

In this section, a family of counters objects is constructed, which will prove to be more `economical' than the basic counter $\Counter_1$ defined in \Cref{ex-counter1}.  Indeed, these counters are the main innovation of this manuscript and provide much of the motivation for the definition of $S^*$-graphs.

First, two more auxiliary proper operations, $\le(a,b)$ and $\check_{\geq0}(a)$, are defined in \Cref{fig-LESSEQ} below.  In accordance with the conventions of minimal compatible hardware, $\I_{\ch\geq0}=\I_{\ch\geq0}^\ext=\{a\}$, $\I_\le=\I_\le^\ext\sqcup\I_\le^\int$, $\I_\le^\ext=\{a,b\}$, $\I_\le^\int=\{a',b'\}$, and $\A_\le(i) = \{\delta\}$ for all $i\in\I_\le$.

Note that the only edge with tail $\le_s$ contains the guard fragments $[a'=1]$ and $[b'=1]$, while the only edge with tail $\le_f$ contains the guard fragment $[b'=1]$.  Further, as $\G_\le-\{e_5^{\pm1}\}$ consists of two connected components, with the one containing $\le_f$ not containing any edge mentioning $a'$, it follows that any path starting with $\le_f$ and containing an edge that mentions $a'$ must have the first such edge be $e_5^{-1}$.  But $e_5$ contains the guard fragment $[a'=1]$.  Hence, $\le$ satisfies condition (O).

\begin{figure}[hbt]
	\centering
	\include{LESSEQ}
	\caption{Operations $\le(a, b)$, and $\check_{\geq0}(a)$.}
	\label{fig-LESSEQ}       
\end{figure} 

\begin{lemma}\label{lem-ndr-le}
	The relations $\IO_{\le}(a,b)$ and $\IO_{\ch\geq0}(a)$ are given by:
	\begin{align*}
		\IO^{s,s}_\le &= \IO^{f,f}_\le = \Id \\
		\IO^{s,s}_{\ch\geq0} &= \IO^{f,f}_{\ch\geq0} = \Id \\
		\IO^{s,f}_\le &= \bigl\{\bigl((\delta^{i}, \delta^{j}), (\delta^{i}, \delta^{j})\bigr)\mid i < j\bigr\} \\
		\IO^{s,f}_{\ch\geq0} &= \bigl\{(\delta^{i}, \delta^{i})~|~i \geq 0\bigr\}
	\end{align*}
	In particular, $a$ and $b$ are immutable in $\le (a,b)$, and $a$ is immutable in $\check_{\geq0} (a)$.
	
	Moreover, $\TMSP_\le,\TMSP_{\ch\geq0}\preceq(n,n)$.
\end{lemma}
\begin{proof} 

The proof of the statement concerning $\check_{\geq0}$ is straightforward.

In $\le$, the external tapes $a$ and $b$ are mentioned only by $e_2$ and $e_3$, respectively.  By \Cref{lem-io-cpy}, these edges mention the tapes as immutable variables.  Hence, by \Cref{lem-immutable}, both $a$ and $b$ are immutable in $\le$. Thus, $\IO^{s,s}_\le, \IO^{f,f}_\le = \Id$, and $\IO_\le^{s,f} \subseteq \Id$.

For all $i,j\in\Z$ such that $i\geq j$, define the subset $C_{i,j}\subseteq\GL_{\ch\geq0}\times S_{\ch\geq0}$ by:
\begin{align*}
	C_{i,j}^s = C^{q_1} &= \emptyset \\
	C_{i,j}^{q_2} &= \{(\delta^i, \delta^j, \delta^k ,1)~|~k< j\} \\
	C_{i,j}^{q_3} &= \{(\delta^i, \delta^j, \delta^k ,\delta^l)~|~k< l\} \\
			&=  \{(\delta^i, \delta^j, \delta^k ,\delta^{k+q})~|~k\in \Z, q\geq 1\} \\	   
	C_{i,j}^{q_4} &= \{(\delta^i, \delta^j, 1,\delta^q)~|~q\geq 1\} \\
	C_{i,j}^f &= \{(\delta^i, \delta^j, 1,1)\}
\end{align*}
It is straightforward to check that $C_{i,j}$ is closed under $\IO_{e_\ell}$ for all $\ell\in\overline{7}$, and so is closed under $\IO_e$ for all $e\in E_{\ch\geq0}$.  As a result, $(\delta^i,\delta^j)\notin\IO_{\ch\geq0}^{s,f}$.

For $i,j\in\Z$ with $i<j$, consider the generic path $\rho_{i,j}=e_1',e_2',e_3',(e_4')^i,e_5',(e_6')^k,e_7'$, where:
\begin{itemize}

\item $e_\ell'=(\s,\f,e_\ell)$ for $\ell\in\overline{7}$,

\item For $m\in\N$ and any generic edge $e'$, $(e')^m$ represents the generic subpath $\underbrace{e',\dots,e'}_{m\text{ times}}$,

\item For $m\in\N$, $(e_\ell')^{-m}$ represents the subpath $(e_\ell'')^m$, where $e_\ell''=(\s,\f,e_\ell^{-1})$.

\item $k\geq\log_2(j-i)$.

\end{itemize}
So, for $w_{i,j}=(\delta^i,\delta^j)$, as in the proof of \Cref{lem-ndr-checkstar}, $\rho_{i,j}$ supports an input-output computation $\eps_{i,j}$ realizing $(w_{i,j},w_{i,j})\in\IO_\le^{s,f}$.

Hence, $\IO_\le^{s,f}=\bigl\{\bigl((\delta^i,\delta^j),(\delta^i,\delta^j)\bigr)\mid i<j\bigr\}$.

Now, let $\F$ be an estimating set for the operation $\le(a,b)$ given by $\F_\cpy=\F_\dt=(n,n)$.  By \Cref{lem-basic-0bounds}, $\TMSP_\inst\preceq\F_\inst$ for all $\inst\in\INST_\le$.

Then, similar to the proof of \Cref{lem-basic-star-bounds}(1), it follows from the construction of $\eps_{i,j}$ that for all $i,j\in\Z$ satisfying $i<j$:
\begin{align*}
\tm^\F(\eps_{i,j})&=1+|i|+|j|+|i|+1+3(j-i)+1\leq5(|i|+|j|)+3 \\
\sp^\F(\eps_{i,j})&=2(|i|+|j|)
\end{align*}
So, since $|w_{i,j}|=|i|+|j|$, it follows that $\TMSP_\le^\F\preceq(n,n)$.

Thus, as $\F$ satisfies the hypotheses of \Cref{prop-complexity-estimate}, $\TMSP_\le\preceq(n,n)$.

\end{proof}

Now, using the object $\text{Counter}_1[a]$ constructed in \Cref{ex-counter1} as a base, for all $d\in\N-\{0\}$ the object $\text{Counter}_{d+1}[a_1, \dots, a_{d+1}]$ is defined iteratively in \Cref{fig-COUNTER_STAR}.

\begin{figure}[hbt]
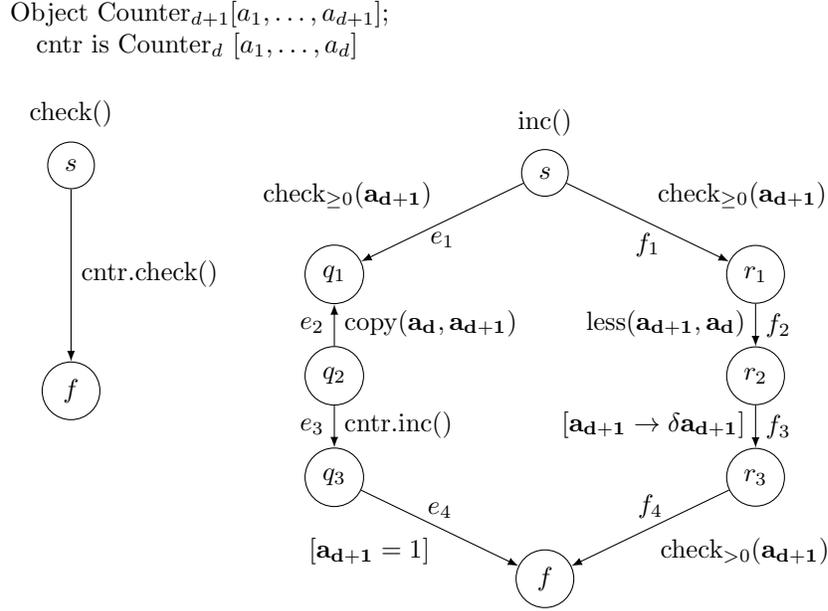

	\centering
	\include{COUNTER_STAR}
	\caption{Object $\text{Counter}_{d+1}[a_1, \dots, a_{d+1}]$.}
	\label{fig-COUNTER_STAR}       
\end{figure} 

Note that for $\op\in\{\ch,\inc\}$, per the minimal compatible hardware conventions, $\I_\op=\I_{\Counter_{d+1}}^\val$ and $\A_\op(a_i)=\{\delta\}$ for all $i\in\overline{d+1}$.

Further, note that the role of the action $\cntr.\check()$ in $\text{Counter}_{d+1}.\check()$ is analogous to that of $\check_{>0}$.  In fact, one can replace this action with $\check_{>0}(a_1)$ without affecting any of the desired properties.

Finally, it is worth noting that the labelled graph defining the operation $\inc()$ satisfies condition (6) of \Cref{def-s1-graph}.  To see this, note that for $i\leq d$, the only edges mentioning $a_i$ as a variable $e_2$ and $f_2$, which mention $a_d$ as a variable; however, by Lemmas \ref{lem-io-cpy} and \ref{lem-ndr-le}, both edges mention $a_d$ as an immutable variable.  Hence, $\G_\inc$ is indeed an $S$-graph.
 
The next goal is to establish that $\text{Counter}_{d+1}$ is indeed a counter in the sense of \Cref{def-counter-star}.  As should be expected, the proof is by induction of $d$, with \Cref{lem-counter-is-counter} establishing a base case.  

First, though, the set of values of these objects must be explicitly determined.

\begin{definition}\label{def-lex}
For $d\geq1$, let $U_d = \bigl\{(i_1, \dots, i_d)\mid i_1\geq i_2\geq \dots\geq i_d\geq 0 \bigr\} \subseteq \N^d$.  Then, for each $d$ inductively define the function $\scc_d:U_d\to U_d$ as follows:
	\begin{itemize}
		\item For $d=1$ and $u=(i)\in U_1$, define $\scc_1(u)=(i+1)$.
		\item For $d>1$, let $u=(i_1,\dots,i_{d-1},i_d)=(w,d)$, where $w=(i_1,\dots,i_{d-1})$.  Then, as $w\in U_{d-1}$, define $$\scc_d(u) = \begin{cases}
			(w, i_d +1) &i_{d-1} > i_d \\
			(\scc_{d-1}(w), 0) &i_{d-1} = i_d\\
		\end{cases}$$
	\end{itemize}
Then, define the relation $\leq_d$ on $U_d$ by $u\leq_d u'$ if and only if there exists $k\in\N$ such that $\scc_d^k(u)=u'$.
\end{definition}

\begin{lemma}\label{lem-lex}
	For all $d\geq 1$, define $\epsilon_d:U_d\to\N$ by $\epsilon_d(0,\dots,0)=0$ and $\epsilon_d(\scc_d(u))=f(u)+1$ for all $u\in U_d$.  Then $\epsilon_d$ defines an isomorphism of posets between $(U_d,\leq_d)$ and $(\N,\leq)$.
\end{lemma}

The isomorphism $\epsilon_d$ is called the \emph{decoding function of} $\Counter_d$.  Note that the decoding function $\epsilon_1$ corresponds to the `identity map' sending $(i)$ to $i$ for all $i\in\N$.

Now, for all $d\geq1$ and $u=(i_1,\dots,i_d)\in U_d$, define $\delta^u=(\delta^{i_1},\dots,\delta^{i_d})\in F(\delta)^d$.  Then, define $\VAL_d=\{\delta^u\mid u\in U_d\}$.  The well order $\leq_d$ on $U_d$ then induces a well order $\leq_d$ on $\VAL_d$, so that the map $f_d$ given by $f_d(\delta^u)=\epsilon_d^{-1}(u)$ defines an isomorphism of posets between $(\VAL_d,\leq_d)$ and $(\N,\leq)$.  Hence, the elements of $\VAL_d$ can be enumerated $\VAL_d=\{v_i^d\}_{i\in\N}$, where $v_i^d=\delta^{\epsilon_d^{-1}(i)}$.  With an abuse of notation, extend the definition of $\scc_d$ by setting $\scc_d(v_i^d)=v_{i+1}^d$.

\begin{lemma}\label{lem-countern}
	For all $d\geq 0$, $\C_{d+1} = \Counter_{d+1}[a_1, \dots, a_{d+1}]$ is a counter with respect to the values $\VAL_{\C_{d+1}}=\VAL_{d+1}$.
\end{lemma}
\begin{proof} 
	The proof follows by induction on $d$, with \Cref{lem-counter-is-counter} providing the proof of the base case of $d=0$.  So, suppose $d\geq1$ and that $\C_d$ is a counter with respect to the values $\VAL_{\C_d}=\VAL_d$.
	
	As a result, since $\C_{d+1}$ uses an instance of $\cntr$, $\VAL_{\C_{d+1}}\subseteq\VAL_d\times F(\delta)=\GL_\op$ for either $\op=\inc,\check$.
	
	By the definition of the operation $\check()$ in a counter, a straightforward inductive argument immediately implies that $\IO_\check^{s,s}=\IO_\check^{f,f}=\Id$ and $$\IO_\check^{s,f}=\{(v,v)\mid v=(v_j^d,\delta^i)\in\VAL_{\C_{d+1}}, j>0\}$$

	Now, fix $m\in\N$ and let $\epsilon_{d+1}^{-1}(m)=(i_1,\dots,i_d,i_{d+1})$ with $u_m=(i_1,\dots,i_d)$.  Then, there exists $j\in\N$ such that $v_m^{d+1}=\delta^{\epsilon_{d+1}^{-1}(m)}=(v_j^d,\delta^{i_{d+1}})$, where $v_j^d=\delta^{u_m}$.  Note that, by the definition of $V_{d+1}$, $j=0$ if and only if $m=0$.
	
	If $i_d>i_{d+1}$, then $v_{m+1}^{d+1}=\scc_{d+1}(v_m^{d+1})=(v_j^d,\delta^{i_{d+1}+1})$.  In this case, there exists an input-output computation
$$\eps_m=\bigl((v_m^{d+1},\inc_s),f_1',(v_m^{d+1},r_1),f_2',(v_m^{d+1},r_2),f_3',(v_{m+1}^{d+1},r_3),f_4',(v_{m+1}^{d+1},\inc_f)\bigr)$$
of $\inc$ between $(v_m^{d+1},\inc_s)$ and $(v_{m+1}^{d+1},\inc_f)$, where $f_j'=(\s,\f,f_j)$ for all $j$.

	Otherwise, if $i_d=i_{d+1}$, then $v_{m+1}^{d+1}=\scc_{d+1}(v_m^{d+1})=(\scc_d(v_j^d),1)$.  But then there exists an input-output computation
$$\eps_m=\bigl((v_m^{d+1},\inc_s),e_1',(v_m^{d+1},q_1),e_2'',((v_j^d,1),q_2),e_3',(v_{m+1}^{d+1},q_3),e_4',(v_{m+1}^{d+1},\inc_f)\bigr)$$
of $\inc$ between $(v_m^{d+1},\inc_s)$ and $(v_{m+1}^{d+1},\inc_f)$, where $e_j'=(\s,\f,e_j)$ and $e_j''=(\f,\s,e_j)$ for all $j$.  

Hence, as $v_0^{d+1}$ is by definition the initial value of $\Counter_{d+1}$, it follows that $\VAL_{d+1}\subseteq\VAL_{\C_{d+1}}$.

    Conversely, define the subsets $C_{m,\check},D_{m,\check}\subseteq\GL_\check\times S_\check$ as follows:
	\begin{itemize}
	
	\item If $m>0$, then $C_{m,\check}^s=C_{m,\check}^f=\{v_m^{d+1}\}$, $D_{m,\check}=\emptyset$
	
	\item If $m=0$, then $C_{m,\check}^s=D_{m,\check}^f=\{v_m^{d+1}\}$ and $D_{m,\check}^s=C_{m,\check}^f=\emptyset$.
	
	\end{itemize}
	
	Similarly, define the subsets $C_{m,\inc}\subseteq\GL_\inc\times S_\inc$ as follows:
	\begin{itemize}
	
	\item If $i_{d+1}<i_d$, then 
	\begin{align*}
	C_{m,\inc}^s=C_{m,\inc}^{q_1}=C_{m,\inc}^{r_1}=C_{m,\inc}^{r_2}&=\{v_m^{d+1}\} \\
	C_{m,\inc}^{r_3}=C_{m,\inc}^f&=\{v_{m+1}^{d+1}\} \\
	C_{m,\inc}^{q_2}=C_{m,\inc}^{q_3}&=\emptyset
	\end{align*}
	
	\item If $i_{d+1}=i_d$, then 
	\begin{align*}
	C_{m,\inc}^s=C_{m,\inc}^{q_1}=C_{m,\inc}^{r_1}&=\{v_m^{d+1}\} \\
	C_{m,\inc}^{q_2}&=\{(v_j^d,1)\} \\
	C_{m,\inc}^{q_3}=C_{m,\inc}^f&=\{v_{m+1}^{d+1}\} \\
	C_{m,\inc}^{r_2}=C_{m,\inc}^{r_3}&=\emptyset
	\end{align*}
	
	\end{itemize}
	
	Finally, define $D_{0,\inc}=\{(v_0^{d+1},\inc_f),(v_0^{d+1},q_3)\}\subseteq\GL_\inc\times S_\inc$.
	
	For each $\op\in\{\check,\inc\}$, it is straightforward to check that $C_{m,\op}$ and $D_{m,\op}$ are each closed under $\IO_e$ for each $e\in E_\op$.  Further, $$\VAL_{d+1}\times\SF_\op=\bigcup\left(C_{m,\op}^{\SF_\op}\cup D_{m,\op}^{\SF_\op}\right)$$
	Hence, \Cref{prop-io-classes2} implies that $\VAL_{\C_{d+1}}=\VAL_{d+1}$.
	
	Moreover, \Cref{prop-io-classes} yields $\IO_\inc\subseteq\bigcup\bigl((C_{m,\inc,\IO})^2\cup(D_{m,\inc,\IO})^2\bigr)$.  In particular, this implies that $\IO_\inc^{s,s}=\IO_\inc^{f,f}=\Id$ and 
	$$\IO_\inc^{s,f}\subseteq\{(v_m^{d+1},v_{m+1}^{d+1})\mid m\in\N\}$$
	Thus, as the existence of the computations $\eps_m$ gives the reverse inclusion, $\Counter_{d+1}$ is indeed a counter with respect to $\VAL_{d+1}$.
	
\end{proof}

\begin{lemma} \label{lem-cntr-complexities}

  For all $d\in\N$, $\TMSP_{\C_{d+1}}\preceq(n,n)$ for $\C_{d+1}=\Counter_{d+1}$.
  
\end{lemma}

\begin{proof} 

The proof proceeds by induction on $d$, with the base case $d=0$ given by \Cref{lem-basic-star-bounds}(2).

Let $\F$ be the estimating set given by $\F_\inst=(n,n)$ for all $\inst\in\INST_{\C_{d+1}}$.  By \Cref{lem-basic-0bounds}, \Cref{lem-basic-star-bounds}, \Cref{lem-ndr-le}, and the inductive hypothesis, it follows that $\TMSP_\inst\preceq\F_\inst$ for each $\inst$.

Note that for $m\in\N$, the computations $\eps_m$ realizing $(v_m^{d+1},v_{m+1}^{d+1})\in\IO_\inc^{s,f}$ are linearly $1$-bounded.  Moreover, if $m\neq0$, then the linearly $1$-bounded computation $$\zeta_m=\bigl((v_m^{d+1},\check_s),e',(v_m^{d+1},\check_f)\bigr)$$ where $e'=(\s,\f,e)$ for the unique edge $e\in E_\check$, realizes $(v_m^{d+1},v_m^{d+1})\in\IO_\check^{s,f}$.  Hence, $\Counter_{d+1}$ is a $1$-bounded object.

Now, for $m\geq1$, note that $\tmt_\check^\F(\zeta_m,1)=|v_m^{d+1}|$ and $\spt_\check^\F(\zeta_m,1)=2|v_m^{d+1}|$.

Next, let $m\in\N$, $M_m=\max\bigl(|v_m^{d+1}|,|v_{m+1}^{d+1}|\bigr)$, and $\epsilon_{d+1}^{-1}(m)=(i_1,\dots,i_d,i_{d+1})$.  Then, in either case where $i_d>i_{d+1}$ or $i_d=i_{d+1}$, it follows that
\begin{align*}
\tmt_\inc^\F(\eps_m,1)&= 3M_m+1 \\
\spt_\inc^\F(\eps_m,1)&=2M_m
\end{align*}

Hence, as $\IO_\inc^{s,s}=\IO_\inc^{f,f}=\Id$ and $\IO_\check^{s,s}=\IO_\check^{f,f}=\Id$ and so any such relation can be realized by an empty computation, $\TMSPT_{\C_{d+1}}^\F(1)\preceq(n,n)$.

Thus, \Cref{lem-K-bounded} implies that $\TMSP_{\C_{d+1}}\preceq(n,n)$.

\end{proof}

Given that the stated purpose of defining $\Counter_d$ is to construct counters that are in some sense more `economical' than $\Counter_1$, \Cref{lem-cntr-complexities} seems to be a departure.  After all, the same asymptotic upper bound is given for the time-space complexity of each of these objects, and so it is unclear how $\Counter_d$ could be superior to the simpler $\Counter_1$.

However, the benefits arise in the use of these counters to construct other objects.  To aid with this, the values of $\Counter_d$ are studied in a more concise way.

For all $d\geq1$, let $|\cdot|$ denote the $\ell_1$-norm on $\N^d$.  That is, for all $m\in\N$, $|\epsilon_d^{-1}(m)|=i_1+\dots+i_m$ where $\epsilon_d^{-1}(m)=(i_1,\dots,i_d)$.  Note that $|\epsilon_d^{-1}(m)|=|v_m^d|$ for all $m,d$.  Then, define the growth function $\size_d$ by:
$$\size_d(n)=\max_{m\in\overline{0,n}}|v_m^d|=\max_{m\in\overline{n}}|v_m^d|$$
where the last equality is given by the observation that $|v_0^d|=0$ for all $d$.

\begin{lemma}\label{lem-lex-efficiency}
For all $d\geq1$, $\size_d\sim n^{1/d}$.
\end{lemma}
\begin{proof} 

For every $k\in\N$, let $k_d=(k,\dots,k)\in\VAL_d$ and let $c_k=\epsilon_d(k_d)$.  For example, if $d=3$ and $k=2$, then $k_d=(2,2,2)$ and $c_k=9$.

By definition, $\size_d(c_k)=\max_{m\in\overline{c_k}}|v_m^d|=\max_{m\in\overline{c_k}}|\epsilon_d^{-1}(m)|\geq|k_d|=kd$.  Moreover, for all $m\leq c_k$, $\epsilon_d^{-1}(m)\leq_d\epsilon_d^{-1}(c_k)=k_d$, so that $\epsilon_d^{-1}(m)\in\{0,1,\dots,k\}^d$.  Hence, $\size_d(c_k)=dk$.

For all $k\geq1$, $(j_1,\dots,j_d)\in\{0,1,,k\}^d-\{0_d\}$, define $\pi_{k,d}(j_1,\dots,j_d)$ as the sequence obtained by permuting the coordinates so that they are in non-increasing order.  This then induces a surjective map $\pi_{k,d}:\{0,1,\dots,k\}^d-\{0_d\}\to\epsilon_d^{-1}(\overline{c_k})$, and so 
\begin{equation}\label{eqn-1}
c_k=|\epsilon_d^{-1}(\overline{c_k})|\leq|\{0,1,\dots,k\}^d-\{0_d\}|=(k+1)^d-1\leq(k+1)^d
\end{equation}

Further, for all $v=(i_1,\dots,i_d)\in\epsilon_d^{-1}(\overline{c_k})$, observe that $\pi_{k,d}^{-1}(v)$ consists of all (distinct) permutations of the indices $i_1,\dots,i_d$.  So, $|\pi_{k,d}^{-1}(v)|\leq d!$.  Hence, there exists an injection of $\{0,1,\dots,k\}^d-\{0_d\}$ into $d!$ disjoint copies of $\epsilon_d^{-1}(\overline{c_k})$ given by mapping distinct elements of $\pi_{k,d}^{-1}(v)$ to distinct copies of $v$.  As a result,
\begin{equation}\label{eqn-2}
k^d\leq(k+1)^d-1=|\{0,1,\dots,k\}^d-\{0_d\}|\leq d! \cdot c_k
\end{equation}
Note that $c_0=0$ implies that both (\ref{eqn-1}) and (\ref{eqn-2}) hold trivially for $k=0$.

Now, for all $n\in\N-\{0\}$, there exists $m\in\N-\{0\}$ such that $m^d\leq n\leq(m+1)^d$.  By equation (\ref{eqn-1}), $m^d\geq c_{m-1}$, and so $c_{m-1}\leq n$.  But then $d(m-1)=\size_d(c_{m-1})\leq\size_d(n)$, and so
\begin{equation}\label{eqn-3}
\size_d(n)\geq d((m+1)-2)\geq dn^{1/d}-2d
\end{equation}

Conversely, for all $n\in\N-\{0\}$, there exists $\ell\in\N-\{0\}$ such that $\frac{\ell^d}{d!}\leq n\leq \frac{(\ell+1)^d}{d!}$.  By (\ref{eqn-2}), $n\leq\frac{(\ell+1)^d}{d!}\leq c_{\ell+1}$ so that $\size_d(n)\leq\size_d(c_{\ell+1})=d(\ell+1)$.  But then $\ell^d\leq d!\cdot n$, so that $\ell\leq \sqrt[d]{d!}\cdot n^{1/d}$ and hence
\begin{equation}\label{eqn-4}
\size_d(n)\leq\left(d\sqrt[d]{d!}\right)n^{1/d}+d
\end{equation}

Thus, as $\size_d(0)=0$ implies that (\ref{eqn-3}) and (\ref{eqn-4}) both hold for $n=0$, the statement follows.

\end{proof}

\begin{corollary}\label{cor-counter-eps}
	For all $\varepsilon\in(0,1)$, there exists $\I_\eps\subseteq\I^*$ and a counter $\Cntr_\varepsilon[\I_\eps]$ such that $\TMSP_{\Cntr_\eps}\preceq(n,n)$ and $\size_{\varepsilon} \preceq n^\varepsilon$. 
\end{corollary}

\begin{proof}

Take $d\in\N$ such that $d\eps>1$.  Then, setting $\I_\eps=\{a_1,\dots,a_d\}$, let $\Cntr_\eps=\Counter_{d+1}[\I_\eps]$.  By \Cref{lem-countern}, \Cref{lem-cntr-complexities}, and \Cref{lem-lex-efficiency}, $\Cntr_\eps$ is a counter satisfying the statement.

\end{proof}

With an abuse of notation, the values of $\Counter_\eps$ are denoted $\VAL_\eps=\{v_i^\eps\}_{i\in\N}$.

\section{Positivity check}\label{sec-positivity-check}

In this section, an alphabet $A$ is fixed.  Following the notation of previous sections, $L_A$ is taken to be $F(A)$, i.e the set of reduced words over $A\cup A^{-1}$.  Then, define $L_A^+$ to be the subset of $L_A$ comprised of the words which consist only of letters from $A$ (and not from $A^{-1}$).  In other words, $L_A^+$ is the set of words in the free monoid on $A$, viewed as a subset of $L_A$.

The goal for the remainder of this section is to construct a family of checkers (in the sense of \Cref{def-acceptor2}) of $L_A^+$ whose complexities are progressively closer to linear.

\subsection{Split} \

\medskip

To aid in this endeavor, the auxiliary proper operation $\splt^A$ is first defined below (see \Cref{fig-SPLIT}).  Note that, per the convention of minimal compatible hardware, $\I_\splt=\I_\splt^\ext=\{a,b\}$ and $\A_\splt(a)=\A_\splt(b)=A$.  Further, $\splt$ is an operation of degree 0, and so its computations can be viewed in the simpler light of Section 4 (as opposed to the generic computations of Section 6).

For any $w\in L_A$, let $S_w=\{(w_1,w_2)\in L_A\times L_A\mid w_1w_2=w\}$, where the equality $w_1w_2=w$ is understood to be in the free group $F(A)$.  Note that the sets $S_w$ form a partition of $L_A\times L_A$.

\begin{figure}[hbt]
	\centering
	\include{SPLIT}
	\vspace{-1cm}
	\caption{Operation $\splt^A(a, b)$.}
	\label{fig-SPLIT}       
\end{figure} 

\begin{lemma}\label{lem-io-split}
	The $\IO$-relation of $\splt^A$ is given by:
$$\IO^{s,s}_\splt = \IO^{s,f}_\splt = \IO^{f,f}_\splt =\bigcup_{w\in L_A}S_w^2$$
	Moreover, $\splt^A$ has linear time-space complexity, i.e $\TMSP_\splt \preceq(n,n)$.
\end{lemma}

\begin{proof} 

As the edge between $\splt_s$ and $\splt_f$ is empty, any input-output computation can be lengthened by 1 to switch the starting (respectively ending) state.  Hence, it follows immediately that $\IO_\splt^{s,s}=\IO_\splt^{s,f}=\IO_\splt^{f,f}$.

Let $e_x\in E_\splt$ be the loop at $\splt_f$ corresponding to the letter $x\in A$.  Then, $e_x^{\pm1}$ supports a computation between $((w_1,w_2),\splt_f)$ and $((w_1',w_2'),\splt_f)$ if and only if there exists $y\in\{x^{\pm1},1\}$ such that $w_1'=w_1y^{-1}$ and $w_2'=yw_2$.  But then $w_1'w_2'=w_1w_2$, i.e $(w_1,w_2),(w_1',w_2')\in S_w$ for $w=w_1w_2$.

It follows that each set $S_w\times S_\splt$ is closed under $\IO_e$ for every edge $e\in E_\splt$.  Hence, $\IO_\splt\subseteq\bigcup (S_w\times\SF_\splt)^2$.

Now, let $(w_1,w_2)\in L_A\times L_A$ and let $w\in L_A$ such that $w=w_1w_2$.  Set $m=|w_1|+|w_2|$ and note that $|w|\leq m$.

If $w_2=1$, then let $\eps_{w_1,w_2}$ be the empty computation between $((w,1),\splt_f)$ and $((w_1,w_2),\splt_f)$.

Otherwise, let $w_2\equiv x_1^{\delta_1}\dots x_k^{\delta_k}$ for $x_i\in A$ and $\delta_i\in\{\pm1\}$ for all $i$.  Then, consider the computation
$$\eps_{w_1,w_2}=\bigl(((u_0,v_0),\splt_f),e_{x_k}^{\delta_k},((u_1,v_1),\splt_f),\dots,e_{x_1}^{\delta_1},((u_k,v_k),\splt_f)\bigr)$$
where $(u_0,v_0)=(w,1)$.  Then for all $j\in\overline{0,k-1}$, $v_{j+1}=x_{k-j}^{\delta_j}v_j$ and $u_{j+1}=u_jx_{k-j}^{-\delta_j}$.  In particular, $v_k=w_2$.  Further, since $((w,1),(u_k,v_k))\in\IO_\splt^{f,f}$ implies that $(u_k,v_k)\in S_w$ and $v_k=w_2$, $u_k=w_1$.

Moreover, as $x_1^{\delta_1}\dots x_k^{\delta_k}$ is assumed to be reduced, $|v_{j+1}|=|v_j|+1$ for all $j\in\overline{0,k-1}$.  So, since $|u_j|=|u_{j+1}x_{k-j}^{\delta_j}|\leq|u_{j+1}|+1$, then $|u_j|+|v_j|\leq|u_{j+1}|+|v_{j+1}|$.  Hence, $\max(|u_j|+|v_j|)=m$.

In either case, $\eps_{w_1,w_2}$ is a computation realizing $\bigl((w,1),(w_1,w_2)\bigr)\in\IO_\splt^{f,f}$ which satisfies the bounds $\tm_\splt(\eps_{w_1,w_2})\leq m$ and $\sp_\splt(\eps_{w_1,w_2})=m$.

For $n\in\N$ and $w\in L_A$, let $(w_1,w_2),(w_1',w_2')\in S_w$, set $m=|w_1|+|w_2|$ and $m'=|w_1'|+|w_2'|$, and suppose $m,m'\leq n$.  Then, the computation $\eps=\eps_{w_1,w_2}^{-1}\eps_{w_1',w_2'}$ is a computation realizing $\bigl((w_1,w_2),(w_1',w_2')\bigr)\in\IO_\splt^{f,f}$ with $\tm_\splt(\eps)\leq2n$ and $\sp_\splt\leq n$.  Thus, $\IO_\splt^{f,f}=\bigcup_{w\in L_A}S_w^2$, and so the $\IO$-relation is as in the statement.

Moreover, as $((w_1,w_2),(w_1',w_2'))\in\IO_{\splt}^{s_1,s_2}$ for any $s_1,s_2$ can be realized by a computation which lengthens $\eps$ by at most 2 without affecting the space, the bounds on the complexity of $\eps$ imply $\TMSP_\splt\preceq(n,n)$.

\end{proof}

\subsection{Right-positivity and a Weak acceptor} \

\medskip

For any $w\in L_A$ with $w\equiv x_1^{\delta_1}\dots x_k^{\delta_k}$ such that $x_i\in A$ and $\delta_i\in\{\pm1\}$ for all $i$, define the \emph{degree of $w$} as $\deg(w)=\sum\delta_i$.  Note that $\deg(w)$ can be defined in terms of any representation of $w$, i.e the degree is invariant under free reduction.  

Then, let $L_A^{R+}$ be the subset of $L_A$ consisting of all words $w$ such that if $w\equiv w_lw_r$, then $\deg(w_r)\geq0$.  In other words, $w\equiv x_1^{\delta_1}\dots x_k^{\delta_k}\in L_A^{R+}$ if and only if $\sum_{i=\ell}^k\delta_i \geq0$ for all $\ell\in\overline{k}$.

Hence, for any $x,y\in A$, $xy^{-2}x^2\in L_A^{R+}$ but $xy^{-2}x\notin L_A^{R+}$.  Note that $L_A^+\subseteq L_A^{R+}$.

\begin{lemma}\label{lem-lplus-decomp}

	Let $w,w_l,w_r\in L_A$ with $w=w_lw_r$.  Suppose $w_l\in L_A^{R+}$ and $w_r\in L_A^+$.  Then $w\in L_A^+$ if and only if $w_l\in L_A^+$.

\end{lemma}
\begin{proof}

The reverse direction is clear, so it suffices to assume that $w\in L_A^+$ and show that $w_l\in L_A^+$.

It is clear that it can be assumed without loss of generality that $w_l\neq1$, so let $w_l\equiv x_1^{\delta_1}\dots x_k^{\delta_k}$ such that $x_i\in A$ and $\delta_i\in\{\pm1\}$ for all $i$.  As $w_l\in L_A^{R+}$, it follows that $\delta_k=1$.  So, since $w_r\in L_A^+$, no cancellation can occur between the last letter of $w_l$ and the first letter of $w_r$.  Hence, since $w_l$ and $w_r$ are themselves reduced, $w\equiv w_lw_r$.

Thus, since $w\in L_A^+$, $\delta_i=1$ for all $i$, and so $w_l\in L_A^+$.

\end{proof}

Note that it is necessary that $w_l\in L_A^{R+}$ for the forward direction of \Cref{lem-lplus-decomp}.  For example, for any $x\in A$, $1=x^{-1}x\in L_A^+$ and $x\in L_A^+$, but $x^{-1}\notin L_A^+$.  The reverse direction, however, is true regardless of the assumption on $w_l$.

\begin{lemma}\label{lem-rplus-decomp}

Let $w_1,w_2\in L_A^{R+}$ and $v_1,v_2\in L_A^+$ such that $w_1v_1=w_2v_2$.  Then there exists $z\in L_A^+$ and $\delta\in\{\pm1\}$ such that $w_1^{-1}w_2=z^\delta$.

\end{lemma}

\begin{proof}

Fixing $i\in\{1,2\}$, if $w_i\neq1$, then there exists $x_{i,1},\dots,x_{i,k}\in A$ and $\delta_{i,1},\dots,\delta_{i,k}\in\{\pm1\}$ such that $w_i=x_{i,1}^{\delta_{i,1}}\dots x_{i,k}^{\delta_{i,k}}$.  Similarly, if $v_i\neq 1$, then by the definition of $L_A^+$ there exists $y_{i,1},\dots,y_{i,\ell}\in A$ such that $v_i=y_{i,1}\dots y_{i,\ell}$.  By the definition of $L_A^{R+}$, it must hold that $\delta_{i,k}=1$.  So, $w_iv_i\equiv x_{i,1}^{\delta_{i,1}}\dots x_{i,k}^{\delta_{i,k}}y_{i,1}\dots y_{i,\ell}$, i.e the product on the right is reduced.

Hence, as the same clearly holds if $w_i=1$ or $v_i=1$, then $w_1v_1\equiv w_2v_2$.  As a result, $v_2$ must be a suffix of $w_1v_1$.

If $v_2$ is a suffix of $v_1$, then there exists $z\in L_A^+$ such that $v_1\equiv zv_2$, so that $w_1^{-1}w_2=v_1v_2^{-1}=z$.  Otherwise, there exists $z\in L_A^+$ such that $v_2\equiv zv_1$, so that $w_1^{-1}w_2=v_1v_2^{-1}=z^{-1}$.

\end{proof}

Now, for $\eps\in(0,1)$, the proper operation $\wacc_{R+}=\waccept_{R+}^\eps(a)$ is defined below in \Cref{fig-CHECK-LPLUS}.  Crucially, note the assumption that $a\notin \I_\eps$.  Hence, $\I_{\wacc_{R+}}=\I_{\wacc_{R+}}^\ext\sqcup\I_{\wacc_{R+}}^\int$, $\I_{\wacc_{R+}}^\ext=\{a\}$, $\I_{\wacc_{R+}}^\int=\I_\eps$, $\A_{\wacc_{R+}}(a)=A$, and $\A_{\wacc_{R+}}(i)=\delta$ for all $i\in\I_\eps$.  

\begin{figure}[hbt]
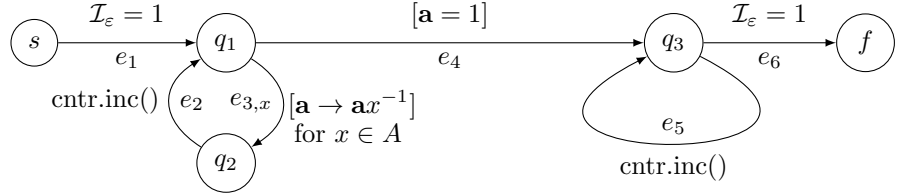

	\centering
	\include{CHECK-LPLUS}
	\caption{A weak acceptor for the language $L^{R+}_A$.}
	\label{fig-CHECK-LPLUS}       
\end{figure} 

Further, the action `$\I_\eps=1$' labelling both $e_1$ and $e_6$ represents the set of guard fragments $[i=1]$ for all $i\in\I_\eps$.  Hence, $e_1,e_6\in E_{\wacc_{R+}}^0$, $e_1^{-1}$ and $e_6^{-1}$ are also labelled by $\I_\eps=1$, and $\wacc_{R+}$ satisfies condition (O).

A computation $\gamma=\bigl((u_0,s_0),e_1',(u_1,s_1),e_2',\dots,e_{2\ell}',(u_{2\ell},s_{2\ell})\bigr)$ of $\wacc_{R+}$ is called a \emph{$(q_1,q_2)$-computation} if $s_{2i}=q_1$ for all $i\in\overline{0,\ell}$ and $s_{2i-1}=q_2$ for all $i\in\overline{\ell}$.  Letting $u_i=(w_i,v_i)$ where $w_i\in L_A$ and $v_i\in\VAL_\eps$, $\gamma$ is said to \emph{accept} $w\in L_A$ if $w_0=w$, $v_0=v_0^\eps$, and $w_{2\ell}=1$.

\begin{lemma}\label{lem-q1-q2-basics}

Let $\bigl((u_1,q_1),e_1',(u_2,q_2),e_2',(u_3,q_1)\bigr)$ be a $(q_1,q_2)$-computation of $\wacc_{R+}$ with $u_i=(w_i,v_i)$ for $w_i\in F(\A)$ and $v_i=v_{j(i)}^\eps\in\VAL_\eps$ for all $i$.  Then:

\begin{enumerate}[label=({\alph*})]

\item $j(1)+\deg(w_1)=j(3)+\deg(w_3)$

\item $\deg(w_2)\leq \deg(w_1)$

\end{enumerate}

\end{lemma}

\begin{proof}

Let $e_j'=(\sf_j,\sf_j',e_j'')$ for $j=1,2$.

First, suppose $e_1''=e_2$.  Then, $e_1'=(\f,\s,e_2)$, so that $w_2=w_1$ and $j(2)=j(1)-1$.  In particular, (b) is satisfied.

If in addition $e_2''=e_2$, then $e_2'=(\s,\f,e_2)$, so that $u_1=u_3$ and (a) is trivially satisfied.  Otherwise, there exists $x\in A$ such that $e_2''=e_{3,x}^{\pm1}$.  In particular, either $e_2'=(\s,\f,e_{3,x}^{-1})$ or $e_2'=(\f,\s,e_{3,x})$.  In either case, $v_3=v_2$ and $w_3=w_2x=w_1x$.  So, $j(3)=j(2)=j(1)-1$ and $\deg(w_3)=\deg(w_1x)=\deg(w_1)+1$.  

Conversely, suppose $e_1''=e_{3,x}^{\pm1}$ for some $x\in A$.  So, $e_1'=(\s,\f,e_{3,x})$ or $e_1'=(\f,\s,e_{3,x}^{-1})$.  In either case, $v_2=v_1$ and $w_2=w_1x^{-1}$.  So, $\deg(w_2)=\deg(w_1)-1$, so that (b) is satisfied.

If in addition $e_2''=e_2$, then $e_2'=(\s,\f,e_2)$, so that $j(3)=j(2)+1=j(1)+1$ and $w_3=w_2$.  Finally, if $e_2''=e_{3,y}^{\pm1}$ for some $y\in A$, then $e_2'=(\s,\f,e_{3,y}^{-1})$ or $e_2'=(\f,\s,e_{3,y})$.  In either case, $v_3=v_2$ and $w_3=w_2y$, so that $j(3)=j(1)$ and $\deg(w_3)=\deg(w_2)+1=\deg(w_1)$.

\end{proof}

\begin{lemma}\label{lem-q1-q2-iff}

There exists a $(q_1,q_2)$-computation accepting $w\in L_A$ if and only if $w\in L_A^{R+}$.  Moreover, letting $\F$ be the estimating set for $\wacc_{R+}$ given by $\F_\cntr=(n,n)$, there exists $K\in\N$ such that for any $w\in L_A^{R+}$, there exists a $(q_1,q_2)$-computation $\gamma_w$ accepting $w$ and satisfying $\tm_{\wacc_{R+}}^\F(\gamma_w)\leq K|w|^{1+\eps}$ and $\sp_{\wacc_{R+}}^\F(\gamma_w)\leq K|w|$.

\end{lemma}

\begin{proof}

Let $\gamma=\bigl((u_0,s_0),e_1',(u_1,s_1),e_2',\dots,e_{2\ell}',(u_{2\ell},s_{2\ell})\bigr)$ be a $(q_1,q_2)$-computation of $\wacc_{R+}$ accepting $w\in L_A$.  Following previous notation, let $u_i=(w_i,v_i)$ with $v_i=v_{j(i)}^\eps$ for all $i$.  Further, let $e_i'=(\sf_i,\sf_i',e_i'')$ for all $i$.

For all $i\in\overline{2\ell}$, define:
$$z_i=
\begin{cases}
1 & e_i''=e_2 \\
x^{-1} & e_i'=(\s,\f,e_{3,x}) \text{ or } e_i'=(\f,\s,e_{3,x}^{-1}) \\
x & e_i'=(\s,\f,e_{3,x}^{-1}) \text{ or } e_i'=(\f,\s,e_{3,x})
\end{cases}
$$ 
Then, let $w_i'=z_1\dots z_i$ for all $i$.

By the definition of the actions labelling $e_2$ and $e_{3,x}$, $w_i=w_{i-1}z_i$ for all $i$, so that $w_i=ww_i'$.  Hence, since $w_{2\ell}=1$, $w_{2\ell}'$ is freely equal to $w^{-1}$.

It follows from \Cref{lem-q1-q2-basics}(a) that $j(2i-2)+\deg(w_{2i-2})=j(2i)+\deg(w_{2i})$ for all $i\in\overline{\ell}$.  In particular, as it is assumed that $j(0)=0$ and $w_0=w$, then $j(2i)+\deg(w_{2i})=\deg(w)$ for all $i\in\overline{\ell}$.  But degrees are additive, so that $\deg(w_{2i})=\deg(ww_{2i}')=\deg(w)+\deg(w_{2i}')$ and then $\deg(w_{2i}')=-j(2i)$.  Hence, since $j(2i)\geq0$ for all $i$, $\deg(w_{2i}')\leq0$.

Further, \Cref{lem-q1-q2-basics}(b) implies that $\deg(w_{2i-1}')\leq\deg(w_{2i}')$ for all $i\in\overline{\ell}$.  Thus, $\deg(w_i')\leq0$ for all $i\in\overline{2\ell}$.

Hence, as any suffix of $w$ must be freely equal to some suffix of $(w_{2\ell}')^{-1}$, which is given by $(w_i')^{-1}$ for some $i$,  the degree of such a suffix must be nonnegative, and so $w\in L_A^{R+}$.

Now, the proof of the existence of $\gamma_w$ for any $w\in L_A^{R+}$ follows by induction on $n=|w|$, with trivial base case $n=0$ given by the empty computation $\bigl(((1,v_0^\eps),q_1)\bigr)$.

First, suppose $n=1$, so that $w=x\in A$.  Consider the computation:
$$\gamma_w=\bigl(((w,v_0^\eps),q_1),(\s,\f,e_{3,x}),((1,v_0^\eps),q_2),(\s,\f,e_2),((1,v_1^\eps),q_1)\bigr)$$
Then, $\tm_{\wacc_{R+}}^\F(\gamma_w)=1+|v_1^\eps|=2$ and $\sp_{\wacc_{R+}}^\F(\gamma_w)=\max(|w|,|v_1^\eps|)=1$.  So, the statement follows for $K\geq2$.

In general, let $w\equiv x_1^{\delta_1}\dots x_n^{\delta_n}$.  Then the suffix $w'=x_2^{\delta_2}\dots x_n^{\delta_n}$ is also in $L_A^{R+}$.  So, by the inductive hypothesis, there exists a $(q_1,q_2)$-computation $\gamma_{w'}$ accepting $w'$ and satisfying $\tm_{\wacc_{R+}}^\F(\gamma_{w'})\leq K|w'|^{1+\eps}$ and $\sp_{\wacc_{R+}}^\F(\gamma_{w'})\leq K|w'|$.

For $\gamma_{w'}=\bigl((u_0',s_0),e_1',(u_1',s_1),e_2',\dots,e_m',(u_m',s_m)\bigr)$ with $u_i'=(w_i',v_i)$, setting $u_i=(x_1^{\delta_1}w_i',v_i)$ for each $i$ yields a computation $$\widetilde{\gamma}_{w'}=\bigl((u_0,s_0),e_1',(u_1,s_1),e_2',\dots,e_m',(u_m,s_m)\bigr)$$ 
Then, $u_0=(w,v_0^\eps)$ and $u_m=(x_1^{\delta_1},v_m)$.  By \Cref{lem-q1-q2-basics}, $v_m=v_j^\eps$ for $j=\deg(w')\leq n-1$.  As $\I_{\cntr}=\I_\eps$, it follows that: 
\begin{align*}
\tm_{\wacc_{R+}}^\F(\widetilde{\gamma}_{w'})&=\tm_{\wacc_{R+}}^\F(\gamma_{w'})\leq K(n-1)^{1+\eps} \\
\sp_{\wacc_{R+}}^\F(\widetilde{\gamma}_{w'})&=\sp_{\wacc_{R+}}^\F(\gamma_{w'})+1\leq K(n-1)+1
\end{align*}

If $\delta_1=1$, then let $\gamma'$ be the computation given by $$\bigl(((x_1^{\delta_1},v_j^\eps),q_1),(\s,\f,e_{3,x_1}),((1,v_j^{\eps}),q_2),(\s,\f,e_2),((1,v_{j+1}^\eps),q_1)\bigr)$$
In this case, the complexity of $\gamma'$ are given by $\tm_{\wacc_{R+}}^\F(\gamma')=1+\max(|v_j^\eps|,|v_{j+1}^\eps|)\leq 1+\size_\eps(n)$ and $\sp_{\wacc_{R+}}^\F(\gamma')=\max(1+|v_j^\eps|,|v_{j+1}^\eps|)\leq1+\size_\eps(n)$.

Conversely, if $\delta_1=-1$, then note that since $w\in L_A^{R+}$, $\deg(w')\geq1$.  Hence, $j\geq1$, so that there exists a computation $\gamma'$ given by
$$\bigl(((x_1^{\delta_1},v_j^\eps),q_1),(\f,\s,e_2),((x_1^{\delta_1},v_{j-1}^\eps),q_2),(\s,\f,e_{3,x_1}^{-1}),((1,v_{j-1}^\eps),q_1)\bigr)$$
As in the previous case, $\gamma'$ satisfies $\tm_{\wacc_{R+}}^\F(\gamma')=\max(|v_j^\eps|,|v_{j-1}^\eps|)+1\leq\size_\eps(n)+1$ and $\sp_{\wacc_{R+}}^\F(\gamma')=\max(1+|v_j^\eps|,1+|v_{j-1}^\eps|)\leq1+\size_\eps(n)$.

In either case, $\widetilde{\gamma}_{w'}$ and $\gamma'$ can be concatenated to yield a $(q_1,q_2)$-computation $\gamma_w$ accepting $w$ and satisfying:
\begin{align*}
\tm_{\wacc_{R+}}^\F(\gamma_w)&\leq K(n-1)^{1+\eps}+1+\size_\eps(n) \\
\sp_{\wacc_{R+}}^\F(\gamma_w)&\leq\max(K(n-1)+1,1+\size_\eps(n))
\end{align*}

By \Cref{cor-counter-eps}, there exists $C\in\N$ such that $\size_\eps(k)\leq Ck^\eps+C$ for all $k\in\N$.  So, taking $K\geq C+1$, it follows that $\sp_{\wacc_{R+}}^\F(\gamma_w)\leq Kn+K$.

Now, it follows from Bernoulli's Inequality that:
\begin{align*}
(n-1)^{1+\eps}&=(n-1)n^\eps\bigl(1-1/n\bigr)^\eps\leq(n-1)n^\eps(1-\eps/n) \\
&\leq n^{1+\eps}-n^\eps-(n-1)n^\eps(\eps/n)
\end{align*}
So, assuming $K\geq C$, $\tm_\wacc^\F(\gamma_w)\leq Kn^{1+\eps}-K\eps(1-1/n)n^\eps+C+1$. Then, taking $K\geq4C\eps^{-1}$, it follows that $K\eps(1-1/n)n^\eps\geq 2Cn^\eps\geq2C$ for $n\geq2$, and hence $\tm_{\wacc_{R+}}^\F(\gamma_w)\leq Kn^{1+\eps}$.

\end{proof}

\begin{lemma}\label{lem-deg-minimal}

Let $\gamma=\bigl((u_0,s_0),e_1',(u_1,s_1),e_2',\dots,e_m',(u_m,s_m)\bigr)$ be a computation of $\wacc_{R+}$ with $e_i'=(\sf_i,\sf_i',e_i'')$ for all $i$.  Suppose the computation path of $\gamma$ has the minimal length amongst all computations between $(u_0,s_0)$ and $(u_m,s_m)$.  Then for all $i$:

\begin{enumerate}

\item $\sf_i\neq\sf_i'$

\item If $\sf_i=\sf_{i+1}'$, then $e_i''\neq e_{i+1}''$

\item If $e_i''\in E_{\wacc_{R+}}^0$, then $e_{i+1}''\neq(e_i'')^{-1}$

\end{enumerate}

\end{lemma}

\begin{proof}

Suppose there exist $i,j\in\overline{0,m}$ with $i\neq j$ such that $(u_i,s_i)=(u_j,s_j)$.  Assuming without loss of generality that $i<j$, the generic path $e_1',\dots,e_i',e_{j+1}',\dots,e_m'$ supports a computation between $(u_0,s_0)$ and $(u_m,s_m)$.  But this path has length strictly less than $m$, contradicting the hypothesis on $\gamma$.

Now, if $\sf_i=\sf_i'$, then the $\IO$-relation of counters and those of actions of degree 0 imply $(u_{i-1},s_{i-1})=(u_i,s_i)$.  Hence, (1) must hold.

Next, suppose $\sf_i=\sf_{i+1}'$ and $e_i''=e_{i+1}''=e$.  Letting $\act_e=\lab(e)$, then $(u_{i-1},u_i)\in\IO_{\act_e}^{\sf_i,\sf_i'}$ and $(u_i,u_{i+1})\in\IO_{\act_e}^{\sf_{i+1},\sf_{i+1}'}$.  As a consequence of (1), $\sf_i'=\sf_{i+1}$, so that \Cref{lem-lio-io}(3) implies that $(u_{i-1},u_{i+1})\in\IO_{\act_e}^{\sf_i,\sf_{i+1}'}$.  The generic edge $e'=(\sf_i,\sf_{i+1}',e)$ then supports a computation between $(u_{i-1},s_{i-1})$ and $(u_{i+1},s_{i+1})$.  But then the generic path $e_1',\dots,e_{i-1}',e',e_{i+2}',\dots,e_m'$ of length $m-1$ supports a computation between $(u_0,s_0)$ and $(u_m,s_m)$, again contradicting the hypothesis on $\gamma$.  Hence, (2) must hold.

Finally, suppose $e_i''\in E_{\wacc_{R+}}^0$ and $e_{i+1}''=(e_i'')^{-1}$.  Then, as no edge of $E_{\wacc_{R+}}^0$ is a loop, (1) implies that $\sf_i'=\sf_{i+1}'\neq\sf_i$.  Let $t$ be the transformation defining $\lab(e_i'')$ and $\delta\in\{\pm1\}$ such that $\delta=1$ if and only if $\sf_i=\s$.  Then $u_i=t^\delta u_{i-1}$ and $u_{i+1}=t^{-\delta}u_i=u_{i-1}$.  But then $(u_{i-1},s_{i-1})=(u_{i+1},s_{i+1})$, so that (3) must hold.

\end{proof}

\begin{lemma}\label{lem-io-deg}
	The proper operation $\wacc_{R+}$ is a weak acceptor of $L^{R+}_A$.  Moreover, the operation satisfies $\IO_{\wacc_{R+}}^{s,f}=\{(w,1)\mid w\in L_A^{R+}\}$ and $\WTMSP_{R+, \varepsilon} \defeq \WTMSP_{\wacc_{R+}} \preceq (n^{1+\varepsilon}, n)$.
\end{lemma}
\begin{proof}

Let $(w,w')\in\IO_{\wacc_{R+}}^{s,f}$ and $\gamma=\bigl((u_0,s_0),e_1',(u_1,s_1),e_2',\dots,e_m',(u_m,s_m)\bigr)$ be an input-output computation with minimal length computation path realizing $(w,w')\in\IO_{\wacc_{R+}}^{s,f}$.

For all $i\in\overline{m}$, let $e_i'=(\sf_i,\sf_i',e_i'')$ with $e_i''\in E_{\wacc_{R+}}$.  Similarly, for all $i\in\overline{0,m}$, let $u_i=(w_i,v_i)$, where $w_i=u_i(a)\in F(A)$ and $v_i=u_i(\I_\eps)=v_{j(i)}^\eps\in\VAL_\eps$.  So, $w_0=w$, $w_m=w'$, and $v_0=v_m=v_0^\eps$.

As the vertices of $\SF_{\wacc_{R+}}$ are in distinct components of $\G_{\wacc_{R+}}-\{q_1\}$, there must exist $j\in\overline{m}$ such that $s_{j-1}=q_1$.  Letting $j$ be the maximal such index, $e_j''=e_4$ and $e_i\in\{e_5^{\pm1},e_6^{\pm1}\}$ for all $i>j$. Then, the guard fragment labelling $e_4$ implies $w_j=1$.  Hence, as the actions of $e_5$ and $e_6$ do not mention $a$, $w'=1$.

Similarly, the vertices of $\SF_{\wacc_{R+}}$ are in distinct components of $\G_{\wacc_{R+}}-\{q_3\}$, and so there exists a minimal $k\in\overline{m-1}$ such that $s_{k+1}=q_3$.  By \Cref{lem-deg-minimal}, the computation 
$$\gamma'=\bigl((u_1,s_1),e_2',\dots,e_k',(u_k,s_k)\bigr)$$ 
is a $(q_1,q_2)$-computation and $e_{k+1}''=e_4^{\pm1}$.  As above, the guard fragment labelling $e_4$ then implies that $w_k=1$.  Further, as it must hold that $e_1''=e_1^{\pm1}$, $u_1=(w_1,v_1)=(w,v_0^\eps)$.  As a result, $\gamma'$ is a $(q_1,q_2)$-computation accepting $w$, so that $w\in L_A^{R+}$ by \Cref{lem-q1-q2-iff}. 

Finally, for $w\in L_A^{R+}$, letting $\rho'$ be the generic path supporting the computation $\gamma_w$ in \Cref{lem-q1-q2-iff}, the generic path $\rho''_s=(\s,\f,e_1),\rho',(\s,\f,e_4)$ supports a computation between the configurations $((w,v_0^\eps),(\wacc_{R+})_s)$ and $((w,v_d^\eps),q_3)$ for some $d\in\N$.  By \Cref{lem-q1-q2-basics}, $\deg(w)=d\geq0$, and so $\rho_f''=(\f,\s,e_5)^d,(\s,\f,e_6)$ supports a computation between the configurations $((1,v_d^\eps),q_3)$ and $((1,v_0^\eps),(\wacc_{R+})_s)$.  Hence, $\rho''=\rho''_s,\rho''_f$ supports a computation $\sigma_w$ of $\wacc_{R+}$ realizing $(w,1)\in\IO_{\wacc_{R+}}^{s,f}$.

Thus, $\IO_{\wacc_{R+}}^{s,f}=\{(w,1)\mid w\in L_A^{R+}\}$, and so $\wacc_{R+}$ is a weak acceptor for $L_A^{R+}$.

Now, let $\F$ be the estimating set for $\wacc_{R+}$ given by $\F_\cntr=(n,n)$ as in \Cref{lem-q1-q2-iff}.  By \Cref{lem-cntr-complexities} and \Cref{lem-wtmsp}, it suffices to find $L\in\N$ such that $\tm_{\wacc_{R+}}^\F(\sigma_w)\leq L|w|^{1+\eps}+L$ and $\sp_{\wacc_{R+}}^\F(\sigma_w)\leq L|w|+L$ for every $w\in L_A^{R+}$.

By construction and \Cref{lem-q1-q2-iff}:
\begin{align*}
\tm_{\wacc_{R+}}^\F(\sigma_w)&\leq 1+K|w|^{1+\eps}+1+\sum_{j=1}^d\max(|v_j^\eps|,|v_{j-1}^\eps|)+1 \\
&\leq K|w|^{1+\eps}+3+d\size_\eps(d)\leq 2K|w|^{1+\eps}+3 \\
\sp_{\wacc_{R+}}^\F(\sigma_w)&\leq K|w|
\end{align*}
Hence, the desired inequalities are satisfied for $L\geq\max\{2K,3\}$.

\end{proof}

\subsection{Weak acceptors for positive words} \

\medskip

For $\eps\in(0,1)$ the proper operation $\wacc_+^{1,\eps}=\waccept^{1,\eps}_+(a)$ is defined in \Cref{fig-POSCHECK1} below.  As indicated by the naming, $\I_{\wacc_+^{1,\eps}}^\val=\emptyset$ and $\I_{\wacc_+^{1,\eps}}^\ext=\{a\}$.  Further, it can be deduced from the actions labelling the edges of the graph that $\I_{\wacc_+^{1,\eps}}^\int=\{b\}$ and $\A_{\wacc_+^{1,\eps}}(a)=\A_{\wacc_+^{1,\eps}}(b)=A$.

\begin{figure}[hbt]
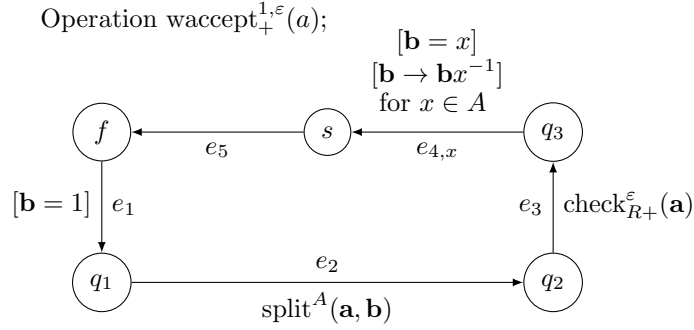

	\centering
	\include{POSCHECK1}
	\caption{A weak acceptor for the language $L^{+}_A$.}
	\label{fig-POSCHECK1}       
\end{figure} 

Since the positive actions labelling $e_2$ and $e_3$ are formed from proper operations, the convention is followed that they are not declared in the title.  However, it should be noted that $\check^\eps_{R+}$ is yet undefined.  As \Cref{lem-io-deg} established that the proper operation $\wacc^\eps_{R+}$ is a weak acceptor of $L_A^{R+}$ satisfying certain complexity bounds, \Cref{prop-acc-ch-wacc} yields a corresponding checker satisfying the same complexity bounds, which is used here to define the action labelling $e_3$.

As $e_5$ is labelled with the trivial action, it follows from the definition of the actions labelling $e_1$ and $e_{4,x}^{-1}$ that $\wacc_+^{1,\eps}$ satisfies condition (O).

Throughout the rest of the discussion of the operation, let $\F$ be the estimating set for $\wacc_+^{1,\eps}$ given by $\F_\splt=(n,n)$ and $\F_{\ch_{R+}}=(n^{1+\eps},n)$.

A computation $\gamma=\bigl((u_0,s_0),e_1',(u_1,s_1),e_2',\dots,e_m',(u_m,s_m)\bigr)$ of $\wacc_+^{1,\eps}$ is called a \emph{circuit computation} if it is an input-output computation satisfying $s_i=(\wacc_+^{1,\eps})_s$ if and only if $i=0,m$.  So, $\gamma$ is a computation realizing $(u_0(a),u_m(a))\in\IO_{\wacc_+^{1,\eps}}^{s,s}$. 

Let $\CIO_{\wacc_+^{1,\eps}}^{s,s}$ be the subset of $\IO_{\wacc_+^{1,\eps}}^{s,s}$ consisting of all pairs realized by circuit computations.  Note that, as trivial computations are circuit computations, $\CIO_{\wacc_+^{1,\eps}}^{s,s}\supseteq\Id$.  Further, as the inverse of a circuit computation is a circuit computation, $\CIO_{\wacc_+^{1,\eps}}^{s,s}$ is symmetrically closed.

The following statement provides an analogue of \Cref{lem-deg-minimal} which will be of similar utility:

\begin{lemma}\label{lem-pos-check-circuit-minimal}

Let $\gamma=\bigl((u_0,s_0),e_1',(u_1,s_1),e_2',\dots,e_m',(u_m,s_m)\bigr)$ be a circuit computation of $\wacc_+^{1,\eps}$ with $e_i'=(\sf_i,\sf_i',e_i'')$ for all $i$.  Suppose $m$ is minimal, i.e the computation path of $\gamma$ has the minimal length amongst all circuit computations realizing $(u_0(a),u_m(a))\in\IO_{\wacc_+^{1,\eps}}^{s,s}$.  Then for all $i$:

\begin{enumerate}

\item $\sf_i\neq\sf_i'$

\item $e_i''\neq e_{i+1}''$

\item If $e_i''\in E_{\wacc_+^{1,\eps}}^0$, then $e_{i+1}''\neq(e_i'')^{-1}$

\end{enumerate}

\end{lemma}

\begin{proof}

As in the proof of \Cref{lem-deg-minimal}, if there exist $i,j\in\overline{0,m}$ with $i<j$ such that $(u_i,s_i)$ and $(u_j,s_j)$ are the same configuration, then the generic path $\rho'=e_1',\dots,e_i',e_{j+1}',\dots,e_m'$ supports a computation between $(u_0,s_0)$ and $(u_m,s_m)$.  But then this is a circuit computation realizing $(u_0(a),u_m(a))\in\IO_{\wacc_+^{1,\eps}}^{s,s}$ which is supported by a generic path of length at most $m-1$, contradicting the choice of $\gamma$.  

By the definition of the $\IO$-relation for checkers and actions of degree 0, it then follows that (1) holds if $e_i''\neq e_2$.

Suppose $e_i''=e_{i+1}''=e$.  Then, as $\G_{\wacc_{R+}}$ contains no loops, $\sf_i'=\sf_{i+1}$.  So, by \Cref{lem-lio-io}(3), $(u_{i-1},u_{i+1})\in\IO_{\lab(e)}^{\sf_i,\sf_{i+1}'}$.  Hence, the same argument as presented for \Cref{lem-deg-minimal}(2) shows that (2) must hold.

Now, let $u_i(a)=w_i$ and $u_i(b)=v_i$ for all $i$.

Suppose $e_i'=(\s,\s,e_2)$, so that $s_{i-1}=q_1=s_i$.  Since $\gamma$ is a circuit computation, $i\neq0,m$.  So, by (2), $e_{i-1}'',e_{i+1}''\in\{e_1^{\pm1}\}$.  In particular, since (1) holds when $e_i''\neq e_2$, then the generic edges are given by $e_{i+1}'\in\{(\f,\s,e_1),(\s,\f,e_1^{-1})\}$ and $e_{i-1}'\in\{(\s,\f,e_1),(\f,\s,e_1^{-1})\}$.  The action labelling $e_1$ then implies that $v_{i-1}=1=v_i$.  But \Cref{lem-io-split} then yields $w_{i-1}=w_{i-1}v_{i-1}=w_iv_i=w_i$, so that $(u_{i-1},s_{i-1})=(u_i,s_i)$.

On the other hand, suppose $e_i'=(\f,\f,e_2)$, so that $s_{i-1}=q_2=s_i$.  Then, as above, (2) implies that $e_{i-1}'=(\f,\s,e_3)$ and $e_{i+1}'=(\s,\f,e_3)$.  As a consequence, $(u_{i-2},s_{i-2})=(u_{i-1},q_3)$, $(u_{i+1},s_{i+1})=(u_i,q_3)$, and $w_{i-2},w_{i+1}\in L_A^{R+}$.  Further, there then must exist $x,y\in A$ such that $e_{i-2}''=e_{4,x}^{\pm1}$ and $e_{i+2}''=e_{4,y}^{\pm1}$.  In particular, this yields $s_{i-3}=(\wacc_+^{1,\eps})_s=s_{i+2}$, and thus $i=3$ and $m=5$ by the definition of circuit computation.  Moreover, the guard of the action labelling $e_{4,z}$ for $z\in A$ enforces $v_{i-2}=x$ and $v_{i+1}=y$.  So, by \Cref{lem-io-split}, 
$$w_{i-2}x=w_{i-1}(a)v_{i-1}=w_iv_i=w_{i+1}y$$  
Since $w_{i-2},w_{i+1}\in L_A^{R+}$, both $w_{i-2}x$ and $w_{i+1}y$ must be reduced words, and so are equivalent.  But then $x=y$, so that $(u_{i-1},s_{i-1})=(u_i,s_i)$.

Hence, condition (1) must hold, and so condition (3) is proved in the same manner as in \Cref{lem-deg-minimal}.

\end{proof}

\begin{lemma}\label{lem-pos-check-circuit}

$\CIO_{\wacc_+^{1,\eps}}^{s,s}=\Id\cup\{(w,wx),(wx,w)\mid w\in L_A^{R+},x\in A\}$.
Moreover, for all $w\in L_A^{R+}$ and all $x\in A$, there exists a circuit computation $\gamma_{w,x}$ realizing $(w,wx)\in\IO_{\wacc_+^{1,\eps}}^{s,s}$ satisfying the bounds $\tm_{\wacc_+^{1,\eps}}^{\F}(\gamma_{w,x})\leq2|w|^{1+\eps}+4$ and $\sp_{\wacc_+^{1,\eps}}^{\F}(\gamma_{w,x})\leq|w|+1$.

\end{lemma}

\begin{proof}

Let $(w',w)\in\CIO_{\wacc_+^{1,\eps}}^{s,s}$ such that $w\neq w'$ and let 
$$\gamma=\bigl((u_0,s_0),e_1',(u_1,s_1),e_2',\dots,e_m',(u_m,s_m)\bigr)$$ 
be a circuit computation realizing $(w',w)\in\IO_{\wacc_+^{1,\eps}}^{s,s}$ whose computation path is of minimal length.  For all $i$, let $e_i'=(\sf_i,\sf_i',e_i'')$, $u_i(a)=w_i$, and $u_i(b)=v_i$.

First, suppose $e_1''=e_5^{\pm1}$.  Then, by \Cref{lem-pos-check-circuit-minimal}, $u_1=u_0$ and $s_1=(\wacc_+^{1,\eps})_f$.  \Cref{lem-pos-check-circuit-minimal} then implies that $e_2''=e_1^{\pm1}$, so that $u_2=u_1$ and $s_2=q_1$.  Similarly, it follows that $e_3'=(\s,\f,e_2)$, so that $s_3=q_2$.  Moreover, by \Cref{lem-io-split}, $w_3v_3=w_2=w'$.  Then, as above, $e_4'=(\s,\f,e_3)$, so that $(u_4,s_4)=(u_3,q_3)$.  In particular, the definition of $\IO_{\ch_{R+}}$ yields $w_3\in L_A^{R+}$.  

Then, $e_5''=e_{4,x}^{\pm1}$ for some $x\in A$, so that $s_5=(\wacc_+^{1,\eps})_s$, $v_4=v_3=x$, and $u_5=(w_4,1)=(w_3,1)$.  By the definition of a circuit computation, $m=5$, so that $w_3=w$.  Hence, $w\in L_A^{R+}$ and $w'=w_3v_3=wx$, i.e $(w',w)=(wx,w)$ with $w\in L_A^{R+}$.

Conversely, suppose $e_1''=e_{4,x}^{\pm1}$ for some $x\in A$.  Then, as above \Cref{lem-pos-check-circuit-minimal} implies that $m=5$, $e_2''=e_3$, $e_3''=e_2$, $e_4''=e_1^{\pm1}$, and $e_5''=e_5^{\pm1}$.  But then the inverse computation $\gamma^{-1}$ is a circuit computation whose computation path is minimal amongst all circuit computations realizing $(w,w')\in\IO_{\wacc_+^{1,\eps}}^{s,s}$ and whose first generic edge is constructed from $e_5^{\pm1}$, so that the above arguments imply the existence of $x\in A$ such that $(w',w)=(w',w'x)$ with $w'\in L_A^{R+}$.

Thus, as $\CIO_{\wacc_+^{1,\eps}}$ is symmetrically closed, it suffices to demonstrate the computation $\gamma_{w,x}$ as in the statement for any $w\in L_A^{R+}$ and $x\in A$.

Toward this goal, fix $w\in L_A^{R+}$ and $x\in A$ and define the generic path:
$$\rho_{w,x}=(\s,\f,e_{4,x}^{-1}),(\f,\s,e_3),(\f,\s,e_2),(\s,\f,e_1^{-1}),(\s,\f,e_5^{-1})$$
Denoting $\rho_{w,x}=e_1',\dots,e_5'$, it follows from the definition of a checker and \Cref{lem-io-split} that $\rho_{w,x}$ supports the circuit computation:
\begin{align*}
\gamma_{w,x}=\bigl(((w,1),&(\wacc_+^{1,\eps})_s),e_1',((w,x),q_1),e_2',((w,x),q_2),e_3',((wx,1),q_1), \\
&e_4',((wx,1),(\wacc_+^{1,\eps})_f),e_5',((wx,1),(\wacc_+^{1,\eps})_s)\bigr)
\end{align*}
Hence, $\gamma_{w,x}$ is a circuit computation realizing $(w,wx)\in\IO_{\wacc_+^{1,\eps}}^{s,s}$, so that $\CIO_{\wacc_+^{1,\eps}}^{s,s}$ is as in the statement.

Moreover, the estimated complexity of $\gamma_{w,x}$ is:
\begin{align*}
\tm_{\wacc_+^{1,\eps}}^\F(\gamma_{w,x})&=1+|w|^{1+\eps}+(|w|+1)+1+1\leq2|w|^{1+\eps}+4 \\
\sp_{\wacc_+^{1,\eps}}^\F(\gamma_{w,x})&=|w|+1
\end{align*}

\end{proof}

\begin{lemma}\label{lem-pos-check1}
The operation $\wacc_+^{1,\eps}$ is a weak acceptor of $L^+_A$ with $\WTMSP_{\waccept_{+}^{1, \varepsilon}} \preceq (n^{2+\varepsilon}, n)$.
\end{lemma}

\begin{proof}

	As $e_5$ is labelled by the trivial action, $\IO_{\wacc_+^{1,\eps}}^{s,s}=\IO_{\wacc_+^{1,\eps}}^{s,f}=\IO_{\wacc_+^{1,\eps}}^{f,f}$.
	
	Let $w\in L_A$ such that $(w,1)\in\IO_{\wacc_+^{1,\eps}}^{s,s}$ and fix a computation 
	$$\gamma=\bigl((u_0,s_0),e_1',(u_1,s_1),e_2',\dots,e_m',(u_m,s_m)\bigr)$$ 
	realizing $(w,1)\in\IO_{\wacc_+^{1,\eps}}^{s,s}$. 
	
	Let $i_0,\dots,i_t\in\overline{m}$ be the indices satisfying $s_{i_j}=(\wacc_+^{1,\eps})_s$ and such that $i_{j-1}<i_j$ for all $j\in\overline{t}$.  So, $i_0=0$ and $i_t=m$.  Then, $\gamma$ can be viewed as the concatenation of the circuit computations
	$$\gamma_j=\bigl((u_{i_{j-1}},s_{j-1}),e_{i_{j-1}+1}',\dots,e_{i_j}',(u_{i_j},s_{i_j})\bigr)$$
	for $j\in\overline{t}$.
	
	Now, let $w_j=u_{i_j}(a)$ for all $j\in\overline{0,t}$.  So, $w_0=w$ and $w_t=1$.  Then, $\gamma_j$ is a circuit computation realizing $(w_{j-1},w_j)\in\IO_{\wacc_+^{1,\eps}}^{s,s}$, so that $(w_{j-1},w_j)\in\CIO_{\wacc_+^{1,\eps}}^{s,s}$.  Such a $\gamma_j$ is called a \emph{maximal circuit subcomputation} of $\gamma$.
	
	It can be assumed without loss of generality that $w_{j-1}\neq w_j$ for all $j\in\overline{t}$, as if so then $\gamma_j$ can be `removed' from $\gamma$ to produce a computation realizing $(w,1)\in\IO_{\wacc_+^{1,\eps}}^{s,s}$ with strictly fewer maximal circuit subcomputations.  Similarly, it may be assumed without loss of generality that $w_i\neq w_j$ for all distinct $i,j\in\overline{0,t}$.
	
	So, by \Cref{lem-pos-check-circuit}, for every $j\in\overline{t}$ there exists $x_j\in A$ and $\delta_j\in\{\pm1\}$ such that $w_j=w_{j-1}x_j^{\delta_j}$.  Hence, for all $j\in\overline{t}$, $1=w_t=w_{j-1}x_j^{\delta_j}\dots x_t^{\delta_t}$, and thus $w_{j-1}=x_t^{-\delta_t}\dots x_j^{-\delta_j}$.
	
	Note that, as it is assumed that $w_i\neq w_j$ for all distinct $i,j\in\overline{0,t}$, the word $x_1^{\delta_1}\dots x_t^{\delta_t}$ must be reduced.
	
	Suppose there exists $j\in\overline{t}$ such that $\delta_j=1$.  Then, letting $k$ be the maximal index for which this holds, $w_k\in L_A^+$ and $\gamma_k$ is a circuit computation realizing $(w_{k-1},w_k)=(w_kx_k^{-1},w_k)\in\IO_{\wacc_+^{1,\eps}}^{s,s}$.  By \Cref{lem-pos-check-circuit}, $w_{k-1}\in L_A^{R+}$.  In particular, since $w_tx_t^{-1}=x_t^{-1}\notin L_A^{R+}$, it must hold that $k<t$.  As a result, it follows that $w_{k-1}=w_kx_k^{-1}=x_t\dots x_{k+1}x_k^{-1}\in L_A^{R+}$.  But then there must be a cancellation in $x_t\dots x_k^{-1}=x_t^{-\delta_t}\dots x_k^{-\delta_k}$, and so a cancellation in the word $x_1^{\delta_1}\dots x_t^{\delta_t}$.
	
	Hence, $\delta_j=-1$ for all $j\in\overline{t}$, and thus $w=w_0=x_t\dots x_1\in L_A^+$.
	
	Conversely, fix $w=y_1\dots y_t\in L_A^+$ and set $w_j$ as the prefix of $w$ with $|w_j|=j$, i.e $w_j=y_1\dots y_j$ for all $j\in\overline{t}$ with $w_0=1$.  Then, since $w_j\in L_A^+\subseteq L_A^{R+}$ for all $j\in\overline{t}$, \Cref{lem-pos-check-circuit} yields a circuit computation $\gamma_{w_{j-1},y_j}$ realizing $(w_{j-1},w_j)\in\IO_{\wacc_+^{1,\eps}}^{s,s}$ and satisfying 
	\begin{align*}
	\tm_{\wacc_+^{1,\eps}}^\F(\gamma_{w_{j-1},y_j})&\leq 2|w_{j-1}|^{1+\eps}+4= 2(j-1)^{1+\eps}+4 \\ 
	\sp_{\wacc_+^{1,\eps}}^\F(\gamma_{w_{j-1},y_j})&\leq|w_{j-1}|+1\leq j
	\end{align*}
	So, letting $\gamma_w$ be the concatenation of the computations $\gamma_{w_{j-1},y_j}$ for $j\in\overline{t}$, $\gamma_w$ is a computation realizing $(1,w)\in\IO_{\wacc_+^{1,\eps}}^{s,s}$ and satisfying:
	\begin{align*}
	\tm_{\wacc_+^{1,\eps}}^\F(\gamma_w)&\leq\sum_{j=1}^t 2(j-1)^{1+\eps}+4\leq2t^{2+\eps}+4t\leq 6t^{2+\eps} \\
	\sp_{\wacc_+^{1,\eps}}^\F(\gamma_w)&=\max_{j\in\overline{t}} j\leq t
	\end{align*}
	
	Now, let $\gamma'$ be the computation formed by concatenating the inverse computation $\gamma_w^{-1}$ with the computation realizing $(1,1)\in\IO_{\wacc_+^{1,\eps}}^{s,f}$ and supported by the single generic edge $(\s,\f,e_5)$.  Then, $\gamma_w'$ realizes $(w,1)\in\IO_{\wacc_+^{1,\eps}}^{s,f}$ and satisfies $\tm_{\wacc_+^{1,\eps}}^\F(\gamma_w')\leq 6t^{2+\eps}+1$ and $\sp_{\wacc_+^{1,\eps}}^\F(\gamma_w')\leq t$.
	
	Hence, $\wacc_+^{1,\eps}$ is indeed a weak acceptor for $L_A^+$ and, since $t=|w|$, it follows from the bounds on the complexity of $\gamma_w'$ yield $\WTMSP_{\wacc_+^{1,\eps}}^\F\preceq(n^{2+\eps},n)$.  Thus, the statement follows from \Cref{lem-io-split}, \Cref{lem-io-deg}, \Cref{prop-acc-ch-wacc}, and \Cref{lem-wtmsp}.
	
\end{proof}

\subsection{Improved Weak acceptors for positive words} \

\medskip

The weak acceptor $\wacc_+^{1,\eps}$ is now used as the base for an iterative construction of weak acceptors of $L_A^+$.  These operations are called $\wacc_+^{d+1,\eps}=\waccept_+^{d+1,\eps}(a)$ for all $d\in\N$, as defined below in \Cref{fig-POSCHECK2}.
\begin{figure}[hbt]
	\centering
	\include{POSCHECK2}
	\caption{An efficient weak acceptor for the language $L^{+}_A$.}
	\label{fig-POSCHECK2}       
\end{figure} 

It is the inductive hypothesis of this construction that $\wacc_+^{d,\eps}$ is a weak acceptor of $L_A^+$ satisfying $\WTMSP_{\wacc_+^{d,\eps}}\preceq(n^{\frac{d+1}{d}+\eps},n)$.  Note that \Cref{lem-pos-check1} provides the base case of this induction, while \Cref{lem-pos-check2} provides the proof of the inductive step.

As with the iterative construction of the counters $\Counter_{d+1}$ in \Cref{sec-counters}, these weak acceptors will prove to be more `efficient' than $\waccept_+^{1,\eps}$, as will be made precise in \Cref{lem-pos-check2} and \Cref{cor-counter-eps}.

Note that the edge $e_4$ is labelled by an action which is as yet undefined, namely the action $\acc_+^{d,\eps}=\accept_+^{d,\eps}$.  However, as in the context of $\waccept_+^{1,\eps}$ above, it is implicitly understood that this action corresponds to the acceptor of $L_A^+$ corresponding to the iteratively constructed weak acceptor $\waccept_+^{d,\eps}$ arising from \Cref{prop-acc-ch-wacc}.

Per the naming conventions, the hardware for $\waccept_+^{d+1,\eps}$ is taken to be the same as that of $\waccept_+^{1,\eps}$ for all $d$.  That is, $\I_{\wacc_+^{d+1,\eps}}=\I_{\wacc_+^{d+1,\eps}}^\ext\sqcup\I_{\wacc_+^{d+1,\eps}}^\int$, with $\I_{\wacc_+^{d+1,\eps}}^\ext=\{a\}$ and $\I_{\wacc_+^{d+1,\eps}}^\int=\{a\}$, while $\A_{\wacc_+^{d+1,\eps}}(a)=\A_{\wacc_+^{d+1,\eps}}(b)=A$.  As $e_6$ is labelled with the trivial action while $e_1,e_5^{-1}$ contain the guard fragment $[b=1]$, it is immediate that $\waccept_+^{d+1,\eps}$ satisfies condition (O).

Let $\F^{d+1}$ be the estimating set for $\wacc_+^{d+1,\eps}$ given by $\F^{d+1}_\splt=(n,n)$, and $\F^{d+1}_{\ch_{R+}}=(n^{1+\eps},n)$, and $\F^{d+1}_{\acc_+^{d,\eps}}=(n^{\frac{d+1}{d}+\eps},n)$.  By \Cref{lem-io-split}, \Cref{lem-io-deg}, \Cref{prop-acc-ch-wacc}, and the inductive hypothesis, $\TMSP_\inst\preceq\F^{d+1}_\inst$ for each instance.

Hence, by \Cref{lem-wtmsp}, to prove \Cref{lem-pos-check2} it suffices to prove that $\waccept_+^{d+1,\eps}$ is a weak acceptor of $L_A^+$ satisfying $\WTMSP_{\wacc_+^{d+1,\eps}}^{\F^{d+1}}\preceq(n^{\frac{d+2}{d+1}+\eps},n)$.

As in the discussion of $\wacc_+^{1,\eps}$ in the previous section, an input-output computation $$\gamma=\bigl((u_0,s_0),e_1',(u_1,s_1),e_2',\dots,e_m',(u_m,s_m)\bigr)$$ of $\wacc_+^{d+1,\eps}$ is called a \emph{circuit computation} if $s_i=(\wacc_+^{d+1,\eps})_s$ if and only if $i=0,m$.  

Let $\rho=e_1',\dots,e_m'$ be a generic path with $e_i'=(\sf_i,\sf_i',e_i'')$ for all $i$.  Then $\rho$ is called a \emph{proper generic path} if $\sf_i\neq\sf_i'$ for all $i$.  As such, a circuit computation $\gamma$ is called a \emph{proper circuit computation} if its computation path is proper.

\begin{lemma}\label{lem-pos-check2-circuit-proper}

For any non-trivial proper circuit computation $\gamma$, there exists $w\in L_A^{R+}$ and $p\in L_A^+$ such that $\gamma$ either realizes $(w,wp)\in\IO_{\wacc_+^{d+1,\eps}}^{s,s}$ or realizes $(wp,w)\in\IO_{\wacc_+^{d+1,\eps}}^{s,s}$.  Moreover, for every $w\in L_A^{R+}$ and $p\in L_A^+$, there exists a proper circuit computation $\gamma_{w,p}$ realizing $(w,wp)\in\IO_{\wacc_+^{d+1,\eps}}^{s,s}$ and satisfying the bounds 
\begin{align*}
\tm_{\wacc_+^{d+1,\eps}}^{\F^{d+1}}(\gamma_{w,p})&\leq2|w|^{1+\eps}+2|p|^{\frac{d+1}{d}+\eps}+3 \\
\sp_{\wacc_+^{d+1,\eps}}^{\F^{d+1}}(\gamma_{w,p})&\leq|w|+|p|
\end{align*}

\end{lemma}

\begin{proof}

Let $\gamma=\bigl((u_0,s_0),e_1',(u_1,s_1),e_2',\dots,e_m',(u_m,s_m)\bigr)$ be a proper circuit computation.  For all $i$, let $u_i(a)=w_i$, $u_i(b)=v_i$, and $e_i'=(\sf_i,\sf_i',e_i'')$.

Suppose $e_1'\in\{(\f,\s,e_5),(\s,\f,e_5^{-1})\}$.  Then, setting $w=w_0$, $u_0=u_1=(w,1)$.  Then, as $\gamma$ is proper, it must hold that $e_2'=(\f,\s,e_4)$, $e_3'=(\f,\s,e_3)$, $e_4'=(\f,\s,e_2)$, $e_5'\in\{(\f,\s,e_1),(\s,\f,e_1^{-1})\}$, and finally $e_6'\in\{(\f,\s,e_6),(\s,\f,e_6^{-1})\}$.  In particular, this implies $s_6=(\wacc_+^{d+1,\eps})_s$, so that $m=6$.

Let $v=v_2$.  By the $\IO$-relation defining acceptors, it must follow that $v\in L_A^+$ and $w_2=w$.  Similarly, the definition of checker implies that $w\in L_A^{R+}$ and that $u_3=u_2=(w,v)$.  

Similarly, the actions of degree 0 labelling $e_6$ and $e_1$ imply that $u_4=u_5=u_6$ with $v_4=1$.  Finally, \Cref{lem-io-split} implies that $wv=w_4v_4=w_4$.  

Hence, $\gamma$ realizes $(w,wv)\in\IO_{\wacc_+^{d+1,\eps}}^{s,s}$ with $w\in L_A^{R+}$ and $v\in L_A^+$.

Otherwise, $e_1'\in\{(\s,\f,e_6),(\f,\s,e_6^{-1})\}$.  The same reasoning then implies that $m=6$ and $e_6'\in\{(\s,\f,e_5),(\f,\s,e_5^{-1})\}$.  

As a result, the inverse computation $\gamma^{-1}$ is a proper circuit computation such that the first generic edge of its computation path is $(\f,\s,e_5)$ or $(\s,\f,e_5^{-1})$.  But then for $w=w_6$, it follows from the argument above that $\gamma^{-1}$ realizes $(w,wv)\in\IO_{\wacc_+^{d+1,\eps}}^{s,s}$ for $w\in L_A^{R+}$ and $v'\in L_A^+$.  Hence, $\gamma$ realizes $(wv,w)\in\IO_{\wacc_+^{d+1,\eps}}^{s,s}$.

Conversely, consider the generic path:
$$\rho=(\f,\s,e_5),(\f,\s,e_4),(\f,\s,e_3),(\f,\s,e_2),(\f,\s,e_1),(\f,\s,e_6)$$
Note that $\rho$ is a proper generic path.  Set $\rho=e_1',\dots,e_6'$.

Then, for any $w\in L_A^{R+}$ and $p\in L_A^+$, $\rho$ supports a circuit computation:
\begin{align*}
\gamma_{w,p}=\bigl(((w,1),&(\wacc_+^{d+1,\eps})_s),e_1',((w,1),q_4),e_2',((w,p),q_3),e_3',((w,p),q_2),e_4', \\
&((wp,1),q_1),e_5',((wp,1),(\wacc_+^{d+1,\eps})_f),e_6',((wp,1),(\wacc_+^{d+1,\eps})_f)\bigr)
\end{align*}
Hence, $\gamma_{w,p}$ is a proper circuit computation realizing $(w,wp)\in\IO_{\wacc_+^{d+1,\eps}}^{s,s}$ and satisfies:
\begin{align*}
\tm_{\wacc_+^{d+1,\eps}}^{\F^{d+1}}(\gamma_{w,p})&=1+|p|^{\frac{d+1}{d}+\eps}+|w|^{1+\eps}+(|w|+|p|)+1+1 \\
&\leq2|p|^{\frac{d+1}{d}+\eps}+2|w|^{1+\eps}+3 \\
\sp_{\wacc_+^{d+1,\eps}}^{\F^{d+1}}(\gamma_{w,p})&=|w|+|p|
\end{align*}

\end{proof}

The following statement is the analogue of \Cref{lem-deg-minimal} and \Cref{lem-pos-check-circuit-minimal}:

\begin{lemma}\label{lem-pos-check2-circuit-minimal}

Let $\gamma=\bigl((u_0,s_0),e_1',(u_1,s_1),e_2',\dots,e_m',(u_m,s_m)\bigr)$ be a circuit computation of $\wacc_+^{d+1,\eps}$ with $e_i'=(\sf_i,\sf_i',e_i'')$ for all $i$.  Suppose the computation path of $\gamma$ has the minimal length amongst all circuit computations realizing $(u_0(a),u_m(a))\in\IO_{\wacc_+^{d+1,\eps}}^{s,s}$.  Then for all $i$:

\begin{enumerate}

\item If $e_i''\notin\{e_2,e_4\}$, then $\sf_i\neq\sf_i'$

\item If $e_i''=e_4$ and $\sf_i=\sf_i'$, then $\sf_i=\s$

\item If $e_i''=e_2$ and $\sf_i=\sf_i'$, then $\sf_i=\f$

\item $e_i''\neq e_{i+1}''$

\item If $e_i''\in E_{\wacc_+^{d+1,\eps}}^0$, then $e_{i+1}''\neq(e_i'')^{-1}$

\end{enumerate}

\end{lemma}

\begin{proof}

As in the proofs of Lemmas \ref{lem-deg-minimal} and \ref{lem-pos-check-circuit-minimal}, for all $i,j\in\overline{0,m}$ with $i\neq j$, $(u_i,s_i)\neq(u_j,s_j)$.  Then, as in those settings, (1) follows from the definition of the $\IO$-relations for checkers and actions of degree 0.  

Similarly, if $e_i''=e_4$ and $\sf_i=\sf_i'=\f$, then the definition of the $\IO$-relation defining an acceptor yields $u_{i-1}=u_i$ and $s_{i-1}=q_4=s_i$.  So, (2) must hold.

Now, suppose $e_i''=e_2$ and $\sf_i=\sf_i'=\s$.  Then, as in the proof of \Cref{lem-pos-check-circuit-minimal}(1), it follows that $e_{i-1}''=e_{i+1}''=e_1$, so that $u_{i-1}(b)=1=u_i(b)$.  But then \Cref{lem-io-split} implies $(u_{i-1},s_{i-1})=(u_i,s_i)$.  Hence, (3) must hold.

Finally, as there are no loops in $\G_{\wacc_+^{d+1,\eps}}$, (4) and (5) follow from the same proofs as those presented for conditions (2) and (3) of \Cref{lem-pos-check-circuit-minimal}.

\end{proof}

\begin{lemma}\label{lem-pos-check2-minimal-to-proper}

For every $(w,w')\in\IO_{\wacc_+^{d+1,\eps}}^{s,s}$, there exists a realizing computation with proper computation path.

\end{lemma}

\begin{proof}

If $w=w'$, then a trivial computation realizes $(w,w')\in\IO_{\wacc_+^{d+1,\eps}}^{s,s}$.  Hence, the statement holds as the trivial generic path is vacuously proper.

Now, suppose $w\neq w'$ and let $\gamma=\bigl((u_0,s_0),e_1',(u_1,s_1),e_2',\dots,e_m',(u_m,s_m)\bigr)$ be a computation realizing $(w,w')\in\IO_{\wacc_+^{d+1,\eps}}^{s,s}$ with computation path $\rho$ of minimal length.  As in previous settings, let $w_i=u_i(a)$, $v_i=u_i(b)$, and $e_i'=(\sf_i,\sf_i',e_i'')$ for all $i$.

Let $j_0,\dots,j_t\in\overline{0,m}$ be the ascending sequence given by the condition $s_i=(\wacc_+^{d+1,\eps})_s$ if and only if $i=j_k$ for some $k\in\overline{0,t}$.  So, $j_0=0$ and $j_t=m$.

For all $i\in\overline{t}$, define the generic path $\rho_i=e_{j_{i-1}+1}',\dots,e_{j_i}'$.  Then, $\rho_i$ supports a subcomputation $\gamma_i$ of $\gamma$ which realizes $(w_{j_{i-1}},w_{j_i})\in\IO_{\wacc_+^{d+1,\eps}}^{s,s}$.  Note that, by the definition of the indices $j_0,\dots,j_r$, $\gamma_i$ is a circuit computation.

Suppose $\rho$ is not a proper generic path.  

For any index $\ell\in\overline{m}$ such that $\sf_\ell=\sf_\ell'$, there exists an index $k\in\overline{t}$ such that $j_{k-1}<\ell\leq j_k$.  Then, as $\gamma_k$ must be a circuit computation with minimal length computation path, \Cref{lem-pos-check2-circuit-minimal} implies that either $e_\ell'=(\s,\s,e_4)$ or $e_\ell'=(\f,\f,e_2)$.

Let $\ell$ be the minimal index for which $\sf_\ell=\sf_\ell'$.  Then, let $\rho_{k,1}'$ and $\rho_{k,2}'$ be the generic paths such that $\rho_k=\rho_{k,1}',e_\ell',\rho_{k,2}'$.

First, suppose $e_\ell'=(\s,\s,e_4)$.  Then, $s_{\ell-1}=s_\ell=q_3$ and $w_{\ell-1}=w_\ell$, so that the minimality of $\gamma_k$ yields $v_{\ell-1}\neq v_\ell$.  The $\IO$-relation defining acceptors then implies that $v_{\ell-1},v_\ell\in L_A^+$.  Let $u_\ell'=(w_{\ell-1},1)=(w_\ell,1)$ and $s_\ell'=(\wacc_+^{d+1,\eps})_s$.

Consider the generic paths $\rho_{k,1}''=(\s,\f,e_4),(\s,\f,e_5)$ and $\rho_{k,2}''=(\f,\s,e_5),(\f,\s,e_4)$.  Then, each are proper, with $\rho_{k,1}''$ supporting a computation between $(u_{\ell-1},s_{\ell-1})$ and $(u_\ell',s_\ell')$ and $\rho_{k,2}''$ supporting a computation between $(u_\ell',s_\ell')$ and $(u_\ell,s_\ell)$.

Conversely, suppose $e_\ell'=(\f,\f,e_2)$.  Then, by \Cref{lem-io-split}, there exists a word $w_\ell''\in L_A$ such that $w_\ell''=w_{\ell-1}v_{\ell-1}=w_\ell v_\ell$.  Set $u_\ell'=(w'',1)$ and $s_\ell'=(\wacc_+^{d+1,\eps})_s$.  In this case, define the proper generic paths $\rho_{k,1}''=(\f,\s,e_2),(\f,\s,e_1),(\f,\s,e_6)$ and $\rho_{k,2}''=(\s,\f,e_6),(\s,\f,e_1),(\s,\f,e_2)$.  Then, as above, $\rho_{k,1}''$ supports a computation between $(u_{\ell-1},s_{\ell-1})$ and $(u_\ell',s_\ell')$, while $\rho_{k,2}''$ supports a computation between $(u_\ell',s_\ell')$ and $(u_\ell,s_\ell)$.

In either case, for the generic paths $\rho_{k,1}=\rho_{k,1}',\rho_{k,1}''$ and $\rho_{k,2}=\rho_{k,2}'',\rho_{k,2}'$, it follows that $\rho_{k,1}$ supports a circuit computation $\gamma_{k,1}$ between $(u_{j_{k-1}},s_{j_{k-1}})$ and $(u_\ell',s_\ell')$, while $\rho_{k,2}$ supports a circuit computation $\gamma_{k,2}$ between $(u_\ell',s_\ell')$ and $(u_{j_k},s_{j_k})$.  

As a result, letting $\rho'$ be the generic path obtained from $\rho$ by replacing the generic edge $e_\ell'$ with the generic subpath $\rho_{k,1}'',\rho_{k,2}''$, it follows that $\rho'$ supports a computation $\gamma'$ that realizes $(w,w')\in\IO_{\wacc_+^{d+1,\eps}}^{s,s}$.  Moreover, any generic edge of $\rho'$ of the form $(\sf,\sf,e)$ must either be $(\s,\s,e_4)$ or $(\f,\f,e_2)$, while the number of such generic edges comprising $\rho'$ is one less than the number comprising $\rho$.

Thus, the statement follows by iterating the above process.

\end{proof}

\begin{lemma}\label{lem-pos-check2}
The proper operation $\wacc_+^{d+1,\eps}$ is a weak acceptor of the set of positive words $L_A^+$ and satisfies $\WTMSP_{\wacc_{+}^{d+1, \varepsilon}} \preceq (n^{\frac{d+2}{d+1}+\varepsilon}, n)$.
\end{lemma}

\begin{proof}

As in \Cref{lem-pos-check1}, the trivial action labelling the edge $e_6$ immediately implies that the $\IO$-relation does not depend on state, i.e $\IO_{\wacc_+^{d+1,\eps}}^{s,s}=\IO_{\wacc_+^{d+1,\eps}}^{s,f}=\IO_{\wacc_+^{d+1,\eps}}^{f,f}$.

Suppose $(1,w)\in\IO_{\wacc_+^{d+1}}^{s,s}$ for some $w\neq1$.  By \Cref{lem-pos-check2-minimal-to-proper}, there exists a computation $\gamma$ realizing $(1,w)\in\IO_{\wacc_+^{d+1,\eps}}^{s,s}$ whose computation path $\rho$ is proper.

As in the proof of \Cref{lem-pos-check2-minimal-to-proper}, let $\gamma=\bigl((u_0,s_0),e_1',(u_1,s_1),e_2',\dots,e_m',(u_m,s_m)\bigr)$ be a computation with $w_i=u_i(a)$, $v_i=u_i(b)$, and $e_i'=(\sf_i,\sf_i',e_i'')$.  Moreover, let $j_0,\dots,j_t\in\overline{0,m}$ be the ascending sequence of indices satisfying $s_i=(\wacc_+^{d+1,\eps})_s$ if and only if $i=j_k$ for some $k\in\overline{0,t}$.  In particular, $j_0=0$ and $j_t=m$.

Without loss of generality, suppose the value of $t$ is minimal with respect to all computations realizing $(1,w)\in\IO_{\wacc_+^{d+1,\eps}}^{s,s}$ whose computation paths are proper.  That is, $\gamma$ is comprised of the least number of proper circuit computations. Note that as $w\neq1$, it is immediate that $t\geq1$.  

For all $i\in\overline{t}$, let $\rho_i=e_{j_{i-1}+1}',\dots,e_{j_i}'$.  Then, $\rho_i$ supports a proper circuit computation $\gamma_i$ realizing $(w_{j_{i-1}},w_{j_i})\in\IO_{\wacc_+^{d+1,\eps}}^{s,s}$.  In particular, $\gamma_1$ is a proper circuit computation realizing $(1,w_{j_1})\in\IO_{\wacc_+^{d+1,\eps}}^{s,s}$.  

If $w_{j_1}=1$, then the proper generic path $\rho_2,\dots,\rho_t$ supports a computation $\gamma'$ that realizes $(1,w)\in\IO_{\wacc_+^{d+1,\eps}}^{s,s}$.  But then $\gamma'$ is comprised of the proper circuit computations $\gamma_2,\dots,\gamma_t$, contradicting the minimality of the number of proper circuit computations comprising $\gamma$.

So, $w_{j_1}\neq1$.  By \Cref{lem-pos-check2-circuit-proper}, there then exist $w'\in L_A^{R+}$ and $p\in L_A^+$ such that either $(w',w'p)=(1,w_{j_1})$ or $(w'p,w')=(1,w_{j_1})$.  If the latter is true, then $w'=p^{-1}$ for $p\in L_A^+$.  But this can only happen if $p=1$, in which case $w_{j_1}=w'=1$.  Hence, $(w',w'p)=(1,w_{j_1})$, so that $w_{j_1}=p\in L_A^+$.

Now, suppose $t\geq2$.  Then, $\gamma_2$ is a proper circuit computation realizing $(w_{j_1},w_{j_2})\in\IO_{\wacc_+^{d+1,\eps}}^{s,s}$, and hence \Cref{lem-pos-check2-circuit-proper} implies that there exists $w''\in L_A^{R+}$ and $p'\in L_A^+$ such that either $(w'',w''p')=(w_{j_1},w_{j_2})$ or $(w''p',w'')=(w_{j_1},w_{j_2})$.  In the first case, $w''=w_{j_1}\in L_A^+$, so that $w_{j_2}=w''p'\in L_A^+$; in the latter, $w''\in L_A^{R+}$ and $w''p'=w_{j_1}\in L_A^+$, so that $w_{j_2}=w''\in L_A^+$ by \Cref{lem-lplus-decomp}.

\Cref{lem-pos-check2-circuit-proper} then implies the existence of a proper circuit computation $\gamma_{1,w_{j_2}}$ that realizes $(1,w_{j_2})\in\IO_{\wacc_+^{d+1,\eps}}^{s,s}$.  Letting $\rho_2'$ be the computation path of $\gamma_{1,w_{j_2}}$, it follows that the proper generic path $\rho_2',\rho_3,\dots,\rho_t$ supports a computation $\gamma'$ realizing $(1,w)\in\IO_{\wacc_+^{d+1,\eps}}^{s,s}$.  But then $\gamma'$ is comprised of the proper circuit computations $\gamma_{1,w_{j_2}},\gamma_3,\dots,\gamma_t$, again contradicting the minimality of the number of proper circuit computations comprising $\gamma$.

Hence, $t=1$, so that $\gamma$ is itself a proper circuit computation.  But then as with the arguments above for $\gamma_1$, it follows that $w\in L_A^+$.

Conversely, for all $w\in L_A^+$, \Cref{lem-pos-check2-circuit-proper} yields a computation $\gamma_{1,w}$ realizing $(1,w)\in\IO_{\wacc_+^{d+1,\eps}}^{s,s}$.

By the symmetry of the $\IO$-relation, $$\IO_{\wacc_+^{d+1,\eps}}^{s,f}\cap\{(w,1)\mid w\in L_A\}=\IO_{\wacc_+^{d+1,\eps}}^{s,s}\cap\{(w,1)\mid w\in L_A\}=\{(w,1)\mid w\in L_A^+\}$$
so that $\wacc_+^{d+1,\eps}$ is indeed a weak acceptor for $L_A^+$.

Thus, by the inductive hypothesis and the discussion in the introduction to this section, it suffices to prove that $\WTMSP_{\wacc_+^{d+1,\eps}}^{\F^{d+1}}\preceq(n^{\frac{d+2}{d+1}+\eps},n)$.

Now, let $n\in\N-\{0\}$ and $w=x_1\dots x_n\in L_A^+$ with $x_1,\dots,x_n\in A$.  For all $\ell,r\in\overline{n}$ such that $\ell<r$, let $w_{\ell,r}=x_{\ell+1}\dots x_r$.  Further, for all $r\in\overline{n}$, let $w_r=x_1\dots x_r$, with $w_0=1$.

Then, for all $\ell,r\in\overline{0,n}$ such that $\ell<r$, \Cref{lem-pos-check2-circuit-proper} yields a proper circuit computation $\gamma_{\ell,r}\defeq\gamma_{w_\ell,w_{\ell,r}}$ realizing $(w_\ell,w_r)\in\IO_{\wacc_+^{d+1,\eps}}^{s,s}$ and satisfying:

\begin{align*}
\tm_{\wacc_+^{d+1,\eps}}^{\F^{d+1}}(\gamma_{\ell,r})&\leq 2|w_\ell|^{1+\eps}+2|w_{\ell,r}|^{\frac{d+1}{d}+\eps}+3\leq2\ell^{1+\eps}+2(r-\ell)^{\frac{d+1}{d}+\eps}+3 \\
\sp_{\wacc_+^{d+1,\eps}}^{\F^{d+1}}(\gamma_{\ell,r})&\leq|w_\ell|+|w_{\ell,r}|=\ell+(r-\ell)=r\leq n
\end{align*}

Now, let $\a=\frac{d}{d+1}$ and $t=\lceil n^\a \rceil$.  Let $k\in\N$ be the greatest positive integer such that $kt\leq n$.  Note that, as $n^{1-\a}t\geq n^{1-\a}n^\a=n$, it must hold that $k\leq n^{1-\a}$.  Then, for every $m\in\overline{k}$, it follows that $\gamma_{(m-1)t,mt}$ is a proper computation realizing $(w_{(m-1)t},w_{mt})\in\IO_{\wacc_+^{d+1,\eps}}^{s,s}$ such that:
\begin{align*}
\tm_{\wacc_+^{d+1,\eps}}^{\F^{d+1}}(\gamma_{(m-1)t,mt})&\leq 2((m-1)t)^{1+\eps}+2t^{\frac{d+1}{d}+\eps}+3 \\
\sp_{\wacc_+^{d+1,\eps}}^{\F^{d+1}}(\gamma_{(m-1)t,mt})&\leq n
\end{align*}

Finally, there exists a computation $\gamma_{kt,n}$ realizing $(w_{kt},w)\in\IO_{\wacc_+^{d+1,\eps}}^{s,s}$ and satisfying:
\begin{align*}
\tm_{\wacc_+^{d+1,\eps}}^{\F^{d+1}}(\gamma_{kt,n})&\leq 2(kt)^{1+\eps}+2t^{\frac{d+1}{d}+\eps}+3 \\
\sp_{\wacc_+^{d+1,\eps}}^{\F^{d+1}}(\gamma_{kt,n})&\leq n
\end{align*}

Hence, the computation $\gamma_w'$ obtained by concatenating $\gamma_{0,t},\dots,\gamma_{(k-1)t,kt},\gamma_{kt,n}$ is a computation realizing $(1,w)\in\IO_{\wacc_+^{d+1,\eps}}^{s,s}$ and satisfying:
\begin{align*}
\tm_{\wacc_+^{d+1,\eps}}^{\F^{d+1}}(\gamma_w')&\leq \left(\sum_{m=1}^k \tm_{\wacc_+^{d+1,\eps}}^{\F^{d+1}}(\gamma_{(m-1)t,mt})\right)+\tm_{\wacc_+^{d+1,\eps}}^{\F^{d+1}}(\gamma_{kt,n}) \\
&\leq\left(\sum_{m=1}^k 2((m-1)t)^{1+\eps}+2t^{\frac{d+1}{d}+\eps}+3\right)+2(kt)^{1+\eps}+2t^{\frac{d+1}{d}+\eps}+3 \\
&\leq 2k(kt)^{1+\eps}+2kt^{\frac{d+1}{d}+\eps}+3(k+1) \\
\sp_{\wacc_+^{d+1,\eps}}^{\F^{d+1}}(\gamma_w')&\leq\max_{m\in\overline{k}}\left(\sp_{\wacc_+^{d+1,\eps}}^{\F^{d+1}}(\gamma_{(m-1)t,mt}),\sp_{\wacc_+^{d+1,\eps}}^{\F^{d+1}}(\gamma_{kt,n})\right)\leq n
\end{align*}

As $k\leq n^{1-\a}$ and $kt\leq n$, $k(kt)^{1+\eps}\leq n^{1-\a}n^{1+\eps}$.

Further, as $t\leq n^\a+1\leq 2n^\a$ and $\frac{d+1}{d}+\eps\leq 3$, it follows that $$2kt^{\frac{d+1}{d}+\eps}\leq 2n^{1-\a}(2n^\a)^{\frac{d+1}{d}+\eps}\leq2n^{1-\a}2^3n^{\a\frac{d+1}{d}+\a\eps}\leq16n^{1-\a}n^{1+\a\eps}$$

Finally, note that $k\leq n^{1-\a}\leq n^{1-\a}n^{1+\eps}$, so that $3(k+1)\leq 3n^{1-\a}n^{1+\eps}+3$.

Hence, as $\a<1$, it follows that:
$$\tm_{\wacc_+^{d+1,\eps}}^{\F^{d+1}}(\gamma_w')\leq 21n^{1-\a}n^{1+\eps}+3=21n^{\frac{d+2}{d+1}+\eps}+3$$
As a result, letting $\rho_w$ be the computation path supporting $\gamma_w'$, the generic path $\rho_w^{-1},(\s,\f,e_6)$ supports a computation $\gamma_w$ realizing $(w,1)\in\IO_{\wacc_+^{d+1,\eps}}^{s,f}$ and satisfying:
\begin{align*}
\tm_{\wacc_+^{d+1,\eps}}^{\F^{d+1}}(\gamma_w)&\leq 21|w|^{\frac{d+2}{d+1}+\eps}+4 \\
\sp_{\wacc_+^{d+1,\eps}}^{\F^{d+1}}(\gamma_w)&\leq |w|
\end{align*}

Thus, as a trivial computation $\gamma_1$ realizes $(1,1)\in\IO_{\wacc_+^{d+1,\eps}}^{s,f}$, $\WTMSP_{\wacc_+^{d+1,\eps}}^{\F^{d+1}}\preceq(n^{\frac{d+2}{d+1}+\eps},n)$.

\end{proof}

\begin{corollary}\label{cor-acceptor-eps}
	For any alphabet $A$ and any $\varepsilon>0$ , there exists an acceptor $\accept^\varepsilon_{A+}$ of $L^+_A$ such that $\TMSP_{\accept_{A+}^\varepsilon} \preceq (n^{1+\varepsilon}, n)$.
\end{corollary}

\begin{proof}

By induction, for all $d\in\N-\{0\}$, \Cref{lem-pos-check2} yields a weak acceptor $\wacc_+^{d,\eps/2}$ of $L_A^+$ with weak-time-space complexity asymptotically bounded by $(n^{1+\frac{1}{d}+\frac{\eps}{2}},n)$.  So, letting $d\geq2/\eps$, it follows that $\wacc_+^{d,\eps/2}$ satisfies $\WTMSP_{\wacc_+^{d,\eps/2}}\preceq(n^{1+\eps},n)$.

But then \Cref{prop-acc-ch-wacc} produces a corresponding acceptor $\accept_{A+}^\eps$ of $L_A^+$ satisfying the statement.
\end{proof}

\section{Simulating Turing machines} \label{sec-turing}

\medskip

\subsection{Turing Machines} \

\medskip

As discussed in the introduction, the goal of \Cref{main-theorem} (and indeed of this manuscript) is to produce an $S$-machine (through means of an $S$-graph) which in some sense `simulates' a given multi-tape, non-deterministic Turing machine.  To this end, this section introduces the framework which will be used to discuss Turing machines.  The definition of Turing machine presented here is similar to that used in \cite{SBR} and \cite{O19}, with \Cref{lem-TM-def} a useful providing a useful tool for simplifying the considerations.

\begin{definition}[Turing machine]\label{def-turing}

A multi-tape Turing machine with $k$ tapes and $k$ heads is a six-tuple $M=\gen{A,Y,Q,\Theta,\vec{s}_1,\vec{s}_0}$, where $A$ is the input alphabet, $Y=\cup_{i=1}^kY_i$ is the tape alphabet (with $A\subseteq Y_1$), $Q=\cup_{i=1}^k Q_i$ is the set of states of the heads of the machine, $\Theta$ is a set of transitions (or commands), $\vec{s}_1$ is the $k$-vector of start states, and $\vec{s}_0$ is the $k$-vector of accept states.

A \emph{configuration of tape $i$} in $M$ is a word of the form $U=uqv$ where $u$ and $v$ are words over $Y_i$ and $q\in Q_i$.  A \emph{configuration of $M$} is then a word of the form $C=\a U_1\omega \a U_2\omega \dots \a U_k\omega$, where $\a$ and $\omega$ are fixed symbols and each $U_i=u_iq_iv_i$ is a configuration of tape $i$.  The \emph{size} of $C$ is $|C|=\sum_i(|u_i|+|v_i|)$.

In essence, a configuration of $M$ determines the status of the machine at a given instance.  One can think of this as requiring the leftmost and rightmost squares of each tape to always be marked with $\a$ and $\omega$, respectively, while the head of the $i$-th tape is placed between the squares containing the last letter of $u_i$ (or $\a$ if $u_i$ is empty) and the first letter of $v_i$ (or $\omega$ if $v_i$ is empty).  So, $u_i$ and $v_i$ are the words written to the left and the right of the head of the $i$-th tape, respectively, while $q_i$ is the state of the $i$-th head at that moment.

A transition of $M$ is determined by the states of the heads and the $2k$ letters observed by the heads.  As a result of the transition, some of the $2k$ letters can be changed to other letters in the appropriate alphabet, new squares can be added between $\a$ and $\omega$, the head can be moved one square to the left or right, and the state can be changed.

In particular, a transition is of the form $$\theta=[a_1q_1b_1\to a_1'q_1'b_1',\dots,a_kq_kb_k\to a_k'q_k'b_k']$$
where:

\begin{itemize}

\item $q_i,q_i'\in Q_i$,

\item each of $a_i,a_i'$ is empty, a letter in $Y_i$, or $\a$, and

\item each of $b_i,b_i'$ is empty, a letter in $Y_i$, or $\omega$.

\end{itemize}  
Then, a configuration $C= \a U_1 \omega \dots \a U_k \omega$ is said to be \emph{$\theta$-applicable} if and only if $U_i=u_iq_iv_i$ for all $i$ and such that:

\begin{itemize}

\item $u_i$ ends with $a_i$ if $a_i\in Y_i$,

\item $u_i$ is empty if $a_i=\a$, 

\item $v_i$ starts with $b_i$ if $b_i\in Y_i$, and

\item $v_i$ is empty if $b_i=\omega$.

\end{itemize}
In this case, the \emph{result of the application of $\theta$ to $C$} is the configuration $$C\cdot\theta=\a U_1' \omega \dots \a U_k' \omega$$ such that $U_i'=u_i'q_i'v_i'$ for all $i$, where $u_i'$ and $v_i'$ are obtained from $u_i$ and $v_i$ by `replacing' $a_i$ and $b_i$ with $a_i'$ and $b_i'$ in the natural way.

Notice that for any command $\theta$, one can construct the \emph{inverse transition} $$\theta^{-1}=[a_1'q_1'b_1'\to a_1q_1b_1, \dots, a_k'q_k'b_k' \to a_kq_kb_k]$$
This terminology is justified by the following observation:  Any configuration $C$ is $\theta$-applicable if and only if $C\cdot\theta$ is $\theta^{-1}$-applicable, with $(C\cdot\theta)\cdot\theta^{-1}=C$.

The Turing machine $M$ is then called \emph{symmetric} if it is closed under inverses, i.e $\theta\in\Theta$ if and only if $\theta^{-1}\in\Theta$.

A \emph{computation} of $M$ is a sequence $\mathcal{C}:C_0\to C_1\to C_2\to \dots \to C_n$ such that for every $i\in\overline{n}$, there exists $\theta_i\in\Theta$ such that $C_i=C_{i-1}\cdot\theta_i$.  In this case, the \emph{time} and \emph{space} of $\mathcal{C}$ are defined to be $\tm(\mathcal{C})=n$ and $\sp(\mathcal{C})=\max|C_i|$, respectively.

An \emph{input configuration} of $M$ is a configuration of the form $\a U_1 \omega\dots \a U_k \omega$ with $U_i=u_iq_i^sv_i$ such that $(q_1^s,\dots,q_k^s)=\vec{s}_1$, $v_i$ is empty for all $i\in\overline{k}$, $u_j$ is empty for $j\neq1$, and $u_1$ is a word over $A$.  Note that the input configuration is uniquely determined by the word $u_1$; accordingly, such a configuration is denoted $I(u_1)$.

An \emph{accept configuration} of $M$ is a configuration of the form $\a U_1 \omega \dots \a U_k \omega$ where $U_i=u_iq_i^av_i$ such that $(q_1^a,\dots,q_k^a)=\vec{s}_0$.

The \emph{language} of $M$, denoted $L_M$, is then the set of words $w$ over $A$ for which there exists an \emph{accepting computation} of $w$, i.e a computation $\mathcal{C}:C_0\to C_1\to\dots C_n$ such that $C_0=I(w)$ and $C_n$ is an accept configuration.

For $w\in L_M$, the \emph{time} and \emph{space} of $w$, denoted $\tm_M(w)$ and $\sp_M(w)$, are the minimal time and space of an accepting computation of $w$.  The \emph{time} and \emph{space} functions of $M$, denoted $\TM_M,\SP_M:\N\to\N$, are then defined as their analogues for $S$-machines and acceptors, i.e
\begin{align*}
\TM_M(n)&=\max\{\tm_M(w)\mid w\in L_M, \ |w|\leq n\} \\
\SP_M(n)&=\max\{\sp_M(w)\mid w\in L_M, \ |w|\leq n\}
\end{align*}
As in previous settings, these functions are growth functions, and so are susceptible to being analyzed with respect to the preorders $\preceq$ and $\preceq_1$, and so up to the corresponding equivalences.

\end{definition}

Note that the definition of Turing machine presented is, in general, of a non-deterministic nature:  While a transition is `deterministic' in the sense that it applies to a configuration in a fixed way, a configuration may be $\theta$-applicable for several distinct $\theta\in\Theta$.  As such, the time and space functions merely measure an `optimal path' for accepting a given word.

The following statement provides a simpler setting in which Turing machines can be studied.  The proof in a more general setting can be found in \cite{SBR}, while the statement can be compared with Lemma 5.1 of \cite{O19}.

\begin{lemma}[Lemma 3.1 in \cite{SBR}]\label{lem-TM-def}

For every Turing machine $M$ with language $L_M$, there exists a Turing machine $M'$ with the following properties:

\setdefaultleftmargin{22pt}{}{}{}{}{}
\begin{enumerate}[label=({\alph*})]

\item $L_{M'}=L_M$.

\item $M'$ is symmetric.

\item $\TM_{M'}\sim_1 \SP_{M'}\sim_1 \TM_M$.


\item For any accepting computation $\mathcal{C}:C_0\to C_1\to\dots\to C_n$, $C_n$ is the empty accept configuration, i.e $C_n=\a q_1^a \omega\dots \a q_k^a \omega$.

\item Every transition of $M'$ or its inverse is of one of the following two forms for some $i$:
\begin{subequations}\label{eqn}
\begin{align*}
&[q_1\omega \to q_1'\omega, \dots, q_{i-1}\omega \to q_{i-1}'\omega, \ aq_i\omega\to q_i'\omega, \ q_{i+1}\omega\to q_{i+1}'\omega, \dots] \tag{1} \\
&[q_1\omega \to q_1'\omega, \dots, q_{i-1}\omega\to q_{i-1}'\omega, \ \a q_i\omega \to \a q_i'\omega, \ q_{i+1}\omega\to q_{i+1}'\omega, \dots] \tag{2}
\end{align*}
where $a\in Y_i$ and $q_j,q_j'\in Q_j$.

\end{subequations}

\end{enumerate}

\end{lemma}

\medskip

\subsection{Turing rewriting systems} \

\medskip

A new framework for Turing machines is introduced in this section.  This machinery is based on that of $S$-graphs and, as such, provides a bridge between the two concepts.  Crucially, \Cref{lem-TM-def} will quickly imply that this framework will suffice for the considerations of this manuscript (see \Cref{lem-tm-system}).

\begin{definition}[Turing rewriting system]\label{def-turing-system}

A \emph{Turing system hardware} $\H$ is a tuple $(\I,\A)$, where $\I$ is a finite \emph{index set} and $\A$ is an $\I$-tuple of (not necessarily disjoint) finite sets $(A_i)_{i\in \I}$.  The set $A_i$ is called an \emph{alphabet} of $\H$.  Let $L_i$ be the set of all (finite) words over $A_i$, i.e so that $L_i$ is the free monoid $A_i^*$ over $A_i$.  Then the \emph{Turing language} of $\H$ is $L_\H=\bigtimes_{i\in\I} L_i$.  As in previous settings, with an abuse of terminology, the elements of $L_\H$ are called \emph{$L_\H$-words}.

An \emph{exact fragment} over $\H$ is an expression $Q_i$ of the form $[u_i\to v_i]$ for some $u_i,v_i\in L_i$.  In this case, $Q_i$ is said to be \emph{applicable} to a word $w_i\in L_i$ if and only if $w_i=u_i$, while $v_i$ is called the \emph{result of the application} of $Q_i$ to $u_i$.  Moreover, the \emph{size} of $Q_i$ is taken to be $|Q_i|=\max(|u_i|,|v_i|)$.

Conversely, a \emph{(right) matching fragment} over $\H$ is an expression $Q_i$ of the form $[*u_i\to*v_i]$, where $u_i,v_i\in L_i$ and $*$ is a designated symbol.  In this case, $Q_i$ is \emph{applicable} to $w_i\in L_i$ if and only if there exists $w_i'\in L_i$ such that $w_i=w_i'u_i$, in which case the \emph{result of the application} of $Q_i$ to $w_i$ is $Q(w_i)=w_i'v_i$.  The \emph{size} of $Q_i$ is again taken to be $|Q_i|=\max(|u_i|,|v_i|)$.

An exact or matching fragment defined with respect to $i\in\I$ is called a \emph{fragment of tape $i$}.  Note that for any fragment $Q_i$ of tape $i$ and any word $w_i\in L_i$ to which $Q_i$ is applicable, $\bigl||Q_i(w_i)|-|w_i|\bigr|\leq|Q_i|$.

A \emph{Turing action} over $\H$ is an $\I$-tuple $Q=(Q_i)_{i\in\I}$ such that $Q_i$ is a fragment of tape $i$ called the \emph{$Q$-fragment of tape $i$}.  In this case, $Q$ is said to be \emph{applicable} to the $L_\H$-word $w=(w_i)_{i\in\I}$ if and only if $Q_i$ is applicable to $w_i$ for all $i\in \I$, while the \emph{result of the application of $Q$ to $w$} is the $L_\H$-word $Q(w)=(Q_i(w_i))_{i\in\I}$.  Further, the \emph{size} of $Q$ is defined to be $|Q|=\max_{i\in\I}|Q_i|$.  Hence, if $Q$ is applicable to $w$, then $\bigl| |Q(w)|-|w| \bigr|\leq|Q| \cdot |\I|$.

A \emph{Turing rewriting system} $\T$ over $\H$ is then a finite directed graph $(S_\T,E_\T)$ (where edges are permitted to be parallel or loops) equipped with a labelling $\lab$ of the edges with Turing actions over $\H$.  The set of Turing actions labelling the edges is denoted $\Q_\T$.

A \emph{configuration} of $\T$ is a pair $(w, s) \in L_\H \times S_\T$.   An edge $e = (s_1, s_2)\in E_\T$ with $Q=\lab(e)$ is said to be \emph{applicable} to a configuration $(w, s)$ if $s_1 = s$ and if $Q$ is applicable to $w$.  In this case, the \emph{result of the application} of $e$ to $(w, s)$ is defined to be the configuration $e(w,s)=(Q(w), s_2)$.

A \emph{computation} of $\T$ is a sequence 
$$\mathcal{C} = \bigl((w_0,s_0),e_1,(w_1,s_1),e_2,\dots,e_n,(w_n,s_n)\bigr)$$ 
such that for all $i$, $(w_i,s_i)$ is a configuration of $\T$, $e_i\in E_\T$ is applicable to $(w_{i-1},s_{i-1})$, and $e_i(w_{i-1},s_{i-1})=(w_i,s_i)$.  In this case, $\mathcal{C}$ is said to be a computation \emph{between} the configurations $(w_0,s_0)$ and $(w_n,s_n)$.  As in previous settings, the \emph{time} and \emph{space} of $\mathcal{C}$ are $n$ and $\max_{i\in\overline{0,n}}|w_i|$, respectively.  Further, letting $\lab(e_i)=Q^i\in\Q_\T$ for all $i\in\overline{n}$, the \emph{history} of $\mathcal{C}$ is the sequence $(Q^1,\dots,Q^n)$.

As in the definition of $S$-machines and $S$-graphs, it is often assumed that $\T$ has two (distinct) distinguished vertices $s_s,s_f\in S_\T$ and that $\I$ has a distinguished index $i^*\in\I$.  To match terminology, $s_s$ and $s_f$ are called the \emph{starting state} and \emph{finishing state} of $\T$, respectively, while $i^*$ is called the \emph{input tape}.  In this setting, though, the input alphabet $A$ of $\T$ is taken to be some subset $A\subseteq A_{i^*}$.

For any $w\in A^*$, define the $L_\H$-word $w^*=(w^*_i)_{i\in\I}$ by setting $w^*_i=w$ if $i=i^*$ and $w^*_i=1$ otherwise (where $1$ represents the empty word in the free monoid).  Then, a computation $\mathcal{C}_w$ between $(w^*,s_s)$ and $(1^*,s_f)$ is said to \emph{recognize} $w$.  Accordingly, the \emph{language recognized by $\T$} is the subset $L_\T\subseteq A^*$ consisting of all words over $A$ recognized by $\T$.

Given $w\in L_\T$, the \emph{time} and \emph{space} of $w$ is the minimal time and space, respectively, of a computation recognizing $w$.  The time and space complexities of $\T$ are then defined in the analogous manner to how they were defined for $S$-machines.

\end{definition}

Note that, as defined above, the application of every given edge of a Turing rewriting system $\T$ is deterministic.  However, a general Turing rewriting system is non-deterministic as distinct edges may be applicable to one given configuration.

Next, an analogue of the `symmetry' of $S$-machines, $S$-graphs, and Turing machines is introduced for Turing rewriting systems.

\begin{definition}[Inverse actions and symmetry] Let $\H=(\I,\A)$ be a Turing system hardware.  For $i\in\I$, if $Q_i$ is a fragment of tape $i$, then the \emph{inverse} fragment $Q_i^{-1}$ of tape $i$ is defined by `reversing arrows', i.e:

\begin{itemize}

\item If $Q_i$ is the exact fragment $[u_i\to v_i]$, then $Q_i^{-1}$ is the exact fragment $[v_i \to u_i]$

\item If $Q_i$ is the matching fragment $[*u_i \to * v_i]$, then $Q_i^{-1}$ is the matching fragment $[* v_i \to * u_i]$

\end{itemize}

Then, given a Turing action $Q=(Q_i)_{i\in\I}$, the \emph{inverse Turing action} $Q^{-1}$ is defined to be $(Q_i^{-1})_{i\in\I}$.  Note that $|Q_i^{-1}|=|Q_i|$ for all $i$, and hence $|Q^{-1}|=|Q|$.  

It is clear from construction that $(Q^{-1})^{-1}=Q$.  Moreover, $Q$ is applicable to $w\in L_\H$ if and only if $Q^{-1}$ is applicable to $Q(w)$, with $Q^{-1}(Q(w))=w$.

Finally, a Turing rewriting system $\T=(S_\T,E_\T)$ over $\H$ is said to be \emph{symmetric} if for every $e=(s_1,s_2)\in E_\T$, there exists an \emph{inverse edge} $e^{-1}=(s_2,s_1)\in E_\T$ with $\lab(e^{-1})=Q^{-1}$.

As above, given a symmetric Turing rewriting system $\T$, an edge $e\in E_\T$ is applicable to a configuration $(w,s)$ if and only if $e^{-1}$ is applicable to $e(w,s)$, in which case $e^{-1}(e(w,s))=(w,s)$.

\end{definition}

\begin{lemma}\label{lem-tm-system}

For every Turing machine $M$, there exists a symmetric Turing rewriting system $\T_M$ such that $L_{\T_M}=L_M$ and $\TM_{\T_M}\sim_1\SP_{\T_M}\sim_1 \TM_M$.

\end{lemma}

\begin{proof}

By \Cref{lem-TM-def}, it may be assumed without loss of generality that $M$ satisfies properties (a)--(e) detailed in that setting.

Let $M=\gen{A,Y,Q,\Theta,\vec{s}_1,\vec{s}_0}$ with $Y=\bigcup_{i=1}^kY_i$ and $Q=\bigcup_{i=1}^kQ_i$.  Then, let $\H_M=(\I_M,\A_M)$ be the Turing system hardware given by $\I_M=\overline{k}$ and $\A_M(i)=Y_i$ for all $i\in\I_M$.

Now, let $S_M=\bigtimes_{i\in\I_M}Q_i$ and define the labelled edges $E_M=\{e_\theta\}_{\theta\in\Theta}$ in correspondence with the transitions of $M$ as follows:

\begin{itemize}

\item If $\theta$ is of form (1) in \Cref{lem-TM-def}, then let $e_\theta$ be the directed edge $((q_j),(q_j'))$ and labelled by the Turing action $Q_\theta=(Q_{j,\theta})$, where $Q_{i,\theta}=[*a\to*1]$ and $Q_{j,\theta}=[*1\to*1]$ for all $j\neq i$.

\item If $\theta^{-1}$ is of form (1) in \Cref{lem-TM-def}, then let $e_\theta$ be the directed edge corresponding to $e_{\theta^{-1}}^{-1}$.

\item If $\theta$ is of form (2) in \Cref{lem-TM-def}, then let $e_\theta$ be the directed edge $((q_j),(q_j'))$ and labelled by the Turing action $Q_\theta=(Q_{j,\theta})$, where $Q_{i,\theta}=[1\to1]$ and $Q_{j,\theta}=[*1\to*1]$ for all $j\neq i$.

\end{itemize}

Finally, let $\T_M=(S_M,E_M)$ be the Turing rewriting system over $\H_M$ with input tape $1$, input alphabet $A$, starting state $\vec{s}_1$, and finishing state $\vec{s}_0$.

Note that since $M$ is a symmetric Turing machine, it follows from construction that $\T_M$ is symmetric.  Define the map $\eta$ between the configurations of $M$ and the configurations of $\T_M$ by $$\eta(\a u_1 q_1 v_1 \omega \a u_2 q_2 v_2 \omega \dots \a u_k q_k v_k \omega)=((u_iv_i),(q_i))$$
By construction, it follows that for any configuration $C$ of $M$ and any $\theta\in\Theta$, $\theta$ is applicable to $C$ if and only if $e_\theta$ is applicable to $\eta(C)$, with $\eta(C\cdot\theta)=e_\theta(\eta(C))$.

Hence, $\eta$ establishes a correspondence between the computations of $M$ and those of $\T_M$, and thus the statement follows.

\end{proof}

The following statement provides a partial converse of \Cref{lem-tm-system}, which together with \Cref{lem-tm-system} implies that the framework of Turing rewriting systems is essentially equivalent to that of Turing machines.  However, this equivalence is superfluous to the purposes of this manuscript, and so its proof is omitted.

\begin{lemma}

For every symmetric Turing rewriting system $\T$, there exists a Turing machine $M_\T$ satisfying the following properties:

\begin{enumerate}[label=({\alph*})]

\item $L_{M_\T}=L_\T$.

\item $M_\T$ is symmetric.

\item $\TM_{M_\T}\sim\TM_\T$, $\SP_{M_\T}\sim\SP_\T$.

\item Any accepting computation of $M_\T$ ends with the accept configuration.

\item Every transition of $M_\T$ or its inverse is of one of the forms outlined in \Cref{lem-TM-def}(e).

\end{enumerate}

\end{lemma}

Now, as the goal of \Cref{main-theorem} is to construct an $S$-machine that emulates a particular Turing machine in a certain sense, it follows from \Cref{lem-tm-system} and \Cref{prop-machine-acc} that it suffices to construct an acceptor $S$-graph that emulates a particular Turing rewriting system in that sense.  However, it is evident from \Cref{def-turing-system} that there are several fundamental differences between $S$-graphs and Turing rewriting systems.

Perhaps the most obvious distinction is that the Turing language of a hardware is constructed from words over the hardware's alphabets, while the language of an $S$-machine is constructed from (freely reduced) group words formed from the hardware's alphabets.  In essence, this amounts to the difference between working with elements of semigroups instead of groups.  This should elucidate the importance of the acceptors $\accept_{A+}^\eps$ defined in \Cref{sec-positivity-check}.

Perhaps even more crucially, though, is that while an action of an $S$-graph can evidently simulate an exact fragment, there is no direct analogue of a matching fragment.  For example, while a guard fragment of an $S$-graph is useful for checking that a particular word is written on a given tape, it is not clear how it can be checked that the word written on that tape simply starts with a particular letter, a feat easily prescribed by a matching fragment.  However, as will be outlined in the coming sections, this issue will again be overcome with the use of the positive acceptors $\accept_{A+}^\eps$.

\section{Turing Tapes}\label{sec-turing-tapes}

Throughout this section, a Turing rewriting system $\T$ is fixed.  Further, letting $\H_\T=(\I_\T,\A_\T)$ be the Turing system hardware over which $\T$ is constructed, an index $i\in\I_\T$ is fixed.  For ease of notation, $\A_\T(i)$ is simply denoted $A_i$.

As a preliminary step in the main construction of this section, a few simple proper operations are first defined.

\medskip

\subsection{Some Auxiliary Operations} \

\medskip

The proper operation $\swap^A(a,b)$ is in \Cref{fig-SWAP} below for a given alphabet $A$.  It can be deduced from the graph that the operation's hardware $(\I_\swap,\A_\swap)$ is given by $\I_\swap=\I_\swap^\ext=\{a,b\}$ with $\A_\swap(a)=\A_\swap(b)=A$.
\begin{figure}[hbt]
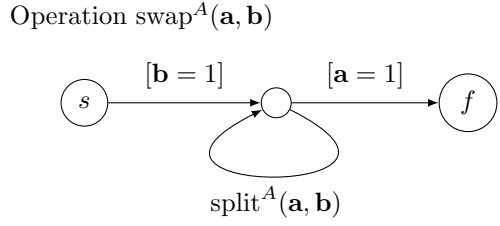

	\centering
	\include{SWAP}
	\caption{The operation $\swap^A(a,b)$}
	\label{fig-SWAP}
\end{figure}

The following statement follows quickly from \Cref{lem-io-split} and \Cref{prop-complexity-estimate}.  As the operation is a simple embellishment of the proper operation $\splt^A$ discussed in \Cref{lem-io-split}, the statement's proof is omitted.

\begin{lemma}\label{lem-io-swap}

The $\IO$-relation for the operation $\swap^A(a,b)$ is given by:
\begin{align*}
\IO_\swap^{s,s}&=\IO_\swap^{f,f}=\Id \\
\IO_\swap^{s,f}&=\{(w_1,w_2)\mid w_1(a)=w_2(b), \ w_1(b)=w_2(a)=1\}
\end{align*}
Moreover, $\TMSP_\swap\preceq(n,n)$.

\end{lemma}

Next, for a given alphabet $A$, the proper operation $\rea_A^\eps(a,b)$ is constructed below.  Its utility is in rearranging the contents of two tapes, though while keeping the order the letters appear consistent.  Its hardware can be seen to be $\I_{\rea}=\I_{\rea}^\ext=\{a,b\}$ with $\A_{\rea}(a)=\A_{\rea}(b)=A$.

\begin{figure}[hbt]
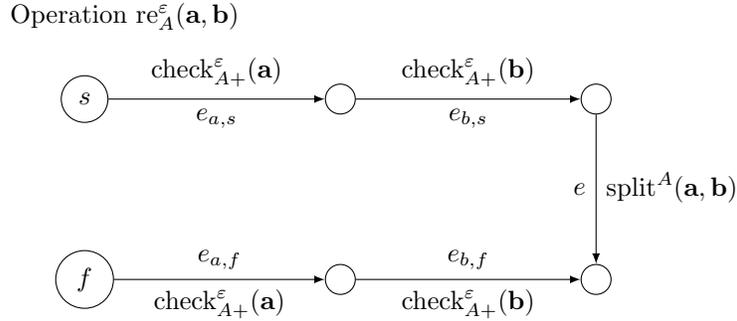

	\centering
	\include{REA2}
	\caption{The proper operation $\rea_A^\varepsilon(a, b, c)$.}
	\label{fig-REA2}       
\end{figure} 

\begin{lemma}\label{lem-io-rea}
	The $\IO$-relation of $\rea_A^\eps$ is given by:
	\begin{align*}
	\IO_{\rea}^{s,f}&=\{(w_1,w_2)\mid w_1(a)w_1(b)=w_2(a)w_2(b), \ w_j(a),w_j(b)\in L_A^+ \text{ for } j=1,2\} \\
	\IO_{\rea}^{s,s}&=\IO_{\rea}^{f,f}=\IO_{\rea}^{s,f}\cup\Id
	\end{align*}
Moreover, $\TMSP_{\rea}\preceq(n^{1+\eps},n)$.
\end{lemma}

\begin{proof} 

Let $(w_1,s_1,w_2,s_2)\in\IO_{\rea}$ such that $w_1\neq w_2$, so that $s_1,s_2\in\SF_{\rea}$.  Further, let $$\gamma=\bigl((u_0,t_0),e_1',(u_1,t_1),e_2',\dots,e_m',(u_m,t_m)\bigr)$$ be a computation realizing $(w_1,w_2)\in\IO_{\rea}^{s_1,s_2}$ with computation path of minimal length.  

For all $i\in\overline{m}$, set $e_i'=(\sf_i,\sf_i',e_i'')$.  Then, as the $\IO$-relation defining checkers implies that its external tape is immutable, there exists $j\in\overline{0,m}$ such that $e_j''=e$.  Moreover, by \Cref{lem-io-swap} it follows that $w_1(a)w_1(b)=w_2(a)w_2(b)$.

Further, the immutability of the external tape of a checker and the minimality hypothesis on the computation path combine to imply that $\sf_i\neq\sf_i'$ if $e_i''\neq e$.  So, since $e$ separates the graph of $\rea_A^\eps$, it follows that:
\begin{align*}
e_1',e_2'&=(\s,\f,e_{a,s_1}),(\s,\f,e_{b,s_1}) \\
e_{m-1}',e_m'&=(\f,\s,e_{b,s_2}),(\f,\s,e_{a,s_2})
\end{align*}
But then the $\IO$-relation of checkers implies that $w_j(a),w_j(b)\in L_A^+$.  As a result, letting $$S=\{(w_1,w_2)\mid w_1(a)w_1(b)=w_2(a)w_2(b), \ w_j(a)i,w_j(b)\in L_A^+ \text{ for } j=1,2\}$$ then $\IO_{\rea}^{s,f}\subseteq S$ and $\IO_{\rea}^{s,s},\IO_{\rea}^{f,f}\subseteq \Id\cup S$.

Now, let $(w_1,w_2)\in S$ and $s_1,s_2\in\SF_{\rea}$.  Define $\sf_1,\sf_2\in\{\s,\f\}$ by setting $\sf_j=\s$ if and only if $s_j=\rea_s$.  Then, the generic path:
\begin{align*}
\rho=(\s,\f, &e_{a,s_1}),(\s,\f,e_{b,s_1}),(\sf_1,\sf_2,e),(\f,\s,e_{b,s_2}),(\f,\s,e_{a,s_2})
\end{align*}
supports a computation $\gamma_{w_1,w_2}^{s_1,s_2}$ realizing $(w_1,w_2)\in\IO_{\rea}^{s_1,s_2}$.

Let $\F$ be the estimating set given by $\F_{\check_{A+}^\eps}=(n^{1+\eps},n)$ and $\F_{\splt3}=(n,n)$.  Then, the estimated complexity of the computation $\gamma_{w_1,w_2}^{s_1,s_2}$ above is bounded by:
\begin{align*}
\tm_{\rea}^\F(\gamma_{w_1,w_2}^{s_1,s_2})&=3|w_1|^{1+\eps}+\max(|w_1|,|w_2|)+3|w_2|^{1+\eps}\leq 7\max(|w_1|,|w_2|)^{1+\eps} \\
\sp_{\rea}^\F(\gamma_{w_1,w_2}^{s_1,s_2})&=\max(|w_1|,|w_2|)
\end{align*}
Thus, the statement follows from \Cref{lem-io-swap}, \Cref{cor-acceptor-eps}, and \Cref{prop-complexity-estimate}.

\end{proof}

\medskip

\subsection{Turing tape objects} \

\medskip

The main topic of this section is the construction of a particular type of object called a \emph{Turing tape}.  As with the introduction of counters (see \Cref{fig-COUNTER_STAR}) and acceptors of positive words (see \Cref{fig-POSCHECK2}), the construction follows three steps:

\begin{enumerate}

\item A general class of objects is defined (see \Cref{def-counter} and \Cref{def-acceptor})

\item A particular example is constructed (see \Cref{lem-counter-is-counter} and \Cref{lem-pos-check1})

\item An iterative step is performed to produce a sequence of such objects which have successively `better' qualities (see \Cref{lem-countern} and \Cref{lem-pos-check2}).

\end{enumerate}

\begin{definition}\label{def-ttape}

For all $d\in\N-\{0\}$, set $\I_d=\{a_1,\dots,a_d\}$.  Then, let $V^{d}$ be the set of $\I_d$-tuples of words in $L_{A_i}^+$ (recall $A_i$ is the alphabet $\A_\T(i)$ for the fixed $i\in\I_\T$).

For all $w\in L_{A_i}^+$, let $V^{d}_w$ be the subset of $V^{d}$ consisting of all $v\in V^{d}$ satisfying $w=v(a_1)\dots v(a_d)$.  Conversely, for any $v\in V^{d}$, let $w_d(v)\in A_i^*$ be the (unique) word over $A_i$ satisfying $v\in V^{d}_{w_d(v)}$.  Note that $V^{d}=\bigsqcup\limits_{w\in A_i^*} V^{d}_w$ and that $|v|=|w_d(v)|$ for all $v\in V^{d}$.

Now, suppose $\Obj=\Obj[\I_d]$ is an object comprised of an operation $\set=\set(b)$ and, for each $Q=(Q_j)_{j\in\I_\T}\in\Q_\T$, an operation $Q=Q()$.  Further, suppose that $\A_\set(b)=A_i$ and $\A_\Obj^\val(a_j)=A_i$ for all $a_j\in\I_d$.  Finally, suppose that $\VAL_\Obj=V^{d}$ and that the $\IO$-relations are given by:
	\begin{align*}
		\IO_\set^{s,s}&=\Id \\
		\IO_\set^{s,f}&= \ \bigr\{(u_1,u_2)\mid u_1(\I_d)=1, \ u_2(b)=1, \ w_d(u_2(\I_d))=u_1(b)\bigl\} \\
		\IO_\set^{f,f}&=\Id \ \cup \ \bigr\{(u_1,u_2)\mid u_1(b)=u_2(b)=1, \ w_d(u_1(\I_d))=w_d(u_2(\I_d))\bigl\} \\
		\IO_Q^{s,s}&=\IO_Q^{f,f}= \ \bigr\{(v_1,v_2)\mid w_d(v_1)=w_d(v_2)\bigl\} \\
		\IO_Q^{s,f}&=\bigl\{(v_1,v_2) \mid Q_i \text{ is applicable to } w_d(v_1) \text{ with } Q_i(w_d(v_1))=w_d(v_2)\bigr\}
	\end{align*}
Then, $\Obj$ is called an \emph{$i$-Turing tape of $\T$ of rank $d$}.

\end{definition}

Now, for any fixed $\eps>0$, consider the object $\TTape_{\T,i}^{1,\eps}\defeq\TTape_{\T,i}^{1,\eps}[\I_1]$ defined in \Cref{fig-TTAPE1} below.  Note that the operations $Q()$ are in correspondence with the elements of $\Q_\T$; the two graphs presented in the figure bearing the name $Q()$ define the operation in the case that the corresponding $Q$-fragment of tape $i$ is exact or (right) matching, respectively.

\begin{figure}[hbt]
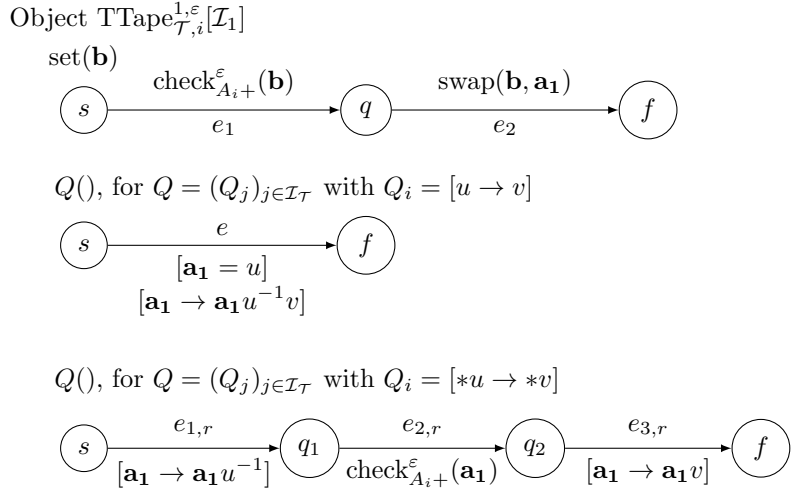

	\centering
	\include{TTAPE1}
	\caption{An object $\TTape_{\T,i}^{1, \varepsilon}[a]$ which is an $i$-Turing tape of $\T$.}
	\label{fig-TTAPE1}       
\end{figure} 

From the established naming conventions, it is implicit that $\I_{\TTape_{\T,i}^{1,\eps}}^\val=\{a_1\}$, $\I_\set^\ext=\{b\}$, and $\I_Q^\ext=\emptyset$ for all $Q\in \Q_\T$.  Moreover, the convention of minimal hardware implies that $\A_{\TTape_{\T,i}^{1,\eps}}(a_1)=\A_\set(b)=A_i$ and that $\I_\op^\int=\emptyset$ for all $\op\in\TTape_{\T,i}^{1,\eps}$.

\begin{lemma}\label{lem-ttape1}
The object $\TTape^1=\TTape_{\T,i}^{1, \varepsilon}$ is an $i$-Turing tape of $\T$ of rank $1$. Moreover, $\TMSP_{\TTape^1} \preceq (n^{1+\varepsilon}, n)$. 
\end{lemma}

\begin{proof}

Note that $V^{1}=L_{A_i}^+$, and so can be viewed as a subset of $L_1=L_{\TTape^1}^\val$.  In the setting of \Cref{prop-io-classes2}, it then follows that $\widehat{V}_Q^{1}=V^{1}$ for all $Q\in \Q_\T$, while $\widehat{V}_\set^{1}$ is the subset of $L_\set$ given by all $u\in\GL_\set$ such that $u(a_1)\in L_{A_i}^+$.

For simplicity, for any $u\in L_\set$, denote $u=(u(a_1),u(b))$, i.e so that the $b$-component is listed second.  Then, for all $v\in L_{A_i}^+$ and $w\in L_{A_i}$, define the subset $C_{v,w,\set}\subseteq\GL_\set\times S_\set$ as follows:

\begin{itemize}

\item $C_{v,w,\set}^s=\{(v,w)\}$

\item If $w\notin L_{A_i}^+$, then $C_{v,w,\set}^q=C_{v,w,\set}^f=\emptyset$

\item If $w\in L_{A_i}^+$ and $v\neq1$, then $C_{v,w,\set}^q=\{(v,w)\}$ and $C_{v,w,\set}^f=\emptyset$

\item If $w\in L_{A_i}^+$ and $v=1$, then $C_{v,w,\set}^q=\{(v,w)\}$ and $C_{v,w,\set}^f=\{(w,1)\}$

\end{itemize}

Then, for all $w\in L_{A_i}-\{1\}$, define the subset $C_{w,\set}\subseteq\GL_\set\times S_\set$ by:
\begin{align*}
C_{w,\set}^f&=\{(v,w)\mid v\in L_{A_i}^+\} \\
C_{w,\set}^s&=C_{w,\set}^q=\emptyset
\end{align*}

Using the definition of checkers and \Cref{lem-io-swap}, it follows that each of $C_{v,w,\set}$ and $C_{w,\set}$ is closed under $\IO_e$ for each $e\in E_\set$.  Moreover, $\bigl(\bigcup_{v,w}C_{v,w,\set}^{\SF_\set}\bigr)\cup\bigl(\bigcup_w C_{w,\set}^{\SF_\set}\bigr)=\widehat{V}^{1}_\set\times\SF_\set$.

Now, let $Q\in\Q_\T$.  First, 
suppose the $Q$-fragment of tape $i$ is the exact fragment $[u\to v]$.  Then, for all $w\in L_{A_i}^+$, define $C_{w,s,Q}\subseteq\GL_Q\times S_Q$ by $C_{w,s,Q}^s=\{w\}$, $C_{w,s,Q}^f=\{v\}$ if $w=u$, and $C_{w,s,Q}^f=\emptyset$ otherwise.  Moreover, for all $w\in L_{A_i}^+-\{v\}$, define the subset $C_{w,f,Q}\subseteq\GL_Q\times S_Q$ by $C_{w,f,Q}^f=\{w\}$ and $C_{w,f,Q}^s=\emptyset$.

Then, $\bigl(\bigcup_w C_{w,s,Q}^{\SF_Q}\bigr)\cup\bigl(\bigcup_wC_{w,f,Q}^{\SF_Q}\bigr)=\widehat{V}^{1}_Q\times\SF_Q$.  Further, by the definition of the action of degree 0, each $C_{w,s,Q}$ and $C_{w,f,Q}$ is closed under $\IO_e$ for any edge $e\in E_Q$.

Conversely, suppose the $Q$-fragment of tape $i$ is the right matching fragment $[* u \to * v]$.  Then, for all $w\in L_{A_i}^+$, define the subset $C_{w,Q}\subseteq\GL_Q\times S_Q$ by:
\begin{align*}
C_{w,Q}^s&=\{wu\} \\
C_{w,Q}^{q_1}&=C_{w,Q}^{q_2}=\{w\} \\
C_{w,Q}^f&=\{wv\}
\end{align*}

Further, for $w\in L_{A_i}$ such that $wu^{-1}\notin L_{A_i}^+$, define $C_{w,s,Q}\subseteq\GL_Q\times S_Q$ by $C_{w,s,Q}^s=\{w\}$, $C_{w,s,Q}^{q_1}=\{wu^{-1}\}$, and $C_{w,s,Q}^{q_2}=C_{w,s,Q}^f=\emptyset$.

Lastly, for $w\in L_{A_i}$ such that $wv^{-1}\notin L_{A_i}^+$, define the subset $C_{w,f,Q}\subseteq\GL_Q\times S_Q$ by setting $C_{w,f,Q}^f=\{w\}$, $C_{w,f,Q}^{q_2}=\{wv^{-1}\}$, and $C_{w,f,Q}^s=C_{w,f,Q}^{q_1}=\emptyset$.

It again follows that each subset is closed under $\IO_e$ for each $e\in E_Q$ and that
\begin{equation}\label{eq-values}
\left(\bigcup_w C_{w,Q}^{\SF_Q}\right)\cup\left(\bigcup_w C_{w,s,Q}^{\SF_Q}\right)\cup\left(\bigcup_w C_{w,f,Q}^{\SF_Q}\right)=\widehat{V}^{1}_Q\times \SF_Q
\end{equation}

Hence, \Cref{prop-io-classes2} implies that $\VAL_{\TTape^1}\subseteq V^{1}$.

On the other hand, for all $v\in V^{1}$, the generic path $(\s,\f,e_1),(\s,\f,e_2)$ in the graph of $\set$ supports an input-output computation $\gamma_v$ between $((1,v),\set_s)$ and $((v,1),\set_f)$.  This implies the existence of an edge between the vertices corresponding to the initial value and $v$ in $\G_{\TTape^1}^\val$, so that $v\in\VAL_{\TTape^1}$.

Thus, $V^{1}$ is indeed the set of values of $\TTape_{\T,i}^{1,\eps}$.

The sets constructed above then combine with \Cref{prop-io-classes} to imply that the $\IO$-relations satisfy the following containments:
\begin{align*}
\IO_\set^{s,s}&,\IO_\set^{f,f},\IO_Q^{s,s},\IO_Q^{f,f}\subseteq\Id \\
\IO_\set^{s,f}&\subseteq\{\bigl((1,v),(v,1)\bigr)\mid v\in L_{A_i}^+\} \\
\IO_Q^{s,f}&\subseteq\{(u,v)\} && \text{ if $Q_i=[u \to v]$} \\
\IO_Q^{s,f}&\subseteq\{(wu,wv)\mid w\in L_{A_i}^+\} && \text{ if $Q_i=[* u \to * v]$}
\end{align*}
Note that the final two containments can be described by the single containment $$\IO_Q^{s,f}\subseteq\{(v,Q_i(v))\mid Q_i \text{ is applicable to }v\}$$
Further, note that for any $v_1,v_2\in V^{1}$, $w_1(v_1)=w_1(v_2)$ if and only if $v_1=v_2$.

Hence, to show that $\TTape_{\T,i}^{1,\eps}$ is an $i$-Turing tape of $\T$ of rank 1, it suffices to show that the containments above are equalities.

The reflexivity of $\IO$-relations implies $\IO_\set^{s,s}=\IO_\set^{f,f}=\IO_Q^{s,s}=\IO_Q^{f,f}=\Id$.  What's more, for any $v\in V^{1}$, the computation $\gamma_v$ of $\set$ defined above realizes $((1,v),(v,1))\in\IO_\set^{s,f}$, so that $\IO_\set^{s,f}=\{((1,v),(v,1))\mid v\in L_{A_i}^+\}$.

Meanwhile, the two cases of $Q$ are given as follows:

\begin{itemize}

\item If $Q_i$ is the exact fragment $[u\to v]$, then the generic path $(\s,\f,e)$ supports a computation $\gamma_{u,v}^Q$ realizing $(u,v)\in\IO_Q^{s,f}$

\item  If $Q_i$ is the right matching fragment $[* u \to * v ]$, then for all $w\in L_{A_i}^+$ the generic path $(\s,\f,e_{1,r}),(\s,\f,e_{2,r}),(\s,\f,e_{3,r})$ supports a computation $\gamma_{w}^Q$ realizing $(wu,wv)\in\IO_Q^{s,f}$

\end{itemize}

Thus, $\TTape_{\T,i}^{1,\eps}$ is an $i$-Turing tape of $\T$ of rank 1.

Now, let $\F$ be the estimating set for $\TTape_{\T,i}^{1,\eps}$ given by $\F_{\check}=(n^{1+\eps},n)$ and $\F_\swap=(n,n)$.  
Note the following inequalities:

\begin{itemize}

\item For all $v\in V^{1}$, 
\begin{align*}
\tm_\set^\F(\gamma_v)&=|v|^{1+\eps}+|v|\leq2|v|^{1+\eps} \\
\sp_\set^\F(\gamma_v)&=|v|
\end{align*} 

\item If $Q_i$ is the exact fragment $[u\to v]$, then 
\begin{align*}
\tm_Q^\F(\gamma_{u,v}^Q)&=1 \\
\sp_Q^\F(\gamma_{u,v}^Q)&=\max(|u|,|v|)
\end{align*}

\item If $Q_i$ is the right matching fragment $[* u \to * v]$, then for all $w\in L_{A_i}^+$,
\begin{align*}
\tm_Q^\F(\gamma_{w}^Q)&=1+|w|^{1+\eps}+1\leq|w|^{1+\eps}+2 \\
\sp_Q^\F(\gamma_{w}^Q)&=\max(|wu|,|w|,|wv|)
\end{align*}
But as $u,v\in L_{A_i}^+$, $|wu|=|w|+|u|$ and $|wv|=|w|+|v|$.  Hence, for $M=\max(|wu|,|wv|)$,
\begin{align*}
\tm_Q^\F(\gamma_{w}^Q)&\leq M^{1+\eps}+2 \\
\sp_Q^\F(\gamma_{w}^Q)&=M
\end{align*}

\end{itemize}

As trivial computations realize all other pairs defining the $\IO$-relations, it then follows that $\TMSP_{\TTape^1}^\F\preceq(n^{1+\eps},n)$.  As \Cref{lem-io-swap}, \Cref{cor-acceptor-eps}, and \Cref{prop-acc-ch-wacc} imply that $\TMSP_{\check}\preceq\F_{\check}$ and $\TMSP_\swap\preceq\F_\swap$, \Cref{prop-complexity-estimate} then yields $\TMSP_{\TTape^1}\preceq(n^{1+\eps},n)$.

\end{proof}

\medskip

\subsection{Improved Turing tapes}\label{sec-improved-ttape} \

\medskip

The goal of this section is to define a sequence of Turing tapes of increasing rank which will prove to get `more efficient' in a particular sense.  However, it should be noted upfront that the setting for this improvement is not the typical time-space complexity, but rather given in terms of the complexity of its `sequential computations' (see \Cref{sec-seq-comp}).

Now, to this end, it is assumed that the object $\TTape^d=\TTape_{\T,i}^{d,\eps}$ is an $i$-Turing tape of $\T$ of rank $d$ satisfying $\TMSP_{\TTape^d}\preceq(n^{1+\eps},n)$.  As such, note that \Cref{lem-ttape1} implies that $\TTape_{\T,i}^{1,\eps}$ provides a base for this iterative construction.

First, the auxiliary object $\IRT^{d}\defeq\IRT_{\T,i}^{d,\eps}[\I_{d+1}]$, whose name indicates that it is an `intermediate rearranging tape', is constructed in \Cref{fig-rea} below.  There, the set $\I_d'$ is taken to be the subset $\{a_2,\dots,a_{d+1}\}$ of $\I_{d+1}$.  Hence, the set is in bijection with $\I_d$, as exploited in its role in the instance $\TT^d$.

\begin{figure}[hbt]
	\centering
	\include{TTd}
	\caption{The auxiliary object $\IRT_{\T,i}^{d,\eps}[\I_{d+1}]$}
	\label{fig-rea}       
\end{figure} 

The object consists of $|\Q_\T|+2$ operations: $|\Q_\T|+1$ operations in correspondence with those of $\TTape^d[\I_d']$ which function analogously and the operation $\rea=\rea()$ which is meant to rearrange the contents of $\I_{d+1}$ without affecting their `order'.

It can be deduced from the makeup of the graphs that $\I_{\set_0}^\int=\I_{Q_0}^\int=\emptyset$, that $\I_\rea^\int=\{c\}$, and that $\A_\op(j)=A_i$ for all $\op\in\IRT^d$ and all $j\in\I_\op$.  In the operation $\rea$, it follows immediately from the trivial action labelling $e_0$ and the guard fragments labelling $e_1$ and $e_5$ that condition (O) is satisfied.

For any $\op\in\IRT^d$ and $u\in\GL_\op$, $u(\I_d')$ is a copy of a value of $\TTape^d$, i.e of a word in $V^{d}$.  As such, $u(\I_d')$ is often identified with this value and the map $w_d$ is extended naturally, so that:
\begin{align*}
w_d(u(\I_d'))&=u(a_2)u(a_3)\dots u(a_{d+1})
\end{align*}
Note that $w_d$ will be used exclusively in this context as opposed to being applied on the `canonical' copy of $\I_d$ in $\I_{d+1}$.

Conversely, for any $v\in V^{d}$, define $v_d$ to be the copy of $v$ over $\I_d'$.  Hence, $v(a_j)=v_d(a_{j+1})$ for all $j\in\overline{d}$.  

Note that for any $v\in\VAL_{\TT^d}$, $v\in V^{d+1}$ if and only if $v(a_1)\in L_A^+$.

\begin{lemma}\label{lem-TTd}

The object $\IRT^d=\IRT_{\T,i}^{d,\eps}$ satisfies $\VAL_{\IRT^d}=V^{d+1}$ with $\IO$-relations given by:
\begin{align*}
\IO_\rea^{s,s}&=\IO_\rea^{s,f}=\IO_\rea^{f,f}=\{(v_1,v_2)\mid w_{d+1}(v_1)=w_{d+1}(v_2)\} \\
\IO_{\set_0}^{s,s}&=\IO_{\set_0}^{f,f}=\Id \\
\IO_{\set_0}^{s,f}&=\{(u_1,u_2)\mid u_1(a_1)=u_2(b)=1, \ u_1(b)=u_2(a_1), \ u_1(\I_d')=u_2(\I_d')\} \\
\IO_{Q_0}^{s,s}&=\IO_{Q_0}^{f,f}=\Id\cup\{(u_1,u_2)\mid u_1(a_1)=u_2(a_1), \ w_d(u_1(\I_d'))=w_d(u_2(\I_d'))\} \\
\IO_{Q_0}^{s,f}&=\{(u_1,u_2)\mid u_1(a_1)=u_2(a_1), \ Q_i(w_d(u_1(\I_d')))=w_d(u_2(\I_d'))\} 
\end{align*}
Moreover, $\TMSP_{\IRT^d}\preceq(n^{1+\eps},n)$.

\end{lemma}

\begin{proof}

For simplicity, let $\GL_\rea^+$ be the subset of $\GL_\rea$ consisting of all words $u$ such that $u(j)\in L_{A_i}^+$ for all $j\in\I_\rea$.  Then, for any $w\in L_{A_i}^+$, define the subset $C_w\subseteq \GL_\rea^+\times S_\rea$ as follows:
\begin{align*}
C_w^s=C_w^{q_1}=C_w^{q_4}=C_w^f&=\{u\mid u(c)=1, \ w_{d+1}(u(\I_{d+1}))=w\} \\
C_w^{q_2}=C_w^{q_3}&=\{u\mid u(\I_d')=1, \ u(a_1)u(c)=w\}
\end{align*}

Note that $\bigcup_w C_w^{\SF_\rea}=(V^{d+1}\times1_\rea^\int)\times\SF_\rea=\widehat{V}^{d+1}_\rea\times\SF_\rea$, where $\widehat{V}^{d+1}_\rea$ is defined as in the setting of \Cref{prop-io-classes2}.  Further, by \Cref{lem-io-rea} and the definition of the $\set$ operation, each subset $C_w$ is closed under $\IO_e$ for all $e\in E_\rea$.

Next, for all $w\in L_{A_i}^+$ and all $v\in V^{d}$, define $C_{w,v}\subseteq\GL_{\set_0}\times S_{\set_0}$ by:
\begin{align*}
C_{w,v}^s=C_{w,v}^{q}&=\{u\mid u(\I_d')=v_d, \ u(a_1)=1, \ u(b)=w\} \\
C_{w,v}^f&=\{u\mid u(\I_d')=v_d, \ u(a_1)=w, \ u(b)=1\}
\end{align*}
By the definition of checkers and \Cref{lem-io-swap}, each $C_{w,v}$ is closed under $\IO_e$ for all $e\in E_{\set_0}$.

Further, for any $w\in L_{A_i}$ with $w\neq1$ and any $v\in V^{d+1}$, define $C_{w,v,f}\subseteq\GL_{\set_0}\times S_{\set_0}$ by:
\begin{align*}
C_{w,v,f}^s&=C_{w,v,f}^q=\emptyset \\
C_{w,v,f}^f&=\{u\mid u(\I_{d+1})=v, \ u(b)=w\}
\end{align*}
Conversely, for $w\in L_{A_i}-L_{A_i}^+$ and $v\in V^{d+1}$, define $C_{w,v,s}\subseteq\GL_{\set_0}\times S_{\set_0}$ by:
\begin{align*}
C_{w,v,s}^s&=\{u\mid u(\I_{d+1})=v, \ u(b)=w\} \\
C_{w,v,s}^q&=C_{w,v,s}^f=\emptyset
\end{align*}

Finally, for any $w\in L_{A_i}^+$ and $v\in V^{d+1}$ such that $v(a_1)\neq1$, define $C_{w,v,s}\subseteq\GL_{\set_0}\times S_{\set_0}$ by:
\begin{align*}
C_{w,v,s}^s&=C_{w,v,s}^q=\{u\mid u(\I_{d+1})=v, \ u(b)=w\} \\
C_{w,v,s}^f&=\emptyset
\end{align*}

Again, the definition of checkers and \Cref{lem-io-swap} imply that each of these subsets is closed under $\IO_e$ for all $e\in E_{\set_0}$.  Moreover, note that
$$\left(\bigcup C_{w,v}^{\SF_{\set_0}}\right)\cup\left(\bigcup C_{w,v,s}^{\SF_{\set_0}} \right)\cup\left(\bigcup C_{w,v,f}^{\SF_{\set_0}}\right)=\widehat{V}^{d+1}_{\set_0}\times\SF_{\set_0}$$

Now, fix $Q=(Q_i)_{i\in\I_\T}\in\Q_\T$.  Then, for all $w\in L_{A_i}^+$ such that $Q_i$ is applicable to $w$ and any $w'\in L_A^+$, define the subset $C_{w,w'}\subseteq\GL_{Q_0}\times S_{Q_0}$ by:
\begin{align*}
C_{w,w'}^s&=\{u\mid w_d(u(\I_d'))=w, \ u(a_1)=w'\} \\
C_{w,w'}^f&=\{u\mid w_d(u(\I_d'))=Q_i(w), \ u(a_1)=w'\}
\end{align*}

Conversely, for any $w\in L_{A_i}^+$ such that $Q_i$ is not applicable to $w$ and any $w'\in L_{A_i}^+$, define the subset $C_{w,w',s}\subseteq\GL_{Q_0}\times S_{Q_0}$ by:
\begin{align*}
C_{w,w',s}^s&=\{u\mid w_d(u(\I_d'))=w, \ u(a_1)=w' \} \\
C_{w,w',s}^f&=\emptyset
\end{align*}
Similarly, for any $w\in L_{A_i}^+$ such that $Q_i^{-1}$ is not applicable to $w$ and any $w'\in L_{A_i}^+$, define the subset $C_{w,w',f}\subseteq\GL_{Q_0}\times S_{Q_0}$ by:
\begin{align*}
C_{w,w',f}^s&=\emptyset \\
C_{w,w',f}^f&=\{u\mid w_d(u(\I_d'))=w, \ u(a_1)=w' \}
\end{align*}
By $\IO$-relations defining a Turing tape, it follows that each of these sets are closed under $\IO_e$ for the corresponding edge $e$.  Further,
\begin{align*}
\left(\bigcup C_{w,w'}^{\SF_{Q_0}} \right) \cup \left(\bigcup C_{w,w',s}^{\SF_{Q_0}} \right) \cup \left(\bigcup C_{w,w',f}^{\SF_{Q_0}} \right) &= V^{d+1}=\widehat{V}^{d+1}_{Q_0}
\end{align*}

Hence, as $V^{d+1}$ contains the initial value, it follows from \Cref{prop-io-classes2} that $\VAL_{\IRT^d}\subseteq V^{d+1}$.

For the reverse inclusion, note the following:
\begin{itemize}

\item Let $w\in L_{A_i}^+$ and $v\in V^{d}$.  Then, by the definition of checkers and \Cref{lem-io-swap}, the generic path $(\s,\f,e_1),(\s,\f,e_2)$ formed in $\G_{\set_0}$ supports a computation $\gamma_{w,v}^{\set_0}$ of $\set_0$ between $(u_1,(\set_0)_s)$ and $(u_2,(\set_0)_f)$ such that $u_j(\I_d')=v$, $u_1(b)=u_2(a_1)=w$, and $u_1(a_1)=u_2(b)=1$.

\item Let $w\in L_{A_i}^+$ and $v_1,v_2\in V^{d+1}_w$.  Then, let $u_1,u_2\in\GL_\rea$ such that $u_j(\I_{d+1})=v_j$ and $u_j(c)=1$.  Then, by the definition of the set operation and \Cref{lem-io-rea}, the generic path $e_1',\dots,e_5'$ given by setting $e_j'=(\s,\f,e_j)$ supports a computation $\gamma_{v_1,v_2}^\rea$ between $(u_1,\rea_s)$ and $(u_2,\rea_f)$.  

\end{itemize}

Now, let $v\in V^{d+1}$.  Then, setting $w=w_{d+1}(v)$, the computation $\gamma_{w,1}^{\set_0}$ implies the existence of an edge in the graph $\G_{\IRT^d}^\val$ between the initial value and the word $v_w\in V^{d+1}$ given by $v_w(\I_d')=1$ and $v_w(a_1)=w$.  Further, as $v_w,v\in V^{d+1}_w$, the computation $\gamma_{v_w,v}^\rea$ implies the existence of an edge in $\G_{\IRT^d}^\val$ between $v_w$ and $v$.  Hence, $v\in\VAL_{\IRT^d}$, so that $\VAL_{\IRT^d}=V^{d+1}$.

The sets constructed above then combine with \Cref{prop-io-classes} to imply that the $\IO$-relations are contained in those in the statement.  It is easy to see that the definition of Turing tape implies that the containments involving $Q_0$ are equalities.  Further, trivial computations imply that $\IO_{\set_0}^{s,s}=\IO_{\set_0}^{f,f}=\Id$, while the trivial action labelling $e_0$ implies that $\IO_\rea^{s,s}=\IO_\rea^{s,f}=\IO_\rea^{f,f}$.  Finally, the computation $\gamma_{w,v}^{\set_0}$ implies the reverse containment for $\IO_{\set_0}^{s,f}$, while the computation $\gamma_{v_1,v_2}^\rea$ implies the reverse containment for $\IO_\rea^{s,f}$.  Thus, the $\IO$-relation of $\IRT^d$ is as in the statement.

Now, let $\F$ be the estimating set for $\IRT^d$ given by $\F_\inst=(n^{1+\eps},n)$ for every $\inst\in\INST_{\IRT^d}$.  By the inductive hypothesis, \Cref{cor-acceptor-eps}, \Cref{lem-io-swap}, \Cref{lem-io-rea}, and \Cref{prop-complexity-estimate}(3), it suffices to show that $\TMSP_{\IRT^d}^\F\preceq(n^{1+\eps},n)$.

For any $(u_1,s_1,u_2,s_2)\in\IO_{Q_0}$, it follows immediately from the inductive hypothesis that there exists a single generic edge supporting a computation that realizes $(u_1,u_2)\in\IO_{Q_0}^{s_1,s_2}$ and satisfies the desired complexity bounds.

For any $(u_1,u_2)\in\IO_{\set_0}^{s,f}$, let $w=u_1(b)=u_2(a_1)$ and set $v\in V^{d}$ such that $u_j(\I_d')=v$.  Then, $\gamma_{w,v}^{\set_0}$ realizes $(u_1,u_2)\in\IO_{\set_0}^{s,f}$ and satisfies:
\begin{align*}
\tm_{\set_0}^\F(\gamma_{w,v}^{\set_0})&=2|w|^{1+\eps}\leq2|u_1|^{1+\eps} \\
\sp_{\set_0}^\F(\gamma_{w,v}^{\set_0})&=|u_1|=|u_2|
\end{align*}

Further, for $(v_1,v_2)\in\IO_\rea^{s,f}$, the computation $\gamma_{v_1,v_2}^\rea$ realizes $(v_1,v_2)\in\IO_\rea^{s,f}$ and satisfies:
\begin{align*}
\tm_\rea^\F(\gamma_{v_1,v_2}^\rea)&=1+|v_1|^{1+\eps}+\max(|v_1|,|v_2|)^{1+\eps}+|v_2|^{1+\eps}+1 \\
&= 3|v_1|^{1+\eps}+2 \\
\sp_\rea^\F(\gamma_{v_1,v_2}^\rea)&=|v_1|=|v_2|
\end{align*}
Hence, by the existence of $e_0$, for any $(v_1,s_1,v_2,s_2)\in\IO_\rea$, there exists a computation $\gamma$ realizing $(v_1,v_2)\in\IO_\rea^{s_1,s_2}$ and satisfying $\tm_\rea^\F(\gamma)\leq3|v_1|^{1+\eps}+3$ and $\sp_\rea^\F(\gamma)=|v_1|=|v_2|$.

Thus, as any remaining $\IO$-relation is reflexive and so can be realized by an empty computation, $\TMSP_{\IRT^d}^\F\preceq(n^{1+\eps},n)$.

\end{proof}

With the object $\IRT^d$ at hand, the iterative step is now carried out in \Cref{fig-TTAPE2} below, with the definition of the object $\TTape^{d+1}\defeq\TTape_{\T,i}^{d+1,\eps}[\I_{d+1}]$.  Note that the two graphs labelled $Q()$ represent the operation corresponding to the Turing action $Q=(Q_j)_{j\in\I_\T}\in\Q_\T$, with the two cases corresponding to the possible type of fragment for $Q_i$.

As in previous settings, the notation $\I_d'=1$ represents the set of guard fragments $[a_j=1]$ for $a_j\in\I_d'$.  The notation $\I_{d+1}=1$ functions similarly.

\begin{figure}[hbt]
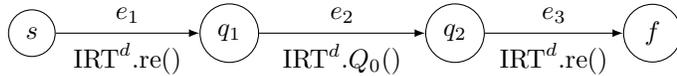

	\centering
	\include{TTAPE2}
	\caption{An $i$-Turing tape of rank $d+1$}
	\label{fig-TTAPE2}       
\end{figure} 

The following statement completes the iterative construction:

\begin{lemma}\label{lem-ttape2}

The object $\TTape^{d+1}$ is an $i$-Turing tape of $\T$ of rank $d+1$ and satisfies the bound $\TMSP_{\TTape^{d+1}}\preceq(n^{1+\eps},n)$.

\end{lemma}

\begin{proof}

As it was proved that $\VAL_{\IRT^d}=V^{d+1}$ in \Cref{lem-TTd}, the definition of graph language immediately implies that $\VAL_{\TTape^{d+1}}\subseteq V^{d+1}$.

Conversely, let $v\in V^{d+1}$.  Then, let $u_1,u_2,u_3\in\GL_\set$ be given by:
\begin{itemize}
\item $u_1(\I_{d+1})=1$, $u_1(b)=w_{d+1}(v)$ 

\item $u_2(a_1)=w_{d+1}(v)$, $u_2(\I_d')=1$, $u_2(b)=1$ 

\item $u_3(\I_{d+1})=v$, $u_3(b)=1$
\end{itemize}  
Then, by the $\IO$-relations outlined in \Cref{lem-TTd}, there exists a computation $\gamma_v^\set$ of $\set$ given by:
\begin{align*}
\gamma_v^\set=\bigl((u_1,\set_s),(\s,&\f,e_1),(u_1,q_1),(\s,\f,e_2),(u_2,q_2), \\
&(\s,\f,e_3),(u_3,q_3),(\s,\f,e_4),(u_3,\set_f)\bigr)
\end{align*}

Consequently, there exists an edge in $\G_{\TTape^{d+1}}^\val$ between the initial value and $v$, and as a result $v\in\VAL_{\TTape^{d+1}}$.

Hence, $\VAL_{\TTape^{d+1}}=V^{d+1}$.

Now, for any $w\in L_{A_i}^+$, define the subset $C_{w,\set}\subseteq\GL_\set\times S_\set$ by:
\begin{align*}
C_{w,\set}^s=C_{w,\set}^{q_1}&=\{u\mid u(\I_{d+1})=1, \ u(b)=w\} \\
C_{w,\set}^{q_2}&=\{u\mid u(a_1)=w, \ u(\I_d')=1, \ u(b)=1\} \\
C_{w,\set}^{q_3}=C_{w,\set}^f&=\{u\mid w_{d+1}(u(\I_{d+1}))=w, \ u(b)=1\}
\end{align*}

Meanwhile, for any $w\in L_{A_i}$ and $v\in V^{d+1}-\{1\}$, define $C_{w,v,\set}\subseteq\GL_\set\times S_\set$ by:
\begin{align*}
C_{w,v,\set}^s&=\{u\mid u(\I_{d+1})=v, \ u(b)=w\} \\
C_{w,v,\set}^{q_1}&=C_{w,v,\set}^{q_2}=C_{w,v,\set}^{q_3}=C_{w,v,\set}^f=\emptyset
\end{align*}
Further, for any $w\in L_{A_i}-L_{A_i}^+$, define $C_{w,s,\set}\subseteq\GL_\set\times S_\set$ by:
\begin{align*}
C_{w,s,\set}^s&=C_{w,s,\set}^{q_1}=\{u\mid u(\I_{d+1})=1, \ u(b)=w\} \\
C_{w,s,\set}^{q_2}&=C_{w,s,\set}^{q_3}=C_{w,s,\set}^f=\emptyset
\end{align*}
Conversely, for any $v\in V^{d+1}$ and $w\in L_{A_i}-\{1\}$, define $C_{w,v,f,\set}\subseteq\GL_\set\times S_\set$ by:
\begin{align*}
C_{w,v,f,\set}^s&=C_{w,v,f,\set}^{q_j}=\emptyset \\
C_{w,v,f,\set}^f&=\{u\mid u(\I_{d+1})=v, \ u(b)=w\}
\end{align*}

By \Cref{lem-TTd}, each of the sets above is closed under $\IO_e$ for all $e\in E_\set$.  So, since $\ND_\set^1\times\SF_\op$ is contained in their union, it follows from \Cref{prop-io-classes} that the $\IO$-relations are contained in those as in the definition (see \Cref{def-ttape}).  The existence of the computation $\gamma_v$ for any $v\in V^{d+1}$ then implies the reverse inclusion for $\IO_\set^{s,f}$, while empty computations and transitivity imply the same for $\IO_\set^{s,s}$ and $\IO_\set^{f,f}$.  Hence, the $\IO$-relation of the operation $\set$ is as in the definition of $i$-Turing tape.

Now, let $Q=(Q_j)_{j\in\I_\T}$.

\medskip

\noindent\textbf{Case 1.} $Q_i$ is the exact fragment $[w_1\to w_2]$.

Then, define the subset $C_{ex}\subseteq\GL_Q\times S_Q$ by:
\begin{align*}
C_{ex}^s=C_{ex}^{q_1}&=V^{d+1}_{w_1} \\
C_{ex}^{q_2}&=\{v\in V^{d+1}_{w_1}\mid w_d(v(\I_d'))=w_1\} \\
C_{ex}^{q_3}&=\{v\in V^{d+1}_{w_2}\mid w_d(v(\I_d'))=w_2\} \\
C_{ex}^{q_4}=C_{ex}^f&=V^{d+1}_{w_2}
\end{align*}
Note that for any $v\in C_{ex}^{q_2}\cup C_{ex}^{q_3}$, it necessarily holds that $v(a_1)=1$.  

For any $w\in L_{A_i}^+$ such that $w\neq w_1$, define $C_{w,s,ex}\subseteq\GL_Q\times S_Q$ by:
\begin{align*}
C_{w,s,ex}^s=C_{w,s,ex}^{q_1}&=V^{d+1}_w \\
C_{w,s,ex}^{q_2}&=\{v\in V^{d+1}_w\mid w_d(v(\I_d'))=w\} \\
C_{w,s,ex}^{q_3}=C_{w,s,ex}^{q_4}&=C_{w,s,ex}^f=\emptyset
\end{align*}

Further, for any $w\in L_{A_i}^+$ such that $w\neq w_2$, define $C_{w,f,ex}\subseteq\GL_Q\times S_Q$ by:
\begin{align*}
C_{w,f,ex}^s=C_{w,f,ex}^{q_1}&=C_{w,f,ex}^{q_2}=\emptyset \\
C_{w,f,ex}^{q_3}&=\{v\in V^{d+1}_w\mid w_d(v(\I_d'))=w\} \\
C_{w,f,ex}^{q_4}=C_{w,f,ex}^f&=V^{d+1}_w
\end{align*}

Using \Cref{lem-TTd}, it is straightforward to check that each of these sets is closed under $\IO_e$ for each $e\in E_Q$.  So, \Cref{prop-io-classes} yields the following containments:
\begin{align*}
\IO_Q^{s,s},\IO_Q^{f,f}&\subseteq\{(v_1,v_2)\mid w_{d+1}(v_1)=w_{d+1}(v_2)\} \\
\IO_Q^{s,f}&\subseteq\{(v_1,v_2)\mid w_{d+1}(v_1)=w_1, \ w_{d+1}(v_2)=w_2\}
\end{align*}

Conversely, for any $v_1,v_2\in V^{d+1}$ such that $w_{d+1}(v_1)=w_{d+1}(v_2)$, it follows from \Cref{lem-TTd} that the single generic edge $(\s,\s,e_1)$ supports a computation $\gamma_{v_1,v_2,s}^{ex}$ between $(v_1,Q_s)$ and $(v_2,Q_s)$.  Similarly, $(\f,\f,e_5)$ supports a computation $\gamma_{v_1,v_2,f}^{ex}$ between $(v_1,Q_f)$ and $(v_2,Q_f)$.

Meanwhile, for any $v_1,v_2\in V^{d+1}$ such that $w_{d+1}(v_j)=w_d(v_j(\I_d'))=w_j$, by \Cref{lem-TTd} there exists a computation $\gamma_{v_1,v_2}^{ex}$ given by
\begin{align*}
\gamma_{v_1,v_2}^{ex}=\bigl((v_1,Q_s),e_1',(v_1,q_1),e_2',(v_1,q_2),e_3',(v_2,q_3),e_4',(v_2,q_4),e_5',(v_2,Q_f)\bigr)
\end{align*}
where $e_j'=(\s,\f,e_j)$ for all $j\in\overline{5}$.  

Hence, by transitivity, all containments given above are equalities, and so the definition of exact fragment implies that the $\IO$-relation of the operation is equivalent to that in the definition of $i$-Turing tape.

\medskip

\noindent \textbf{Case 2.} $Q_i$ is the (right) matching fragment $[*w_1 \to *w_2]$.  

For any $w\in L_{A_i}^+$ such that $ww_1^{-1}\in L_{A_i}^+$, letting $w'=ww_1^{-1}w_2$, define $C_{w,m}\subseteq\GL_Q\times S_Q$ by:
\begin{align*}
C_{w,m}^s&=C_{w,m}^{q_1}=V^{d+1}_w \\
C_{w,m}^{q_2}&=C_{w,m}^f=V^{d+1}_{w'}
\end{align*}

Conversely, for any $w\in L_{A_i}^+$ such that $ww_1^{-1}\in L_{A_i}-L_{A_i}^+$, define $C_{w,s,m}\subseteq\GL_Q\times S_Q$ by:
\begin{align*}
C_{w,s,m}^s&=C_{w,s,m}^{q_1}=V^{d+1}_w \\
C_{w,s,m}^{q_2}&=C_{w,s,m}^f=\emptyset
\end{align*}

Similarly, for $w\in L_{A_i}^+$ such that $ww_2^{-1}\notin L_{A_i}^+$, define $C_{w,f,m}\subseteq\GL_Q\times S_Q$ by:
\begin{align*}
C_{w,f,m}^s&=C_{w,f,m}^{q_1}=\emptyset \\
C_{w,f,m}^{q_2}&=C_{w,f,m}^f=V^{d+1}_w
\end{align*}

Note that for any $v\in V^{d+1}_w$, $w_d(v(\I_d'))$ is a suffix of $w$.  So, if $Q_i$ is applicable to $w_d(v(\I_d'))$, then it is also applicable to $w$.

As a result, by \Cref{lem-TTd}, each of the sets defined above is closed under $\IO_e$ for each $e\in E_Q$.  So, \Cref{prop-io-classes} implies the following containments:
\begin{align*}
\IO_Q^{s,s},\IO_Q^{f,f}&\subseteq\{(v_1,v_2)\mid w_{d+1}(v_1)=w_{d+1}(v_2)\} \\
\IO_Q^{s,f}&\subseteq\{(v_1,v_2)\mid Q_i(w_{d+1}(v_1))=w_{d+1}(v_2)\}
\end{align*}

As in the previous case, for any $v_1,v_2\in V^{d+1}$ such that $w_{d+1}(v_1)=w_{d+1}(v_2)$, the makeup of $\IO_\rea$ implies the existence of computations $\gamma_{v_1,v_2,s}^m$ and $\gamma_{v_1,v_2,f}^m$ supported by single generic edges, yielding the reverse inclusion for $\IO_Q^{s,s}$ and $\IO_Q^{f,f}$.  Meanwhile, for any $v_1,v_2\in V^{d+1}$ such that $w_{d+1}(v_j)=w_d(v_j(\I_d'))=ww_j$ for some $w\in L_{A_i}^+$, \Cref{lem-TTd} implies the existence of a computation $\gamma_{v_1,v_2}^m$ given by
\begin{align*}
\gamma_{v_1,v_2}^m=\bigl((v_1,Q_s),e_1',(v_1,q_1),e_2',(v_2,q_2),e_3',(v_2,Q_f)\bigr)
\end{align*}
where $e_j'=(\s,\f,e_j)$ for all $j\in\overline{3}$.

Hence, again all containments above are equalities by transitivity, i.e the $\IO$-relation of the operation is equivalent to that in the definition of $i$-Turing tape.

\medskip

Thus, $\TTape^{d+1}$ is an $i$-Turing tape of rank $d+1$.

Now, let $\F$ be the estimating set for $\TTape^{d+1}$ given by $\F_{\IRT^d}=(n^{1+\eps},n)$.  By \Cref{lem-TTd} and \Cref{prop-complexity-estimate}(3), it suffices to show $\TMSP_{\TTape^{d+1}}^\F\preceq(n^{1+\eps},n)$.

For $Q=(Q_j)_{j\in\I_\T}$, note the following:

\begin{itemize}

\item If $Q_i$ is exact, then for any $(v_1,v_2)\in\IO_Q^{s,f}$ such that $v_j(a_1)=1$ for $j=1,2$:
\begin{align*}
\tm_Q^\F(\gamma_{v_1,v_2}^{ex})&=|v_1|^{1+\eps}+1+\max(|v_1|,|v_2|)^{1+\eps}+1+|v_2|^{1+\eps} \\
&\leq 3\max(|v_1|,|v_2|)^{1+\eps}+2 \\
\sp_Q^\F(\gamma_{v_1,v_2}^{ex})&=\max(|v_1|,|v_2|)
\end{align*}
Further, for any $v_1,v_2\in V^{d+1}$ such that $w_{d+1}(v_1)=w_{d+1}(v_2)$, 
\begin{align*}
\tm_Q^\F(\gamma_{v_1,v_2,s}^{ex})=\tm_Q^\F(\gamma_{v_1,v_2,f}^{ex})&=\max(|v_1|,|v_2|)^{1+\eps}=|v_j|^{1+\eps} \\
\sp_Q^\F(\gamma_{v_1,v_2,s}^{ex})=\sp_Q^\F(\gamma_{v_1,v_2,f}^{ex})&=|v_j|
\end{align*}

\item If $Q_i$ is (right) matching, then for any $(v_1,v_2)\in\IO_Q^{s,f}$ such that $v_j(a_1)=1$ for $j=1,2$:
\begin{align*}
\tm_Q^\F(\gamma_{v_1,v_2}^m)&=|v_1|^{1+\eps}+\max(|v_1|,|v_2|)^{1+\eps}+|v_2|^{1+\eps} \\
&\leq3\max(|v_1|,|v_2|)^{1+\eps} \\
\sp_Q^\F(\gamma_{v_1,v_2}^m)&=\max(|v_1|,|v_2|)
\end{align*}
Further, for any $v_1,v_2\in V^{d+1}$ such that $w_{d+1}(v_1)=w_{d+1}(v_2)$, 
\begin{align*}
\tm_Q^\F(\gamma_{v_1,v_2,s}^m)=\tm_Q^\F(\gamma_{v_1,v_2,f}^m)&=\max(|v_1|,|v_2|)^{1+\eps}=|v_j|^{1+\eps} \\
\sp_Q^\F(\gamma_{v_1,v_2,s}^m)=\sp_Q^\F(\gamma_{v_1,v_2,f}^m)&=|v_j|
\end{align*}

\end{itemize}

Hence, by transitivity, for any $(v_1,s_1,v_2,s_2)\in\IO_Q$, there exists a computation $\gamma$ realizing $(v_1,v_2)\in\IO_Q^{s_1,s_2}$ and satisfying $\tm_Q^\F(\gamma)\leq5M^{1+\eps}+2$ and $\sp_Q^\F(\gamma)\leq M$ for $M=\max(|v_1|,|v_2|)$.

For the operation $\set$, note that any $(u_1,u_2)\in\IO_\set^{s,f}$ is realized by $\gamma_v^\set$ for $v=u_2(\I_{d+1})$.  In this case, $|u_1|=|v|=|u_2|$, $\tm_\set^\F(\gamma_v^\set)=2|v|^{1+\eps}+2$, and $\sp_\set^\F(\gamma_v^\set)=|v|$.

Further, for any $(u_1,u_2)\in\IO_\set^{f,f}$ such that $u_1\neq u_2$, $u_1(b)=1=u_2(b)$ and $w_{d+1}(v_1)=w_{d+1}(v_2)$ for $v_j=u_j(\I_{d+1})$.  But then setting $u\in\GL_\set$ such that $u(b)=w_{d+1}(v_j)$ and $u(\I_{d+1})=1$, the computations $\gamma_{v_j}^\set$ realize $(u,u_j)\in\IO_\set^{s,f}$.  Hence, letting $\gamma_{v_1,v_2}^\set$ be the concatenation of the inverse computation $(\gamma_{v_1}^\set)^{-1}$ with $\gamma_{v_2}^\set$, it follows that $\gamma_{v_1,v_2}^\set$ realizes $(u_1,u_2)\in\IO_\set^{f,f}$ and satisfies 
\begin{align*}
\tm_\set^\F(\gamma_{v_1,v_2}^\set)&=\sum\tm_\set^\F(\gamma_{v_j}^\set)=4|v_j|^{1+\eps}+4 \\
\sp_\set^\F(\gamma_{v_1,v_2}^\set)&=\max\bigl(\sp_\set^\F(\gamma_{v_j}^\set)\bigr)=|v_j|
\end{align*}
with $|u_1|=|v_1|=|v_2|=|u_2|$.

Thus, as all other relations of $\set$ are realized by an empty computation, it follows that $\TMSP_{\TTape^{d+1}}^\F\preceq(n^{1+\eps},n)$.

\end{proof}

\medskip

\subsection{Sequential Computations}\label{sec-seq-comp} \

\medskip

While the iterative construction of the previous section produced Turing tapes of any rank, \Cref{lem-ttape2} gives the same upper bound on the complexity regardless of rank.  As such, there is no evident improvement achieved from the iterative step, and so its purpose is as yet unclear.  To address this, a new measure of computational efficiency must be introduced.  This method is specific to Turing tapes and, crucially, is tailored to gauge how well the object simulates the Turing rewriting system to which it corresponds.

\begin{definition}[Sequential computations]\label{def-seq-comp}

Let $Q^1,\dots,Q^k\in\Q_\T$ and denote $Q^j=(Q^j_\ell)_{\ell\in\I_\T}$ for all $j\in\overline{k}$.  Then, an \emph{$i$-computation of $\T$ with history $(Q^1,\dots,Q^k)$} is a tuple $\mathcal{C}=(w_0,w_1,\dots,w_k)$ of words over $A_i$ such that $Q^j_i(w_{j-1})=w_j$ for all $j\in\overline{k}$.  Note that any computation of $\T$ determines an $i$-computation of $\T$, while any tuple of $\ell$-computations $(\mathcal{C}_\ell)_{\ell\in\I_\T}$ of $\T$ with a fixed history uniquely determine a computation of $\T$ with that history.

Now, suppose $\TTape_d$ is an $i$-Turing tape of $\T$ of some rank $d$.   Then, for any $i$-computation $\mathcal{C}=(w_0,w_1,\dots,w_k)$ of $\T$ with history $(Q^1,\dots,Q^k)$, a \emph{sequential computation of $\TTape_d$ approximating $\mathcal{C}$} is a sequence $\gamma_\mathcal{C}=(\gamma_1,\dots,\gamma_k)$ such that for all $j\in\overline{k}$, $\gamma_j$ is an input-output computation of $\EXP(Q^j)$ that $\EXP$-realizes $(v_{j-1},v_j)\in\IO_{Q^j}^{s,f}$, where $v_0,v_1,\dots,v_k\in V^{d}$ satisfy $w_d(v_\ell)=w_\ell$ for all $\ell\in\overline{0,k}$.  In this case, $\gamma_\mathcal{C}$ is said to \emph{start} and to \emph{end} with $v_0$ and $v_k$, respectively.  Moreover, the \emph{time} and \emph{space} of $\gamma_\mathcal{C}$ are defined to be $\tm(\gamma_\mathcal{C})=\sum_{j=1}^n\tm_{Q^j}(\gamma_j)$ and $\sp(\gamma_\mathcal{C})=\max(\sp_{Q^j}(\gamma_j))$, respectively.

If $f$ and $g$ are a pair of growth functions, then $\TTape_d$ is said to have \emph{sequential complexity at most $(f,g)$}, denoted $\SEQ(\TTape_d)\leq(f,g)$, if for any $i$-computation $\mathcal{C}=(w_0,w_1,\dots,w_k)$ of $\T$, there exists a sequential computation $\gamma_\mathcal{C}$ of $\TTape_d$ approximating $\mathcal{C}$ such that $\tm(\gamma_\mathcal{C})\leq f(n)$ and $\sp(\gamma_\mathcal{C})\leq g(n)$ for $n=\max(|w_0|,|w_k|,k)$.  As in previous settings, if $f',g'$ are another pair of growth functions satisfying $(f',g')\sim(f,g)$, then $(f',g')$ is called an \emph{asymptotic upper bound on the sequential complexity} of $\TTape_d$, denoted $\SEQ(\TTape_d)\preceq(f',g')$.

\end{definition}

Note that the definition of sequential complexity presented above is given exclusively in terms of computations of the expansion of a Turing tape.  As in the case with standard complexity, estimating sets can be introduced, yielding a concept of estimated sequential complexity which can be related to sequential complexity with an analogue of \Cref{prop-complexity-estimate}.  However, given the level of care needed for the proof of the following crucial statement, this concept is not elaborated upon in this manuscript.

\begin{proposition}\label{prop-ttape-seq-compl}
For all $d\in\N-\{0\}$, $\SEQ(\TTape_{\T,i}^{d,\eps})\preceq(n^{\frac{d+1}{d}+\varepsilon},n)$ for any $\eps>0$.
\end{proposition}

\begin{proof} 

For any $d\in\N$, set $\a_d=\frac{d}{d+1}$ and $\b_d=1/\a_d$.

Fix an $i$-computation $\mathcal{C}=(w_0,w_1,\dots,w_k)$ of $\T$ with history $(Q^1,\dots,Q^k)$.  Without loss of generality, suppose $k\geq1$, i.e the $i$-computation is nonempty.  The proof proceeds by induction on $d$, showing that $\SEQ(\TTape^d)\preceq(n^{\b_d+\eps},n)$ for $\TTape^d=\TTape_{\T,i}^{d,\eps}$.

For the base case, let $d=1$.  By \Cref{lem-ttape1}, there exists a constant $C_1\in\N$ such that $\TMSP_{\TTape^1}\leq(C_1n^{1+\eps}+C_1,C_1n+C_1)$.  So, for all $j\in\overline{k}$, there exists a computation $\gamma_j$ of $\EXP(Q^j)$ that $\EXP$-realizes $(w_{j-1},w_j)\in\IO_{Q^j}^{s,f}$ and satisfies:
\begin{align*}
\tm_{Q^j}(\gamma_j)&\leq C_1\max(|w_{j-1}|,|w_j|)^{1+\eps}+C_1 \\
\sp_{Q^j}(\gamma_j)&\leq C_1\max(|w_{j-1}|,|w_j|)+C_1
\end{align*}

Further, letting $K_\T=\max_{Q\in\Q_\T}|Q|$, then $\bigl||w_j|-|w_{j-1}|\bigr|\leq K_\T$ for all $j\in\overline{k}$.  As a result, $|w_j|\leq|w_0|+k\cdot K_\T$ for all $j$.  In particular, for $n=\max(|w_0|,|w_k|,k)$, $|w_j|\leq(1+K_\T)n$ for all $j\in\overline{k}$, and so
\begin{align*}
\tm_{Q^j}(\gamma_j)&\leq C_1(1+K_\T)^{1+\eps}n^{1+\eps}+C_1 \\
\sp_{Q^j}(\gamma_j)&\leq C_1(1+K_\T)n+C_1
\end{align*}
But then $\gamma_\mathcal{C}=(\gamma_1,\dots,\gamma_k)$ is a sequential computation of $\TTape_{\T,i}^{1,\eps}$ approximating $\mathcal{C}$ and satisfying
\begin{align*}
\tm(\gamma_\mathcal{C})&=\sum\tm_{Q^j}(\gamma_j)\leq k\bigl(C_1(1+K_\T)^{1+\eps}n^{1+\eps}+C_1\bigr) \\
&\leq C_1(1+K_\T)^{1+\eps}n^{2+\eps}+C_1n\leq 2C_1(1+K_\T)^{1+\eps}n^{2+\eps} \\
&=2C_1(1+K_\T)^{1+\eps}n^{\b_1+\eps} \\
\sp(\gamma_\mathcal{C})&\leq C_1(1+K_\T)n+C_1
\end{align*}

Hence, as $C_1$ and $K_\T$ are constants that do not depend on the $i$-computation $\mathcal{C}$, the asymptotic bound $\SEQ(\TTape_{\T,i}^{1,\eps})\preceq(n^{2+\eps},n)=(n^{\b_1+\eps},n)$ follows.

Now, suppose $\SEQ(\TTape_{\T,i}^{d,\eps})\preceq(n^{\b_d+\eps},n)$.  So, there exists a constant $C_d\in\N$ such that $\SEQ(\TTape_{\T,i}^{d,\eps})\leq(C_dn^{\b_d+\eps}+C_d,C_dn+C_d)$.  Further, by \Cref{lem-TTd}, it can be assumed that $C_d$ satisfies $\TMSP_{\IRT^d}\leq(C_dn^{1+\eps}+C_d,C_dn+C_d)$.

As above, let $n=\max(|w_0|,|w_k|,k)$.  Then, fix $t\leq2n^{1-\a_d}$ and $\ell_j,r_j\in\overline{k}$ for all $j\in\overline{t}$ such that:

\begin{itemize}

\item $\ell_1=0$, $r_t=k$

\item $\ell_{j+1}=r_j$ for $j\in\overline{t-1}$

\item $1\leq r_j-\ell_j\leq 2n^{\a_d}$ for all $j\in\overline{t}$.

\end{itemize}
For example, letting $m=\lceil n^{\a_d}\rceil$, one can set $t=\lceil k/m \rceil$ and then take $r_j=jm$ for $j\in\overline{t-1}$ and $r_t=k$.  In this case, note that $n^{1-\a_d}m\geq n\geq k$, so that $n^{1-\a_d}+1\geq t$.

Then, for all $p\in\overline{t}$, let $\mathcal{C}_p=(w_{\ell_p},w_{\ell_p+1},\dots,w_{r_p})$ be the $i$-computation of $\T$ with history $(Q^{\ell_p+1},\dots,Q^{r_p})$.  For each $p$, set $m_p=\max(|w_{\ell_p}|,|w_{r_p}|,r_p-\ell_p)$.  Note that by the definition of the indices, $m_p\leq\max(|w_{\ell_p}|,|w_{r_p}|,2n^{\a_d})$.

Fix $p\in\overline{t}$ and $v_p\in V^{d+1}$ such that $w_{d+1}(v_p)=w_{\ell_p}$.  In the next two cases, a sequential computation $\gamma_{\mathcal{C}_p}^{v_p}$ of $\TTape^{d+1}$ is constructed which approximates $\mathcal{C}_p$, starts with $v_p$, and satisfies certain sequential complexity bounds.

\medskip

\noindent
\textbf{Case 1.}  $|w_{\ell_p}|,|w_{r_p}|\leq 4K_\T n^{\a_d}$.

Then, $m_p\leq(4K_\T+2)n^{\a_d}$.

By the inductive hypothesis, there then exists a sequential computation $\gamma_{\mathcal{C}_p}^d=(\gamma_{\ell_p+1}^d,\dots,\gamma_{r_p}^d)$ of $\TTape^d$ approximating $\mathcal{C}_p$ and satisfying 
\begin{align*}
\tm(\gamma_{\mathcal{C}_p}^d)&\leq C_dm_p^{\b_d+\eps}+C_d \\
\sp(\gamma_{\mathcal{C}_p}^d)&\leq C_dm_p+C_d
\end{align*}

So, there exist $v_{\ell_p}^d,\dots,v_{r_p}^d\in V^{d}$ with $w_d(v_j^d)=w_j$ for all $j$ such that $\gamma_j^d$ is an input-output computation of the operation $\EXP(Q^j)$ of $\EXP(\TTape^d)$ which $\EXP$-realizes $(v_{j-1}^d,v_j^d)\in\IO_{Q^j}^{s,f}$.

Let $v_j^{d+1}\in V^{d+1}$ such that $v_j^{d+1}(a_1)=1$ and $v_j^{d+1}(\I_d')=(v_j^d)_d$.  By \Cref{lem-op-e}(1), there exists a computation $\eta_j$ of the operation $\EXP(Q^j_0)$ of $\EXP(\IRT^d)$ that $\EXP$-realizes $(v_{j-1}^{d+1},v_j^{d+1})\in\IO_{\IRT^d}^{s,f}$ and has the same complexity as $\gamma_j^d$, i.e it satisfies $\tm_{Q^j_0}(\eta_j)=\tm_{Q^j}(\gamma_j^d)$ and $\sp_{Q^j_0}(\eta_j)=\sp_{Q^j}(\gamma_j^d)$.

Fix $j\in\overline{\ell_p+1,r_p}$ and let $u_{j-1},u_j\in L_{Q^j}^\exp$ such that $u_\ell(\I_{d+1})=v_\ell^{d+1}$ and $u_\ell(\I_{Q^j}^\inv)=1_{Q^j}^\inv$.

If $Q^j_i$ is matching, then as above, \Cref{lem-op-e}(1) then yields a computation $\chi_j$ of the operation $\EXP(Q^j)$ of $\EXP(\TTape^{d+1})$ between $(u_{j-1},q_1)$ and $(u_j,q_2)$ and satisfying $\tm_{Q^j}(\chi_j)=\tm_{Q^j_0}(\eta_j)$, $\sp_{Q^j}(\chi_j)=\sp_{Q^j_0}(\eta_j)$.  Then, $\chi_j$ can be augmented by the edges $\pi_{e_1}(e_0)$ and $\pi_{e_3}(e_0)$, where $e_0\in E_{\rea}^\exp$ is the unlabelled edge of $\rea$, yielding an input-output computation $\gamma_j^{d+1}$ of $\EXP(Q^j)$.

Conversely, if $Q^j_i$ is exact, then again \Cref{lem-op-e}(1) yields a computation $\chi_j$ of the operation $\EXP(Q^j)$ between $(u_{j-1},q_2)$ and $(u_j,q_3)$ and with the same complexity as $\eta_j$.  Since $u_j(a_1)=1$, $e_4$ is applicable to $(u_j,q_3)$; similarly, $e_2^{-1}$ is applicable to $(u_{j-1},q_2)$.  Augmenting $\chi_j$ with these two applications, the resulting computation can then be augmented by the image of the unlabelled edge of $\EXP(\rea)$ as above, again producing an input-output computation $\gamma_j^{d+1}$ of $\EXP(Q^j)$.

In both cases, $\gamma_j^{d+1}$ $\EXP$-realizes $(v_{j-1}^{d+1},v_j^{d+1})\in\IO_{Q^j}^{s,f}$ and satisfies:
\begin{align*}
\tm_{Q^j}(\gamma_j^{d+1})&\leq\tm_{Q^j}(\chi_j)+4=\tm_{Q^j_0}(\eta_j)+4 \\
\sp_{Q^j}(\gamma_j^{d+1})&=\sp_{Q^j}(\chi_j)=\sp_{Q^j_0}(\eta_j)
\end{align*}

Next, note that $w_{d+1}(v_{\ell_p}^{d+1})=w_d(v_{\ell_p}^d)=w_{\ell_p}=w_{d+1}(v_p)$.  As a result, \Cref{lem-TTd} implies the existence of a computation $\phi_p$ of $\EXP(\rea)$ that $\EXP$-realizes $(v_p,v_{\ell_p}^{d+1})\in\IO_{\rea}^{s,s}$ and satisfies $\tm_\rea(\phi_p)\leq C_d|v_p|^{1+\eps}+C_d$ and $\sp_\rea(\phi_p)\leq C_d|v_p|+C_d$.  Using \Cref{lem-op-e}(1) again then yields a computation $\phi_p'$ of $\EXP(Q^{\ell_p+1})$ that $\EXP$-realizes $(v_p,v_{\ell_p})\in\IO_{Q^{\ell_p+1}}^{s,s}$ and has the same complexity bounds as $\phi_p$.

Then, $\phi_p'$ can be concatenated with $\gamma_{\ell_p+1}^{d+1}$ to produce a computation $\psi_p$ of the operation $\EXP(Q^{\ell_p+1})$ of $\EXP(\TTape^{d+1})$ that $\EXP$-realizes $(v_p,v_{\ell_p+1})\in\IO_{Q^{\ell_p+1}}^{s,f}$.

As a result, $\gamma_{\mathcal{C}_p}^{v_p}=(\psi_p,\gamma_{\ell_p+2}^{d+1},\dots,\gamma_{r_p}^{d+1})$ is a sequential computation of $\TTape^{d+1}$ approximating $\mathcal{C}_p$ starting with $v_p$ and satisfying:
\begin{align*}
\tm(\gamma_{\mathcal{C}_p}^{v_p})&=\tm_{Q^{\ell_p+1}}(\phi_p')+\sum_j\tm_{Q^j}(\gamma_j^{d+1})\leq C_d|v_p|^{1+\eps}+C_d+\sum_j(\tm_{Q^j}(\gamma_j^d)+4) \\
&=C_d|w_{\ell_p}|^{1+\eps}+\tm(\gamma_{\mathcal{C}_p}^d)+4(r_p-\ell_p)+C_d \\
&\leq C_dm_p^{1+\eps}+C_dm_p^{\b_d+\eps}+8n^{\a_d}+2C_d\leq2C_dm_p^{\b_d+\eps}+8n^{\a_d}+2C_d \\
&\leq2C_d(4K_\T+2)^{\b_d+\eps}n^{\a_d(\b_d+\eps)}+8n^{\a_d}+2C_d \\
&\leq 2C_d(4K_\T+2)^{\b_d+\eps}n^{1+\eps}+8n^{\a_d}+2C_d \\
\medskip
\sp(\gamma_{\mathcal{C}_p}^{v_p})&=\max(\sp_{Q^{\ell_p+1}}(\phi_p'),\sp_{Q^j}(\gamma_j^{d+1}))\leq\max(C_d|v_p|+C_d,\sp_{Q^j}(\gamma_j^d)) \\
&\leq\max(C_d|v_p|+C_d,\sp(\gamma_{\mathcal{C}_p}^d))\leq C_dm_p+C_d \\
&\leq C_d(4K_\T+2)n^{\a_d}+C_d
\end{align*}

\medskip

\noindent
\textbf{Case 2.}  $\max(|w_{\ell_p}|,|w_{r_p}|)>4K_\T n^{a_d}$.

By the definition of $K_\T$, for any $j\in\overline{\ell_p+1,r_p}$, $\bigl| |w_j|-|w_{j-1}| \bigr|\leq K_\T$.  So, for all $j\in\overline{\ell_p,r_p}$, $|w_j|\geq\max(|w_{\ell_p}|,|w_{r_p}|)-K_\T(r_p-\ell_p)>2K_\T n^{\a_d}$.  In particular, $|w_j|>K_\T$ for all $j$, and so no exact fragment of a Turing action in $\Q_\T$ is applicable to $w_j$.  Hence, $Q^j_i$ is matching for all $j\in\overline{\ell_p+1,r_p}$.  

Suppose $|w_{\ell_p}|>4K_\T n^{\a_d}$ and set $w_{\ell_p}=ww_{\ell_p}'$ such that $|w_{\ell_p}'|=4K_\T n^{\a_d}$.  

Suppose that $w_j=ww_j'$ with $|w_j'|\geq(4n^{\a_d}-2(j-\ell_p))K_\T$ for some $j\in\overline{\ell_p,r_p-1}$.  Then, letting $Q^{j+1}_i=[*y_j\to*z_j]$, $y_j$ is a suffix of $w_j$ and satisfies $|y_j|\leq K_\T$.  As $j-\ell_p\leq r_p-\ell_p\leq2n^{\a_d}$, it follows that $|w_j'|\geq K_\T\geq|y_j|$.  So, $y_j$ is also a suffix of $w_j'$, meaning that $Q^{j+1}_i$ is applicable to $w_j'$.  As a result, $w_{j+1}=ww_{j+1}'$ for $w_{j+1}'=w_j'y_j^{-1}z_j$ with
$$|w_{j+1}'|\geq|w_j'|-|y_j|\geq|w_j'|-K_\T\geq(4n^{\a_d}-2((j+1)-\ell_p))K_\T$$
Hence, $w_j=ww_j'$, $Q^j_i$ is applicable to $w_{j-1}'$, and $Q^j_i(w_{j-1})=wQ^j_i(w_{j-1}')=w_j$ for all $j$.

Conversely, if $|w_{\ell_p}|\leq4K_\T n^{\a_d}$, then necessarily $|w_{r_p}|>4K_\T n^{\a_d}$.  But then setting $w_{r_p}=ww_{r_p}'$ such that $|w_{r_p}'|=4K_\T n^{\a_d}$, analogous arguments apply to the inverse actions to produce $w_j'$ for all $j$.

So, there exists an $i$-computation $\mathcal{C}_p'=(w_{\ell_p}',w_{\ell_p+1}',\dots,w_{r_p}')$ of $\T$ with history $(Q^{\ell_p+1},\dots,Q^{r_p})$.  Letting $m_p'=\max(|w_{\ell_p}'|,|w_{r_p}'|,r_p-\ell_p)$, there then exists a sequential computation $\rho_{\mathcal{C}_p'}^d=(\rho_{\ell_p+1}^d,\dots,\rho_{r_p}^d)$ approximating $\mathcal{C}_p'$ and satisfying
\begin{align*}
\tm(\rho_{\mathcal{C}_p'}^d)&\leq C_d(m_p')^{\b_d+\eps}+C_d \\
\sp(\rho_{\mathcal{C}_p'}^d)&\leq C_dm_p'+C_d
\end{align*}
Note that in either case, one of $|w_{\ell_p}'|$ or $|w_{r_p}'|$ is equal to $4K_\T n^{\a_d}$, so that the other is at most $4K_\T n^{\a_d}+K_\T(r_p-\ell_p)\leq6K_\T n^{\a_d}$.  This yields $m_p'\leq(6K_\T+2)n^{\a_d}$.

So, there exist $v_{\ell_p}',\dots,v_{r_p}'\in V^{d}$ with $w_d(v_j')=w_j'$ for all $j$ such that $\rho_j^d$ is an input-output computation of the operation $\EXP(Q^j)$ of $\EXP(\TTape^d)$ which $\EXP$-realizes $(v_{j-1}',v_j')\in\IO_{Q^j}^{s,f}$.  

Let $v_j^{d+1}\in V^{d+1}$ such that $v_j^{d+1}(a)=w$ and $v_j^{d+1}(\I_d')=(v_j')_d$.  By \Cref{lem-op-e}(1), there exists a computation $\eta_j$ of the operation $\EXP(Q_0^j)$ of $\EXP(\IRT^d)$ that $\EXP$-realizes $(v_{j-1}^{d+1},v_j^{d+1})\in\IO_{\IRT^d}^{s,f}$ and has complexity given by $\tm_{Q_0^j}(\eta_j)=\tm_{Q^j}(\rho_j^d)$ and $\sp_{Q_0^j}(\eta_j)=\sp_{Q^j}(\rho_j^d)+|w|$.

As in Case 1, for all $j\in\overline{\ell_p+1,r_p}$, there then exists a computation $\gamma_j^{d+1}$ of the computation $\EXP(Q^j)$ of $\EXP(\TTape^{d+1})$ which $\EXP$-realizes $(v_{j-1}^{d+1},v_j^{d+1})\in\IO_{Q^j}^{s,f}$ and satisfies:
\begin{align*}
\tm_{Q^j}(\gamma_j^{d+1})&\leq\tm_{Q_0^j}(\eta_j)+4 \\
\sp_{Q^j}(\gamma_j^{d+1})&=\sp_{Q_0^j}(\eta_j)
\end{align*}
Further, note that $w_{d+1}(v_j^{d+1})=ww_j'=w_j$ for all $j$.  In particular, as $w_d(v_p)=w_{\ell_p}$, there again exists a computation $\phi_p'$ of the operation $\EXP(Q^{\ell_p+1})$ of $\EXP(\TTape^{d+1})$ that $\EXP$-realizes $(v_p,v_{\ell_p})\in\IO_{Q^{\ell_p+1}}^{s,s}$ and satisfies 
\begin{align*}
\tm_{Q^{\ell_p+1}}(\phi_p')&\leq C_d|v_p|^{1+\eps}+C_d \\
\sp_{Q^{\ell_p+1}}(\phi_p')&\leq C_d|v_p|+C_d
\end{align*}

As a result, letting $\psi_p$ be the computation of $\EXP(Q^{\ell_p+1})$ obtained by concatenating $\phi_p'$ and $\gamma_{\ell_p+1}^{d+1}$, then again $\gamma_{\mathcal{C}_p}^{v_p}=(\psi_p,\gamma_{\ell_p+2}^{d+1},\dots,\gamma_{r_p}^{d+1})$ is a sequential computation of $\TTape^{d+1}$ approximating $\mathcal{C}_p$ starting with $v_p$.  Further, the sequential computation satisfies:

\begin{align*}
\tm(\gamma_{\mathcal{C}_p}^{v_p})&=\tm_{Q^{\ell_p+1}}(\phi_p')+\sum_j\tm_{Q^j}(\gamma_j^{d+1})\leq C_d|v_p|^{1+\eps}+C_d+\sum_j(\tm_{Q^j}(\rho_j^d)+4) \\
&=C_d|w_{\ell_p}|^{1+\eps}+\tm(\rho_{\mathcal{C}_p'}^d)+4(r_p-\ell_p)+C_d \\
&\leq C_d|w_{\ell_p}|^{1+\eps}+C_d(m_p')^{\b_d+\eps}+8n^{\a_d}+2C_d \\
&\leq C_d|w_{\ell_p}|^{1+\eps}+C_d(6K_\T+2)^{\b_d+\eps}n^{\a_d(\b_d+\eps)}+8n^{\a_d}+2C_d \\
&\leq C_d|w_{\ell_p}|^{1+\eps}+C_d(6K_\T+2)^{\b_d+\eps}n^{1+\eps}+8n^{\a_d}+2C_d \\
\medskip
\sp(\gamma_{\mathcal{C}_p}^{v_p})&=\max(\sp_{Q^{\ell_p+1}}(\phi_p'),\sp_{Q^j}(\gamma_j^{d+1}))\leq\max(C_d|v_p|+C_d,\sp_{Q^j}(\rho_j^d)+|w|) \\
&\leq C_d\max(|v_p|,m_p')+C_d+|w| \\
&\leq C_d\max(|w_{\ell_p}|,(6K_\T+2)n^{\a_d})+C_d+|w| \\
&\leq C_d(3|w_{\ell_p}|+2n^{\a_d})+C_d+|w_{\ell_p}| \\
&\leq(3C_d+1)|w_{\ell_p}|+2C_dn^{\a_d}+C_d
\end{align*}

\medskip

In either case, letting $L_d=3C_d(6K_\T+2)^{\b_d+\eps}+8$, $\gamma_{\mathcal{C}_p}^{v_p}$ is a sequential computation of $\TTape^{d+1}$ approximating $\mathcal{C}_p$, starting with $v_p$, and satisfying:
\begin{align*}
\tm(\gamma_{\mathcal{C}_p}^{v_p})&\leq L_d|w_{\ell_p}|^{1+\eps}+L_dn^{1+\eps}+L_d \\
\sp(\gamma_{\mathcal{C}_p}^{v_p})&\leq L_d|w_{\ell_p}|+L_dn+L_d
\end{align*}

Now, fix $v_{1,s}\in V^{d+1}$ such that $w_{d+1}(v_1^s)=w_0$.  As $\ell_1=0$, the computation $\gamma_{\mathcal{C}_1}^{v_{1,s}}$ exists as above.  Then, let $v_{2,s}\in V^{d+1}$ be the value with which $\gamma_{\mathcal{C}_1}$ ends.  By construction, it then follows that $w_{d+1}(v_{2,s})=w_{r_1}=w_{\ell_2}$, so that $\gamma_{\mathcal{C}_2}^{v_{2,s}}$ exists.

Continuing this way, there exist $v_{1,s},\dots,v_{t,s},v_{t+1,s}\in V^{d+1}$ such that for all $p\in\overline{t}$, $\gamma_{\mathcal{C}_p}^{v_{p,s}}$ is a sequential computation as above which starts and ends with $v_{p,s}$ and $v_{p+1,s}$, respectively.

Hence, $\gamma_{\mathcal{C}_1}^{v_{1,s}},\dots,\gamma_{\mathcal{C}_t}^{v_{t,s}}$ can be concatenated to produce a sequential computation $\gamma_{\mathcal{C}}$ of $\TTape^{d+1}$ approximating $\mathcal{C}$ and satisfying:
\begin{align*}
\tm(\gamma_{\mathcal{C}})&=\sum_p\tm(\gamma_{\mathcal{C}_p}^{v_{p,s}})\leq L_d\sum_p|w_{\ell_p}|^{1+\eps}+L_dtn^{1+\eps}+L_d \\
&\leq L_dt\max_p|w_{\ell_p}|^{1+\eps} +L_dtn^{1+\eps}+L_d \\
\medskip
\sp(\gamma_{\mathcal{C}})&=\max_p\sp(\gamma_{\mathcal{C}_p}^{v_{p,s}})\leq L_d\max_p|w_{\ell_p}|+L_dn+L_d
\end{align*}

But since $\bigl| |w_j|-|w_{j-1}| \bigr|\leq K_\T$ for all $j\in\overline{k}$, 
$$|w_j|\leq|w_0|+K_\T k\leq n+K_\T n=(K_\T+1)n$$
In particular, $\max_p|w_{\ell_p}|\leq(K_\T+1)n$.

Thus, as $t\leq2n^{1-\a_d}$, letting $C_{d+1}=4L_d(K_\T+2)^{1+\eps}$ yields:
\begin{align*}
\tm(\gamma_\mathcal{C})&\leq L_dt((K_\T+1)^{1+\eps}+1)n^{1+\eps}+L_d \\
&\leq2L_dn^{1-\a_d}\cdot2(K_\T+1)^{1+\eps}n^{1+\eps}+L_d \\
&\leq4L_d(K_\T+1)^{1+\eps}n^{2-\a_d+\eps}+L_d \\
&\leq C_{d+1}n^{1+\b_{d+1}+\eps}+C_{d+1} \\
\medskip
\sp(\gamma_{\mathcal{C}})&\leq L_d(K_\T+2)n+L_d\leq C_{d+1}n+C_{d+1}
\end{align*}

\end{proof}

\begin{corollary}\label{cor-ttape-eps}
	Let $\T$ be a Turing rewriting system over the Turing system hardware $(\I_\T,\A_\T)$.  Then for any $i\in\I_\T$ and any $\eps>0$, there exists an $i$-Turing tape $\TTape_{\T,i}^{\varepsilon}=\TTape_{\T,i}^\eps[\I_{i,\eps}]$ of $\T$ of some rank satisfying both $\TMSP_{\TTape_{\T,i}^\eps}\preceq(n^{1+\eps},n)$ and $\SEQ(\TTape_{\T,i}^\eps)\preceq(n^{1+\varepsilon},n)$.
\end{corollary}

\begin{proof}

For fixed $\eps>0$, let $d\in\N$ such that $d\eps>2$.  Then, by \Cref{prop-ttape-seq-compl}, the $i$-Turing tape $\TTape_{\T,i}^{d,\eps/2}$ of $\T$ of rank $d$ satisfies the statement.

\end{proof}


\section{Proof of \Cref{main-theorem}}\label{sec-final-proof}

By \Cref{prop-machine-acc} and \Cref{lem-tm-system}, it suffices to show that for any $\eps>0$ and any Turing rewriting system $\T$ such that $\TM_\T\succeq n$, there exists an acceptor $\acc^\eps$ of $L_\T$ which satisfies $\TMSP_{\acc^\eps}\preceq(\TM_\T^{1+\eps},\TM_\T)$.  Moreover, \Cref{prop-acc-ch-wacc} then implies that it suffices to find a weak acceptor $\wacc^\eps$ of $L_\T$ satisfying $\WTMSP_{\wacc^\eps}\preceq_1(\TM_\T^{1+\eps},\TM_\T)$.

To this end, fix a Turing rewriting system $\T$ over a Turing system hardware $\H_\T=(\I_\T,\A_\T)$.  Let $k=|\I_\T|$ and fix an enumeration $\I_\T=\{i_1,\dots,i_k\}$ such that $i_1$ is the input tape of $\T$.  As adding an index to $\I_\T$ with an empty corresponding alphabet will not affect the language or complexity of $\T$, it may be assumed without loss of generality that $k\geq2$.  Further, let $(S_\T,E_\T)$ be the underlying directed graph defining $\T$.

Then, define the proper operation $\wacc^\eps=\wacc^\eps(a)$ as follows:

\medskip

\begin{enumerate}

\item For all $j\in\overline{k}$, there exists an instance $\TT_j\in\INST_{\wacc^\eps}$ constructed from the $i$-Turing tape $\TTape_{\T,i_j}^\eps$ defined in \Cref{cor-ttape-eps} and given by renaming the value tapes to disjoint index sets $\I_{j,\eps}$, i.e so that $\I_{j,\eps}\cap\I_{\ell,\eps}=\emptyset$ for $j\neq \ell$.

\medskip

\item $\I_{\wacc^\eps}^\int=\left(\bigsqcup_{j=1}^k\I_{j,\eps}\right)\sqcup\{b\}$, where $\A_{\wacc^\eps}(b)=\A_\T(i_1)$.

\medskip

\item For every $q\in S_\T$, there exists a corresponding state $p_q$ of $\wacc^\eps$.

\medskip

\item For every $e\in E_\T$, there exist $k-1$ corresponding states $s_e(1),\dots,s_e(k-1)$.  Moreover, setting $s_e(0)=p_{t(e)}$, $s_e(k)=p_{h(e)}$, and $\lab(e)=Q\in\Q_{\wacc^\eps}$, then there exist corresponding edges $e(1),\dots,e(k)$ where $e(j)=(s_e(j-1),s_e(j),\TT_{j}.Q())$ (see \Cref{fig-QEDGE} for the case where $e$ is not a loop).

\medskip

\item In addition to the $|S_\T|+(k-1)|E_\T|$ states introduced in (3), there are four more states:  The starting and finishing states comprising $\SF_{\wacc^\eps}$ and two more named $q$ and $q'$.

\medskip

\item Let $\act_0$ be the $0$-action over $\H_{\wacc^\eps}$ consisting of the guard fragments $[c=1]$ for all $c\in\I_{\wacc^\eps}^\int$.  Further, let $s_\T,f_\T\in S_\T$ be the starting and finishing states of $\T$.  Then, in addition to the $k|E_\T|$ edges introduced in (4), there are six more edges:
\begin{itemize}

\item $e_s=((\wacc^\eps)_s,q,\act_0)$

\item $e_\swap=(q,q',\swap^A(a,b))$

\item $e_\set=(q',p_{s_\T},\TT_{1}(b))$

\item $e_f=(p_{f_\T},(\wacc^\eps)_f,\act_0)$

\item The inverse edges $e_s^{-1}$ and $e_f^{-1}$

\end{itemize}

\end{enumerate}

\begin{figure}[hbt]
	\centering
	\include{QEDGE}
	\caption{The $k$ edges of $\G_e'$ in $\wacc^\eps$ corresponding to a single edge of $\G_\T$.}
	\label{fig-QEDGE}
\end{figure}

Note that the naming conventions necessitate $\I_{\wacc^\eps}^\ext=\{a\}$ and $\I_{\wacc^\eps}^\val=\emptyset$.  Further, the convention of minimal hardware implies that $\A_{\wacc^\eps}(c)=\A_\T(i_j)$ for all $c\in\I_{j,\eps}$.

The disjointness of the index sets $\I_{j,\eps}$ described in (1) guarantees that $\G_{\wacc^\eps}$ satisfies condition (5) of \Cref{def-s1-graph}.  Similarly, condition (4) of \Cref{def-s1-graph} is ensured by the existence of the inverse edges $e_s^{-1}$ and $e_f^{-1}$ in (6).  The rest of the conditions are similarly verifiable, and hence $\G_{\wacc^\eps}$ is indeed an $S^*$-graph.

What's more, the definition of the action $\act_0$ labelling $e_s$ and $e_f^{-1}$ immediately implies that condition (O) is satisfied.  Thus, $\wacc^\eps$ is indeed a proper operation.

The subgraph of $\G_{\wacc^\eps}$ which consists of the states and edges constructed in (3) and (4) is denoted $\G_\T'$.  The states of $\G_\T'$ arising from (3) are called \emph{regular}, while those arising from (4) are called \emph{extra}.  

Further, for any $e\in E_\T$, the subgraph of $\G_\T'$ constructed from $e$ as in \Cref{fig-QEDGE} is denoted $\G_e'$.  As such, $\G_\T'$ can be viewed as a copy of the graph of $\G_\T$ where each single edge $e$ is replaced by a subgraph consisting of $k$ edges.  

With this terminology in hand, \Cref{fig-wacc-eps} then depicts the general layout of the proper operation $\wacc^\eps$, with the subgraph $\G_\T'$ `left blank' as its makeup depends on the makeup of the fixed Turing rewriting system $\T$.

\begin{figure}[hbt]
	\centering
	\include{WACCEPS}
	\caption{The proper operation $\wacc^\eps(a)$.}
	\label{fig-wacc-eps}
\end{figure}

Now, let $\gamma=\bigl((u_0,s_0),e_1',(u_1,s_1),e_2',\dots,e_m',(u_m,s_m)\bigr)$ be a computation of $\wacc^\eps$ and denote $e_i'=(\sf_i,\sf_i',e_i'')$ for all $i\in\overline{m}$.  Then, $\gamma$ is called a \emph{$\G_\T'$-computation} if each state $s_i$ and each edge $e_i''$ is contained in $\G_\T'$.  

If in addition $s_i$ is a regular state if and only if $i\in\{0,m\}$, then $\gamma$ is called a \emph{$\T$-edge computation}.  By the construction of $\G_\T'$, there then exists $e\in E_\T$ and $j_1,\dots,j_m\in\overline{k}$ such that $e_i''=e(j_i)$ for all $i\in\overline{m}$.  In this case, $\gamma$ is called a \emph{$\T$-edge computation of $e$}.

Let $V^{j,\eps}$ be the set of values of $\TTape_{\T,i_j}^\eps$.  So, letting $d_{j,\eps}$ be the rank of this $i_j$-Turing tape, $V^{j,\eps}=V^{d_{j,\eps}}$.  Then, since any $v\in\VAL_{\TT_j}$ is a copy of an element of $V^{j,\eps}$, the map $w_{d_{j,\eps}}$ can be extended naturally to a map $w_{j,\eps}$ on $\VAL_{\TT_j}$, so that $w_{j,\eps}(v)\in\A_\T(i_j)$.

For every $u\in\GL_{\wacc^\eps}$ and every $j\in\overline{k}$, denote $v_j(u)=u(\I_{\eps,j})\in\VAL_{\TT_j}$.  Then, define $w_\T(u)\in L_{\H_\T}$ by $w_\T(u)(i_j)=w_{j,\eps}(v_j(u))$ for all $j\in\overline{k}$.

\begin{lemma}\label{lem-T-edge-comp}

Let $e=(q_1,q_2,Q)\in E_\T$ with $Q=(Q_{i_j})_{j\in\overline{k}}\in\Q_\T$.  Suppose $$\gamma=\bigl((u_0,s_0),e_1',(u_1,s_1),e_2',\dots,e_m',(u_m,s_m)\bigr)$$ is a $\T$-edge computation of $e$ with $e_i'=(\sf_i,\sf_i',e_i'')$ for all $i\in\overline{m}$.  Then:

\begin{enumerate}

\item If $s_0=s_m$, then $w_\T(u_0)=w_\T(u_m)$.

\item If $s_0=p_{q_1}$ and $s_m=p_{q_2}$, then $w_\T(u_0)$ is $Q$-applicable with $Q(w_\T(u_0))=w_\T(u_m)$.

\item If $s_0=p_{q_2}$ and $s_m=p_{q_1}$, then $w_\T(u_0)$ is $Q^{-1}$-applicable with $Q^{-1}(w_\T(u_0))=w_\T(u_m)$.

\end{enumerate}

\end{lemma}

\begin{proof}

For every $j\in\overline{k}$, let $R_j=\{(u_1',u_2')\in\GL_{\wacc^\eps}^2\mid u_1'(\I_{\wacc^\eps}-\I_{j,\eps})=u_2'(\I_{\wacc^\eps}-\I_{j,\eps})\}$.

Then, by the definition of $i_j$-Turing tape, for every $j\in\overline{k}$ the $\IO$-relation of $e(j)$ is given by:
\begin{align*}
\IO_{e(j)}^{s_e(j-1),s_e(j-1)}=\IO_{e(j)}^{s_e(j),s_e(j)}&=R_j\cap\{(u_1',u_2')\mid w_{j,\eps}(v_j(u_1'))=w_{j,\eps}(v_j(u_2'))\} \\
\IO_{e(j)}^{s_e(j-1),s_e(j)}&=R_j\cap\{(u_1',u_2')\mid Q_{i_j}(w_{j,\eps}(v_j(u_1')))=w_{j,\eps}(v_j(u_2'))\}
\end{align*}
where, as in (4), $s_e(0)=p_{q_1}$ and $s_e(k)=p_{q_2}$.

Fix $i\in\overline{0,m}$ and $j\in\overline{k}$.  Set $\ell\in\overline{0,k}$ such that $s_i=s_e(\ell)$.

\bigskip

\noindent
(I) Suppose $\ell<j$.

Let $r\in\overline{i,m}$ and suppose $s_t\neq s_e(j)$ for all $t\in\overline{i,r}$.  Then, for all $t\in\overline{i,r}$, $e_t''=e(j)$ implies $(u_{t-1},u_t)\in\IO_{e(j)}^{s_e(j-1),s_e(j-1)}$.  The makeup of the $\IO$-relations described above then implies that $w_{j,\eps}(v_j(u_i))=w_{j,\eps}(v_j(u_r))$.

So, if $s_t\neq s_e(j)$ for all $t\in\overline{i,m}$, then $w_{j,\eps}(v_j(u_i))=w_{j,\eps}(v_j(u_m))$.  Further, since $s_e(\ell)$ and $p_{q_2}$ lie in distinct components of $\G_e'-\{s_e(j)\}$, then necessarily $s_m=p_{q_1}$.

Otherwise, let $r\in\overline{i,m}$ be the minimal index satisfying $s_r=s_e(j)$.  The reasoning above then yields $w_{j,\eps}(v_j(u_i))=w_{j,\eps}(v_j(u_{r-1}))$.  Further, for all $p\in\overline{j+1,k}$, $s_e(\ell)$ and $s_e(p)$ lie in distinct components of $\G_e'-\{s_e(j)\}$, and so $s_{r-1}=s_e(j-1)$.  The makeup of $\IO_{e(j)}$ described above then implies that $w_{j,\eps}(v_j(u_{r-1}))$ is $Q_{i_j}$-applicable with $Q_{i_j}(w_{j,\eps}(v_j(u_{r-1})))=w_{j,\eps}(v_j(u_r))$.  Hence, $Q_{i_j}(w_{j,\eps}(v_j(u_i)))=w_{j,\eps}(v_j(u_r))$.

\bigskip

\noindent
(II) Conversely, suppose $\ell\geq j$.

By a similar argument, if $s_t\neq s_e(j-1)$ for all $t\in\overline{i,r}$, then $w_{j,\eps}(v_j(u_i))=w_{j,\eps}(v_j(u_r))$.  Hence, as in (I) $s_t\neq s_e(j-1)$ for all $t\in\overline{i,m}$ implies $w_{j,\eps}(v_j(u_i))=w_{j,\eps}(v_j(u_m))$ with $s_m=p_{q_2}$.

Further, if $r\in\overline{i,m}$ is the minimal index satisfying $s_r=s_e(j-1)$, then analogous arguments yield $w_{j,\eps}(v_j(u_i))=w_{j,\eps}(v_j(u_{r-1}))$ and $s_{r-1}=s_e(j)$.  But then the makeup of the $\IO_{e(j)}$ yields $Q_{i_j}(w_{j,\eps}(v_j(u_r)))=w_{j,\eps}(v_j(u_{r-1}))$.  Hence, $Q_{i_j}^{-1}(w_{j,\eps}(v_j(u_i)))=w_{j,\eps}(v_j(u_r))$.

\bigskip

Now, suppose $s_0=p_{q_1}=s_e(0)$ and fix $j\in\overline{k}$.  

If there is no $i\in\overline{m}$ such that $s_i=s_e(j)$, then as in (I), $w_{j,\eps}(v_j(u_i))=w_{j,\eps}(v_j(u_m))$ and $s_m=p_{q_1}$.

Otherwise, let $r_1\in\overline{m}$ be the minimal index such that $s_{r_1}=s_e(j)$.  Then, iteratively define $r_2,\dots,r_n$ as follows:

\begin{itemize}

\item For $\ell$ odd, let $r_{\ell+1}\in\overline{r_\ell,m}$ be the minimal index such that $s_{r_{\ell+1}}=s_e(j-1)$; if no such index exists, set $n=\ell$.

\item For $\ell$ even, let $r_{\ell+1}\in\overline{r_\ell,m}$ be the minimal index satisfying $s_{r_{\ell+1}}=s_e(j)$; if no such index exists, set $n=\ell$.

\end{itemize}
Setting $r_0=0$, then (I) and (II) above imply that for all $\ell\in\overline{t}$,
$$w_{j,\eps}(v_j(r_{\ell}))=
\begin{cases}
Q_{i_j}(w_{j,\eps}(v_j(u_{r_{\ell-1}}))) & \text{ if } \ell \text{ is odd} \\
Q_{i_j}^{-1}(w_{j,\eps}(v_j(u_{r_{\ell-1}}))) & \text{ if } \ell \text{ is even}
\end{cases}
$$
In particular, for $z=\lfloor n/2 \rfloor$, $w_{j,\eps}(v_j(u_{r_{2z}}))=w_{j,\eps}(v_j(u_{r_0}))=w_{j,\eps}(v_j(u_0))$.

If $n$ is even, then $n=2z$, $s_n=s_e(j-1)$, and $s_t\neq s_e(j)$ for all $t\in\overline{r_n+1,m}$.  (I) then implies that $s_m=p_{q_1}=s_0$ and $w_{j,\eps}(v_j(u_m))=w_{j,\eps}(v_j(u_{r_n}))=w_{j,\eps}(v_j(u_0))$.

Conversely, if $n$ is odd, then $n=2z+1$, $s_t=s_e(j)$, and $s_t\neq s_e(j-1)$ for all $t\in\overline{r_n+1,m}$.  (II) then implies that $s_m=p_{q_2}$ and $$w_{j,\eps}(v_j(u_m))=w_{j,\eps}(v_j(u_{r_n}))=Q_{i_j}(w_{j,\eps}(v_j(u_{r_{n-1}})))=Q_{i_j}(w_{j,\eps}(v_j(u_0)))$$

In particular, 
\begin{itemize}

\item $s_0=p_{q_1}=s_m$ implies that $w_{j,\eps}(v_j(u_0))=w_{j,\eps}(v_j(u_m))$ for all $j$, and so $w_\T(u_0)=w_\T(u_m)$.

\item $s_0=p_{q_1}$ and $s_m=p_{q_2}$ implies that $Q_{i_j}(w_{j,\eps}(v_j(u_0)))=w_{j,\eps}(v_j(u_m))$ for all $j$, and so $Q(w_\T(u_0))=w_\T(u_m)$.

\end{itemize}

Symmetric arguments then imply the analogues for the case where $s_0=p_{q_2}$.

\end{proof}

\begin{lemma} \label{lem-wacc-eps-prime-comp}

Suppose $\gamma=\bigl((u_0,s_0),e_1',(u_1,s_1),e_2',\dots,e_m',(u_m,s_m)\bigr)$ is a $\G_T'$-computation such that $s_0$ and $s_m$ are regular states.  Then there exists a computation $$\mathcal{C}=\bigl((w_0,q_0),e_1,(w_1,q_1),e_2,\dots,e_n,(w_n,q_n)\bigr)$$ of $\T$ such that $w_\T(u_0)=w_0$, $w_\T(u_m)=w_n$, $s_0=p_{q_0}$, and $s_m=p_{q_n}$.

\end{lemma}

\begin{proof}

Let $\ell_0,\ell_1,\dots,\ell_t\in\overline{0,m}$ be the ascending sequence consisting of all indices for which $s_{\ell_i}$ is a regular state, with $\ell_0=0$ and $\ell_t=m$.  

Then, by construction, for all $i\in\overline{t}$, the subcomputation $$\gamma_i=\bigl((u_{\ell_{i-1}},s_{\ell_{i-1}}),e_{\ell_{i-1}+1}',\dots,e_{\ell_i}',(u_{\ell_i},s_{\ell_i})\bigr)$$ is a $\T$-edge computation of some $f_i\in E_\T$.  Letting $Q^i=\lab(f_i)$, \Cref{lem-T-edge-comp} then implies:
\begin{enumerate}[label={(\alph*)}]

\item If $s_{\ell_{i-1}}=s_{\ell_i}$, then $w_\T(u_{\ell_{i-1}})=w_\T(u_{\ell_i})$.

\item If $s_{\ell_{i-1}}=p_{q_1}$ and $s_{\ell_i}=p_{q_2}$, then $Q^i(w_\T(u_{\ell_{i-1}}))=w_\T(u_{\ell_i})$

\item If $s_{\ell_{i-1}}=p_{q_2}$ and $s_{\ell_i}=p_{q_1}$, then $(Q^i)^{-1}(w_\T(u_{\ell_{i-1}}))=w_\T(u_{\ell_i})$.

\end{enumerate}

If $\gamma_i$ satisfies (a), then set $e_i=\nan$, where $\nan$ is a fixed symbol.  Otherwise, set $g_i=f_i$ if $\gamma_i$ satisfies (b) and set $g_i=f_i^{-1}$ if $\gamma_i$ satisfies (c).

For all $i\in\overline{0,t}$, let $w_i'=w_\T(u_{\ell_i})$ and let $q_i'\in S_\T$ be the state to which $s_{\ell_i}$ corresponds.  Then, consider the sequence 
$$\mathcal{C}'=\bigl((w_0',q_0'),g_1,(w_1',q_1'),g_2,\dots,g_t,(w_t',q_t')\bigr)$$
As some `edges' may be denoted $\nan$, $\mathcal{C}'$ certainly need not be a computation of $\T$; however, by (a) it is immediate that $g_i=\nan$ implies $(w_{i-1}',q_{i-1}')=(w_i',q_i')$.  Meanwhile, if $g_i\neq\nan$, then (b) and (c) imply that $g_i(w_{i-1}',q_{i-1}')=(w_i',q_i')$.

So, removing any occurrences of $\nan$ produces a computation $\mathcal{C}$ satisfying the statement.

\end{proof}

The following thus finishes the proof of \Cref{main-theorem}.

\begin{proposition} \label{prop-finale}

The proper operation $\wacc^\eps=\wacc^\eps(a)$ is a weak acceptor of $L_\T$ that satisfies $\IO_{\wacc^\eps}^{s,f}=\{(w,1)\mid w\in L_\T\}$.  Moreover, $\WTMSP_{\wacc^\eps}\preceq(n^{1+\eps},n)$.

\end{proposition}

\begin{proof}

Let $\gamma=\bigl((u_0,s_0),e_1',(u_1,s_1),e_2',\dots,e_m',(u_m,s_m)\bigr)$ be a computation of $\wacc^\eps$ realizing $(w,1)\in\IO_{\wacc^\eps}^{s,f}$ for some $w\in F(A)$.  Set $e_i'=(\sf_i,\sf_i',e_i'')$ for all $i\in\overline{m}$.

Suppose $\rho=e_1',\dots,e_m'$ is of minimal length amongst all generic paths supporting computations that realize $(w,1)\in\IO_{\wacc^\eps}^{s,f}$.  

Then, as in \Cref{lem-deg-minimal}, no two configurations of the computation can be the same, as otherwise any generic edges between them may be removed.  The definition of actions of degree 0, \Cref{lem-io-swap}, and the definition of the set operation of a Turing tape then immediately imply:

\begin{enumerate}

\item If $e_i''\in\{e_s^{\pm1},e_f^{\pm1},e_\swap\}$, then $\sf_i\neq\sf_i''$.

\item If $e_i''=e_\set$, then $\sf_i=\f$ or $\sf_i'=\f$.

\item If $e_i''\in\{e_s^{\pm1},e_f^{\pm1}\}$, then $e_{i+1}'',e_{i-1}''\neq e_i^{-1}$.

\end{enumerate}

Next, suppose $e_i''=e_{i+1}''$ for some $i\in\overline{m-1}$.  Then, since no edge of $\wacc^\eps$ is a loop (this holds in $\G_\T'$ since $k\geq2$), then necessarily $\sf_i'=\sf_{i+1}$.  The transitivity of the $\IO$-relation implies that the single generic edge $e'=(\sf_i,\sf_{i+1}',e_i'')$ supports a computation between $(u_{i-1},s_{i-1})$ and $(u_{i+1},s_{i+1})$.  But then replacing the subpath $e_i',e_{i+1}'$ with $e'$ in $\rho$ produces a generic path $\rho'$ with length $m-1$ supporting a computation realizing $(w,1)\in\IO_{\wacc^\eps}^{s,f}$, contradicting the hypothesis on $\rho$.  Hence, this yields:
\begin{enumerate}

\item[(4)] $e_i''\neq e_{i+1}''$ for all $i\in\overline{m-1}$.

\end{enumerate}

Now, as a consequence of (1)--(4), then it follows immediately that:

\begin{itemize}

\item $e_1'\in\{(\s,\f,e_s),(\f,\s,e_s^{-1})\}$

\item $e_2'=(\s,\f,e_\swap)$

\item $e_3'=(\s,\f,e_\set)$

\item $e_4''$ is an edge of $\G_\T'$

\end{itemize}

In particular, \Cref{lem-io-swap} and the definition of the set operation of a Turing tape imply that $\gamma$ has a prefix:
$$\gamma_0=\bigl((u,\wacc^\eps_s),e_1',(u,q),(\s,\f,e_\swap),(u',q'),(\s,\f,e_\set),(u_3,p_{s_\T})\bigr)$$
where $u,u',u_3\in\GL_{\wacc^\eps}$ satisfy:
\begin{itemize}

\item $u(a)=w$, $u(i)=1$ for $i\in\I_{\wacc^\eps}-\{a\}$

\item $u'(b)=w$, $u'(i)=1$ for $i\in\I_{\wacc^\eps}-\{b\}$

\item $w_{1,\eps}(v_1(u_3))=w$, $u(i)=1$ for $i\in\I_{\wacc^\eps}-\I_{1,\eps}$.

\end{itemize}

In particular, note that $w_\T(u_3)$ is the input configuration $I(w)$ of $\T$.

Let $\ell_1\in\overline{4,m}$ be the maximal index such that the subcomputation 
$$\gamma_1=\bigl((u_3,s_3),e_4',(u_4,s_4),e_5',\dots,e_{\ell_1}',(u_{\ell_1},s_{\ell_1})\bigr)$$ 
is a $\G_\T'$-computation.  Then, \Cref{lem-wacc-eps-prime-comp} yields a computation $\mathcal{C}_1$ of $\T$ between $(I(w),s_\T)$ and $(w_\T(u_{\ell_1}),t_1)$, where $t_1\in S_\T$ is the state to which $s_{\ell_1}$ corresponds.

Suppose, $t_1=s_\T$ and so $s_{\ell_1}=p_{s_\T}$.  The maximality of $\ell_1$ then implies $e_{\ell_1+1}''=e_\set$.  Further, conditions (1)--(4) necessitate $e_{\ell_1+1}'=(\f,\f,e_\set)$, so that $s_{\ell_1+1}=p_{s_\T}$.  The $\IO$-relations defining Turing tapes then imply that $w_{1,\eps}(v_1(u_{\ell_1}))=w_{1,\eps}(v_1(u_{\ell_1+1}))$, and so $w_\T(u_{\ell_1})=w_\T(u_{\ell_1+1})$.

Condition (4) then implies that $e_{\ell_1+2}''$ must be an edge of $\G_\T'$, so that there exists a maximal index $\ell_2\in\overline{\ell_1+2,m}$ such that the subcomputation
$$\gamma_2=\bigl((u_{\ell_1+1},s_{\ell_1+1}),e_{\ell_1+2}',\dots,e_{\ell_2}',(u_{\ell_2},s_{\ell_2})\bigr)$$
is a $\G_\T'$-computation.  Again, \Cref{lem-wacc-eps-prime-comp} yields a computation $\mathcal{C}_2$ of $\T$ between $(w_\T(u_{\ell_1}),s_\T)$ and $(w_\T(u_{\ell_2}),t_2)$, where $t_2\in S_\T$ is the state to which $s_{\ell_2}$ corresponds.

If $t_2=s_\T$, then iterate the process.  As $s_m=\wacc^\eps_f$ and $p_{f_\T}$ separates all other vertices from it in $\G_{\wacc^\eps}$, there exists a natural number $r$ such that $t_r=f_\T$.  But then concatenating $\mathcal{C}_1,\dots,\mathcal{C}_r$ produces a computation of $\T$ accepting $w$, and thus $w\in L_\T$.

Conversely, let $w\in L_\T$ and fix a computation of $\T$ $$\mathcal{C}=\bigl((w_0,s_0),e_1,(w_1,s_1),e_2,\dots,e_m,(w_m,s_m)\bigr)$$ accepting $w$ and satisfying $m\leq\TM_\T(|w|)$.  Then, for all $j\in\overline{k}$, $\mathcal{C}(j)=(w_0(i_j),w_1(i_j),\dots,w_m(i_j))$ is an $i_j$-computation of $\T$ with the same history as $\mathcal{C}$, i.e $(Q^1,\dots,Q^m)$ for $Q^\ell=\lab(e_\ell)\in\Q_\T$ for all $\ell\in\overline{m}$.

By \Cref{cor-ttape-eps}, there exists $K_s\in\N$ such that $\SEQ(\TTape_{\T,i_j}^\eps)\leq(K_sn^{1+\eps}+K_s,K_sn+K_s)$ for all $j\in\overline{k}$.  In particular, for all $j\in\overline{k}$ there exists a sequential computation $\gamma_j=(\gamma_{j,1},\dots,\gamma_{j,m})$ of $\TTape_{\T,i_j}^\eps$ approximating $\mathcal{C}(j)$ and satisfying 
\begin{align*}
\tm(\gamma_j)&\leq K_sn_j^{1+\eps}+K_s \\
\sp(\gamma_j)&\leq K_sn_j+K_s
\end{align*}
for $n_j=\max(|w_0(i_j)|,|w_m(i_j)|,m)$.  Note that $w_0(i_j)=w_m(i_j)=1$ for $j\neq1$, so that $n_j=m$.  Further, $w_0(i_1)=w$ and $w_m(i_1)=1$, so that $n_1=\max(|w|,m)$.

For all $j\in\overline{k}$, let $v_{j,0},v_{j,1},\dots,v_{j,m}\in V^{j,\eps}$ be the values such that for all $\ell\in\overline{m}$, $\gamma_{j,\ell}$ is a computation of the operation $\EXP(Q^\ell)$ of $\EXP(\TTape_{\T,i_j}^\eps)$ which $\EXP$-realizes $(v_{j,\ell-1},v_{j,\ell})\in\IO_{Q^\ell}^{s,f}$.  Note that by the definition of sequential computation, $w_{j,\eps}(v_{j,\ell})=w_\ell(i_j)$ for all $j\in\overline{k}$ and $\ell\in\overline{0,m}$.

Then, for all $j\in\overline{k}$ and all $\ell\in\overline{m}$, \Cref{lem-op-e}(1) produces a computation $\delta_{j,\ell}$ of $\EXP(\wacc^\eps)$ between the configurations $(u_{j-1,\ell},s_{e_\ell}(j-1))$ and $(u_{j,\ell},s_{e_\ell}(j))$ such that:

\begin{enumerate}[label={(\alph*)}]

\item $u_{j-1,\ell}(\I_{t,\eps})=u_{j,\ell}(\I_{t,\eps})=v_{t,\ell}$ for all $t\in\overline{j-1}$

\medskip

\item $u_{j-1,\ell}(\I_{j,\eps})=v_{j,\ell-1}$

\medskip

\item $u_{j,\ell}(\I_{j,\eps})=v_{j,\ell}$

\medskip

\item $u_{j-1,\ell}(\I_{t,\eps})=u_{j,\ell}(\I_{t,\eps})=v_{t,\ell-1}$ for all $t\in\overline{j+1,k}$

\medskip

\item $u_{j-1,\ell}(i)=u_{j,\ell}(i)=1$ for all $i\in\I_{\wacc^\eps}^\exp-\bigsqcup_{t\in\overline{k}}\I_{t,\eps}$

\medskip

\item $\tm_{\wacc^\eps}(\delta_{j,\ell})=\tm_{Q^\ell}(\gamma_{j,\ell})$

\medskip

\item $\displaystyle\sp_{\wacc^\eps}(\delta_{j,\ell})=\sp_{Q^\ell}(\gamma_{j,\ell})+\sum\limits_{t=1}^{j-1}|v_{t,\ell}|+\sum\limits_{t=j+1}^k|v_{t,\ell-1}|$.

\end{enumerate}

\medskip

As $w_{j,\eps}(v_{j,\ell})=w_\ell(i_j)$ for all $j,\ell$, $|v_{j,\ell}|=|w_\ell(i_j)|$.  In particular, $\sum_{j=1}^k |v_{j,\ell}|=|w_\ell|$.  Hence, (g) implies that
\medskip
\begin{enumerate}

\item[(g')] $\sp_{\wacc^\eps}(\delta_{j,\ell})\leq\sp_{Q^\ell}(\gamma_{j,\ell})+|w_\ell|+|w_{\ell-1}|$

\end{enumerate}
\medskip

For fixed $\ell\in\overline{m}$, the computations $\delta_{1,\ell},\dots,\delta_{k,\ell}$ to form a computation $\rho_\ell$ between $(u_{0,\ell},s_{e_\ell}(0))$ and $(u_{k,\ell},s_{e_\ell}(k))$ which, by (f) and (g'), satisfy:
\begin{align*}
\tm_{\wacc^\eps}(\rho_\ell)&=\sum_{j=1}^k\tm_{\wacc^\eps}(\delta_{j,\ell})=\sum_{j=1}^k\tm_{Q^\ell}(\gamma_{j,\ell}) \\
\sp_{\wacc^\eps}(\rho_\ell)&=\max_{j\in\overline{k}}\sp_{\wacc^\eps}(\delta_{j,\ell})\leq\max_{j\in\overline{k}}\sp_{Q^\ell}(\gamma_{j,\ell})+2\max(|w_\ell|,|w_{\ell-1}|)
\end{align*}

Further, (a)--(e) imply that $u_{k,\ell}=u_{0,\ell+1}$ for all $\ell\in\overline{m-1}$.  So, since $s_{e_\ell}(0)=p_{t(e_\ell)}=p_{s_{\ell-1}}$ and $s_{e_\ell}(k)=p_{h(e_\ell)}=p_{s_\ell}$ for all $\ell\in\overline{m}$, $\rho_1,\dots,\rho_m$ can be concatenated to form a computation $\psi$ between $(u_{0,1},p_{s_0})$ and $(u_{k,m},p_{s_m})$ which, by (f) and (g'), satisfies:
\begin{align*}
\tm_{\wacc^\eps}(\psi)&=\sum_{\ell=1}^m\tm_{\wacc^\eps}(\rho_\ell)=\sum_{\ell=1}^m\sum_{j=1}^k\tm_{Q^\ell}(\gamma_{j,\ell})=\sum_{j=1}^k\sum_{\ell=1}^m\tm_{Q^\ell}(\gamma_{j,\ell}) \\
&=\sum_{j=1}^k\tm(\gamma_j)\leq\sum_{j=1}^k (K_sn_j^{1+\eps}+K_s)\leq K_sk\max(|w|,m)^{1+\eps}+K_sk\\
&\leq K_sk\TM_\T(|w|)^{1+\eps}+K_sk|w|^{1+\eps}+K_sk \\
\sp_{\wacc^\eps}(\psi)&=\max_{\ell\in\overline{m}}\sp_{\wacc^\eps}(\rho_\ell)\leq\max_{\ell\in\overline{m}}\left(\max_{j\in\overline{k}}\sp_{Q^\ell}(\gamma_{j,\ell})+2\max(|w_{\ell-1}|,|w_\ell|)\right) \\
&\leq\max_{\ell\in\overline{m}}\max_{j\in\overline{k}}\sp_{Q^\ell}(\gamma_{j,\ell})+\max_{\ell\in\overline{m}}2\max(|w_{\ell-1}|,|w_\ell|) \\
&=\max_{j\in\overline{k}}\max_{\ell\in\overline{m}}\sp_{Q^\ell}(\gamma_{j,\ell})+2\max_{\ell\in\overline{0,m}}|w_\ell|=\max_{j\in\overline{k}}\sp(\gamma_j)+2\max_{\ell\in\overline{0,m}}|w_\ell| \\
&\leq\max_{j\in\overline{k}}(K_sn_j+K_s)+2\max_{\ell\in\overline{0,m}}|w_\ell|\leq K_sm+K_s|w|+2\max_{\ell\in\overline{0,m}}|w_\ell| \\
&\leq K_s\TM_\T(|w|)+K_s|w|+2\max_{\ell\in\overline{0,m}}|w_\ell|+K_s
\end{align*}

Letting $K_\T=\max_{Q\in Q^\T}|Q|$, as in the proof of \Cref{prop-ttape-seq-compl}, $\bigl||w_\ell|-|w_{\ell-1}|\bigr|\leq kK_\T$ for all $\ell\in\overline{m}$, so that $|w_\ell|\leq|w_0|+\ell kK_\T\leq|w|+mkK_\T\leq|w|+mk\TM_\T(|w|)$.  So, since $K_s$, $k$, and $K_\T$ are constants not dependent on $\mathcal{C}$, there exists a constant $L_1\in\N$ not dependent on $\mathcal{C}$ such that:
\begin{align*}
\tm_{\wacc^{\eps}}(\psi)&\leq L_1\TM_\T(|w|)^{1+\eps}+L_1|w|^{1+\eps}+L_1 \\
\sp_{\wacc^\eps}(\psi)&\leq L_1\TM_\T(|w|)+L_1|w|+L_1
\end{align*}

By (a)--(d), $u_{k,m}(\I_{j,\eps})=v_{j,m}$ for all $j\in\overline{k}$.  So, since $w_{j,\eps}(v_{j,m})=w_m(i_j)=1$ for all $j\in\overline{k}$, $u_{m,1}(\I_{j,\eps})=1$.  Combining this with (e), $u_{k,m}=1$.

By a similar argument, $u_{0,1}(\I_{\wacc^\eps}-\I_{1,\eps})=1$ and $u_{0,1}(\I_{1,\eps})=v_{1,0}$.  

Now, by \Cref{cor-ttape-eps}, there exists $K_1\in\N$ such that $\TMSP_{\TTape_{\T,i_1}^\eps}\leq(K_1n^{1+\eps},K_1)$.  So, since $w_{1,\eps}(v_{1,0})=w$, by the definition of $i_j$-Turing tape, for the operation $\set$ of $\TTape_{\T,i_1}^\eps$, there exists $(u',u'')\in\IO_\set^{s,f}$ such that $u'(b)=w$, $u'(\I_{i_1,\eps})=1$, $u''(b)=1$, and $u''(\I_{i_1,\eps})=v_{1,0}$.  So, there exists a computation $\chi$ of $\EXP(\set)$ that $\EXP$-realizes $(u',u'')\in\IO_\set^{s,f}$ and satisfies $\tm_\set(\chi)\leq K_1|w|^{1+\eps}+K_1$, $\sp_\set(\chi)\leq K_1|w|+K_1$.

So, letting $u_b\in\GL_{\wacc^\eps}$ be such that $u_b(b)=w$ and $u_b(i)=1$ for all $i\in\I_{\wacc^\eps}^\exp-\{b\}$, \Cref{lem-op-e}(1) then yields a computation $\chi'$ of $\EXP(\wacc^\eps)$ between the configurations $(u_b,q')$ and $(u_{1,0},p_{s_\T})$ and with the same complexity bounds as $\chi$.

As $s_0=s_\T$ and $s_m=f_\T$, $\chi'$ and $\psi$ can be concatenated to form a computation $\phi$ of $\EXP(\wacc^\eps)$ between $(u_b,q')$ and $(1,p_{f_\T})$ which satisfies:
\begin{align*}
\tm_{\wacc^\eps}(\phi)&=\tm_{\wacc^\eps}(\psi)+\tm_{\wacc^\eps}(\chi') \\
&\leq L_2\TM_\T(|w|)^{1+\eps}+L_2|w|^{1+\eps}+L_2 \\
\sp_{\wacc^\eps}(\phi)&=\max(\sp_{\wacc^\eps}(\psi),\sp_{\wacc^\eps}(\chi'))\\
&\leq L_2\TM_\T(|w|)+L_2|w|+L_2
\end{align*}
for $L_2=L_1+K_1$.

Next, consider the operation $\swap=\swap^A(a,b)$.  By \Cref{lem-io-swap}, there exists a relation $(t_a,t_b)\in\IO_\swap^{s,f}$ where $t_a(a)=t_b(b)=w$ and $t_a(b)=t_b(a)=1$.  Further, there exists a constant $K_2\in\N$ such that $\TMSP_\swap\leq(K_2n+K_2,K_2n+K_2)$.  So, there exists a computation $\tau$ of $\EXP(\swap)$ that $\EXP$-realizes $(t_a,t_b)\in\IO_\swap^{s,f}$ and satisfies $\tm_\swap(\tau),\sp_\swap(\tau)\leq K_2|w|+K_2$.

Then, letting $u_w\in\GL_{\wacc^\eps}$ be given by $u_w(a)=w$ and $u_w(i)=1$ for all $i\in\I_{\wacc^\eps}^\exp-\{a\}$, \Cref{lem-op-e}(1) again yields a computation $\tau'$ between $(u_w,q)$ and $(u_b,q')$ with the same complexity bounds as $\tau$.

As above, $\tau'$ can then be combined with $\phi$ to produce a computation $\zeta$ of $\EXP(\wacc^\eps)$ between $(u_w,q)$ and $(1,p_{f_\T})$ and satisfying:
\begin{align*}
\tm_{\wacc^\eps}(\zeta)&=\tm_{\wacc^\eps}(\phi)+\tm_{\wacc^\eps}(\tau') \\
&\leq L_3\TM_\T(|w|)^{1+\eps}+L_3|w|^{1+\eps}+L_3 \\
\sp_{\wacc^\eps}(\zeta)&=\max(\sp_{\wacc^\eps}(\phi),\sp_{\wacc^\eps}(\tau'))\\
&\leq L_3\TM_\T(|w|)+L_3|w|+L_3
\end{align*}
for $L_3=L_2+K_2$.

Finally, noting that $u_w$ and $1$ satisfy all the guard fragments comprising the action $\act_0$, $\kappa$ can be extended by $e_s$ and $e_f$ to produce a computation $\gamma_w$ between $(u_w,\wacc^\eps_s)$ and $(u_w,\wacc^\eps_f)$ and satisfying $\tm_{\wacc^\eps}(\gamma_w)=\tm_{\wacc^\eps}(\zeta)+2$ and $\sp_{\wacc^\eps}(\gamma_w)=\sp_{\wacc^\eps}(\zeta)$.  But then $\gamma_w$ $\EXP$-realizes $(w,1)\in\IO_{\wacc^\eps}^{s,f}$, so that the statement follows.

\end{proof}

\bibliographystyle{plain}
\bibliography{biblio}

\end{document}

%% file: S2G.tex
\begin{tikzpicture} 
[point/.style={inner sep = 1.7pt, circle,draw,fill=white},
mpoint/.style={inner sep = 1.7pt, circle,draw,fill=black},
FIT/.style args = {#1}{rounded rectangle, draw,  fit=#1, rotate fit=45, yscale=0.5},
FITR/.style args = {#1}{rounded rectangle, draw,  fit=#1, rotate fit=-45, yscale=0.5},
FIT1/.style args = {#1}{rounded rectangle, draw,  fit=#1, rotate fit=45, scale=2},
vecArrow/.style={
		thick, decoration={markings,mark=at position
		   1 with {\arrow[thick]{open triangle 60}}},
		   double distance=1.4pt, shorten >= 5.5pt,
		   preaction = {decorate},
		   postaction = {draw,line width=0.4pt, white,shorten >= 4.5pt}
	},
myptr/.style={decoration={markings,mark=at position 1 with %
    {\arrow[scale=3,>=stealth]{>}}},postaction={decorate}}
]

\begin{scope} [xscale=0.85, yscale=0.7]      
	\node (p1q1) at (0,0) [point] {};
	\node (p2q1) at (0,-5) [point] {};
	\node (p1q2) at (5,0) [point] {};
	\node (p2q2) at (5,-5) [point] {};
	\node [left]at (p1q1.west) {$(p_2', q_2', p_1, q_1, r_1, E)$};
	\node [below]at (p2q1.south) {$(p_2', q_2', p_2, q_1, r_1, E)$};
	\node [below]at (p1q2.south) {$(p_2', q_2', p_1, q_2, r_1, E)$};
	\node [below]at (p2q2.south) {$(p_2', q_2', p_2, q_2, r_1, E)$};
	
	\draw [->] (p1q1) to node[midway, left]{
		$\begin{gathered}
		 [\mathbf{2} \rightarrow \mathbf{2}~\delta^{-1}]\\[-1ex]
						[\mathbf{3} \rightarrow a~\mathbf{3}]\\[-1ex]
						[\mathbf{5} \rightarrow \delta~\mathbf{5}] \\[-1ex]
					  [\mathbf{1} = 1] \\[-1ex] [\mathbf{3} = a] \\[-1ex] [\mathbf{4} = 1] \end{gathered}$
	} (p2q1);
	
	\draw [->]  (p2q1) edge [loop above, looseness=50, out=75,in=-10] node[pos = 0.3, above]  
	    	{$[\mathbf{3} \rightarrow a~\mathbf{3}]$} (p2q1);

	\path [->]  (p2q2) edge [loop above, looseness=50, out=75,in=-10] node[pos = 0.25, above]  
	    	{$[\mathbf{3} \rightarrow a~\mathbf{3}]$} (p2q2);
\end{scope}	

\end{tikzpicture}	

%% file: G2S.tex
\begin{tikzpicture} 
[point/.style={inner sep = 1.7pt, circle,draw,fill=white},
mpoint/.style={inner sep = 1.7pt, circle,draw,fill=black},
FIT/.style args = {#1}{rounded rectangle, draw,  fit=#1, rotate fit=45, yscale=0.5},
FITR/.style args = {#1}{rounded rectangle, draw,  fit=#1, rotate fit=-45, yscale=0.5},
FIT1/.style args = {#1}{rounded rectangle, draw,  fit=#1, rotate fit=45, scale=2},
vecArrow/.style={
		thick, decoration={markings,mark=at position
		   1 with {\arrow[thick]{open triangle 60}}},
		   double distance=1.4pt, shorten >= 5.5pt,
		   preaction = {decorate},
		   postaction = {draw,line width=0.4pt, white,shorten >= 4.5pt}
	},
myptr/.style={decoration={markings,mark=at position 1 with %
    {\arrow[scale=3,>=stealth]{>}}},postaction={decorate}}
]

\begin{scope} [xscale=0.85, yscale=0.7]      
	\node (p1q1) at (0,0) [point] {$P1Q1$};
	\node (p2q1) at (0,-5) [point] {$P2Q1$};
	\node (p2q2) at (5,-5) [point] {$P2Q2$};
	
	\draw [->] (p1q1) to node[midway, left]{
		$\begin{gathered}
		 [\mathbf{D_1} \rightarrow \mathbf{D_1}\delta^{-1}]\\[-1ex]
						[\mathbf{A} \rightarrow a\mathbf{A}]\\[-1ex]
						[\mathbf{D_2} \rightarrow \delta\mathbf{D_2}] \\[-1ex]
					 [\mathbf{I} = 1] \\[-1ex] [\mathbf{A} = a]\\[-1ex]  \end{gathered}$
	}  node [pos = 0.3, right] {$e_1$}  (p2q1);
	
	\draw [->]  (p2q1) edge [loop above, looseness=10, out=75,in=-10] node[pos = 0.3, above]  
	    	{$[\mathbf{A} \rightarrow a\mathbf{A}]$} 
	    	node [pos=0.65, right] {$e_2$} (p2q1);

	\path [->]  (p2q2) edge [loop above, looseness=10, out=75,in=-10] node[pos = 0.25, above]  
	    	{$[\mathbf{A} \rightarrow a\mathbf{A}]$} 
	    	node [pos=0.65, right] {$e_3$} (p2q2);
\end{scope}	

\end{tikzpicture}	

%% file: COPY.tex
\begin{tikzpicture} 
[point/.style={inner sep = 4pt, circle,draw,fill=white},
mpoint/.style={inner sep = 1.7pt, circle,draw,fill=black},
FIT/.style args = {#1}{rounded rectangle, draw,  fit=#1, rotate fit=45, yscale=0.5},
FITR/.style args = {#1}{rounded rectangle, draw,  fit=#1, rotate fit=-45, yscale=0.5},
FIT1/.style args = {#1}{rounded rectangle, draw,  fit=#1, rotate fit=45, scale=2},
vecArrow/.style={
		thick, decoration={markings,mark=at position
		   1 with {\arrow[thick]{open triangle 60}}},
		   double distance=1.4pt, shorten >= 5.5pt,
		   preaction = {decorate},
		   postaction = {draw,line width=0.4pt, white,shorten >= 4.5pt}
	},
myptr/.style={decoration={markings,mark=at position 1 with %
    {\arrow[scale=3,>=stealth]{>}}},postaction={decorate}}
]

\begin{scope} [xscale=0.75, yscale=0.5]      
	\node [anchor=west] at (-1, 1.5) {$\cpy(a,b)$};
	\node (start) at (0,0) [point] {$s$};
	\node (n1) at (5,0) [point] {$s_1$};
	\node (n2) at (10,0) [point] {$s_2$};
	\node (finish) at (15, 0) [point] {$f$};
	
	\draw [->] (start) to node[midway, below]{
		$\begin{gathered}
		 [\mathbf{b} =1]\\[-1ex]
		[\mathbf{c} =1]
		\end{gathered}$
	}  (n1);

	\draw [->] (n1) to node[midway, below]{
		$\begin{gathered}
		 [\mathbf{a} =1]\\[-1ex]
 		 [\mathbf{b} =1]
		\end{gathered}$
	}  (n2);

	\draw [->] (n2) to node[midway, below]{
		$\begin{gathered}
		 [\mathbf{c} =1]
		\end{gathered}$
	}  (finish);
	
	\draw [->]  (n1) edge [loop, looseness=25, out=125,in=55] node[pos = 0.5, above]  
	    	{		$\begin{gathered}
	    			 [\mathbf{a} \rightarrow x^{-1}\mathbf{a}]\\[-1ex]
	    			[\mathbf{c} \rightarrow \mathbf{c}x] \\[-1ex]
	    			\textrm{for $x \in A$ }
	    			\end{gathered}$}  (n1);

	\draw [->]  (n2) edge [loop, looseness=25, out=125,in=55] node[pos = 0.5, above]  
	    	{		$\begin{gathered}
	    			[\mathbf{a} \rightarrow x\mathbf{a}] \\[-1ex]
	    			[\mathbf{b} \rightarrow x\mathbf{b}] \\[-1ex]
	    			[\mathbf{c} \rightarrow \mathbf{c}x^{-1}]\\[-1ex]
	    			\textrm{for $x \in A$ }
	    			\end{gathered}$}  (n2);
\end{scope}	

\end{tikzpicture}	

%% file: COPY2.tex
\begin{tikzpicture} 
[point/.style={inner sep = 4pt, circle,draw,fill=white},
mpoint/.style={inner sep = 1.7pt, circle,draw,fill=black},
FIT/.style args = {#1}{rounded rectangle, draw,  fit=#1, rotate fit=45, yscale=0.5},
FITR/.style args = {#1}{rounded rectangle, draw,  fit=#1, rotate fit=-45, yscale=0.5},
FIT1/.style args = {#1}{rounded rectangle, draw,  fit=#1, rotate fit=45, scale=2},
vecArrow/.style={
		thick, decoration={markings,mark=at position
		   1 with {\arrow[thick]{open triangle 60}}},
		   double distance=1.4pt, shorten >= 5.5pt,
		   preaction = {decorate},
		   postaction = {draw,line width=0.4pt, white,shorten >= 4.5pt}
	},
myptr/.style={decoration={markings,mark=at position 1 with %
    {\arrow[scale=3,>=stealth]{>}}},postaction={decorate}}
]

\begin{scope} [xscale=0.75, yscale=0.25]      
	\node [anchor=west] at (-1, 2.5) {$\cpy(x, y)$};
	\node (start) at (0,0) [point] {$s$};
	\node (n1) at (5,0) [point] {$s_1$};
	\node (n2) at (10,0) [point] {$s_2$};
	\node (finish) at (15, 0) [point] {$f$};
	
	\draw [->] (start) to node[midway, below]{
		$\begin{gathered}
		 [\mathbf{y} =1]\\[-1ex]
		[\mathbf{c} =1]
		\end{gathered}$
	}  (n1);

	\draw [->] (n1) to node[midway, below]{
		$\begin{gathered}
		[\mathbf{x} =1]\\[-1ex]
 		[\mathbf{y} =1]
		\end{gathered}$
	}  (n2);

	\draw [->] (n2) to node[midway, below]{
		$\begin{gathered}
		 [\mathbf{c} =1]
		\end{gathered}$
	}  (finish);
	
	\draw [->]  (n1) edge [loop, looseness=25, out=125,in=55] node[pos = 0.5, above]  
	    	{		$\begin{gathered}
	    			 [\mathbf{x} \rightarrow z^{-1}\mathbf{x}]\\[-1ex]
	    			[\mathbf{c} \rightarrow \mathbf{c}z] \\[-1ex]
	    			\textrm{for $z \in A$ }
	    			\end{gathered}$}  (n1);

	\draw [->]  (n2) edge [loop, looseness=25, out=125,in=55] node[pos = 0.5, above]  
	    	{		$\begin{gathered}
	    			[\mathbf{x} \rightarrow z\mathbf{x}] \\[-1ex]
	    			[\mathbf{y} \rightarrow z\mathbf{y}] \\[-1ex]
	    			[\mathbf{c} \rightarrow \mathbf{c}z^{-1}]\\[-1ex]
	    			\textrm{for $z \in A$ }
	    			\end{gathered}$}  (n2);
\end{scope}

\begin{scope} [xscale=0.75, yscale=0.25, yshift = -15cm]      
	\node [anchor=west] at (-1, 2.5) {$\cpy(b, a)$};
	\node (start) at (0,0) [point] {$s$};
	\node (n1) at (5,0) [point] {$s_1$};
	\node (n2) at (10,0) [point] {$s_2$};
	\node (finish) at (15, 0) [point] {$f$};
	
	\draw [->] (start) to node[midway, below]{
		$\begin{gathered}
		[\mathbf{a} =1]\\[-1ex]
		[\mathbf{c} =1]
		\end{gathered}$
	}  (n1);

	\draw [->] (n1) to node[midway, below]{
		$\begin{gathered}
		 [\mathbf{b} =1]\\[-1ex]
 		 [\mathbf{a} =1]
		\end{gathered}$
	}  (n2);

	\draw [->] (n2) to node[midway, below]{
		$\begin{gathered}
		 [\mathbf{c} =1]
		\end{gathered}$
	}  (finish);
	
	\draw [->]  (n1) edge [loop, looseness=25, out=125,in=55] node[pos = 0.5, above]  
	    	{		$\begin{gathered}
	    			[\mathbf{b} \rightarrow x^{-1}\mathbf{b}]\\[-1ex]
	    			[\mathbf{c} \rightarrow \mathbf{c}x] \\[-1ex]
	    			\textrm{for $x \in A$ }
	    			\end{gathered}$}  (n1);

	\draw [->]  (n2) edge [loop, looseness=25, out=125,in=55] node[pos = 0.5, above]  
	    	{		$\begin{gathered}
	    			[\mathbf{b} \rightarrow x\mathbf{b}] \\[-1ex]
	    			[\mathbf{a} \rightarrow x\mathbf{a}] \\[-1ex]
	    			[\mathbf{c} \rightarrow \mathbf{c}x^{-1}]\\[-1ex]
	    			\textrm{for $x \in A$ }
	    			\end{gathered}$}  (n2);
\end{scope}

\begin{scope} [xscale=0.75, yscale=0.25, yshift = -30cm]      
	\node [anchor=west] at (-1, 2.5) {$\cpy(a, c)$};
	\node (start) at (0,0) [point] {$s$};
	\node (n1) at (5,0) [point] {$s_1$};
	\node (n2) at (10,0) [point] {$s_2$};
	\node (finish) at (15, 0) [point] {$f$};
	
	\draw [->] (start) to node[midway, below]{
		$\begin{gathered}
		 [\mathbf{c} =1]\\[-1ex]
		[\mathbf{c'} =1]
		\end{gathered}$
	}  (n1);

	\draw [->] (n1) to node[midway, below]{
		$\begin{gathered}
		 [\mathbf{a} =1]\\[-1ex]
 		 [\mathbf{c} =1]
		\end{gathered}$
	}  (n2);

	\draw [->] (n2) to node[midway, below]{
		$\begin{gathered}
		 [\mathbf{c'} =1]
		\end{gathered}$
	}  (finish);
	
	\draw [->]  (n1) edge [loop, looseness=25, out=125,in=55] node[pos = 0.5, above]  
	    	{		$\begin{gathered}
	    			 [\mathbf{a} \rightarrow x^{-1}\mathbf{a}]\\[-1ex]
	    			[\mathbf{c'} \rightarrow \mathbf{c'}x] \\[-1ex]
	    			\textrm{for $x \in A$ }
	    			\end{gathered}$}  (n1);

	\draw [->]  (n2) edge [loop, looseness=25, out=125,in=55] node[pos = 0.5, above]  
	    	{		$\begin{gathered}
	    			[\mathbf{a} \rightarrow x\mathbf{a}] \\[-1ex]
	    			[\mathbf{c} \rightarrow x\mathbf{c}] \\[-1ex]
	    			[\mathbf{c'} \rightarrow \mathbf{c'}x^{-1}]\\[-1ex]
	    			\textrm{for $x \in A$ }
	    			\end{gathered}$}  (n2);
\end{scope}	

\end{tikzpicture}	

%% file: COUNTER1.tex
\begin{tikzpicture} 
[point/.style={inner sep = 4pt, circle,draw,fill=white},
mpoint/.style={inner sep = 1.7pt, circle,draw,fill=black},
FIT/.style args = {#1}{rounded rectangle, draw,  fit=#1, rotate fit=45, yscale=0.5},
FITR/.style args = {#1}{rounded rectangle, draw,  fit=#1, rotate fit=-45, yscale=0.5},
FIT1/.style args = {#1}{rounded rectangle, draw,  fit=#1, rotate fit=45, scale=2},
vecArrow/.style={
		thick, decoration={markings,mark=at position
		   1 with {\arrow[thick]{open triangle 60}}},
		   double distance=1.4pt, shorten >= 5.5pt,
		   preaction = {decorate},
		   postaction = {draw,line width=0.4pt, white,shorten >= 4.5pt}
	},
myptr/.style={decoration={markings,mark=at position 1 with %
    {\arrow[scale=3,>=stealth]{>}}},postaction={decorate}}
]

\begin{scope} [xscale=0.75, yscale=0.25]      
\node [anchor=west] at (-1.5, 10) {Object $\text{Counter}_0[a]$};

\begin{scope} 
	\node at (0, 2.3) {$\check()$};
	\node (start) at (0,0) [point] {$s$};
	\node (n1) at (5,0) [point] {$s_1$};
	\node (n2) at (10,0) [point] {$s_2$};
	\node (n3) at (15, 0) [point] {$s_3$};
	
	\draw [->] (start) to node[midway, above]{
		$\begin{gathered}
		[\mathbf{b} =1]\\[-1ex]
 		[\mathbf{c} =1]
		\end{gathered}$
	}  (n1);

	\draw [->] (n1) to node[midway, above]{
		$\begin{gathered}
		 [\mathbf{a} =1]\\[-1ex]
 		 [\mathbf{b} =1]
		\end{gathered}$
	}  (n2);

	\draw [->] (n2) to node[midway, above]{
		$\begin{gathered}
		 [\mathbf{c} =1]
		\end{gathered}$
	}  (n3);
	
	\draw [->]  (n1) edge [loop, looseness=25, out=125,in=55] node[pos = 0.5, above]  
	    	{		$\begin{gathered}
	    			 [\mathbf{a} \rightarrow \delta^{-1}\mathbf{a}]\\[-1ex]
	    			[\mathbf{c} \rightarrow \mathbf{c}\delta] \\[-1ex]
	    			\end{gathered}$}  (n1);

	\draw [->]  (n2) edge [loop, looseness=25, out=125,in=55] node[pos = 0.5, above]  
	    	{		$\begin{gathered}
	    			[\mathbf{a} \rightarrow \delta\mathbf{a}] \\[-1ex]
	    			[\mathbf{b} \rightarrow \delta\mathbf{b}] \\[-1ex]
	    			[\mathbf{c} \rightarrow \mathbf{c}\delta^{-1}]\\[-1ex]
	    			\end{gathered}$}  (n2);

	\node (n4) at (0, -10) [point] {$s_4$};
	\node (finish) at (15, -10) [point] {$f$};

	\draw [->] (n3) edge[out=-90, in = 90, looseness=1]  (n4);

	\draw [->] (n4) to node[midway, above]{
		$\begin{gathered} 
		[\mathbf{b} =\delta]\\[-1ex]
		[\mathbf{b} \rightarrow \delta^{-1}\mathbf{b}]\\[-1ex]
		\end{gathered}$
	} (finish);

	\node (n5) at (2.5, -20) [point] {$s_5$};
	\node (n6) at (12.5, -20) [point] {$s_6$};

	\draw [->] (n4) to node[midway, left]{
		$\begin{gathered}
		 [\mathbf{d}=1] \ 
		\end{gathered}$
	}  (n5);
	\draw [->] (n6) to node[midway, right]{
		$\begin{gathered}
		 \ \ \ \ \ [\mathbf{d}=1]
		\end{gathered}$
	}  (n4);

	\draw [->] (n5) edge[out=0, in = 180, looseness=1] node[midway, above]{
		 $ \begin{gathered} 
		 [\mathbf{b}=1] \\[-1ex]
		\end{gathered}$
	}  (n6);
	\draw [->] (n5) edge[out=-40, in = -140, looseness=1] node[midway, below]{
		$\begin{gathered}
		 [\mathbf{b}=\delta^{-1}]\\[-1ex]
		 [\mathbf{b}\rightarrow\delta\mathbf{b}]\\[-1ex]
		\end{gathered}$
	}  (n6);

	\draw [->] (n5) edge[loop, out=-125, in = -55, looseness=25] node[pos=0.5, below]{
		$\begin{gathered}
		 [\mathbf{b}\rightarrow\delta^{-2}\mathbf{b}]\\[-1ex]
		 [\mathbf{d}\rightarrow\delta\mathbf{d}]
		\end{gathered}$
	}  (n5);

	\draw [->] (n6) edge[loop, out=-125, in = -55, looseness=25] node[pos=0.5, below]{
		$\begin{gathered}
		 [\mathbf{b}\rightarrow\delta\mathbf{b}]\\[-1ex]
		 [\mathbf{d}\rightarrow\delta^{-1}\mathbf{d}]
		\end{gathered}$
	}  (n6);

\end{scope}

\begin{scope} [yshift = -40cm]      
	\node at (0, 2.3) {$\inc()$};
	\node (start) at (0,0) [point] {$s$};
	\node (n1) at (5,0) [point] {$s_1'$};
	\node (n2) at (10,0) [point] {$s_2'$};
	\node (n3) at (15, 0) [point] {$s_3'$};
	
	\draw [->] (start) to node[midway, above]{
		$\begin{gathered}
			[\mathbf{b} =1]\\[-1ex]
			[\mathbf{c} =1]
		\end{gathered}$
	}  node [midway, below] {
		$[\mathbf{a} \rightarrow \delta\mathbf{a}]$
	}(n1);
	
	\draw [->] (n1) to node[midway, above]{
		$\begin{gathered}
		 [\mathbf{a} =1]\\[-1ex]
 		 [\mathbf{b} =1]
		\end{gathered}$
	}  (n2);

	\draw [->] (n2) to node[midway, above]{
		$\begin{gathered}
		 [\mathbf{c} =1]
		\end{gathered}$
	}  (n3);
	
	\draw [->]  (n1) edge [loop, looseness=25, out=125,in=55] node[pos = 0.5, above]  
	    	{		$\begin{gathered}
	    			 [\mathbf{a} \rightarrow \delta^{-1}\mathbf{a}]\\[-1ex]
	    			[\mathbf{c} \rightarrow \mathbf{c}\delta] \\[-1ex]
	    			\end{gathered}$}  (n1);

	\draw [->]  (n2) edge [loop, looseness=25, out=125,in=55] node[pos = 0.5, above]  
	    	{		$\begin{gathered}
	    			[\mathbf{a} \rightarrow \delta\mathbf{a}] \\[-1ex]
	    			[\mathbf{b} \rightarrow \delta\mathbf{b}] \\[-1ex]
	    			[\mathbf{c} \rightarrow \mathbf{c}\delta^{-1}]\\[-1ex]
	    			\end{gathered}$}  (n2);

	\node (n4) at (0, -10) [point] {$s_4'$};
	\node (finish) at (15, -10) [point] {$f$};

	\draw [->] (n3) edge[out=-90, in = 90, looseness=1]  (n4);
	\draw [->] (n4) to node[midway, above]{		
		$\begin{gathered} 
		[\mathbf{b} =\delta]\\[-1ex]
		[\mathbf{b} \rightarrow \delta^{-1}\mathbf{b}]\\[-1ex]
		\end{gathered}$
	} (finish);

	\node (n5) at (2.5, -20) [point] {$s_5'$};
	\node (n6) at (12.5, -20) [point] {$s_6'$};

	\draw [->] (n4) to node[midway, left]{
		$\begin{gathered}
		 [\mathbf{d}=1] \
		\end{gathered}$
	}  (n5);
	\draw [->] (n6) to node[midway, right]{
		$\begin{gathered}
		\ \ \ \ \ [\mathbf{d}=1]
		\end{gathered}$
	}  (n4);

	\draw [->] (n5) edge[out=0, in = 180, looseness=1] node[midway, above]{
		$\begin{gathered}
		 [\mathbf{b}=1] \\[-1ex]
		\end{gathered}$
	}  (n6);
	\draw [->] (n5) edge[out=-40, in = -140, looseness=1] node[midway, below]{
		$\begin{gathered}
		 [\mathbf{b}=\delta^{-1}]\\[-1ex]
		 [\mathbf{b}\rightarrow\delta\mathbf{b}]\\[-1ex]
		\end{gathered}$
	}  (n6);

	\draw [->] (n5) edge[loop, out=-125, in = -55, looseness=25] node[pos=0.5, below]{
		$\begin{gathered}
		 [\mathbf{b}\rightarrow\delta^{-2}\mathbf{b}]\\[-1ex]
		 [\mathbf{d}\rightarrow\delta\mathbf{d}]\\[-1ex]
		\end{gathered}$
	}  (n5);

	\draw [->] (n6) edge[loop, out=-125, in = -55, looseness=25] node[pos=0.5, below]{
		$\begin{gathered}
		 [\mathbf{b}\rightarrow\delta\mathbf{b}]\\[-1ex]
		 [\mathbf{d}\rightarrow\delta^{-1}\mathbf{d}]\\[-1ex]
		\end{gathered}$
	}  (n6);
\end{scope}	
\end{scope}	

\end{tikzpicture}	

%% file: SILLYCOUNTER.tex
\begin{tikzpicture} 
[point/.style={inner sep = 4pt, circle,draw,fill=white},
mpoint/.style={inner sep = 1.7pt, circle,draw,fill=black},
FIT/.style args = {#1}{rounded rectangle, draw,  fit=#1, rotate fit=45, yscale=0.5},
FITR/.style args = {#1}{rounded rectangle, draw,  fit=#1, rotate fit=-45, yscale=0.5},
FIT1/.style args = {#1}{rounded rectangle, draw,  fit=#1, rotate fit=45, scale=2},
vecArrow/.style={
		thick, decoration={markings,mark=at position
		   1 with {\arrow[thick]{open triangle 60}}},
		   double distance=1.4pt, shorten >= 5.5pt,
		   preaction = {decorate},
		   postaction = {draw,line width=0.4pt, white,shorten >= 4.5pt}
	},
myptr/.style={decoration={markings,mark=at position 1 with %
    {\arrow[scale=3,>=stealth]{>}}},postaction={decorate}}
]

\begin{scope} [xscale=0.75, yscale=0.3]      
	\node [anchor=west] at (-2.4, 5) {Object $\Obj_0[a]$};
	\node at (0, 2.3) {$\check0()$};
	\node (start) at (0,0) [point] {$s$};
	\node (finish) at (5, 0) [point] {$f$};
	
	\draw [->] (start) to node[midway, above]{
		$\begin{gathered}
		 [\mathbf{a} = 1]
		\end{gathered}$
	}  (finish);
\end{scope}

\begin{scope} [xscale=0.75, yscale=0.3, xshift= 7.5cm]      
	\node at (0, 2.3) {$\inc2()$};
	\node (start) at (0,0) [point] {$s$};
	\node (finish) at (5, 0) [point] {$f$};
	
	\draw [->] (start) to node[midway, above]{
		$\begin{gathered}
		 [\mathbf{a} \rightarrow \delta^2 \mathbf{a}]
		\end{gathered}$
	}  (finish);
\end{scope}

\begin{scope} [xscale=0.65, yscale=0.3, xshift = 1cm, yshift = -18.5cm]      
	\node [anchor=west] at (-4, 13.5) {
		\begin{tabular}{l} 
		Object $\Obj_1[d']$;  \\
		\ \ \ $C_{ab}$ is $\Obj_0[d']$ \\
		\ \ \ $C_c$ is $\Obj_0[d'']$
		\end{tabular}};
		
	\node at (0, 10) {syn$(a, b, c)$};
		
	\node (s) at (0,-15) [point] {$s$};
	\node (ss) at (0,0) [point] {$s_3$};
	\node (ss1) at (0,-4) [point] {$s_2$};
	\node (ss2) at (0,-8) [point] {$s_1$};
	\node (s2) at (8, 0) [point] {};
	\node (c2) at (8, 8) [point] {};
	\node (s3) at (16, 0) [point] {};

	\node (c1) at (0, 8) [point] {$s_4$};
	\draw [->] (s) to node[midway, left]{
		$\begin{gathered}
		 [\mathbf{a}'=1] \\[-1ex]
		 [\mathbf{b}'=1] \\[-1ex]
		 [\mathbf{d}'=1] \\[-1ex]
		 [\mathbf{d}''=1]
		\end{gathered}$
	}  node [midway, right] {$e_1$} (ss2);
	\draw [->] (ss2) to node[midway, left]{
		$\cpy^A(\mathbf{a}, \mathbf{a}')$
	}   node [midway, right] {$e_2$} (ss1);
	\draw [->] (ss1) to node[midway, left]{
		$\cpy^B(\mathbf{b}, \mathbf{b}')$
	} node [midway, right] {$e_3$} (ss);
	\draw [->] (ss) edge[out=100, in = -100, looseness=1] node[midway, left]{
		$\begin{gathered}
		 [\mathbf{a}'\rightarrow \mathbf{a}'x^{-1}] \\[-1ex]
		 \text{for $x\in A$}
		\end{gathered}$
	}  (c1);
	\draw [->] (c1) edge[out=-80, in = 80, looseness=1] node[midway, right]{
		$\begin{gathered}
			C_{ab}.\inc2()
		\end{gathered}$
	}  (ss);

	\draw [->] (s2) edge[out=100, in = -100, looseness=1] node[midway, left]{
		$\begin{gathered}
		 [\mathbf{b}'\rightarrow \mathbf{b}'x^{-1}] \\[-1ex]
		 \text{for $x\in B$ }
		\end{gathered}$
	}  (c2);
	\draw [->] (c2) edge[out=-80, in = 80, looseness=1] node[midway, right]{
		$\begin{gathered}
			C_{ab}.\inc2()
		\end{gathered}$
	}  (s2);

	\node (c3) at (16, 8) [point] {};
	\draw [->] (s3) edge[out=100, in = -100, looseness=1] node[midway, left]{
		$\begin{gathered}
		 [\mathbf{c}\rightarrow \mathbf{c}x^{-1}] \\[-1ex]
		 \text{for $x\in C$ }
		\end{gathered}$
	}  (c3);
	\draw [->] (c3) edge[out=-80, in = 80, looseness=1] node[midway, right]{
		$\begin{gathered}
			C_c.\inc2()
		\end{gathered}$
	}  (s3);

	\draw [->] (ss) to node[midway, above]{
		$\begin{gathered}
		 [\mathbf{a}'=1]
		\end{gathered}$
	}  node [midway, below] {$e_4$} (s2);
	\draw [->] (s2) to node[midway, above]{
		$\begin{gathered}
		 [\mathbf{b}'=1]
		\end{gathered}$
	} (s3);
	
	\node (s4) at (6, -8) [point] {$s_5$};
	
	\draw [->] (s3) edge[out=-90, in = 90] node[midway, above]{
		$\begin{gathered}
		 [\mathbf{c}=1]
		\end{gathered}$
	}  (s4);
	
	\node (c4) at (6, -16) [point] {};
	\draw [->] (s4) edge[out = -100, in = 100] node[midway, left]{
		$\begin{gathered}
			C_c.\inc2()
		\end{gathered}$
	}  (c4);
	\draw [->] (c4) edge[out=80, in = -80, looseness=1] node[midway, right]{
		$\begin{gathered}
			C_{ab}.\inc2()
		\end{gathered}$
	}  (s4);
	
	\node (s5) at (10, -8) [point] {};
	\draw [->] (s4) to node[midway, above]{
		$\begin{gathered}
			C_{ab}.\check0()
		\end{gathered}$
	} node [midway, below] {$e_5$} (s5);
	
	\node (f) at (14, -8) [point] {$f$};
	\draw [->] (s5) to node[midway, above]{
		$\begin{gathered}
			[\textbf{d}''=1]
		\end{gathered}$
	}  node [midway, below] {$e_6$} (f);

\end{scope}

\end{tikzpicture}	

%% file: COUNTER2.tex
\begin{tikzpicture} 
[point/.style={inner sep = 4pt, circle,draw,fill=white},
mpoint/.style={inner sep = 1.7pt, circle,draw,fill=black},
FIT/.style args = {#1}{rounded rectangle, draw,  fit=#1, rotate fit=45, yscale=0.5},
FITR/.style args = {#1}{rounded rectangle, draw,  fit=#1, rotate fit=-45, yscale=0.5},
FIT1/.style args = {#1}{rounded rectangle, draw,  fit=#1, rotate fit=45, scale=2},
vecArrow/.style={
		thick, decoration={markings,mark=at position
		   1 with {\arrow[thick]{open triangle 60}}},
		   double distance=1.4pt, shorten >= 5.5pt,
		   preaction = {decorate},
		   postaction = {draw,line width=0.4pt, white,shorten >= 4.5pt}
	},
myptr/.style={decoration={markings,mark=at position 1 with %
    {\arrow[scale=3,>=stealth]{>}}},postaction={decorate}}
]

\begin{scope} [xscale=0.75, yscale=0.3]      
	\node [anchor=west] at (-1.5, 3.) {Operation $\divtwo(b)$};
	\node (start) at (0,0) [point] {$s$};
	\node (n5) at (4, 0) [point] {};
	\node (n6) at (11, 0) [point] {};
	\node (finish) at (15, 0) [point] {$f$};

	\draw [->] (start) to node[midway, above]{
		$\begin{gathered}
		 [\mathbf{d}=1]
		\end{gathered}$
	}  (n5);
	\draw [->] (n6) to node[midway, above]{
		$\begin{gathered}
		 [\mathbf{d}=1]
		\end{gathered}$
	}  (finish);

	\draw [->] (n5) edge[out=30, in = 150, looseness=1] node[midway, above]{
		$\begin{gathered}
		 [\mathbf{b}=1]
		\end{gathered}$
	}  (n6);
	\draw [->] (n5) edge[out=-30, in = -150, looseness=1] node[midway, below]{
		$\begin{gathered}
		 [\mathbf{b}=\delta^{-1}]\\[-1ex]
		 [\mathbf{b}\rightarrow\delta\mathbf{b}]\\[-1ex]
		\end{gathered}$
	}  (n6);

	\draw [->] (n5) edge[loop, out=-125, in = -55, looseness=30] node[pos=0.5, below]{
		$\begin{gathered}
		 [\mathbf{b}\rightarrow\delta^{-2}\mathbf{b}]\\[-1ex]
		 [\mathbf{d}\rightarrow\delta\mathbf{d}]\\[-1ex]
		\end{gathered}$
	}  (n5);

	\draw [->] (n6) edge[loop, out=-125, in = -55, looseness=30] node[pos=0.5, below]{
		$\begin{gathered}
		 [\mathbf{b}\rightarrow\delta\mathbf{b}]\\[-1ex]
		 [\mathbf{d}\rightarrow\delta^{-1}\mathbf{d}]\\[-1ex]
		\end{gathered}$
	}  (n6);
\end{scope}	

\begin{scope} [xscale=0.75, yscale=0.3, yshift = -12cm]      
	\node [anchor=west] at (-1.5, 3) {Operation $\check_{>0} (a)$};
	\node (start) at (0,0) [point] {$s$};
	\node (n2) at (4, 0) [point] {};
	\node (n3) at (11, 0) [point] {};
	\node (finish) at (15, 0) [point] {$f$};
	
	\draw [->] (start) to node[midway, above]{
		$\begin{gathered}
		 [\mathbf{b} = 1]
		\end{gathered}$
	}  (n2);

	\draw [->] (n2) to node[midway, above]{
		$\begin{gathered}
		 \cpy^{\delta}(\mathbf{a}, \mathbf{b})
		\end{gathered}$
	}  (n3);

	\draw [->] (n3) to node[midway, above]{
		$\begin{gathered}
		 [\mathbf{b} =\delta] \\[-1ex]
		 [\mathbf{b} \rightarrow \delta^{-1}\mathbf{b}]
		\end{gathered}$
	}  (finish);

	\draw [->] (n3) edge[loop, out=-135, in = -45, looseness=30] node[pos=0.5, below]{
		$\begin{gathered}
		 \textrm{div$2$} (\mathbf{b})
		\end{gathered}$
	}  (n3);
\end{scope}

\begin{scope}[yshift = -6.5cm]
\begin{scope} [xscale=0.75, yscale=0.3]      
	\node [anchor=west] at (-1.5, 4) {Object Counter$_1[a]$};
	\node at (0, 2.3) {$\check()$};
	\node (start) at (0,0) [point] {$s$};
	\node (finish) at (5, 0) [point] {$f$};
	
	\draw [->] (start) to node[midway, above]{
		$\begin{gathered}
		 \check_{>0} (\mathbf{a})
		\end{gathered}$
	}  (finish);
\end{scope}	

\begin{scope} [xscale=0.75, yscale=0.3, yshift= -5cm]      
	\node at (-0.2, 2.3) {$\inc()$};
	\node (start) at (0,0) [point] {$s$};
	\node (n1) at (5,0) [point] {$q$};
	\node (finish) at (10, 0) [point] {$f$};
	
	\draw [->] (start) to node[midway, above]{
		$\begin{gathered}
		 [\mathbf{a} \rightarrow \delta \mathbf{a}]
		\end{gathered}$
	}  (n1);

	\draw [->] (n1) to node[midway, above]{
		$\begin{gathered}
		\check_{>0} (\mathbf{a})
		\end{gathered}$
	}  (finish);
\end{scope}	
\end{scope}

\end{tikzpicture}	

%% file: DIV2.tex
\begin{tikzpicture} 
[point/.style={inner sep = 4pt, circle,draw,fill=white},
mpoint/.style={inner sep = 1.7pt, circle,draw,fill=black},
FIT/.style args = {#1}{rounded rectangle, draw,  fit=#1, rotate fit=45, yscale=0.5},
FITR/.style args = {#1}{rounded rectangle, draw,  fit=#1, rotate fit=-45, yscale=0.5},
FIT1/.style args = {#1}{rounded rectangle, draw,  fit=#1, rotate fit=45, scale=2},
vecArrow/.style={
		thick, decoration={markings,mark=at position
		   1 with {\arrow[thick]{open triangle 60}}},
		   double distance=1.4pt, shorten >= 5.5pt,
		   preaction = {decorate},
		   postaction = {draw,line width=0.4pt, white,shorten >= 4.5pt}
	},
myptr/.style={decoration={markings,mark=at position 1 with %
    {\arrow[scale=3,>=stealth]{>}}},postaction={decorate}}
]

\begin{scope} [xscale=0.75, yscale=0.3]      
	\node [anchor=west] at (-1.5, 3.) {Operation $\divtwo(b)$};
	\node (start) at (0,0) [point] {$s$};
	\node (n5) at (4, 0) [point] {$q_1$};
	\node (n6) at (11, 0) [point] {$q_2$};
	\node (finish) at (15, 0) [point] {$f$};

	\draw [->] (start) to node[midway, above]{
		$\begin{gathered}
		 [\mathbf{d}=1]
		\end{gathered}$
	}  node[midway, below]{$e_1$} (n5);
	\draw [->] (n6) to node[midway, above]{
		$\begin{gathered}
		 [\mathbf{d}=1]
		\end{gathered}$
	}  node[midway, below]{$e_6$} (finish);

	\draw [->] (n5) edge[out=30, in = 150, looseness=1] node[pos=0.5, above]{
		$\begin{gathered}
		 [\mathbf{b}=1]
		\end{gathered}$
	}  node[pos=0.3, below]{$e_3$}(n6);
	\draw [->] (n5) edge[out=-30, in = -150, looseness=1] node[pos=0.5, below]{
		$\begin{gathered}
		 [\mathbf{b}=\delta^{-1}]\\[-1ex]
		 [\mathbf{b}\rightarrow\delta\mathbf{b}]\\[-1ex]
		\end{gathered}$
	}  node[pos=0.7, above]{$e_4$}(n6);

	\draw [->] (n5) edge[loop, out=-125, in = -55, looseness=30] node[pos=0.5, below]{
		$\begin{gathered}
		 [\mathbf{b}\rightarrow\delta^{-2}\mathbf{b}]\\[-1ex]
		 [\mathbf{d}\rightarrow\delta\mathbf{d}]\\[-1ex]
		\end{gathered}$
	}  node[pos=0.5, above]{$e_2$}  (n5);

	\draw [->] (n6) edge[loop, out=-125, in = -55, looseness=30] node[pos=0.5, below]{
		$\begin{gathered}
		 [\mathbf{b}\rightarrow\delta\mathbf{b}]\\[-1ex]
		 [\mathbf{d}\rightarrow\delta^{-1}\mathbf{d}]\\[-1ex]
		\end{gathered}$
	}   node[pos=0.5, above]{$e_5$} (n6);
\end{scope}

\end{tikzpicture}	

%% file: CHECKSTAR.tex
\begin{tikzpicture} 
[point/.style={inner sep = 4pt, circle,draw,fill=white},
mpoint/.style={inner sep = 1.7pt, circle,draw,fill=black},
FIT/.style args = {#1}{rounded rectangle, draw,  fit=#1, rotate fit=45, yscale=0.5},
FITR/.style args = {#1}{rounded rectangle, draw,  fit=#1, rotate fit=-45, yscale=0.5},
FIT1/.style args = {#1}{rounded rectangle, draw,  fit=#1, rotate fit=45, scale=2},
vecArrow/.style={
		thick, decoration={markings,mark=at position
		   1 with {\arrow[thick]{open triangle 60}}},
		   double distance=1.4pt, shorten >= 5.5pt,
		   preaction = {decorate},
		   postaction = {draw,line width=0.4pt, white,shorten >= 4.5pt}
	},
myptr/.style={decoration={markings,mark=at position 1 with %
    {\arrow[scale=3,>=stealth]{>}}},postaction={decorate}}
]

\begin{scope} [xscale=0.75, yscale=0.3, yshift = -12cm]      
	\node [anchor=west] at (-1.5, 3) {Operation $\check_{>0} (\mathbf{a})$};
	\node (start) at (0,0) [point] {$s$};
	\node (n2) at (4, 0) [point] {$q_1$};
	\node (n3) at (11, 0) [point] {$q_2$};
	\node (finish) at (15, 0) [point] {$f$};
	
	\draw [->] (start) to node[midway, above]{
		$\begin{gathered}
		 [\mathbf{b} = 1]
		\end{gathered}$
	} node[midway, below] {$e_1$} (n2);

	\draw [->] (n2) to node[midway, above]{
		$\begin{gathered}
		 \cpy^\delta(\mathbf{a}, \mathbf{b})
		\end{gathered}$
	} node[midway, below] {$e_2$}  (n3);

	\draw [->] (n3) to node[midway, above]{
		$\begin{gathered}
		 [\mathbf{b} =\delta] \\[-1ex]
		 [\mathbf{b} \rightarrow \delta^{-1}\mathbf{b}]
		\end{gathered}$
	} node[midway, below] {$e_4$} (finish);

	\draw [->] (n3) edge[loop, out=-135, in = -45, looseness=20] node[pos=0.5, below]{
		$\begin{gathered}
		 \textrm{div$2$ } (\mathbf{b})
		\end{gathered}$
	} node[midway, above] {$e_3$}  (n3);
\end{scope}

\end{tikzpicture}	

%% file: SILLYCOUNTER-EXPANDED.tex
\begin{tikzpicture} 
[point/.style={inner sep = 4pt, circle,draw,fill=white},
mpoint/.style={inner sep = 4pt, circle,draw,fill=gray},
FIT/.style args = {#1}{rounded rectangle, draw,  fit=#1, rotate fit=45, yscale=0.5},
FITR/.style args = {#1}{rounded rectangle, draw,  fit=#1, rotate fit=-45, yscale=0.5},
FIT1/.style args = {#1}{rounded rectangle, draw,  fit=#1, rotate fit=45, scale=2},
vecArrow/.style={
		thick, decoration={markings,mark=at position
		   1 with {\arrow[thick]{open triangle 60}}},
		   double distance=1.4pt, shorten >= 5.5pt,
		   preaction = {decorate},
		   postaction = {draw,line width=0.4pt, white,shorten >= 4.5pt}
	},
myptr/.style={decoration={markings,mark=at position 1 with %
    {\arrow[scale=3,>=stealth]{>}}},postaction={decorate}}
]

\begin{scope} [xscale=0.65, yscale=0.3, xshift = 1cm, yshift = -16cm]      
	\node (s) at (0,-24) [point] {$s$};
	\node (ss) at (0,0) [point] {};
	\node (ss1) at (1,-6) [point] {};
	\node (ss2) at (0,-16) [point] {};
	\node (s2) at (7, 0) [point] {};
	\node (c2) at (7, 8) [point] {};
	\node (s3) at (16, 0) [point] {};

	\node (c1) at (-1, 8) [point] {};
	\draw [->] (s) to node[midway, left]{
		$\begin{gathered}
		 [\mathbf{a}'=1] \\[-1ex]
		 [\mathbf{b}'=1] \\[-1ex]
		 [\mathbf{d}'=1] \\[-1ex]
		 [\mathbf{d}''=1]
		\end{gathered}$
	}  (ss2);
	\node (ccaa1) at (5, -12.5) [mpoint]{};
	\node (ccaa2) at (5, -8.5) [mpoint]{};
	\draw [->] (ss2) to node[pos=0.6, below]{
		$\begin{gathered}
		\\[-2ex]
		 [\mathbf{a}'=1] \\[-1ex]
		 [\mathbf{c_a}=1]
		\end{gathered}$
	}   (ccaa1);
	\draw [->] (ccaa1) to node[midway, right]{
		$\begin{gathered}
		 [\mathbf{a}=1] \\[-1ex]
		 [\mathbf{a}'=1]
		\end{gathered}$
	}   (ccaa2);
	\draw [->] (ccaa2) to node[pos = 0.4, above]{
		$\begin{gathered}
		 [\mathbf{c_a}=1]
		\end{gathered}$
	}   (ss1);
	\draw [->] (ccaa1) edge[out=-80, in = 0, looseness=10] node[pos = 0.5, below]{
		$\begin{gathered}
			[\mathbf{a}\rightarrow x^{-1} \mathbf{a}] \\[-1ex]
			[\mathbf{c_a}\rightarrow \mathbf{c_a} x] \\[-1ex]
			 \text{for $x\in A_a$ }
		\end{gathered}$
	}  (ccaa1);

	\draw [->] (ccaa2) edge[out=80, in = 0, looseness=10] node[pos = 0.5, above]{
		$\begin{gathered}
			[\mathbf{a}\rightarrow x \mathbf{a}] \\[-1ex]
			[\mathbf{a'}\rightarrow x \mathbf{a'}] \\[-1ex]
			[\mathbf{c_a}\rightarrow \mathbf{c_a} x^{-1}] \\[-1ex]
			 \text{for $x\in A_a$ }
		\end{gathered}$
	}  (ccaa2);

	\node (ccbb1) at (-2, -10) [mpoint]{};
	\node (ccbb2) at (-2, -6) [mpoint]{};
	\draw [->] (ss1) to node[pos=0.8, right]{
		$\begin{gathered}
		\\
		 [\mathbf{b}'=1] \\[-1ex]
		 [\mathbf{c_b}=1]
		\end{gathered}$
	}   (ccbb1);
	\draw [->] (ccbb1) to node[midway, left]{
		$\begin{gathered}
		 [\mathbf{b}=1] \\[-1ex]
		 [\mathbf{b}'=1]
		\end{gathered}$
	}   (ccbb2);
	\draw [->] (ccbb2) to node[pos = 0.5, right]{
		$\begin{gathered}
		 [\mathbf{c_b}=1]
		\end{gathered}$
	}   (ss);
	\draw [->] (ccbb1) edge[out=180, in = -90, looseness=10] node[pos = 0.5, below]{
		$\begin{gathered}
		\\[-3ex]
			[\mathbf{b}\rightarrow x^{-1} \mathbf{b}] \\[-1ex]
			[\mathbf{c_b}\rightarrow \mathbf{c_b} x] \\[-1ex]
			 \text{for $x\in A_b$ }
		\end{gathered}$
	}  (ccbb1);

	\draw [->] (ccbb2) edge[out=180, in = 90, looseness=10] node[pos = 0.5, above]{
		$\begin{gathered}
			[\mathbf{b}\rightarrow x \mathbf{b}] \\[-1ex]
			[\mathbf{b'}\rightarrow x \mathbf{b'}] \\[-1ex]
			[\mathbf{c_b}\rightarrow \mathbf{c_b} x^{-1}] \\[-1ex]
			 \text{for $x\in A_b$ }
		\end{gathered}$
	}  (ccbb2);

	\draw [->] (ss) edge[out=100, in = -100, looseness=1] node[midway, left]{
		$\begin{gathered}
		 [\mathbf{a}'\rightarrow \mathbf{a}'x^{-1}] \\[-1ex]
		 \text{for $x\in A_a$ }
		\end{gathered}$
	}  (c1);
	\draw [->] (c1) edge[out=-80, in = 80, looseness=1] node[midway, right]{
		$\begin{gathered}
			[\mathbf{d}'\rightarrow \delta^2 \mathbf{d}']
		\end{gathered}$
	}  (ss);

	\draw [->] (s2) edge[out=100, in = -100, looseness=1] node[midway, left]{
		$\begin{gathered}
		 [\mathbf{b}'\rightarrow \mathbf{b}'x^{-1}] \\[-1ex]
		 \text{for $x\in A_b$ }
		\end{gathered}$
	}  (c2);
	\draw [->] (c2) edge[out=-80, in = 80, looseness=1] node[midway, right]{
		$\begin{gathered}
			[\mathbf{d}'\rightarrow \delta^2 \mathbf{d}']
		\end{gathered}$
	}  (s2);

	\node (c3) at (15, 8) [point] {};
	\draw [->] (s3) edge[out=120, in = -100, looseness=1] node[midway, left]{
		$\begin{gathered}
		 [\mathbf{c}\rightarrow \mathbf{c}x^{-1}] \\[-1ex]
		 \text{for $x\in A_c$ }
		\end{gathered}$
	}  (c3);
	\draw [->] (c3) edge[out=-80, in = 80, looseness=1] node[midway, right]{
		$\begin{gathered}
			[\mathbf{d}''\rightarrow \delta^2 \mathbf{d}'']
		\end{gathered}$
	}  (s3);

	\draw [->] (ss) to node[midway, above]{
		$\begin{gathered}
		 [\mathbf{a}'=1]
		\end{gathered}$
	}  (s2);
	\draw [->] (s2) to node[midway, above]{
		$\begin{gathered}
		 [\mathbf{b}'=1]
		\end{gathered}$
	}  (s3);
	
	\node (s4) at (10, -8) [point] {};
	
	\draw [->] (s3) edge[out=-90, in = 90] node[midway, above]{
		$\begin{gathered}
		 [\mathbf{c}=1]
		\end{gathered}$
	}  (s4);
	
	\node (c4) at (12, -17) [point] {};
	\draw [->] (s4) edge[out = -100, in = 100] node[midway, left]{
		$\begin{gathered}
			[\mathbf{d}''\rightarrow \delta^2 \mathbf{d}'']
		\end{gathered}$
	}  (c4);
	\draw [->] (c4) edge[out=80, in = -80, looseness=1] node[midway, right]{
		$\begin{gathered}
			[\mathbf{d}'\rightarrow \delta^2 \mathbf{d}']
		\end{gathered}$
	}  (s4);
	
	\node (s5) at (13, -8) [point] {};
	\draw [->] (s4) to node[midway, above]{
		$\begin{gathered}
			[\textbf{d}'=1]
		\end{gathered}$
	} (s5);
	
	\node (f) at (17, -8) [point] {$f$};
	\draw [->] (s5) to node[midway, above]{
		$\begin{gathered}
			[\textbf{d}''=1]
		\end{gathered}$
	} (f);

\end{scope}	

\end{tikzpicture}	

%% file: ACC-CH-WACC.tex
\begin{tikzpicture} 
[point/.style={inner sep = 4pt, circle,draw,fill=white},
mpoint/.style={inner sep = 1.7pt, circle,draw,fill=black},
FIT/.style args = {#1}{rounded rectangle, draw,  fit=#1, rotate fit=45, yscale=0.5},
FITR/.style args = {#1}{rounded rectangle, draw,  fit=#1, rotate fit=-45, yscale=0.5},
FIT1/.style args = {#1}{rounded rectangle, draw,  fit=#1, rotate fit=45, scale=2},
vecArrow/.style={
		thick, decoration={markings,mark=at position
		   1 with {\arrow[thick]{open triangle 60}}},
		   double distance=1.4pt, shorten >= 5.5pt,
		   preaction = {decorate},
		   postaction = {draw,line width=0.4pt, white,shorten >= 4.5pt}
	},
myptr/.style={decoration={markings,mark=at position 1 with %
    {\arrow[scale=3,>=stealth]{>}}},postaction={decorate}}
]

\begin{scope}
\begin{scope} [xscale=0.75, yscale=0.3]      
	\node [anchor=west] at (-1.5, 4.5) {Operation $\check_L (\mathbf{a})$};
\end{scope}

\begin{scope} [yshift = 0.2cm, xscale=0.85, yscale=0.4]      
	\node (start) at (0,0) [point] {$s$};
	\node (n1) at (3, 0) [point] {$q_1$};
	\node (n2) at (7, 0) [point] {$q_2$};
	\node (n3) at (11, 0) [point] {$q_3$};
	\node (finish) at (15, 0) [point] {$f$};
	
	\draw [->] (start) to node[midway, above]{
		$\begin{gathered}
		 [\mathbf{b} = 1]
		\end{gathered}$
	}  node[midway,below]{$e_1$} (n1);

	\draw [->] (n1) to node[midway, above]{
		$\begin{gathered}
		 \cpy^A(\mathbf{a}, \mathbf{b})
		\end{gathered}$
	} node[midway,below]{$e_2$} (n2);

	\draw [->] (n2) to node[midway, above]{
		$\begin{gathered}
		 \waccept_L(\mathbf{b})
		\end{gathered}$
	} node[midway,below]{$e_3$} (n3);

	\draw [->] (n3) to node[midway, above]{
		$\begin{gathered}
		 [\mathbf{b} = 1]
		\end{gathered}$
	}  node[midway,below]{$e_4$} (finish);
\end{scope}	
\end{scope}

\begin{scope}[yshift = -2.5cm]
\begin{scope} [xscale=0.75, yscale=0.3]      
	\node [anchor=west] at (-1.5, 4.5) {Operation $\accept_L (\mathbf{a})$};
\end{scope}

\begin{scope} [yshift = 0.2cm, xscale=0.85, yscale=0.4]      
	\node (start) at (0,0) [point] {$s$};
	\node (n1) at (3, 0) [point] {$s_1$};
	\node (n2) at (7, 0) [point] {$s_2$};
	\node (n3) at (11, 0) [point] {$s_3$};
	\node (finish) at (15, 0) [point] {$f$};
	
	\draw [->] (start) to node[midway, above]{
		$\begin{gathered}
		 [\mathbf{b} = 1]
		\end{gathered}$
	} node[midway,below]{$f_1$} (n1);

	\draw [->] (n1) to node[midway, above]{
		$\begin{gathered}
		 \cpy^A(\mathbf{a}, \mathbf{b})
		\end{gathered}$
	}  node[midway,below]{$f_2$} (n2);

	\draw [->] (n2) to node[midway, above]{
		$\begin{gathered}
		 \check_L(\mathbf{b})
		\end{gathered}$
	}  node[midway,below]{$f_3$} (n3);

	\draw [->] (n3) to node[midway, above]{
		$\begin{gathered}
		 [\mathbf{a} = 1]\\[-1ex]
		 [\mathbf{b} = 1]
		\end{gathered}$
	} node[midway,below]{$f_5$} (finish);
	
	\draw [->] (n3) edge[loop, out=-45, in = -135, looseness=15] node[pos=0.5, below]{
		$\begin{gathered}		
	    	[\mathbf{a} \rightarrow \mathbf{a} x^{-1}]\\[-1ex]
	    	[\mathbf{b} \rightarrow \mathbf{b} x^{-1}]\\[-1ex]
	    	\textrm{for $x \in A$ }
		\end{gathered}$
	} node[midway,above]{$f_{4,x}$} (n3);
\end{scope}	
\end{scope}
\end{tikzpicture}	

%% file: LESSEQ.tex
\begin{tikzpicture} 
[point/.style={inner sep = 4pt, circle,draw,fill=white},
mpoint/.style={inner sep = 1.7pt, circle,draw,fill=black},
FIT/.style args = {#1}{rounded rectangle, draw,  fit=#1, rotate fit=45, yscale=0.5},
FITR/.style args = {#1}{rounded rectangle, draw,  fit=#1, rotate fit=-45, yscale=0.5},
FIT1/.style args = {#1}{rounded rectangle, draw,  fit=#1, rotate fit=45, scale=2},
vecArrow/.style={
		thick, decoration={markings,mark=at position
		   1 with {\arrow[thick]{open triangle 60}}},
		   double distance=1.4pt, shorten >= 5.5pt,
		   preaction = {decorate},
		   postaction = {draw,line width=0.4pt, white,shorten >= 4.5pt}
	},
myptr/.style={decoration={markings,mark=at position 1 with %
    {\arrow[scale=3,>=stealth]{>}}},postaction={decorate}}
]

\begin{scope} [xscale=0.75, yscale=0.3, yshift = 0cm]      
	\node [anchor=west] at (-1.5, 6) {Operation $\le (a, b)$};
	\node (start) at (0,0) [point] {$s$};
	\node (n2) at (3, 0) [point] {$q_1$};
	\node (n3) at (8, 0) [point] {$q_2$};
	\node (n4) at (13, 0) [point] {$q_3$};
	\node (n5) at (13, -8) [point] {$q_4$};
	\node (finish) at (8, -8) [point] {$f$};
	
	\draw [->] (start) to node[midway, above]{
		$\begin{gathered}
		 [\mathbf{a'} = 1]\\[-1ex]
		 [\mathbf{b'} = 1]
		\end{gathered}$
	} node[midway, below] {$e_1$} (n2);

	\draw [->] (n2) to node[midway, above]{
		$\begin{gathered}
		 \cpy^\delta(\mathbf{a}, \mathbf{a'})
		\end{gathered}$
	} node[midway, below] {$e_2$} (n3);

	\draw [->] (n3) to node[midway, above]{
		$\begin{gathered}
		 \cpy^\delta(\mathbf{b}, \mathbf{b'})
		\end{gathered}$
	} node[midway, below] {$e_3$} (n4);

	\draw [->] (n4) edge[loop, out=60, in = 120, looseness=20] node[pos=0.25, right]{
		$\begin{gathered}
		 [\mathbf{a'}\rightarrow \delta^{-1} \mathbf{a'}]\\[-1ex]
		 [\mathbf{b'}\rightarrow \delta^{-1} \mathbf{b'}]
		\end{gathered}$
	} node[midway, above] {$e_4$} (n4);

	\draw [->] (n4) to node[midway, right]{
		$\begin{gathered}
		 [\mathbf{a'} = 1]
		\end{gathered}$
	}  node[midway, left] {$e_5$} (n5);

	\draw [->] (n5) edge[loop, out=-120, in = -60, looseness=20] node[pos=0.75, right]{
		$\begin{gathered}
		 \textrm{div$2$} (\mathbf{b'})
		\end{gathered}$
	}  node[midway, below] {$e_6$} (n5);

	\draw [->] (n5) to node[midway, above]{
		$\begin{gathered}
		 [\mathbf{b'} =\delta] \\[-1ex]
		 [\mathbf{b'} \rightarrow \delta^{-1}\mathbf{b'}]
		\end{gathered}$
	} node[midway, below] {$e_7$} (finish);
\end{scope}	

\begin{scope} [xscale=0.75, yshift = -5cm , yscale=0.3]      
	\node [anchor=west] at (-1.5, 5) {Operation $\check_{\geq0} (a)$};
	\node (start) at (0,0) [point] {$s$};
	\node (finish) at (5, 0) [point] {$f$};
	
	\draw [->] (start) edge[loop, out=-40, in = -140, looseness=1]  node[midway, below]{
		$\begin{gathered}
		 [\mathbf{a} = 1]
		\end{gathered}$
	}  (finish);

	\draw [->] (start) edge[loop, out=40, in = 140, looseness=1]  node[midway, above]{
		$\begin{gathered}
		 \check_{>0}(\mathbf{a})
		\end{gathered}$
	}  (finish);
\end{scope}	
\end{tikzpicture}	

%% file: COUNTER_STAR.tex
\begin{tikzpicture} 
[point/.style={inner sep = 4pt, circle,draw,fill=white},
mpoint/.style={inner sep = 1.7pt, circle,draw,fill=black},
FIT/.style args = {#1}{rounded rectangle, draw,  fit=#1, rotate fit=45, yscale=0.5},
FITR/.style args = {#1}{rounded rectangle, draw,  fit=#1, rotate fit=-45, yscale=0.5},
FIT1/.style args = {#1}{rounded rectangle, draw,  fit=#1, rotate fit=45, scale=2},
vecArrow/.style={
		thick, decoration={markings,mark=at position
		   1 with {\arrow[thick]{open triangle 60}}},
		   double distance=1.4pt, shorten >= 5.5pt,
		   preaction = {decorate},
		   postaction = {draw,line width=0.4pt, white,shorten >= 4.5pt}
	},
myptr/.style={decoration={markings,mark=at position 1 with %
    {\arrow[scale=3,>=stealth]{>}}},postaction={decorate}}
]

\begin{scope} [xscale=0.75, yscale=0.3]      
	\node [anchor=west] at (-1.5, 6) {
		\begin{tabular}{l} 
		Object Counter$_{d+1}[a_1, \dots, a_{d+1}]$;  \\
		\ \ \ cntr is Counter$_d$ $[a_1, \dots, a_d]$
		\end{tabular}};
	
	\node at (0, 2.3) {$\check()$};
	\node (start) at (0,0) [point] {$s$};
	\node (finish) at (0, -10) [point] {$f$};
	
	\draw [->] (start) to node[midway, right]{
		$\begin{gathered}
		 \textrm{cntr}.\check()
		\end{gathered}$
	}  (finish);
\end{scope}

\begin{scope} [yshift= 1.25cm, xshift = 3.5cm, xscale=0.4, yscale=0.45]      
	\node at (7, -1.5) {$\inc()$};
	\node (start) at (7,-3) [point] {$s$};
	\node (n2) at (0,-6) [point] {$q_1$};
	\node (n5) at (0,-9) [point] {$q_2$};
	\node (n6) at (0,-12) [point] {$q_3$};
	\node (m2) at (14,-6) [point] {$r_1$};
	\node (m5) at (14,-9) [point] {$r_2$};
	\node (m6) at (14,-12) [point] {$r_3$};
	\node (finish) at (7, -15) [point] {$f$};
	
	\draw [->] (start) to node[pos=0.5, above left]{
		$\begin{gathered}
		\check_{\geq0}(\mathbf{a_{d+1}}) 
		\end{gathered}$
	} node [pos=0.5, below] {$e_1$}  (n2);

	\draw [->] (start) to node[pos=0.5,above right]{
		$\begin{gathered}
		\check_{\geq0}(\mathbf{a_{d+1}}) 
		\end{gathered}$
	} node [pos=0.5, below] {$f_1$}  (m2);

	\draw [->] (n5) to node[midway, right]{
		$\begin{gathered}
		 		\cpy(\mathbf{a_d}, \mathbf{a_{d+1}})
		\end{gathered}$
	} node [pos=0.5, left] {$e_2$} (n2);

	\draw [->] (m2) to node[midway, left]{
		$\begin{gathered}
		 		\le(\mathbf{a_{d+1}}, \mathbf{a_d})
		\end{gathered}$
	}  node [pos=0.5, right] {$f_2$} (m5);

	\draw [->] (n5) to node[pos=0.5, right]{
		$\begin{gathered}
		\text{cntr}.\inc()
		\end{gathered}$
	}  node [pos=0.5, left] {$e_3$} (n6);

	\draw [->] (m5) to node[midway, left]{
		$\begin{gathered}
		 [\mathbf{a_{d+1}} \rightarrow \delta \mathbf{a_{d+1}}]
		\end{gathered}$
	}  node [pos=0.5, right] {$f_3$} (m6);

	\draw [->] (n6) to node[midway, below left]{
		$\begin{gathered}
		 [\mathbf{a_{d+1}} =1]
		\end{gathered}$
	} node [pos=0.5, above] {$e_4$}  (finish);

	\draw [->] (m6) to node[midway, below right]{
		$\begin{gathered}
		\check_{>0}(\mathbf{a_{d+1}}) 
		\end{gathered}$
	} node [pos=0.5, above] {$f_4$}  (finish);
=======
\end{scope}

\end{tikzpicture}	

%% file: SPLIT.tex
\begin{tikzpicture} 
[point/.style={inner sep = 4pt, circle,draw,fill=white},
mpoint/.style={inner sep = 1.7pt, circle,draw,fill=black},
FIT/.style args = {#1}{rounded rectangle, draw,  fit=#1, rotate fit=45, yscale=0.5},
FITR/.style args = {#1}{rounded rectangle, draw,  fit=#1, rotate fit=-45, yscale=0.5},
FIT1/.style args = {#1}{rounded rectangle, draw,  fit=#1, rotate fit=45, scale=2},
vecArrow/.style={
		thick, decoration={markings,mark=at position
		   1 with {\arrow[thick]{open triangle 60}}},
		   double distance=1.4pt, shorten >= 5.5pt,
		   preaction = {decorate},
		   postaction = {draw,line width=0.4pt, white,shorten >= 4.5pt}
	},
myptr/.style={decoration={markings,mark=at position 1 with %
    {\arrow[scale=3,>=stealth]{>}}},postaction={decorate}}
]

\begin{scope} [xscale=0.75, yscale=0.3]      
	\node [anchor=west] at (-1.5, 4.5) {Operation $\splt^A(\mathbf{a}, \mathbf{b})$};
\end{scope}

\begin{scope} [yshift = 0.2cm, xscale=0.85, yscale=0.4]      
	\node (start) at (0,0) [point] {$s$};
	\node (finish) at (4, 0) [point] {$f$};
	
	\draw [->] (start) to node[midway, above]{
		$\begin{gathered}
		\end{gathered}$
	}  (finish);

	\draw [->] (finish) edge[loop, out=45, in = -45, looseness=9] node[pos=0.5, right]{
		$\begin{gathered}		
	    	[\mathbf{a} \rightarrow \mathbf{a} x^{-1}]\\[-1ex]
	    	[\mathbf{b} \rightarrow x \mathbf{b}]\\[-1ex]
	    	\textrm{for $x \in A$ }
		\end{gathered}$
	}  (finish);

\end{scope}	
\end{tikzpicture}	

%% file: CHECK-LPLUS.tex
\begin{tikzpicture} 
[point/.style={inner sep = 4pt, circle,draw,fill=white},
mpoint/.style={inner sep = 1.7pt, circle,draw,fill=black},
FIT/.style args = {#1}{rounded rectangle, draw,  fit=#1, rotate fit=45, yscale=0.5},
FITR/.style args = {#1}{rounded rectangle, draw,  fit=#1, rotate fit=-45, yscale=0.5},
FIT1/.style args = {#1}{rounded rectangle, draw,  fit=#1, rotate fit=45, scale=2},
vecArrow/.style={
		thick, decoration={markings,mark=at position
		   1 with {\arrow[thick]{open triangle 60}}},
		   double distance=1.4pt, shorten >= 5.5pt,
		   preaction = {decorate},
		   postaction = {draw,line width=0.4pt, white,shorten >= 4.5pt}
	},
myptr/.style={decoration={markings,mark=at position 1 with %
    {\arrow[scale=3,>=stealth]{>}}},postaction={decorate}}
]

\begin{scope} [xshift=-0cm, xscale=0.75, yscale=0.3, yshift=0cm]      
	\node [anchor=west] at (-1.5, 5) {
		\begin{tabular}{l} 
		Operation $\waccept_{R+}^\varepsilon(a)$; \\
		cntr is Counter$_\varepsilon[\I_\eps]$; 
		\end{tabular}};
\end{scope}

\begin{scope} [yshift = 0cm, xscale=0.85, yscale=0.4]      
	\node (start) at (0,0) [point] {$s$};
	\node (n1) at (3, 0) [point] {$q_1$};
	\node (n3) at (10, 0) [point] {$q_3$};
	\node (n4) at (3, -4) [point] {$q_2$};
	\node (finish) at (13, 0) [point] {$f$};

	\draw [->] (start) to node[midway, above]{
		$\begin{gathered}
		 \I_\eps = 1
		\end{gathered}$
	}  node[midway, below]{$e_1$} (n1);

	\draw [->] (n1) edge[out=-45, in = 45, looseness=1] node[pos=0.7, right]{
		$\begin{gathered}		
	    	[\mathbf{a} \rightarrow  \mathbf{a} x^{-1}] \\[-1ex]
	    	\textrm{for $x \in A$ }
		\end{gathered}$
	} node[pos=0.5, left]{$e_{3,x}$} (n4);

	\draw [->] (n4) edge[out=135, in = -135, looseness=1] node[pos=0.5, left]{
		$\begin{gathered}		
	    	\cntr.\inc()
		\end{gathered}$
	} node[pos=0.5, right]{$e_2$} (n1);

	\draw [->] (n1) to node[midway, above]{
		$\begin{gathered}
		 [\mathbf{a} = 1]
		\end{gathered}$
	} node[midway, below]{$e_4$}  (n3);
	
	\draw [->] (n3) edge[loop, out=-45, in = -135, looseness=20] node[pos=0.5, below]{
		$\begin{gathered}		
	    	\cntr.\inc()
		\end{gathered}$
	}  node[pos=0.5, above]{$e_5$} (n3);

	\draw [->] (n3) to node[midway, above]{
		$\begin{gathered}
		 \I_\eps = 1
		\end{gathered}$
	} node[midway, below]{$e_6$} (finish);
\end{scope}	
\end{tikzpicture}	

%% file: POSCHECK1.tex
\begin{tikzpicture} 
[point/.style={inner sep = 4pt, circle,draw,fill=white},
mpoint/.style={inner sep = 1.7pt, circle,draw,fill=black},
FIT/.style args = {#1}{rounded rectangle, draw,  fit=#1, rotate fit=45, yscale=0.5},
FITR/.style args = {#1}{rounded rectangle, draw,  fit=#1, rotate fit=-45, yscale=0.5},
FIT1/.style args = {#1}{rounded rectangle, draw,  fit=#1, rotate fit=45, scale=2},
vecArrow/.style={
		thick, decoration={markings,mark=at position
		   1 with {\arrow[thick]{open triangle 60}}},
		   double distance=1.4pt, shorten >= 5.5pt,
		   preaction = {decorate},
		   postaction = {draw,line width=0.4pt, white,shorten >= 4.5pt}
	},
myptr/.style={decoration={markings,mark=at position 1 with %
    {\arrow[scale=3,>=stealth]{>}}},postaction={decorate}}
]

\begin{scope} [xscale=0.75, yscale=0.3, yshift=0cm]      
	\node [anchor=west] at (-1.5, 5) {
		\begin{tabular}{l} 
		Operation $\waccept_+^{1, \varepsilon}(a)$;
		\end{tabular}};
\end{scope}	

\begin{scope} [yshift= 0cm, xscale=0.75, yscale=0.5]      

	\node (finish) at (0, 0) [point] {$f$};
	\node (n1) at (0,-4) [point] {$q_1$};
	\node (n2) at (8,-4) [point] {$q_2$};
	\node (n3) at (8,0) [point] {$q_3$};
	\node (start) at (4,0) [point] {$s$};
	
	\draw [->] (finish) to node[midway, left]{
		$\begin{gathered}
		 [\mathbf{b} =1]
		\end{gathered}$
	} node[midway, right] {$e_1$} (n1);

	\draw [->] (n1) to node[midway, below]{
		$\begin{gathered}
		\splt^A(\mathbf{a}, \mathbf{b})
		\end{gathered}$
	} node[midway, above] {$e_2$}  (n2);

	\draw [->] (n2) to node[midway, right]{
		$\begin{gathered}
		\check_{R+}^\varepsilon(\mathbf{a})
		\end{gathered}$
	} node[midway, left] {$e_3$}  (n3);

	\draw [->] (n3) to node[midway, above]{
		$\begin{gathered}		
			[\mathbf{b} = x]\\[-1ex]
	    	[\mathbf{b} \rightarrow \mathbf{b} x^{-1}]\\[-1ex]
	    	\textrm{for $x \in A$ }
		\end{gathered}$
	} node[midway, below] {$e_{4,x}$}  (start);

	\draw [->] (start) to node[midway, above]{
		$\begin{gathered}		
		\end{gathered}$
	} node[midway, below] {$e_5$}  (finish);
\end{scope}	

\end{tikzpicture}	

%% file: POSCHECK2.tex
\begin{tikzpicture} 
[point/.style={inner sep = 4pt, circle,draw,fill=white},
mpoint/.style={inner sep = 1.7pt, circle,draw,fill=black},
FIT/.style args = {#1}{rounded rectangle, draw,  fit=#1, rotate fit=45, yscale=0.5},
FITR/.style args = {#1}{rounded rectangle, draw,  fit=#1, rotate fit=-45, yscale=0.5},
FIT1/.style args = {#1}{rounded rectangle, draw,  fit=#1, rotate fit=45, scale=2},
vecArrow/.style={
		thick, decoration={markings,mark=at position
		   1 with {\arrow[thick]{open triangle 60}}},
		   double distance=1.4pt, shorten >= 5.5pt,
		   preaction = {decorate},
		   postaction = {draw,line width=0.4pt, white,shorten >= 4.5pt}
	},
myptr/.style={decoration={markings,mark=at position 1 with %
    {\arrow[scale=3,>=stealth]{>}}},postaction={decorate}}
]

\begin{scope} [xscale=0.75, yscale=0.3, yshift=0cm]      
	\node [anchor=west] at (-1.5, 5) {
		\begin{tabular}{l} 
		Operation $\waccept_+^{d+1, \varepsilon}(a)$;
		\end{tabular}};
\end{scope}	

\begin{scope} [yshift= 0cm, xscale=0.75, yscale=0.5]      

	\node (finish) at (-1, 0) [point] {$f$};
	\node (n1) at (-1,-4) [point] {$q_1$};
	\node (n2) at (8,-4) [point] {$q_2$};
	\node (n3) at (8,0) [point] {$q_3$};
	\node (n4) at (4,0) [point] {$q_4$};
	\node (start) at (1,0) [point] {$s$};
	
	\draw [->] (finish) to node[midway, left]{
		$\begin{gathered}
		 [\mathbf{b} =1]
		\end{gathered}$
	} node[midway, right] {$e_1$} (n1);

	\draw [->] (n1) to node[midway, below]{
		$\begin{gathered}
		\splt^A(\mathbf{a}, \mathbf{b})
		\end{gathered}$
	} node[midway, above] {$e_2$}  (n2);

	\draw [->] (n2) to node[midway, right]{
		$\begin{gathered}
		\check_{R+}^\varepsilon(\mathbf{a})
		\end{gathered}$
	} node[midway, left] {$e_3$}  (n3);

	\draw [->] (n3) to node[midway, above]{
		$\begin{gathered}		
			\accept_+^{d, \varepsilon}(\mathbf{b})
		\end{gathered}$
	} node[midway, below] {$e_4$}  (n4);

	\draw [->] (n4) to node[midway, above]{
		$\begin{gathered}		
			[\mathbf{b} = 1]
		\end{gathered}$
	} node[midway, below] {$e_5$}  (start);

	\draw [->] (start) to node[midway, above]{
		$\begin{gathered}		
		\end{gathered}$
	} node[midway, below] {$e_6$}  (finish);
\end{scope}	

\end{tikzpicture}	

%% file: SWAP.tex
\begin{tikzpicture} 
[point/.style={inner sep = 4pt, circle,draw,fill=white},
mpoint/.style={inner sep = 1.7pt, circle,draw,fill=black},
FIT/.style args = {#1}{rounded rectangle, draw,  fit=#1, rotate fit=45, yscale=0.5},
FITR/.style args = {#1}{rounded rectangle, draw,  fit=#1, rotate fit=-45, yscale=0.5},
FIT1/.style args = {#1}{rounded rectangle, draw,  fit=#1, rotate fit=45, scale=2},
vecArrow/.style={
		thick, decoration={markings,mark=at position
		   1 with {\arrow[thick]{open triangle 60}}},
		   double distance=1.4pt, shorten >= 5.5pt,
		   preaction = {decorate},
		   postaction = {draw,line width=0.4pt, white,shorten >= 4.5pt}
	},
myptr/.style={decoration={markings,mark=at position 1 with %
    {\arrow[scale=3,>=stealth]{>}}},postaction={decorate}}
]

\begin{scope} [xscale=0.75, yscale=0.3]      
	\node [anchor=west] at (-1.5, 4.5) {Operation $\swap^A(\mathbf{a}, \mathbf{b})$};
\end{scope}

\begin{scope} [yshift = 0.2cm, xscale=0.85, yscale=0.4]      
	\node (start) at (0,0) [point] {$s$};
	\node (n2) at (3, 0) [point] {};
	\node (finish) at (6, 0) [point] {$f$};
	
	\draw [->] (start) to node[midway, above]{
		$\begin{gathered}
		[\mathbf{b} = 1]
		\end{gathered}$
	}  (n2);

	\draw [->] (n2) edge[loop, out=-45, in = -135, looseness=30] node[pos=0.5, below]{
		$\begin{gathered}		
	    	\splt^A(\mathbf{a},\mathbf{b})
		\end{gathered}$
	}  (n2);

	\draw [->] (n2) to node[midway, above]{
		$\begin{gathered}
		 [\mathbf{a} = 1]
		\end{gathered}$
	}  (finish);
\end{scope}	

%
%
%

\end{tikzpicture}	

%% file: REA2.tex
\begin{tikzpicture} 
[point/.style={inner sep = 4pt, circle,draw,fill=white},
mpoint/.style={inner sep = 1.7pt, circle,draw,fill=black},
FIT/.style args = {#1}{rounded rectangle, draw,  fit=#1, rotate fit=45, yscale=0.5},
FITR/.style args = {#1}{rounded rectangle, draw,  fit=#1, rotate fit=-45, yscale=0.5},
FIT1/.style args = {#1}{rounded rectangle, draw,  fit=#1, rotate fit=45, scale=2},
vecArrow/.style={
		thick, decoration={markings,mark=at position
		   1 with {\arrow[thick]{open triangle 60}}},
		   double distance=1.4pt, shorten >= 5.5pt,
		   preaction = {decorate},
		   postaction = {draw,line width=0.4pt, white,shorten >= 4.5pt}
	},
myptr/.style={decoration={markings,mark=at position 1 with %
    {\arrow[scale=3,>=stealth]{>}}},postaction={decorate}}
]

\begin{scope}[yshift = -1.75cm]
\begin{scope} [xscale=0.75, yscale=0.3]      
	\node [anchor=west] at (-1.5, 3) {Operation $\rea_A^\varepsilon (\mathbf{a}, \mathbf{b})$};
\end{scope}

\begin{scope} [yshift =-0.2cm, xscale=0.85, yscale=0.4]      
	\node (start) at (0,0) [point] {$s$};
	\node (n1) at (4, 0) [point] {};
	\node (n2) at (8, 0) [point] {};
	\node (n3) at (8, -6) [point] {};
	\node (n4) at (4, -6) [point] {};
	\node (finish) at (0, -6) [point] {$f$};
	
	\draw [->] (start) to node[midway, above]{
		$\begin{gathered}
		 \check^\varepsilon_{A+} (\mathbf{a})
		\end{gathered}$
	} node[midway, below] {$e_{a,s}$} (n1);

	\draw [->] (n1) to node[midway, above]{
		$\begin{gathered}
		 \check^\varepsilon_{A+} (\mathbf{b})
		\end{gathered}$
	} node[midway, below] {$e_{b,s}$} (n2);


	\draw [->] (n2) to node[midway, right]{
		$\begin{gathered}
		 \splt^A (\mathbf{a}, \mathbf{b})
		\end{gathered}$
	}  node[midway, left] {$e$} (n3);


	\draw [<-] (n3) to node[midway, below]{
		$\begin{gathered}
		 \check^\varepsilon_{A+} (\mathbf{b})
		\end{gathered}$
	} node[midway, above] {$e_{b,f}$} (n4);

	\draw [<-] (n4) to node[midway, below]{
		$\begin{gathered}
		 \check^\varepsilon_{A+} (\mathbf{a})
		\end{gathered}$
	} node[midway, above] {$e_{a,f}$} (finish);
\end{scope}	
\end{scope}

\end{tikzpicture}	

%% file: TTAPE1.tex
\begin{tikzpicture} 
[point/.style={inner sep = 4pt, circle,draw,fill=white},
mpoint/.style={inner sep = 1.7pt, circle,draw,fill=black},
FIT/.style args = {#1}{rounded rectangle, draw,  fit=#1, rotate fit=45, yscale=0.5},
FITR/.style args = {#1}{rounded rectangle, draw,  fit=#1, rotate fit=-45, yscale=0.5},
FIT1/.style args = {#1}{rounded rectangle, draw,  fit=#1, rotate fit=45, scale=2},
vecArrow/.style={
		thick, decoration={markings,mark=at position
		   1 with {\arrow[thick]{open triangle 60}}},
		   double distance=1.4pt, shorten >= 5.5pt,
		   preaction = {decorate},
		   postaction = {draw,line width=0.4pt, white,shorten >= 4.5pt}
	},
myptr/.style={decoration={markings,mark=at position 1 with %
    {\arrow[scale=3,>=stealth]{>}}},postaction={decorate}}
]

\begin{scope}[yshift = -3.5cm]
\begin{scope} [xscale=0.75, yscale=0.3]      
	\node [anchor=west] at (-1.5, 4) {Object $\TTape_{\T,i}^{1, \varepsilon}[\I_1]$};
	\node at (0, 2.3) {$\set(\mathbf{b})$};
	\node (start) at (0,0) [point] {$s$};
	\node (n1) at (5,0) [point] {$q$};
	\node (finish) at (10, 0) [point] {$f$};
	
	\draw [->] (start) to node[midway, above]{
		$\begin{gathered}
		 \check_{A_i+}^\varepsilon (\mathbf{b})
		\end{gathered}$
	}  node[midway, below] {$e_1$} (n1);
	\draw [->] (n1) to node[midway, above]{
		$\begin{gathered}
		\swap(\mathbf{b}, \mathbf{a_1})
		\end{gathered}$
	}  node[midway, below] {$e_2$} (finish);
\end{scope}

\begin{scope} [xscale=0.75, yscale=0.3, yshift= -6cm]      
	\node [anchor=west] at (-0.7, 2.5) {$Q()$, for $Q=(Q_j)_{j\in\I_\T}$ with $Q_i = [u\rightarrow v]$};
	\node (start) at (0,0) [point] {$s$};
	\node (finish) at (5, 0) [point] {$f$};
	
	\draw [->] (start) to node[midway, below]{
		$\begin{gathered}
		 [\mathbf{a_1} = u] \\[-1ex]
		 [\mathbf{a_1} \rightarrow \mathbf{a_1} u^{-1} v]
		\end{gathered}$
	}  node[midway, above] {$e$} (finish);
\end{scope}

%
%
%
%

\begin{scope} [xscale=0.75, yscale=0.3, yshift= -15cm]      
	\node [anchor=west] at (-0.7, 3) {$Q()$, for $Q=(Q_j)_{j\in\I_\T}$ with $Q_i = [* u \rightarrow * v]$};
	\node (start) at (0,0) [point] {$s$};
	\node (n1) at (4,0) [point] {$q_1$};
	\node (n2) at (8,0) [point] {$q_2$};
	\node (finish) at (12,0) [point] {$f$};
	
	\draw [->] (start) to node[midway, below]{
		$\begin{gathered}
		 [\mathbf{a_1} \rightarrow  \mathbf{a_1}u^{-1}]
		\end{gathered}$
	} node[midway, above] {$e_{1,r}$} (n1);

	\draw [->] (n1) to node[midway, below]{
		$\begin{gathered}
		 \check_{A_i+}^\varepsilon(\mathbf{a_1})
		\end{gathered}$
	}  node[midway, above] {$e_{2,r}$} (n2);

	\draw [->] (n2) to node[midway, below]{
		$\begin{gathered}
		 [\mathbf{a_1} \rightarrow  \mathbf{a_1}v]
		\end{gathered}$
	} node[midway, above] {$e_{3,r}$} (finish);
\end{scope}

%

\end{scope}

\end{tikzpicture}	

%% file: TTd.tex
\begin{tikzpicture} 
[point/.style={inner sep = 4pt, circle,draw,fill=white},
mpoint/.style={inner sep = 1.7pt, circle,draw,fill=black},
FIT/.style args = {#1}{rounded rectangle, draw,  fit=#1, rotate fit=45, yscale=0.5},
FITR/.style args = {#1}{rounded rectangle, draw,  fit=#1, rotate fit=-45, yscale=0.5},
FIT1/.style args = {#1}{rounded rectangle, draw,  fit=#1, rotate fit=45, scale=2},
vecArrow/.style={
		thick, decoration={markings,mark=at position
		   1 with {\arrow[thick]{open triangle 60}}},
		   double distance=1.4pt, shorten >= 5.5pt,
		   preaction = {decorate},
		   postaction = {draw,line width=0.4pt, white,shorten >= 4.5pt}
	},
myptr/.style={decoration={markings,mark=at position 1 with %
    {\arrow[scale=3,>=stealth]{>}}},postaction={decorate}}
]

\begin{scope} [xshift=-0cm, xscale=0.75, yscale=0.3, yshift=2cm]      
	\node [anchor=west] at (-1.5, 5) {
		\begin{tabular}{l} 
		Object $\IRT_{\T,i}^{d,\eps}[\I_{d+1}]$; \\
		\ \ \ $\TT^d$ is $\TTape^d[\I_d']$
		\end{tabular}};
\end{scope}	

\begin{scope} [xscale=0.75, yscale=0.3]      
	\node [anchor=west] at (-0.7, 2.5) {$\rea()$};
	\node (start) at (0,0) [point] {$s$};
	\node (n1) at (5, 0) [point] {$q_1$};
	\node (n2) at (10, 0) [point] {$q_2$};
	\node (n3) at (10, -7) [point] {$q_3$};
	\node (n4) at (5, -7) [point] {$q_4$};
	\node (finish) at (0, -7) [point] {$f$};
	
	\draw [->] (start) to node[midway, right]{ $e_0$
	}  (finish);

	\draw [->] (start) to node[midway, above]{
		$\begin{gathered}
		[\mathbf{c} = 1]
		\end{gathered}$
	} node [midway, below]{$e_1$} (n1);

	\draw [->] (n2) to node[midway, above]{
		$\TT^d.\set(\mathbf{c})$
	}   node [midway, below]{$e_2$} (n1);

	\draw [->] (n2) to node[midway, right]{
		$\begin{gathered}
		 \rea_{A_i}^\varepsilon (\mathbf{a_1}, \mathbf{c})
		\end{gathered}$
	}   node [midway, left]{$e_3$} (n3);

	\draw [->] (n3) to node[midway, below]{
		$\TT^d.\set(\mathbf{c})$
	}   node [midway, above]{$e_4$}  (n4);
	
	\draw [->] (n4) to node[midway, below]{
		$\begin{gathered}
		[\mathbf{c} = 1]
		\end{gathered}$
	}  node [midway, above]{$e_5$}  (finish);
\end{scope}

\begin{scope} [xscale=0.75, yscale=0.3, yshift=-14cm]      
	\node [anchor=west] at (-0.7, 2.5) {$\set_0(\mathbf{b})$};
	\node (start) at (0,0) [point] {$s$};
	\node (n1) at (5, 0) [point] {$q$};
	\node (finish) at (10, 0) [point] {$f$};
	
	\draw [->] (start) to node[midway, below]{
		$\check_{A_i+}^\eps(\mathbf{b})$
	}   node [midway, above]{$e_1$} (n1);

	\draw [->] (n1) to node[midway, below]{
		$\swap(\mathbf{b},\mathbf{a_1})$
	}   node [midway, above]{$e_2$} (finish);
\end{scope}

\begin{scope} [xscale=0.75, yscale=0.3, yshift=-21cm]      
	\node [anchor=west] at (-0.7, 2.5) {$Q_0()$ for $Q\in\Q_\T$};
	\node (start) at (0,0) [point] {$s$};
	\node (finish) at (5, 0) [point] {$f$};
	
	\draw [->] (start) to node[midway, below]{
		$\TT_d'.Q()$
	}   node [midway, above]{$e$} (finish);
\end{scope}

%
%

\end{tikzpicture}	

%% file: TTAPE2.tex
\begin{tikzpicture} 
[point/.style={inner sep = 4pt, circle,draw,fill=white},
mpoint/.style={inner sep = 1.7pt, circle,draw,fill=black},
FIT/.style args = {#1}{rounded rectangle, draw,  fit=#1, rotate fit=45, yscale=0.5},
FITR/.style args = {#1}{rounded rectangle, draw,  fit=#1, rotate fit=-45, yscale=0.5},
FIT1/.style args = {#1}{rounded rectangle, draw,  fit=#1, rotate fit=45, scale=2},
vecArrow/.style={
		thick, decoration={markings,mark=at position
		   1 with {\arrow[thick]{open triangle 60}}},
		   double distance=1.4pt, shorten >= 5.5pt,
		   preaction = {decorate},
		   postaction = {draw,line width=0.4pt, white,shorten >= 4.5pt}
	},
myptr/.style={decoration={markings,mark=at position 1 with %
    {\arrow[scale=3,>=stealth]{>}}},postaction={decorate}}
]

\begin{scope} [xshift=-0cm, xscale=0.75, yscale=0.3, yshift=1cm]      
	\node [anchor=west] at (-1.5, 5) {
		\begin{tabular}{l} 
		Object $\TTape_{\T,i}^{d+1, \varepsilon}[\I_{d+1}]$; \\
		\ \ \ $\IRT^d$ is $\IRT_{\T,i}^{d,\eps}[\I_{d+1}]$
		\end{tabular}};
\end{scope}	

%
%
%
%
%
%
%
%
\begin{scope} [xscale=0.75, yscale=0.3]      
	\node at (0, 2.3) {$\set(\mathbf{b})$};
	\node (start) at (0,0) [point] {$s$};
	\node (n1) at (3.5,0) [point] {$q_1$};
	\node (n2) at (8,0) [point] {$q_2$};
	\node (n3) at (12,0) [point] {$q_3$};
	\node (finish) at (15.5, 0) [point] {$f$};
	
	\draw [->] (start) to node[midway, below]{
		$\begin{gathered}
		\I_{d+1}=1
		\end{gathered}$
	} node[midway,above]{$e_1$} (n1);
	
	\draw [->] (n1) to node[midway, below]{
		$\begin{gathered}
		\IRT^d.\set_0(\mathbf{b})
		\end{gathered}$
	} node[midway,above]{$e_2$} (n2);
	
	\draw [->] (n2) to node[midway, below]{
	 $\IRT^d.\rea()$
	} node[midway,above]{$e_3$} (n3);
	
	\draw [->] (n3) to node[midway, below]{
	$[\mathbf{b}=1]$
	} node[midway,above]{$e_4$} (finish);

\end{scope}

\begin{scope} [xscale=0.75, yscale=0.3, yshift= -8cm]      
	\node [anchor=west] at (-0.7, 2.5) {$Q()$ for $Q=(Q_j)_{j\in\I_\T}\in\Q_\T$ such that $Q_i$ is exact};
	\node (start) at (0,-1) [point] {$s$};
	\node (n1) at (3.5,-1) [point] {$q_1$};
	\node (n2) at (6.5,-1) [point] {$q_2$};
	\node (n3) at (10.5,-1) [point] {$q_3$};
	\node (n4) at (13.5,-1) [point] {$q_4$};
	\node (finish) at (17, -1) [point] {$f$};
	
	\draw [->] (start) to node[midway, below]{
		$\begin{gathered}
		 \IRT^d.\rea()
		\end{gathered}$
	} node[midway,above]{$e_1$} (n1);
	
	\draw [->] (n1) to node[midway, below]{
		$\begin{gathered}
		[\mathbf{a_1} = 1]
		\end{gathered}$
	} node[midway,above]{$e_2$} (n2);

	\draw [->] (n2) to node[midway, below]{
		$\IRT^d.Q_0()$
	} node[midway,above]{$e_3$} (n3);

	\draw [->] (n3) to node[midway, below]{
		$\begin{gathered}
		[\mathbf{a_1} = 1]
		\end{gathered}$
	} node[midway,above]{$e_4$} (n4);
	
	\draw [->] (n4) to node[midway, below]{
		$\begin{gathered}
		 \IRT^d.\rea()
		\end{gathered}$
	} node[midway,above]{$e_5$} (finish);	
\end{scope}	

\begin{scope} [xscale=0.75, yscale=0.3, yshift= -18cm]      
	\node [anchor=west] at (-0.7, 2.5) {$Q()$ for $Q=(Q_j)_{j\in\I_\T}\in\Q_\T$ such that $Q_i$ is (right) matching};
	\node (start) at (0,0) [point] {$s$};
	\node (n1) at (3.5,0) [point] {$q_1$};
	\node (n2) at (7.5,0) [point] {$q_2$};
	\node (finish) at (11, 0) [point] {$f$};
	
	\draw [->] (start) to node[midway, below]{
		$\begin{gathered}
		 \IRT^d.\rea()
		\end{gathered}$
	} node[midway,above]{$e_1$} (n1);

	\draw [->] (n1) to node[midway, below]{
		$\IRT^d.Q_0()$
	} node[midway,above]{$e_2$} (n2);
	
	\draw [->] (n2) to node[midway, below]{
		$\begin{gathered}
		 \IRT^d.\rea()
		\end{gathered}$
	} node[midway,above]{$e_3$} (finish);	
\end{scope}

%
%
%
%
%
%

\end{tikzpicture}	

%% file: QEDGE.tex
\begin{tikzpicture} 
[point/.style={inner sep = 4pt, circle,draw,fill=white},
mpoint/.style={inner sep = 1.7pt, circle,draw,fill=black},
FIT/.style args = {#1}{rounded rectangle, draw,  fit=#1, rotate fit=45, yscale=0.5},
FITR/.style args = {#1}{rounded rectangle, draw,  fit=#1, rotate fit=-45, yscale=0.5},
FIT1/.style args = {#1}{rounded rectangle, draw,  fit=#1, rotate fit=45, scale=2},
vecArrow/.style={
		thick, decoration={markings,mark=at position
		   1 with {\arrow[thick]{open triangle 60}}},
		   double distance=1.4pt, shorten >= 5.5pt,
		   preaction = {decorate},
		   postaction = {draw,line width=0.4pt, white,shorten >= 4.5pt}
	},
myptr/.style={decoration={markings,mark=at position 1 with %
    {\arrow[scale=3,>=stealth]{>}}},postaction={decorate}}
]

\begin{scope} [yshift = 0.2cm, xscale=0.85, yscale=0.4]    
	\node [anchor=west] at (-1, 3) {In $\G_\T$:};  
	\node (start) at (0,0) [point] {$q_1$};
	\node (finish) at (16, 0) [point] {$q_2$};
	
	\draw [->] (start) to node[midway, above]{
		$\begin{gathered}
		Q
		\end{gathered}$
	} node[midway, below]{$e$} (finish);
\end{scope}

\begin{scope} [yshift = -2.75cm, xscale = 0.85, yscale=0.4]
	\node [anchor=west] at (-1,3) {In $\G_\T'$:};
	\node at (8,4.5) {\Large$\Downarrow$};
	\node[fill,circle,minimum size=1.2cm] (start) at (0,0) [point] {$p_{q_1}$};
	\node[fill,circle,minimum size=1.2cm] (s1) at (4,0) [point] {$s_e(1)$};
	\node[fill,circle,minimum size=1.2cm] (s2) at (8,0) [point] {$s_e(2)$};
	\node[fill,circle,minimum size=1.2cm] (s3) at (12,0) [point] {\tiny $s_e(k-1)$};
	\node (q1) at (9.5,0) {};
	\node (q2) at (10.4,0) {};
	\node at (10,0) {$\dots$};
	\node[fill,circle,minimum size=1.2cm] (finish) at (16,0) [point] {$p_{q_2}$};
	
	\draw [->] (start) to node[midway, above]{
		$\begin{gathered}
		\TT_1.Q()
		\end{gathered}$
	} node[midway, below]{$e(1)$}(s1);
	
	\draw [->] (s1) to node[midway, above]{
		$\begin{gathered}
		\TT_2.Q()
		\end{gathered}$
	} node[midway, below]{$e(2)$}(s2);
	
	\draw [->] (s3) to node[midway, above]{
		$\begin{gathered}
		\TT_k.Q()
		\end{gathered}$
	} node[midway, below]{$e(k)$}(finish);
	
	\draw[->] (s2) to (q1);
	
	\draw[->] (q2) to (s3);
	
\end{scope}

%
%
%

\end{tikzpicture}	

%% file: WACCEPS.tex
\begin{tikzpicture} 
[point/.style={inner sep = 4pt, circle,draw,fill=white},
mpoint/.style={inner sep = 1.7pt, circle,draw,fill=black},
FIT/.style args = {#1}{rounded rectangle, draw,  fit=#1, rotate fit=45, yscale=0.5},
FITR/.style args = {#1}{rounded rectangle, draw,  fit=#1, rotate fit=-45, yscale=0.5},
FIT1/.style args = {#1}{rounded rectangle, draw,  fit=#1, rotate fit=45, scale=2},
vecArrow/.style={
		thick, decoration={markings,mark=at position
		   1 with {\arrow[thick]{open triangle 60}}},
		   double distance=1.4pt, shorten >= 5.5pt,
		   preaction = {decorate},
		   postaction = {draw,line width=0.4pt, white,shorten >= 4.5pt}
	},
myptr/.style={decoration={markings,mark=at position 1 with %
    {\arrow[scale=3,>=stealth]{>}}},postaction={decorate}}
]

\begin{scope} [yshift = 0.2cm, xscale=0.85, yscale=0.4]    
	\node [anchor=west] at (-1, 4.5) {
	\begin{tabular}{l} 
		Object $\wacc^\eps(\textbf{a})$; \\
		\ \ \ $\TT_j$ is $\TTape_{\T,i_j}^\eps[\I_{j,\eps}]$ for $j\in\overline{k}$
		\end{tabular}};  
	\node (start) at (0,-1) [point] {$s$};
	\node (q1) at (2.75,-1) [point] {$q$};
	\node (q2) at (6,-1) [point] {$q'$};
	\node (ps) at (9.25,-1) [point] {$p_{s_\T}$};
	\node (pf) at (14.25,-1) [point] {$p_{f_\T}$};
	\node (finish) at (17, -1) [point] {$f$};
	
	\draw [->] (start) to node[midway, above]{
		$\begin{gathered}
		[\mathbf{b}=1] \\[-1ex]
		\I_{j,\eps}=1 \\[-1ex]
		\text{for } j\in\overline{k}
		\end{gathered}$
	} node[midway, below]{$e_s$} (q1);
	
	\draw [->] (q1) to node[midway, above]{
		$\swap^A(\mathbf{a},\mathbf{b})$
	} node[midway, below]{$e_\swap$} (q2);
	
	\draw [->] (q2) to node[midway, above]{
		$\TT_1.\set(\mathbf{b})$
	} node[midway,below]{$e_\set$} (ps);
	
	\draw [->] (pf) to node[midway,above]{
		$\begin{gathered}
		[\mathbf{b}=1] \\[-1ex]
		\I_{j,\eps}=1 \\[-1ex]
		\text{for } j\in\overline{k}
		\end{gathered}$
	} node[midway,below]{$e_f$} (finish);
	
	\node at (11.75,-1) {\Large $\G_\T'$};
	
	\node (ps1) at (11.25,1) {};
	\node (ps2) at (11.25,-1) {};
	\node (ps3) at (11.25,-3) {};
	\node (ps4) at (11.25,0) {};
	\node (ps5) at (11.25,-2) {};
	
	\draw [-] (ps) to (ps1);
	\draw [-] (ps) to (ps2);
	\draw [-] (ps) to (ps3);
	\draw [-] (ps) to (ps4);
	\draw [-] (ps) to (ps5);
	
	\node (pf1) at (12.25,1) {};
	\node (pf2) at (12.25,-1) {};
	\node (pf3) at (12.25,-3) {};
	\node (pf4) at (12.25,0) {};
	\node (pf5) at (12.25,-2) {};
	
	\draw [-] (pf) to (pf1);
	\draw [-] (pf) to (pf2);
	\draw [-] (pf) to (pf3);
	\draw [-] (pf) to (pf4);
	\draw [-] (pf) to (pf5);
	
\end{scope}

\end{tikzpicture}	